%% file: Mem-BP.tex
\documentclass{mem-mod}

\usepackage{amssymb,mathtools,xypic}
\usepackage{enumerate,xspace,stmaryrd}
\usepackage{bbm} 
\usepackage[hang,flushmargin]{footmisc}

\usepackage[original]{imakeidx}
\makeindex[name=defs,title={Index of terminology}]
\makeindex[name=symbols,columns=3,title={List of notations}]

\counterwithout{figure}{chapter}

\usepackage[colorlinks=true, citecolor=blue]{hyperref}

\newcommand\Gm{\Gamma}
\newcommand\gm{\gamma}
\newcommand\eps{\varepsilon}
\newcommand\Om{\Omega}
\newcommand\om{\omega}
\newcommand\tGm{{\tilde\Gm}}

\newcommand\CC{\mathbb{C}}
\newcommand\NN{\mathbb{N}}
\newcommand\QQ{\mathbb{Q}}
\newcommand\RR{\mathbb{R}}
\newcommand\ZZ{\mathbb{Z}}
\newcommand\DD{\mathbb{D}}

\DeclareMathOperator{\Rea}{Re} 
\DeclareMathOperator{\Ima}{Im}
\DeclareMathOperator{\image}{im}
\newcommand\oh{\mathrm{O}}
\newcommand\Gf{\Gamma} 
\newcommand{\mc}{\mathcal}
\newcommand{\wt}{\widetilde}
\newcommand{\wh}{\widehat}
\newcommand{\setmid}{\;:\;}
\newcommand{\ceqq}{\,=\,}
\DeclareMathOperator{\pr}{pr}

\newcommand\uhp{{\mathfrak H}}
\newcommand{\shp}{S\uhp}
\newcommand\lhp{{\mathfrak H}^-}
\newcommand\proj[1]{{\mathbb{P}^1_{#1}}}
\newcommand\fd{{\mathfrak{F}}}

\newcommand{\fixTS}{\theta}

\DeclareMathOperator{\Fct}{Fct} 
\DeclareMathOperator{\Op}{Op}
\DeclareMathOperator{\Hom}{Hom}

\DeclareMathOperator{\SL}{SL}
\DeclareMathOperator{\PSL}{PSL}
\DeclareMathOperator{\PGL}{PGL}
\DeclareMathOperator{\GL}{GL}

\newcommand{\mat}[4]{\begin{pmatrix} #1&#2\\#3&#4\end{pmatrix}}
\newcommand{\textmat}[4]{\left(\begin{smallmatrix} #1&#2 \\ #3&#4
\end{smallmatrix}\right)}
\newcommand{\bmat}[4]{\begin{bmatrix} #1&#2\\#3&#4\end{bmatrix}}
\newcommand{\textbmat}[4]{\left[\begin{smallmatrix} #1&#2 \\ #3&#4
\end{smallmatrix}\right]}

\DeclareMathOperator{\id}{id}
\newcommand{\one}{\mathbbm{1}}
\newcommand{\hyp}{{\rm h}}
\DeclareMathOperator{\Tr}{Tr}

\DeclareMathOperator{\TO}{\mathcal L}
\newcommand{\fast}{\text{\rm fast}}
\newcommand{\slow}{\mathrm{slow}} 

\newcommand\tro[1]{\TO^\slow_{#1}}
\newcommand\ftro[1]{\TO^\fast_{#1}}
\newcommand\ftroP[1]{\TO^{\mathrm P}_{#1}}

\newcommand{\Perd}{\,\mc P\hspace*{-0.5mm}}
\newcommand{\Per}{\,\mc P}
\DeclareMathOperator{\Pert}{Per}

\newcommand\av[1]{\mathrm{Av}_{\!#1}}

\newcommand\E{{\mathcal E}}
\newcommand\A{{\mathcal A}}
\DeclareMathOperator{\MCF}{MCF}

\newcommand\FE{{\mathsf{FE}}}
\newcommand\BFE{{\mathsf{BFE}}}

\newcommand\G[2]{{\mathcal{G}_{\!#1}^{#2}}}
\newcommand\N[2]{{\mathcal{N}_{\!#1}^{#2}}}
\newcommand\V[2]{{\mathcal{V}_{\!#1}^{#2}}}
\newcommand\W[2]{{\mathcal{W}_{\!#1}^{#2}}}

\newcommand{\spaceext}{\mc V_{s}^{\omega,\mathrm{ext}}(D_\RR)}
 
\newcommand{\cond}{\mathrm{cd}}
\newcommand{\van}{\mathrm{van}} 
\newcommand{\tsic}{\textsl{sic}\xspace}  
\newcommand{\tvan}{\textsl{van}\xspace} 
\newcommand{\tsmp}{\textsl{smp}\xspace}  
\newcommand{\taj}{\textsl{aj}\xspace} 
\newcommand{\texc}{\textsl{exc}\xspace}  
\newcommand{\tcond}{\textsl{cd}\xspace}
\newcommand{\teext}{\textsl{ext}\xspace}
\newcommand\fs{{\om^\ast}}
\newcommand\fxi{{\om(\Xi)}}
\newcommand\smp{{\mathrm{smp}}}
\newcommand\aj{{\mathrm{aj}}}
\newcommand\exc{{\mathrm{exc}}}
\newcommand\mext{\mathrm{ext}}
\newcommand\extm{\mathrm{ext}}
\newcommand\sic{\mathrm{sic}}

\newcommand\bsing{{\mathsf{bdSing}\,}}
\newcommand\sing[1]{{\mathsf{Sing}_{#1}}}

\newcommand\coh{{\mathbf{r}}}
\newcommand\qcoh{{\mathbf{q}}}
\newcommand\cocu{{\mathbf{c}}}
\newcommand\ucoc{{\mathbf{u}}}
\newcommand\alphu{{\mathbf{t}}}

\newcommand\tess{\mathcal{T}}

\newcommand\pc{{\mathbf{pc}}}

\numberwithin{equation}{section}

\newcounter{sectold}

\newtheorem{thm}{Theorem}[section]
\newtheorem*{lem*}{Lemma}
\newtheorem{mainthm}{Theorem}

\newtheorem{mainthmprim}{Theorem}

\newtheorem{prop}[thm]{Proposition}
\newtheorem{cor}[thm]{Corollary}
\newtheorem{lem}[thm]{Lemma}

\theoremstyle{definition}
\newtheorem{defn}[thm]{Definition}
\theoremstyle{remark}

\newtheorem{rmk}[thm]{Remark}

\begin{document}

\frontmatter

\title[Transfer operators and automorphic forms]{Eigenfunctions of transfer
operators and automorphic forms for Hecke triangle groups of infinite
covolume}

\author{Roelof Bruggeman}
\address{Roelof Bruggeman, Mathematisch Instituut, Universiteit Utrecht, Postbus 80010, 3508 TA\ Utrecht, Nederland}
\email{r.w.bruggeman@uu.nl}

\author{Anke Pohl}
\address{Anke Pohl, University of Bremen, Department 3 -- Mathematics, Bibliothekstr.\@ 5, 28359 Bremen, Germany}
\email{apohl@uni-bremen.de}
\thanks{Funded by the Deutsche Forschungsgemeinschaft (DFG, German Research Foundation) -- project no.~264148330 and PO~1483/2-1.}

\date{}



\begin{abstract}
We develop cohomological interpretations for several types of automorphic forms for Hecke triangle groups of infinite covolume. We then use these interpretations to establish explicit isomorphisms between spaces of automorphic forms, cohomology spaces and spaces of eigenfunctions of transfer operators. These results show a deep relation between spectral entities of Hecke surfaces of infinite volume and the dynamics of their geodesic flows.
\end{abstract}

\maketitle

\tableofcontents

\mainmatter

\include{Mem-BP-intro}

\include{Mem-BP-partI}

\include{Mem-BP-partII}

\include{Mem-BP-partIII}

\include{Mem-BP-partIV}

\include{Mem-BP-partV}
\include{Mem-BP-partVI}
\include{Mem-BP-partVII}

\appendix

\backmatter
\providecommand{\MR}{\relax\ifhmode\unskip\space\fi MR }
\providecommand{\MRhref}[2]{%
  \href{http://www.ams.org/mathscinet-getitem?mr=#1}{#2}
}
\providecommand{\href}[2]{#2}


%
\printindex[defs]
\markboth{INDEX OF TERMINOLOGY}{INDEX OF TERMINOLOGY} 
\printindex[symbols]
\markboth{LIST OF NOTATIONS}{LIST OF NOTATIONS} 

\end{document}

%% file: Mem-BP-intro.tex

\markboth{1. INTRODUCTION}{1. INTRODUCTION}
\section{Introduction}\label{sec:intro}

\subsection*{Motivational background} 

For various classes of geometrically finite hyperbolic surfaces\footnote{For brevity, we refer throughout to $\Gamma\backslash\uhp$ as a hyperbolic surface also in the case that the fundamental group~$\Gamma$ is not torsion-free and hence $\Gamma\backslash\uhp$ has singularity points and should more correctly be called an orbifold.} $X=\Gamma\backslash\uhp$, the resonances of the Laplace--Beltrami operator~$\Delta_X$ on~$X$ (i.\,e., the spectral parameters of the  $L^2$-eigenfunctions, and the scattering resonances of~$\Delta_X$) are characterized as the `non-topological' zeros of the Selberg zeta function~$Z_X$ for~$X$. As is well-known, the Selberg zeta function is a generating function for the geodesic length spectrum of $X$. We refer to Section~\ref{sec:SZF} for more details. Since the introduction of this zeta function by Selberg in the year~1956 the characterization of resonances as zeros of $Z_X$ has been crucial for many results on the existence, distribution and localization of resonances and Laplace eigenvalues. One classical example of such a result---proven by Selberg himself---is the existence of an abundance of Maass cusp forms for the modular surface~$X_{\text{mod}}=\PSL_2(\ZZ)\backslash\uhp$ and certain other arithmetic surfaces, resulting in Weyl's law for their spectral parameters. 

In addition to being of immediate practical use, the Selberg zeta function yields a deep connection between spectral entities of~$X$, namely the resonances, and geometric-dynamical entities of~$X$, namely the lengths of periodic geodesics. From the point of view of physics, the Selberg zeta function establishes a relation between quantum and classical mechanics, and thereby contributes to the understanding of Bohr's correspondence principle.

For certain classes of hyperbolic surfaces~$X=\Gamma\backslash\uhp$ of \emph{finite area} an even deeper relation between certain spectral objects of~$X$, namely the Maass cusp forms, and the periodic geodesics of~$X$ is known: By using well-chosen discretizations of the geodesic flow on~$X$ 
and carefully designed cohomological interpretations of automorphic forms and functions, Maass cusp forms (and occasionally also some other automorphic functions) can be characterized as eigenfunctions of certain operators---known as \emph{transfer operators}---that derive purely from the dynamics of the periodic geodesics on~$X$. 
As a by-product, this approach provides a notion of period functions for Maass cusp forms that is purely based on dynamics and transfer operators. 

This transfer-operator-based characterization of Maass cusp forms connects the forms themselves, not only their spectral parameters, with the dynamics on the hyperbolic surface~$X$, and hence establishes a link between spectral and dynamical properties of~$X$ that is deeper than the one provided by the Selberg zeta function.  In what follows we briefly survey the existing results. 

The seminal result of this kind was for the modular surface~$X_{\text{mod}}=\PSL_2(\ZZ)\backslash\uhp$, by combination of \cite{Mayer_thermo, Mayer_thermoPSL, Bruggeman_lewiseq, Chang_Mayer_transop, LZ01} and taking advantage of a discretization of the geodesic flow that can be found in~\cite{Artin, Series}. We present this result in more detail. For the geodesic flow on~$X_{\text{mod}}$, there exist a discretization and a symbolic dynamics that are closely related to the Gauss map
\[
 F\colon [0,1]\smallsetminus\QQ \to [0,1]\smallsetminus\QQ\,,\quad x\mapsto \tfrac1x \mod 1\,,
\]
as was noted, for example, by E.~Artin~\cite{Artin}. See also Series's more detailed and geometric presentation~\cite{Series}. Mayer \cite{Mayer_thermo,Mayer_thermoPSL} studied the associated transfer operator~$\TO_{F,s}$ with parameter~$s\in\CC$, $\Rea s\gg 1$. This is the operator
\begin{equation}\label{TO_Mayer}
 \TO_{F,s} f(x) = \sum_{n\in\NN} (x+n)^{-2s} f\left( \frac{1}{x+n}\right)\,,
\end{equation}
acting on appropriate spaces of functions~$f$ defined on (a complex neighborhood of) the interval~$[0,1]$. He found a Banach space~$\mc B$ on which, for $\Rea s > 1/2$, the operator~$\TO_{F,s}$ acts as a self-map, is nuclear of order zero, and the map $s\mapsto\TO_{F,s}$ has a meromorphic extension to all of~$\CC$ (whose images for $\Rea s \leq \frac12$ we continue to denote by~$\TO_{F,s}$). The Fredholm determinant of the transfer operator family~$\TO_{F,s}$, $s\in\CC$, determines the Selberg zeta function~$Z_{X_{\text{mod}}}$ via
\begin{equation}\label{SZF_modular}
 Z_{X_{\text{mod}}}(s) = \det\left(1-\TO_{F,s}\right)\det\left(1+\TO_{F,s}\right)\,.
\end{equation}
Thus, the zeros of~$Z_{X_{\text{mod}}}$ are determined by the eigenfunctions with eigenvalue~$\pm 1$ of~$\TO_{F,s}$ in~$\mc B$. Since certain of the zeros of~$Z_{X_{\text{mod}}}$ are identical to the spectral parameters of the Maass cusp forms for~$\PSL_2(\ZZ)$, the natural question on the explicit relation between eigenfunctions of~$\TO_{F,s}$ and Maass cusp forms arose. 

Both Lewis--Zagier \cite{LZ01} and Chang--Mayer \cite{Chang_Mayer_transop} showed that, for any spectral parameter~$s\in\CC$ with $\Rea s\in (0,1)$, the vector space spanned by the $\pm1$-eigen\-func\-tions of~$\TO_{F,s}$ is linear isomorphic to the space of highly regular solutions of the functional equation
\begin{equation}\label{modular_funceq}
 f(x) = f(x+1) + (x+1)^{-2s} f\left(\frac{x}{x+1}\right)\,,\qquad x\in\RR_{>0}\,,
\end{equation}
where the $+1$-eigenfunctions of~$\TO_{F,s}$ correspond to those solutions that satisfy in addition $f(x)=f(1/x)$, and the $-1$-eigenfunctions of~$\TO_{F,s}$ correspond to the solutions satisfying $f(x)=-f(1/x)$. By Lewis--Zagier~\cite{LZ01} (using $L$-series, Mellin transform, and Fourier expansions) and, alternatively, by Bruggeman~\cite{Bruggeman_lewiseq} (using hyperfunction cohomology), these highly regular solutions of the functional equation~\eqref{modular_funceq} are isomorphic to the Maass cusp forms for~$\PSL_2(\ZZ)$ with spectral parameter~$s$. 

This isomorphism between Maass cusp forms and highly-regular solutions of the functional equation in~\eqref{modular_funceq} can be expressed in several ways. One realization---that will be of interest for our work---is by a certain \emph{integral transform} that is given by integrating the Maass cusp form against a Poisson-type kernel in the Green's form. (This integral transform is closely related to the Poisson transformation.) In analogy to the concept of period polynomials in Eichler--Shimura theory, the solutions of~\eqref{modular_funceq} are therefore called \emph{period functions}. 

These two isomorphisms combined provide, for $\Rea s \in (0,1)$, an isomorphism between the sum of the $\pm1$-eigenspaces of~$\TO_{F,s}$ and the space~$\MCF_s\bigl(\PSL_2(\ZZ)\bigr)$ of the Maass cusp forms for~$\PSL_2(\ZZ)$ with spectral parameter~$s$:
\begin{equation}\label{iso_TO_MCF}
 \left\{ f\in\mc B \setmid \TO_{F,s}f = f\right\} \oplus \left\{ f\in\mc B \setmid \TO_{F,s}f=-f\right\} \;\cong\; \MCF_s\bigl(\PSL_2(\ZZ)\bigr)\,.
\end{equation}

We emphasize that the isomorphism in~\eqref{iso_TO_MCF} does not make use of the equality in~\eqref{SZF_modular}, which we included above for completeness of the historical line of thoughts. Again for the purpose of completeness we remark that in combination with the Selberg trace formula, the equality in~\eqref{SZF_modular} implies that the sum of the dimensions of the Jordan blocks of~$\TO_{F,s}$ with eigenvalues~$\pm 1$ equals the dimension of the space of Maass cusp forms with spectral parameter~$s$. Taking advantage of tools from number theory and harmonic analysis, then also a precise isomorphism map between the Maass cusp forms and the eigenfunctions of transfer operators can be found (at least in the situation of the modular surface~$X_\text{mod}$). Within the transfer operator approach, however, such an isomorphism can be given without relying on the Selberg trace formula, and this isomorphism has a clear geometric motivation. (See the discussion in Section~\ref{sec:heuristic}.)

This transfer operator approach for the modular surface~$X_{\text{mod}}$ allows dynamical characterizations also for certain other eigenfunctions of the Laplacian \cite{Chang_Mayer_period, Chang_Mayer_transop, LZ01}, and the factorization in~\eqref{SZF_modular} could be explained \cite{Efrat_spectral, Bruggeman_lewiseq, LZ01, Moeller_Pohl, Pohl_spectral_hecke}. An alternative transfer-operator-based characterization of the Maass cusp forms is provided by  the combination of~\cite{Mayer_Stroemberg, BM09, Mayer_Muehlenbruch_Stroemberg}. This approach proceeds from a discretization of the geodesic flow giving rise to a transfer operator with properties similar to the one in~\eqref{TO_Mayer}. (See below for a further, more recent alternative that is of a different nature and shows additional features.) Taking advantage of twists by finite-dimensional representations these results could be extended to certain finite covers of the modular surface \cite{Chang_Mayer_transop, Deitmar_Hilgert, Fraczek_Mayer_Muehlenbruch}. All of these transfer operator approaches have in common that the transfer operators contain \emph{infinite sums} (as in~\eqref{TO_Mayer}) and hence their eigenfunctions with eigenvalue~$1$ are not automatically seen to satisfy a finite-term functional equation (as, e.\,g., \eqref{modular_funceq}).

With~\cite{Pohl_Symdyn2d}, Pohl provided a construction of discretizations and symbolic dynamics for the geodesic flow on a rather large class of hyperbolic surfaces~$\Gamma\backslash\uhp$ (of finite or infinite area) for which the arising transfer operators are \emph{sums} of only \emph{finitely many terms}. Further, for all hyperbolic surfaces~$\Gamma\backslash\uhp$ of \emph{finite area}, Bruggeman--Lewis--Zagier~\cite{BLZm} provided a characterization of the Maass cusp forms of~$\Gamma$ as cocycles in certain parabolic $1$-cohomology spaces. 

For all those Fuchsian groups~$\Gamma$ that are both cofinite and admissible for the construction in~\cite{Pohl_Symdyn2d}, the transfer operator eigenfunctions with eigenvalue~$1$ serve as building blocks for the parabolic $1$-cocycles. The eigenspaces with eigenvalue~$1$ of the transfer operators are then seen to be isomorphic to the parabolic $1$-cohomology spaces \cite{Pohl_mcf_general, Pohl_mcf_Gamma0p, Moeller_Pohl}. In combination with~\cite{BLZm}, it follows that the space of these transfer operator eigenfunctions is thus isomorphic to the space of Maass cusp forms for~$\Gamma$. As in the case of the modular group, for any such~$\Gamma$, the isomorphism is given by an (explicit) integral transform. The defining equation (rather, system of equations) of these transfer operator eigenfunctions consists of finitely many terms, by the very choice of the discretization of the geodesic flow. Therefore, these eigenfunctions constitute (dynamically defined) period functions for the Maass cusp forms. We refer to the survey~\cite{Pohl_Zagier} for a rather informal presentation of this type of transfer operator approaches to Maass cusp forms.

For the modular surface, this construction recovers the functional equation in~\eqref{modular_funceq} without a detour via an infinite-term transfer operator. However, an extension of this construction finds again Mayer's transfer operator~\eqref{TO_Mayer} as well as a dynamical explanation of the factorization in~\eqref{SZF_modular} (alternative to Efrat's approach~\cite{Efrat_spectral}) and separated period functions and functional equations for odd and even Maass cusp forms, recovering Lewis' equation in~\cite{Lewis}. See~\cite{Moeller_Pohl,Pohl_spectral_hecke}.

\subsection*{Aim of this article}

The construction of the discretizations of the geodesic flow in~\cite{Pohl_Symdyn2d} and of slow transfer operators applies to (certain) hyperbolic surfaces of \emph{infinite area} as well. Therefore it is a natural question to which extent such dynamical and transfer-operator-based characterizations of automorphic forms are possible for hyperbolic surfaces of \emph{infinite area}. With this paper we initiate the investigation of the realm of such characterizations. 

Since hyperbolic surfaces~$X=\Gamma\backslash\uhp$ of infinite area have no embedded $L^2$-eigen\-val\-ues, and hence their pure point spectrum is finite, the major bulk of the interesting spectral values is given by the scattering resonances. For this reason we will consider also automorphic forms that are more general than Maass cusp forms, namely the set of \emph{funnel forms}, its subset of \emph{resonant funnel forms} and its even more restrictive subset of \emph{cuspidal funnel forms}. All these automorphic forms are characterized by natural conditions on their growth behavior at the cusps and funnels. The essence of their properties is sketched in what follows.

The set~$\A_s(\Gamma)$ of \emph{funnel forms} for~$\Gamma$ with spectral parameter~$s$ consists of the \mbox{$\Gamma$-}invariant Laplace eigenfunctions with eigenvalue~$s(1-s)$ that have an \emph{$s$-analytic boundary behavior} near all funnels. See Section~\ref{sect-daf} for more details. The $s$-analytic boundary behavior at funnels is the minimal property of Laplace eigenfunctions to allow for an identification with (equivalence classes) of complex period functions. Funnel forms may have large growth towards any cusp.

Its subset~$\A_s^1(\Gamma)$ of \emph{resonant funnel forms} consists of those funnel forms that---along any geodesic $\gamma$ going to a cusp of~$X$---behave like
\[
 p\,\gamma(t)^{1-s} + O\big(e^{-c\gamma(t)}\big)\qquad\text{as $t\to\infty$}
\]
for suitable $p\in\CC$, $c>0$. We remark that while the behavior at funnels is related to the spectral parameter~$s$, the behavior towards cusps is related to the opposite spectral parameter~$1-s$. The subset~$\A_s^0(\Gamma)$ of \emph{cuspidal funnel forms} are those resonant funnel forms that have exponential decay in any cusp. We refer to Section~\ref{sect-daf} for precise definitions.

In this paper we restrict the discussion to Hecke triangle groups of infinite covolume. This allows us to provide explicit formulas and calculations, and to avoid bookkeeping of several orbits of cuspidal points and funnel representatives. Moreover, we impose the restriction $\Rea s\in (0,1)$ on the spectral parameters~$s\in\CC$. Due to the latter restriction we can build on several results on parabolic cohomology from~\cite{BLZm} for the necessary development of cohomological interpretations of automorphic forms. Otherwise these constructions would have been needed to be extended to~$s$ running through all of~$\CC$ and would have resulted in a much longer treatise. We hope to return to such generalizations in a future paper. 

For the remainder of this introduction let $\Gamma\coloneqq \Gamma_\lambda$ denote the Hecke triangle group with cusp width~$\lambda>2$ (which implies that $\Gamma_\lambda$ has infinite covolume). As a subgroup of~$\PSL_2(\RR)$, it is generated by the two elements
\[
 S \coloneqq \bmat{0}{1}{-1}{0} \quad\text{and}\quad T_\lambda\coloneqq \bmat{1}{\lambda}{0}{1}\,.
\]
In \cite{Pohl_hecke_infinite}, Pohl provided so-called slow and fast transfer operators for $\Gamma$. The \emph{slow transfer operator}~$\tro s$ of~$\Gamma$ with parameter~$s\in\CC$ acts on vectors of functions
\begin{equation}\label{functionvector}
 f = \begin{pmatrix} f_1 \colon (-1,\infty)\to\CC \\ f_2\colon (-\infty,1)\to\CC \end{pmatrix}
\end{equation}
by
\begin{align*}
 \left( \tro s f\right)_1(x) & = (\lambda+x)^{-2s} f_1 \left(\frac{-1}{\lambda+x}\right) + f_1(x+\lambda) + (\lambda+x)^{-2s}f_2\left(\frac{-1}{\lambda+x}\right)\,,
 \\
 \left( \tro s f\right)_2(x) & = (\lambda-x)^{-2s} f_1\left(\frac{1}{\lambda-x}\right) + f_2(x-\lambda) + (\lambda-x)^{-2s}f_2\left(\frac{1}{\lambda-x}\right)\,.
\end{align*}
The notion and essence of a slow transfer operator (as opposed to a fast transfer operator) will be  explained in Section~\ref{sect-to}. 

We let $\FE_s^\omega(\CC)$ denote the space of the $1$-eigenfunctions~$f=(f_1,f_2)$ of~$\tro s$ (i.\,e., $f=\tro s f$) for which $f_1$ and~$f_2$ extend holomorphically to~$\CC\smallsetminus (-\infty,-1]$ and~$\CC\smallsetminus [1,\infty)$, respectively. In other words, these functions are the holomorphic solutions to the functional equation~$f=\tro s f$ with a large domain of definition. Further let $\BFE_s^\omega(\CC)$ be the subspace of~$\FE_s^\omega(\CC)$ consisting of the elements of the form~$(-b,b)$ for some entire~$\lambda$-periodic function~$b$. We call $\FE_s^\omega(\CC)$ the space of \emph{period functions} and $\BFE_s^\omega(\CC)$ the space of \emph{boundary period functions}, justified by Theorems~\ref{thmA_new} and \ref{thmB_new} below. The spaces~$\FE_s^\omega(\CC)$ and $\BFE_s^\omega(\CC)$ allow us to characterize the funnel forms~$\A_s(\Gamma)$.

\begin{mainthm}\label{thmA_new}
For $s\in\CC$, $\Rea s \in (0,1)$, $s\neq\frac12$, there is a surjective linear map
\[ 
A_s\colon \FE^\om_s(\CC) \rightarrow \A_s(\Gm) 
\]
from the space of period functions~$\FE^\om_s(\CC)$ to the space of
funnel forms~$\A_s(\Gm)$. The map~$A_s$ descends to an isomorphism
\[
 \FE^\omega_s(\CC)/\BFE_s^\omega(\CC) \to \A_s(\Gamma)\,.
\]
\end{mainthm} 
In a nutshell, this correspondence identifies a funnel form~$u\in\nobreak\A_s(\Gamma)$ with a period function~$f=(f_1,f_2)\in \FE_s^\omega(\CC)$ if 
\begin{align*}
 f_1(t) &= \int_{-1}^\infty \{ u, R(t;\cdot)^s \} \qquad\text{on $(-1,\infty)$}
\intertext{and}
 f_2(t) &= \int_1^\infty \{u, R(t;\cdot)^s \} \qquad\text{on $(-\infty,1)$\,.}
\end{align*}
Here $\{\cdot,\cdot\}$ denotes the Green's form on~$C^\infty(\uhp)$ and $R(\cdot;\cdot)$ the Poisson-type kernel. (See Section~\ref{sect-coaief} for details.) The integration is along any path in~$\uhp$ with endpoints as indicated. In this way we find a linear map from funnel forms to period functions, subject to some corrections due to problems of convergence of the integrals. (In the case of non-convergence, a regularization is possible.) A natural map~$A_s$ in the opposite direction is seen to have a non-trivial kernel that is precisely~$\BFE_s^\om(\CC)$. We refer to Chapter~\ref{part:autom} for precise statements.

To identify which of the period functions arise from resonant and cuspidal funnel forms we will
need to understand how the growth properties of funnel forms at the cusp of a Hecke triangle group are characterized in terms of properties of the period functions. For that we will take advantage of a certain \emph{fast transfer operator}~$\ftro s$. This fast transfer operator is not precisely the one provided in~\cite{Pohl_hecke_infinite} but can easily be deduced from it. As any fast transfer operator, also this one arises from a certain induction process on parabolic elements in the slow transfer operator~$\tro s$, and hence it captures very well the growth behavior at the cusp. See Section~\ref{sect-to} for the detailed construction of~$\ftro s$ and its properties.

The operator~$\ftro s$ acts on functions~$f$ of the form as in~\eqref{functionvector} by
\begin{align*}
 \left(\ftro s f\right)_1(x) & = \sum_{n\geq 1} \frac{1}{(n\lambda +x)^{2s}} \left(f_1+f_2\right)\left(\frac{-1}{n\lambda+x}\right)
 \intertext{and}
 \left(\ftro s f\right)_2(x) & = \sum_{n\geq 1} \frac{1}{(n\lambda -x)^{2s}} \left(f_1+f_2\right)\left(\frac{1}{n\lambda-x}\right)\,.
\end{align*}
We let
\[
 \FE_s^{\omega,1}(\CC) \coloneqq  \left\{ f\in \FE_s^\omega(\CC) \setmid f = \ftro s f\right\}
\]
denote the set of period functions that are also $1$-eigenfunctions of the fast transfer operator, and we set
\[
 \FE_s^{\omega,0}(\CC) \coloneqq  \left\{ f=(f_1,f_2)\in \FE_s^{\omega,1}(\CC) \setmid f_1(0) = -f_2(0) \right\}\,.
\]
The boundary period functions contained in~$\FE_s^{\omega,1}(\CC)$ or $\FE_s^{\omega,0}(\CC)$ are trivial (see Proposition~\ref{prop:coboundaries}):
\[
 \FE_s^{\omega,1}(\CC)\cap \BFE_s^\omega(\CC) = \FE_s^{\omega,0}(\CC)\cap\BFE_s^\omega(\CC) = \{0\}\,.
\]
The spaces~$\FE_s^{\omega,1}(\CC)$ and~$\FE_s^{\omega,0}(\CC)$ allow us to characterize the spaces of resonant funnel forms~$\A_s^1(\Gamma)$ and cuspidal funnel forms~$\A_s^0(\Gamma)$, respectively.

\begin{mainthm}\label{thmB_new}
Let $s\in \CC$, $\Rea s\in (0,1)$. 
\begin{enumerate}[{\rm (i)}]
\item For~$s\not=\frac12$, the map~$A_s$ in Theorem~\ref{thmA_new} induces a bijective linear map
\[
 \FE_s^{\omega,1}(\CC) \to \A_s^1(\Gamma)\,.
\]
\item On~$\FE_s^{\omega,0}(\CC)$, the map~$A_s$ in Theorem~\ref{thmA_new} has a natural extension to $s=\frac12$. This extension yields a bijective linear map 
\[
 \FE_s^{\omega,0}(\CC) \to \A_s^0(\Gamma)
\]
for all~$s\in\CC$ with~$\Rea s\in (0,1)$.
\end{enumerate}
\end{mainthm}

As in the case of hyperbolic surfaces of finite area and dynamical characterizations of Maass cusp forms, we will prove Theorems~\ref{thmA_new} and~\ref{thmB_new} in a two-step approach with cohomology spaces as mediator between automorphic forms and eigenfunctions of the transfer operators.

We will first develop several cohomological interpretations of funnel forms, resonant funnel forms and cuspidal funnel forms. The cohomology spaces that we will use are not completely standard, and---due to the presence of a funnel---are more involved than the parabolic cohomology spaces from~\cite{BLZm} for finite area surfaces. Furthermore, in order to fully capture the behavior of the considered automorphic forms at cusps and funnels we will use subspaces of these cohomology spaces determined by conditions that appear to be non-cohomological but of geometric nature. 

A distinguished role is played by cocycles that are defined on the $\Gamma$-invariant set 
\[
 \Xi \coloneqq \Gamma\,1 \cup \Gamma\,\infty
\]
and that have values in carefully chosen $\Gamma$-modules~$M$ that depend on whether we aim at a cohomological interpretation of all funnel forms~$\A_s(\Gamma)$ or all resonant funnel forms~$\A_s^1(\Gamma)$ or all cuspidal funnel forms~$\A_s^0(\Gamma)$. The set~$\Xi$ consists of all representatives of the cusp of~$\Gamma\backslash\uhp$ and all boundaries of all connected representatives of a well-chosen kind of the funnel of~$\Gamma\backslash\uhp$. These are exactly the points which are crucial for distinguishing generic Laplace eigenfunctions from funnel forms, and generic funnel forms from resonant funnel forms, and generic resonant funnel forms from cuspidal funnel forms. The $\Gamma$-modules~$M$ consist of semi-analytic functions on $\proj\RR$ with additional regularity-like properties at (and near) the points in~$\Xi$.  Thus, for any $*\in \{\underbar{\ \ }, 0, 1\}$ and any~$s\in\CC$, $\Rea s\in (0,1)$ (and $s\not=\tfrac12$ for non-cuspidal funnel forms)  we will define $\Gamma$-modules $M(\A_s^*(\Gamma))$ and certain additional conditions $\cond(\A_s^*(\Gamma))$ on cohomology classes such that
\[
 \A_s^*(\Gamma) \cong H^1_\Xi\big(\Gamma; M(\A_s^*(\Gamma))\big)^{\cond(\A_s^*(\Gamma))}\,.
\]
In Chapter~\ref{part:autom} we will see that this isomorphism is rather explicit and constructive. 

Even though in this article we will use these cohomological interpretations of automorphic forms only to establish an isomorphism between funnel forms (and subspaces) to certain spaces of eigenfunctions of transfer operators, these intermediate cohomological results are clearly of independent interest. In particular they contribute to the increasing zoo of cohomological interpretations of automorphic functions and forms that started with the work by Eichler~\cite{Eichler} and Shimura~\cite{Shimura}, and is continued by, e.\,g., \cite{Knopp_cohom,  Bunke_Olbrich_1, Bunke_Olbrich_2, Bunke_Olbrich_annals, Deitmar_Hilgert_cohom, Knopp_Mawi, BLZm, BCD}.

After having established the cohomological interpretation of the funnel forms we will show that the eigenfunctions with eigenvalue~$1$ of the transfer operator~$\TO_s^\slow$ of sufficient regularity or the joint eigenfunctions of~$\TO_s^\slow$ and~$\TO_s^\fast$ serve as building blocks for well-chosen representatives of the cocycle classes, resulting in explicit and constructive isomorphisms 
\[
 \FE_s^\omega(\CC)/\BFE_s^\omega(\CC) \cong H^1_\Xi\big(\Gamma; M(\A_s(\Gamma))\big)^{\cond(\A_s(\Gamma))}
\]
and
\[
 \FE_s^{\omega,*}(\CC) \cong H^1_\Xi\big(\Gamma; M(\A_s^*(\Gamma))\big)^{\cond(\A_s^*(\Gamma))}
\]
for~$*\in\{0,1\}$. Here again we will see that the set~$\Xi$ plays a special role: It naturally emerges from the discretization of the geodesic flow (which is the basis for the transfer operators and hence for the sets~$\FE_s^{\om,*}(\CC)$).

The restriction to Hecke triangle groups allows us to establish also the following refinement of the isomorphisms between the eigenspaces of the transfer operators, the cohomology spaces and the automorphic forms. Hecke triangle groups enjoy an exterior symmetry which allows to separate even and odd automorphic forms. The transfer operators $\TO_s^\slow$ and $\TO_s^\fast$, the defining properties of period functions as well as the cohomology spaces are compatible with this symmetry. For that reason Theorems~\ref{thmA_new} and \ref{thmB_new} have rather straightforward extensions with respect to this parity. 

To state the refined result we add a `$+$' to the notation if we restrict to the objects invariant under the exterior symmetry, and a `$-$' if we restrict to the anti-invariant objects.

\begin{mainthm}\label{thmC}
The isomorphisms between period functions, cohomology spaces and funnel forms from above induce the following isomorphisms:
\begin{enumerate}[{\rm (i)}]
\item For $s\in\CC$, $\Rea s \in (0,1)$, $s\not=\tfrac12$, 
\[
 \FE_s^{\om,\pm}(\CC)/\BFE_s^{\om,\pm}(\CC) \cong H^{1,\mp}_\Xi\bigl(\Gamma; M(\A_s(\Gamma))\bigr)^{\cond(\A_s(\Gamma))} \cong \A_s^{\pm}(\Gamma)\,.
\]
\item For $s\in\CC$, $\Rea s\in (0,1)$, $s\not=\tfrac12$,
\[
 \FE_s^{\om,1,\pm}(\CC) \cong H^{1,\mp}_\Xi\bigl(\Gamma; M(\A_s^1(\Gamma))\bigr)^{\cond(\A_s^1(\Gamma))} \cong \A_s^{1,\pm}(\Gamma)\,.
\]
\item For $s\in\CC$, $\Rea s \in (0,1)$,
\[
 \FE_s^{\om,0,\pm}(\CC) \cong H^{1,\mp}_\Xi\bigl(\Gamma; M(\A_s^0(\Gamma))\bigr)^{\cond(\A_s^0(\Gamma))} \cong \A_s^{0,\pm}(\Gamma)\,.
\]
\end{enumerate}
\end{mainthm}

We provide a brief overview of the structure of this paper. It is divided into seven chapters. In Chapter~\ref{part:intro} we will provide the necessary background information on Hecke triangle groups, transfer operators, automorphic forms, period functions and related objects. After having introduced all these objects, we will provide---in Section~\ref{sec:heuristic}---an informal presentation of some insights about the isomorphisms and their constitutions.  In Chapter~\ref{part:cohom} we will discuss the cohomology spaces that we will use for the proofs of Theorems~\ref{thmA_new} and~\ref{thmB_new}. In Chapter~\ref{part:autom} we will establish isomorphisms between the spaces of (resonant and cuspidal) funnel forms and suitable cohomology spaces. In Chapter~\ref{part:TO} we will construct the isomorphisms between the spaces of eigenfunctions of the transfer operators and the cohomology spaces. In Chapter~\ref{part:proof_mainthms} we will combine all these results to proofs of Theorems~\ref{thmA_new} and~\ref{thmB_new}. In addition we will provide a brief summary of the constructions in Chapters~\ref{part:autom} and~\ref{part:TO}. In Chapter~\ref{part:parity} we will discuss the extensions with respect to parity and prove Theorem~\ref{thmC}. In the final Chapter~\ref{part:complements} we will  show a complementary result used for the motivation in Chapter~\ref{part:intro} and we will discuss future research directions.

\subsection*{Acknowledgement} RB thanks the University of Bremen for support, great hospitality and working conditions during a research stay. AP wishes to thank the Max Planck Institute for Mathematics in Bonn for great hospitality and excellent working conditions during part of the preparation of this manuscript. Both authors wish to thank the referees for helpful comments and suggestions.

%% file: Mem-BP-partI.tex

\setcounter{sectold}{\arabic{section}}

\chapter{Preliminaries, properties of period functions, and some insights}\label{part:intro}
\markboth{I. PRELIMINARIES}{1. PRELIMINARIES}

\setcounter{section}{\arabic{sectold}}

This chapter serves two purposes: 
\begin{enumerate}[1)]
 \item We introduce automorphic forms, transfer operators, and related objects for Hecke triangle groups of infinite co-area. After fixing some basic notations, reviewing the necessary background information on hyperbolic geometry, and surveying Hecke triangle groups in Sections~\ref{sec:notation}-\ref{sec:hecke}, we will define, in Section~\ref{sect-daf}, the notions of funnel forms, resonant funnel forms and cuspidal funnel forms. 
 
 For the transfer operators and the cohomology spaces that we will use for the proofs of Theorems~\ref{thmA_new} and \ref{thmB_new}, principal series representations are essential. We will recall these representations in Section~\ref{sect-psa}. In Section~\ref{sect-to} we will present the two families of slow and fast transfer operators families, define the notion of period functions, and discuss some of their first properties.
 \item We will present, in Section~\ref{sec:heuristic}, some insights and a motivation for the structure of the isomorphism between eigenfunctions of transfer operators and cocycles.
\end{enumerate}

\section{Notations}\label{sec:notation}\markright{2. NOTATIONS}

We set $\one\coloneqq \textmat{1}{0}{0}{1}$, and let 
\index[symbols]{P@$\PSL_2(\RR)$}
\[
 \PSL_2(\RR) \coloneqq  \SL_2(\RR)/\{\pm\one\}
\]
denote the projective group of $\SL_2(\RR)$. We denote the image of~$\textmat{a}{b}{c}{d}\in\SL_2(\RR)$ by
\[
 \bmat{a}{b}{c}{d}\,.
\]
\index[symbols]{$\textbmat{a}{b}{c}{d}$}
For any complex number~$z\in\CC$ we set $x = x(z) = \Rea z$ and $y = y(z) = \Ima z$ if not noted otherwise.\index[symbols]{x@$x(z)$}\index[symbols]{Y@$y(z)$} For~$r\in\RR$ we set\index[symbols]{R@$\RR_{\geq r}$}
\[
 \RR_{\geq r} \coloneqq  \{ x\in\RR \setmid x\geq r\}
\]
and\index[symbols]{C@$\CC_{\Rea \geq r}$}
\[
 \CC_{\Rea \geq r} \coloneqq  \{ z\in\CC\setmid \Rea z\geq r\}\,,
\]
and define analogously $\RR_{\leq r}$, $\RR_{>r}$, $\RR_{<r}$, $\CC_{\Rea \leq r}$, $\CC_{\Rea > r}$, $\CC_{\Rea < r}$. Further, $\NN = \{1,2,\ldots\}$ \index[symbols]{N@$\NN$} denotes the set of natural numbers (without~$0$), and $\NN_0 \coloneqq  \NN\cup\{0\}$. \index[symbols]{N@$\NN_0$}

\section{Elements from hyperbolic geometry}\label{sec:hypgeom}
\markright{3. HYPERBOLIC GEOMETRY}

\subsection{Models and isometries} 

For the hyperbolic plane,\index[defs]{hyperbolic plane} we will use almost exclusively the upper half plane model\index[symbols]{Haa@$\uhp$}\index[defs]{upper half plane model}
\[
\uhp\coloneqq  \{ z\in\CC \setmid \Ima z >0 \}\,,\quad ds_z^2 = \frac{dz\,d\overline{z}}{(\Ima z)^2}\,.
\]
For a very few figures we will use the disk model\index[symbols]{D@$\DD$}\index[defs]{disk model}
\[
 \DD \coloneqq  \{ w\in\CC \setmid |w|<1\}\,,\quad ds_w^2 = 4 \frac{dw\,d\overline{w}}{\left( 1 - |w|^2\right)^2}\,.
\]
The equivalence between these two models is given by any Cayley transform, e.\,g.,
\[
 \uhp\to\DD\,,\quad z\mapsto \frac{i-z}{i+z}\,.
\]
In what follows we present all additional objects only for the upper half plane model~$\uhp$. 

The (positive) Laplace--Beltrami operator~$\Delta$\index[symbols]{D@$\Delta$}\index[defs]{Laplace--Beltrami operator} on~$\uhp$ is 
\begin{equation}\label{Laplace}
 \Delta = -y^2\left( \partial_x^2 + \partial_y^2\right)\,.
\end{equation}
We identify the group of orientation-preserving Riemannian isometries\index[defs]{isometries of hyperbolic plane} of~$\uhp$ with the group~$\PSL_2(\RR)$. The action on~$\uhp$ is then given by fractional linear transformations, thus
\begin{equation}\label{fraclinaction}
 gz \coloneqq  \frac{az +b}{cz+d}
\end{equation}
for~$g=\textmat{a}{b}{c}{d}\in \PSL_2(\RR)$, $z\in\uhp$. This action extends continuously to all of the Riemann sphere\index[defs]{Riemann sphere}\index[defs]{complex projective line}\index[defs]{projective line!complex} (complex projective line)~$\proj\CC = \CC\cup\{\infty\}$.\index[symbols]{P@$\proj\CC$} It preserves the upper half plane~$\uhp$, the lower half plane\index[symbols]{Hab@$\uhp^-$}
\[
 \uhp^- \coloneqq  \{z\in\CC \setmid \Ima z<0\}\,,
\]
and the real projective line\index[defs]{real projective line}\index[defs]{projective line!real}\index[symbols]{P@$\proj\RR$}
\[
 \proj\RR = \RR \cup\{\infty\}\,. 
\]
We note that 
\[
 \proj\RR = \partial\uhp = \partial\uhp^-\,,
\]
where the boundaries of~$\uhp$ and~$\uhp^-$ are taken in~$\proj\CC$.\index[symbols]{D@$\partial\uhp$}\index[symbols]{D@$\partial\uhp^-$} The \index[defs]{boundary of~$\uhp$} boundary of~$\uhp$ in~$\proj\CC$ coincides with the geodesic boundary of~$\uhp$. We let
\[
 \overline\uhp = \uhp \cup \proj\RR
\]
denote the \index[defs]{closure of~$\uhp$}\index[symbols]{Hac@$\overline\uhp$} closure of~$\uhp$ in~$\proj\CC$.

\subsection{Classification of isometries}\label{sec:hypclass}

We call an element~$g\in\PSL_2(\RR)$ \emph{hyperbolic}\index[defs]{hyperbolic element}\index[defs]{element!hyperbolic} if the action of~$g$ on~$\overline{\uhp}$ has exactly two fixed points. We call it \emph{elliptic}\index[defs]{elliptic element}\index[defs]{element!elliptic} if it has a single fixed point in~$\uhp$; we call it \emph{parabolic}\index[defs]{parabolic element}\index[defs]{element!parabolic} if it has a single fixed point in~$\partial\uhp$. Equivalently, $g$ is hyperbolic, elliptic or parabolic if and only if~$g\not=\id$ and, for any representative~$h\in\SL_2(\RR)$ of~$g$, we have $|\Tr h|>2$, $|\Tr h|<2$ or $|\Tr h|=2$, respectively.

\subsection{Cusps, funnels, limit set, and ordinary points}\label{sec:ordinary}
Let $\Gamma$ be a Fuchsian group, that is, a discrete subgroup of~$\PSL_2(\RR)$. We call a point~$x\in\proj\RR$ \emph{cuspidal}\index[defs]{cuspidal} if it is fixed by a parabolic element of~$\Gamma$. In that case, we call its~$\Gamma$-orbit a 
\index[defs]{cusp}\emph{cusp} of~$\Gamma$. We remark that the cuspidality of a point in~$\proj\RR$ depends on the choice of~$\Gamma$.

The \emph{limit set}~$\Lambda = \Lambda(\Gamma)$ \index[defs]{limit set}\index[symbols]{L@$\Lambda,\Lambda(\Gamma)$} of~$\Gamma$ is the set of accumulation points of~$\Gamma z$ for any~$z\in\uhp$ that is not fixed by an element of~$\Gamma$ other than the identity. The set~$\Lambda$ is independent of the choice of~$z$, and it is contained in~$\proj\RR$. Each element of~$\Lambda$ is called a \emph{limit point}\index[defs]{limit point}\index[defs]{point!limit} of~$\Gamma$. By a slight abuse\footnote{Classically, the set of ordinary points is the complement of~$\Lambda$ in~$\overline\uhp$, see the definition in~\cite[Chapter~III, Section~1]{Leh}.} of notion, we call 
\begin{equation}
\Omega \coloneqq  \Omega(\Gamma) \coloneqq  \proj\RR\smallsetminus\Lambda
\end{equation}
the set of \index[symbols]{O@$\Omega,\Omega(\Gamma)$}\index[defs]{ordinary point}\index[defs]{point!ordinary}\emph{ordinary points} of~$\Gamma$. The cuspidal points of~$\Gamma$ are contained in~$\Lambda$. We call a connected component of~$\Omega$ a \emph{funnel interval}\index[defs]{funnel interval} of~$\Gamma$. Each funnel interval contains at least one maximal interval of points that are pairwise non-equivalent under the action of~$\Gamma$. Any such maximal interval we call a \emph{funnel representative}.\index[defs]{funnel representative} A \emph{funnel}\index[defs]{funnel} is the~$\Gamma$-orbit of a funnel representative.

Let $X=\Gamma\backslash\uhp$ be the hyperbolic surface with fundamental group~$\Gamma$. We allow $X$ to have conical singularities, in which case $X$ is not a surface in the strict sense but an orbifold. A conical singularity\index[defs]{conical singularity} of~$X$ (a proper orbifold point) corresponds to an equivalence class of elliptic elements in~$\Gamma$. The fundamental group~$\Gamma$ is called \emph{geometrically finite}\index[defs]{group!geometrically finite} if it has a fundamental domain~$\mathfrak{F}$ in~$\uhp$ with finitely many geodesic sides. It is called \emph{cofinite}\index[defs]{group!cofinite} if the area of~$\mathfrak{F}$ is finite with respect to the hyperbolic measure induced by the invariant measure~$y^{-2}\, dx\,dy$ on~$\uhp$.

The surface~$X$ may have \emph{ends}. An \index[defs]{end}\emph{end} of~$X$ refers here to a connected component of the complement of the compact core of~$X$, or almost equivalently, a connected component of $X\smallsetminus K$ for a very large compact subset~$K$ of~$X$. (We refer to~\cite{Borthwick_book} for the notion of the compact core of hyperbolic surfaces.) The ends of~$X$ of finite area correspond to the cusps of~$\Gamma$, the ends of infinite area correspond to the funnels. 

We refer to Section~\ref{sec:hecke} for examples, where we present the Hecke triangle groups of infinite covolume, which are the Fuchsian groups whose automorphic forms we discuss in this article.

\subsection{Geodesics, resonances, and the Selberg zeta function}\label{sec:SZF}

The Selberg zeta function is an important entity in the study of the spectral theory of hyperbolic surfaces. In this article, we use this zeta function only for motivational purposes. Nevertheless, for completeness, we now provide its full definition and briefly recall some of its properties. In a nutshell, it is a generating function for the geodesic length spectrum of the hyperbolic surface under consideration whose zero set contains all resonances. As such, it is a mediator between geometric and spectral properties of hyperbolic surfaces. In what follows we give more details.

The (unit speed) \emph{geodesics}\index[defs]{geodesic}\index[defs]{geodesic!on~$\uhp$} on the upper half plane~$\uhp$ are precisely the images of the curve 
\[
 \RR\to\uhp\,,\quad t\mapsto ie^t
\]
under the action of~$\PSL_2(\RR)$. The images of the geodesics, the complete \index[defs]{geodesic arc}\emph{geodesic arcs}, are the vertical lines in~$\uhp$ and the euclidean semi-circles in~$\uhp$ with center in~$\RR$. 

Let $\Gamma$ be a geometrically finite Fuchsian group, and let 
\[
 \pi\colon \uhp \to \Gamma\backslash\uhp
\]
be the canonical projection map. The \emph{geodesics}\index[defs]{geodesic!on~$\Gamma\backslash\uhp$} on $X\coloneqq \Gamma\backslash\uhp$ are the images of the geodesics on~$\uhp$ under~$\pi$. That is, if $\gamma\colon \RR\to\uhp$ is a geodesic on~$\uhp$, then 
\[
 \hat\gamma\coloneqq \pi\circ\gamma \colon \RR\to X\,,\quad t\mapsto \pi(\gamma(t))
\]
is a geodesic on~$X$. Conversely, every geodesic on~$X$ arises in this way. 

Whereas the arcs of geodesics on~$\uhp$ are all very similiar to each other and have only simple shapes, those of the geodesics on~$X$ can be quite different from each other. In particular, a geodesic~$\hat\gamma$ on~$X$ might be \index[defs]{periodic geodesic}\index[defs]{geodesic!periodic}\emph{periodic}, i.\,e., there exists~$t_0>0$ such that for all~$t\in\RR$ we have 
\[
 \hat\gamma(t+t_0) = \hat\gamma(t)\,.
\]
In this case, the minimal such~$t_0$ is called the \emph{primitive period length} of~$\hat\gamma$. We consider two geodesics~$\hat\gamma$ and~$\hat\eta$ on~$X$ as \emph{equivalent}\index[defs]{equivalent geodesics}\index[defs]{geodesic!equivalent} if they trace out the same arc in the same orientation, that is, if there exists $s\in\RR$ such that $\hat\gamma(\cdot) = \hat\eta(\cdot + s)$. We remark that equivalent geodesics have the same primitive period length (if one of the geodesics, and hence the other, is periodic). We let $L_X$ denote the \emph{primitive geodesic length spectrum} of~$X$, that is the multiset of the primitive period lengths of all equivalence classes of periodic geodesics on~$X$.

The \index[symbols]{Z@$Z_X$}\index[defs]{Selberg zeta function}\emph{Selberg zeta function}~$Z_X$ of~$X$ is determined by the Euler product 
\begin{equation}
 Z_X(s) \coloneqq \prod_{\ell\in L_X} \prod_{k=0}^\infty \left(1-e^{-(s+k)\ell}\right)\,,
\end{equation}
which is known to converge for $\Rea s \gg 1$ and which has a meromorphic continuation to all of~$\CC$. See~\cite{Selberg, Venkov_book, Guillope}. 

The zero set of~$Z_X$ has a spectral meaning, as we will explain now. For accuracy, we remark that this spectral interpretation of the zeros of the Selberg zeta function is known for all hyperbolic surfaces of \emph{finite area}. In the case that the hyperbolic surface $X=\Gamma\backslash\uhp$ has \emph{infinite area}, the spectral interpretation has so far been established for \emph{torsion-free} fundamental groups~$\Gamma$ only. (See the references below.) However, we suppose that the existing methods apply to non-torsion-free, non-cofinite geometrically finite Fuchsian groups~$\Gamma$ as well and that the spectral interpretation remains valid. The results in this paper also support this speculation.

A \index[defs]{resonance}\emph{resonance} of the hyperbolic surface~$X$ (of finite or infinite area) is a pole of the meromorphic continuation of the resolvent 
\begin{equation}
 R(s) \coloneqq \bigl(\Delta - s(1-s)\bigr)^{-1}
\end{equation}
of the Laplace--Beltrami operator~$\Delta$ on~$X$, understood as a map 
\[
 R(s)\colon L^2_{\textnormal{comp}}(X) \to H^2_{\textnormal{loc}}(X)\,.
\]
(Here, $L^2_{\textnormal{comp}}(X)$ is the space of compactly supported $L^2$-functions, and $H^2_{\textnormal{loc}}(X)$ is the space of functions that are locally in the Sobolev space~$H^2(X)$.) In particular, each spectral parameter~$s$ of any $L^2$-eigenfunction of~$\Delta$ is a resonance. The zero set of~$Z_X$ consists of the resonances and the so-called topological zeros~\cite{Selberg, Hejhal1, Hejhal2, Venkov_book, Patterson_Perry, BJP_szf}. This relation gives both a spectral meaning to the zeros of the Selberg zeta function and a dynamical interpretation of the resonances, and it has been useful for proving many results on resonances and geodesics.

\subsection{Intervals and rounded neighborhoods}

In extension of the notion of intervals in~$\RR$, we  call connected subsets of~$\proj\RR$ also \emph{intervals}\index[defs]{interval in~$\proj\RR$}.  We take advantage of the cyclic order\index[defs]{cyclic order} of~$\proj\RR = \RR\cup\{\infty\}$ to write intervals in~$\proj\RR$ other than~$\emptyset$, $\proj\RR$ and~$\proj\RR\smallsetminus\{\text{point}\}$. Thus, for any~$a,b\in\RR$, in terms of intervals in~$\RR$, the interval $(a,b)_c\subseteq\proj\RR$\index[symbols]{$(\cdot,\cdot)_c$} equals
\begin{equation}\label{cycle_int}
(a,b)_c = 
\begin{cases}
 (a,b) & \text{if $a < b$\,,}
 \\
 (a,\infty)\cup\{\infty\}\cup (-\infty,b) & \text{if $a>b$\,,}
\end{cases}
\end{equation}
and analogously if~$a=\infty=-\infty$ or~$b=\infty$. In particular,
\[
 (\infty,b)_c = (-\infty,b)
\]
for any~$b\in\RR$. Note that the symbol~$(a,a)_c$ cannot be defined consistently to~\eqref{cycle_int}, and is therefore left undefined.

The real projective line~$\proj\RR$ is embedded into the complex projective line~$\proj\CC=\CC\cup\{\infty\}$ (in the canonical way). We now introduce properties of open subsets of~$\proj\CC$ that we will use further below (e.\,g., in Section~\ref{sec:real_complex}) for neighborhoods in~$\proj\CC$ of intervals in~$\proj\RR$. 

For~$z_0\in\CC$, $r>0$ we let\index[symbols]{B@$B_r(z)$}
\[
 B_r(z_0) \coloneqq  \left\{ z\in\CC \setmid |z-z_0|<r \right\}
\]
denote the open ball in~$\CC$ with center~$z_0$ and radius~$r$. Let~$U$ be an open subset of~$\proj\CC$, and~$x\in\RR$. We say that~$U$ is
\begin{itemize}
\item \emph{left-rounded at~$x$}\index[defs]{left-rounded} if $B_\eps(x-\eps) \subseteq U$ for some~$\eps>0$,
\item \emph{right-rounded at~$x$}\index[defs]{right-rounded} if $B_\eps(x+\eps)\subseteq U$ for some~$\eps>0$,
\item \emph{rounded at~$x$}\index[defs]{rounded} if $U$ is left- and right-rounded at~$x$.
\end{itemize}
Likewise, we say that $U$ is
\begin{itemize}
 \item \emph{left-rounded at~$\infty$} if there exists~$x_0\in\RR$ such that $\CC_{\Rea > x_0}\subseteq U$,
 \item \emph{right-rounded at~$\infty$} if there exists~$x_0\in\RR$ such that $\CC_{\Rea < x_0}\subseteq U$,
 \item \emph{rounded at~$\infty$}\index[defs]{rounded at~$\infty$} if $U$ is left- and right-rounded at $\infty$.
\end{itemize}
One easily checks that a set~$U$ is rounded at~$\infty$ if and only if for each~$g\in\PSL_2(\RR)$ the set~$gU$ is rounded at~$g\infty$ (cf.\@ \cite[Section~3.1.1, 3.3]{Adam_Pohl}). Thus, the properties of being rounded at some point is stable under conjugation by the elements in~$\PSL_2(\RR)$. Also the properties of being left- or right-rounded are stable. Figure~\ref{fig:rounded} illustrates the notions of left-, and right-roundedness, and of roundedness.
\begin{figure}
\centering
\includegraphics[width=0.95\linewidth]{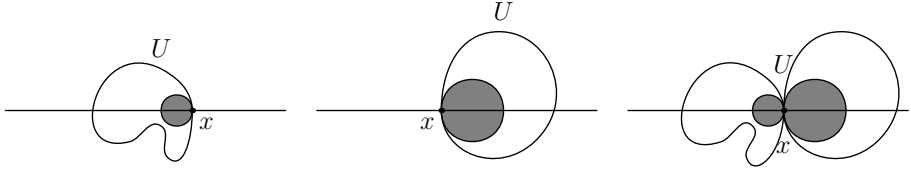}
\caption{Examples of open sets~$U$ that are left-, right-, and rounded at~$x$, respectively. The grey circles indicate the disks touching~$x$ from the left, right, and both sides, respectively.}\label{fig:rounded}
\end{figure}

\section{Hecke triangle groups with infinite covolume}\label{sec:hecke}
\markright{4. HECKE TRIANGLE GROUPS}

Hecke triangle groups were used by Hecke in his studies of functional equations for Dirichlet series. See \cite{Hecke_groups} or Hecke's lecture notes \cite[Chap.~II]{He}. The Hecke triangle groups with infinite covolume form a $1$-parameter family~$(\Gamma_\lambda)_{\lambda\in\RR_{>2}}$ of Fuchsian groups. More precisely, the \index[defs]{Hecke triangle group}\index[symbols]{Gaam@$\Gamma_\lambda$} Hecke triangle group~$\Gamma = \Gamma_\lambda$ with parameter~$\lambda\in\RR$, $\lambda>2$, is the subgroup of~$\PSL_2(\RR)$ generated by\index[symbols]{T@$T=\bmat 1\lambda01$}\index[symbols]{S@$S=\bmat 0{-1}10$}
\begin{equation}
 T=T_\lambda=\bmat 1\lambda01 \quad\text{and}\quad S=\bmat 0{-1}10\,.
\end{equation}
Figure~\ref{fig-fd} shows two fundamental domains for~$\Gamma\backslash\uhp$.\index[symbols]{F@$\fd_0$}\index[symbols]{F@$\fd_1$}
\begin{figure}
\centering
\includegraphics[width=0.95\linewidth]{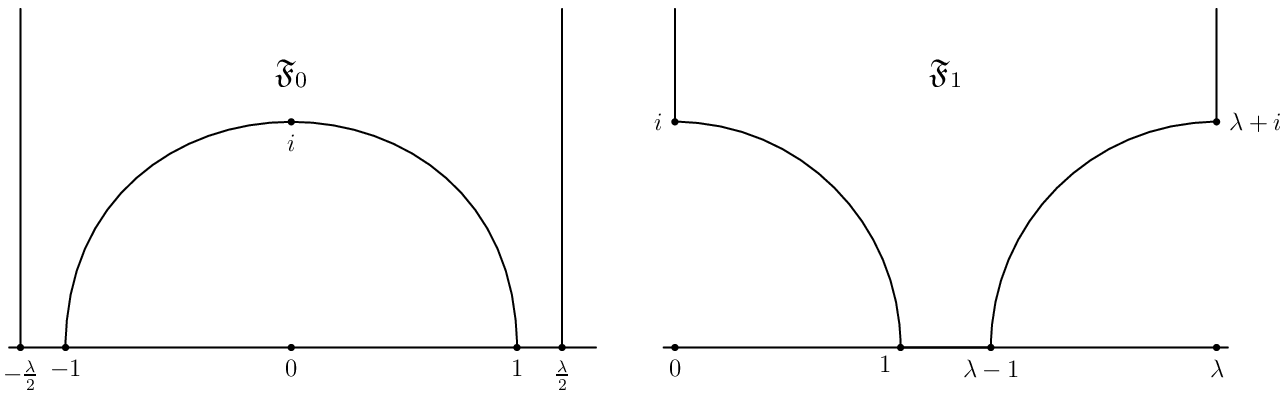}
\caption{Two fundamental domains for~$\Gm\backslash\uhp$.}\label{fig-fd}
\end{figure}

The Hecke triangle surface~$\Gamma\backslash\uhp$ has one cusp, one conical singularity, and one funnel. In~$\overline{\uhp}$, the cusp is represented by~$\infty$, and hence the set of all cuspidal points of~$\Gamma$ is the~$\Gamma$-orbit~$\Gamma\infty$. The conical singularity is represented by~$i$. The funnel of~$\Gamma\backslash\uhp$ is represented by~$[1,\lambda-1)$, and the set of ordinary points is
\begin{equation}
 \Omega = \bigcup_{\gamma\in\Gamma} \gamma[1,\lambda-1]\,.
\end{equation}
We can deduce the funnel interval~$I$ containing~$[1,\lambda-1]$ from the fundamental domain~$\fd_1$ in~Figure~\ref{fig-fd} (right) and its side-pairings, as explained in what follows. The side-pairings of~$\fd_1$ are given by~$T$ and~$TS$. The map~$T$ identifies the two vertical sides both ending in~$\infty$, and~$TS$ identifies the two sides which are touching the funnel representative~$[1,\lambda-1)$. The location and shape of the neighboring translates~$TS\fd_1$ and~$(TS)^{-1}\fd_1$ of~$\fd_1$ (see Figure~\ref{fig:funnel_interval}) show that 
\[
 TS[1,\lambda-1] \cup (TS)^{-1}[1,\lambda-1]\subseteq I\,.
\]
A straightforward induction yields that 
\[
 I = \bigcup_{n\in\ZZ} (TS)^n [1,\lambda-1] = (\fixTS^-,\fixTS^+)\,,
\]
where\index[symbols]{T@$\fixTS^\pm$}
\begin{equation}
 \fixTS^\pm \coloneqq \frac{\lambda \pm \sqrt{\lambda^2-4}}{2} 
\end{equation}
are the repelling~($-$) and attracting~($+$) fixed points of~$TS$. 

On the funnel interval~$(\fixTS^-,\fixTS^+)$, the cyclic group~$\langle TS\rangle$ acts discontinuously.  A fundamental domain for this action is given by~$[1,\lambda-1)$, which is a funnel representative. Since~$\Gamma\backslash\uhp$ has a single funnel, any other funnel interval is of the form~$g(\fixTS^-,\fixTS^+)$ for some~$g\in\Gamma$, and each such interval is a funnel interval. 

\begin{figure}
\centering
\includegraphics[width=0.95\linewidth]{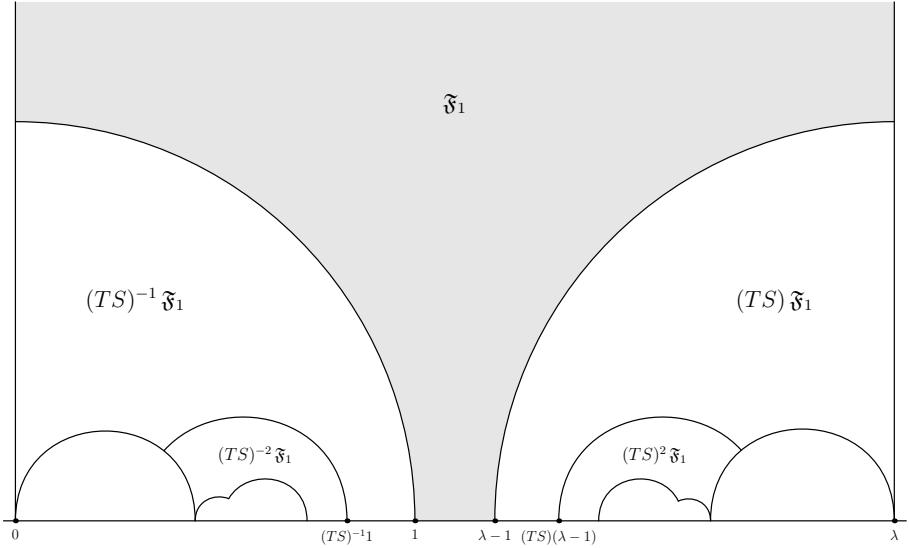}
\caption{Neighboring $\Gamma$-translates of~$\fd_1$ at the funnel.}
\label{fig:funnel_interval}
\end{figure}

The fundamental domain~$\fd_0$ of~$\Gm$ in~Figure~\ref{fig-fd} leads to the presentation
\begin{equation}\label{presentation_Gamma} \Gamma = \langle T,S \mid S^2 = 1 \rangle \end{equation}
by considering the neighboring fundamental domains (see
Figure~\ref{fig:Gamma_nbh}) or, equivalently, by applying Poincar\'e's Theorem on fundamental polyhedra~\cite{Maskit}.
\begin{figure}
\centering
\includegraphics[width=0.95\linewidth]{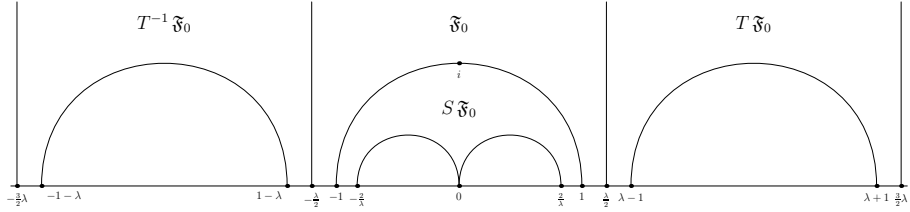}
\caption{Fundamental domain~$\fd_0$ for~$\Gamma$, and neighboring
translates.}\label{fig:Gamma_nbh}
\end{figure}

\section{Automorphic forms}\label{sect-daf}
\markright{5. AUTOMORPHIC FORMS}

We state the definitions of the relevant spaces of automorphic forms only specified for Hecke triangle groups, which is sufficient for our purposes. Their generalizations to arbitrary Fuchsian groups are straightforward. Throughout let~$\Gamma$ be a Hecke triangle group. 

For any~$s\in \CC$ we let $\E_s\coloneqq  \E_s^{\Gm}$\index[symbols]{Eaa@$\E_s$, $\E_s^\Gamma$} denote the space of~\emph{$\Gamma$-invariant Laplace eigenfunctions with spectral parameter}~$s$. Recall the hyperbolic Laplacian from~\eqref{Laplace}. Thus, the space~$\E_s$ consists of all functions~$u\colon\uhp\rightarrow\nobreak\CC$ that satisfy
\begin{enumerate}[(a)]
\item\label{constantorbit} $u(\gm  z) = u(z)$ for all~$\gm\in \Gm$, all~$z\in\uhp$, and
\item $\Delta u = s(1-s)\, u$.
\end{enumerate}
The $\Gamma$-invariance in~\eqref{constantorbit} shows that each~$u\in\E_s$ descends to a function on $\Gamma\backslash\uhp$. Hence we may as well consider $\E_s$ to be a space of functions $\Gamma\backslash\uhp\to\CC$. 

\subsection{Funnel forms of different types}\label{sec:def_ff} 

We define subspaces of~$\E_s$ by imposing conditions of regularity near the funnel and of growth at the cusp of~$\Gamma\backslash\uhp$. Recall that $\Omega$ denotes the set of ordinary points of~$\Gamma$ in~$\proj\RR$. 

We say that a function~$u\in\nobreak\mc E_s$ is a \emph{funnel form} if for each open interval~$I\subseteq\Omega$ the map
\begin{equation}
 \uhp\to\CC\,,\quad z\mapsto y^{-s}u(z)
\end{equation}
extends to a real-analytic function on a neighborhood of~$I$ in~$\CC$. In other words, there exists a neighborhood~$U$ of~$I$ in~$\CC$ and a real-analytic function~$A$ on~$U$ such that
\begin{equation}\label{def_ff}
u(z) = y^s \, A(z)\quad\text{for all~$z\in U\cap\uhp$\,.}
\end{equation}
We call this property \index[defs]{$s$-analytic boundary behavior}\emph{$s$-analytic boundary behavior near~$I$}, and $A$ a \index[defs]{real-analytic core}\emph{real-analytic core of~$u$ near~$I$}.  We denote the space of funnel forms by~$\A_s\coloneqq  \A_s(\Gamma)$.\index[defs]{funnel form}\index[symbols]{AAA@$\A_s, \A_s(\Gamma)$}

We remark that the map~$u$ is real-analytic on~$\uhp$ as being a Laplace eigenfunction. Thus, also the core~$A$ is real-analytic on all of~$\uhp$, and the requirement in~\eqref{def_ff} is on the real-analytic extendability of $z\mapsto y^{-s}u(z)$ beyond~$\uhp$. In particular, we may always assume that $U$ contains all of~$\uhp$, a property that will simplify the proof of Lemma~\ref{lem-io} below.

In view of Theorem~\ref{thmA_new}, the definition of funnel forms is seen to be natural. One of our objectives is to find an identification of those period functions that do not necessarily satisfy any further restrictions with suitable Laplace eigenfunctions. The $s$-analytic boundary behavior is just the right property of Laplace eigenfunctions to be required at the funnel in order to reach this goal. 

Since $\Gamma$ is a Hecke triangle group, the set~$\Omega$ of ordinary points of~$\Gamma$ is the union
\[
 \Omega = \bigcup_{\gamma\in\Gamma} \gamma (\fixTS^-,\fixTS^+)\,,
\]
where $\fixTS^\pm$ are the attracting and repelling fixed points of~$TS$ (see Section~\ref{sec:hecke}). This union can be made disjoint by restricting to only certain~$\Gamma$-translates of the interval~$(\fixTS^-,\fixTS^+)$. More precisely, let~$R\subseteq\Gamma$ be a set of representatives for the coset spaces~$\Gamma/\langle TS\rangle$. Then
\begin{equation}\label{Omega_funddom}
 \Omega = \bigsqcup_{\gamma \in R} \gamma(\fixTS^-,\fixTS^+) \qquad\text{(disjoint union).}
\end{equation}
The property of having $s$-analytic boundary behavior is stable under the action of~$\PSL_2(\RR)$. In other words, if $u\in\E_s$ has $s$-analytic boundary behavior near the interval~$I\subseteq\Omega$, and $g$ is an element in~$\Gamma$ then, using~\eqref{constantorbit}, one easily checks that $u$ also has \mbox{$s$-}ana\-lytic behavior near~$gI$. From~\eqref{Omega_funddom} it thus follows that $u\in \E_s$ is a funnel form if and only if~$u$ satisfies~\eqref{def_ff} for~$I=(\fixTS^-,\fixTS^+)$.

We define the space~$\A_s^1 \coloneqq  \A_s^1(\Gamma)$ of \emph{resonant funnel
forms}\index[defs]{resonant funnel form}\index[defs]{funnel form!resonant}\index[symbols]{AAC@$\A_s^1,\A_s^1(\Gamma)$} to be all those funnel forms~$u\in \A_s$ for which there exist constants~$p\in \CC$ and~$c>0$ (both depending on~$u$) such that
\begin{equation}\label{resgrwth} 
u(z) \ceqq  p\, y^{1-s} + \oh\left( e^{-cy}\right)\qquad\text{as $y\uparrow\infty$}\,. 
\end{equation}

Finally, we define the space~$\A_s^0\coloneqq  \A_s^0(\Gm)$ of \emph{cuspidal funnel forms}\index[defs]{cuspidal funnel form}\index[defs]{funnel form!cuspidal}\index[symbols]{AAE@$\A_s^0,\A_s^0(\Gamma)$} to be the subspace of~$\A_s^1$ consisting of those funnel forms~$u$ for which there exists~$c>0$ (depending on~$u$) such that
\begin{equation}\label{expdecay} 
u(z) \ceqq  \oh\left( e^{-cy}\right)\quad\text{as $y\uparrow\infty$}\,.
\end{equation}

We note that even though~$\E_s =\E_{1-s}$, each of the defined subspaces~$\A_s$, $\A_s^1$, $\A_s^0$ of forms with spectral parameter~$s$ might be different from the corresponding subspace for the spectral parameter~$1-s$.

\subsection{Fourier expansion}
The subspaces~$\A_s^1$ and~$\A_s^0$ are characterized within $\A_s$ by properties of the Fourier expansions at the cusp.

For each~$s\in\CC$, the functions in~$\E_s$ are $\lambda$-periodic along the real axis. If $s\not=1/2$, then the Fourier expansion\index[defs]{Fourier expansion} of~$u\in\E_s$ at the cusp~$\infty$ is given by (see \cite[Section~1.2 and 8.1]{BLZm})
\begin{equation}\label{Feu}
\begin{aligned}
u(z) &= A_0 \lambda^{s-1} y^{1-s} + y^\frac12 \lambda^{s-\frac12}\sum_{\substack{n\in\ZZ\\ n\not=0}} A_n K_{s-\frac12} \left( 2\pi |n| \frac{y}{\lambda}\right) e^{2\pi i n \frac{x}{\lambda}}
 \\
  & \quad + B_0\lambda^{-s} y^s + y^\frac12 \Gamma\left(s+\tfrac12\right) \pi^{\frac12 - s} \lambda^{s-\frac12}
  \\
  & \hphantom{ \quad + B_0\lambda^{-s} y^s + y^\frac12 \Gamma\left(s+\tfrac12\right) } \times\sum_{\substack{n\in\ZZ\\ n\not=0}} B_n |n|^{\frac12-s} I_{s-\frac12}\left( 2\pi |n| \frac{y}{\lambda}\right) e^{2\pi i n \frac{x}{\lambda}}\,,
\end{aligned}
\end{equation}
where~$I_{\nu}$ and~$K_{\nu}$ are the modified Bessel functions (exponentially increasing and exponentially decreasing, respectively). For~$n\in\ZZ$, the coefficients~$A_n, B_n$ are suitable complex numbers (independent of~$x$ and~$y$). 

If~$s=1/2$, then the zero-th term of the Fourier expansion of $u\in\E_s$ has a different form than in~\eqref{Feu}. For $u\in\E_{\frac12}$ the expansion is 
\begin{equation}
\begin{aligned}
 u(z) &= A_0 y^\frac12 + y^\frac12 \sum_{\substack{n\in\ZZ\\ n\not=0}} A_n K_{0} \left( 2\pi |n| \frac{y}{\lambda}\right) e^{2\pi i n \frac{x}{\lambda}}
\\
 & \quad 
 + B_0 y^\frac12 \ln y + y^\frac12 \sum_{\substack{n\in\ZZ\\ n\not=0}} B_n  I_{0}\left( 2\pi |n| \frac{y}{\lambda}\right) e^{2\pi i n \frac{x}{\lambda}}\,.
\end{aligned}
\end{equation}
For any~$s\in\CC$, the factor in front of~$I_{s-1/2}$ ensures that the terms behave like~$y^s$ as~$y\downarrow 0$. See the discussion in \cite[Sections~1.2 and 8.1]{BLZm}. 

The subspace~$\A_s^1$ consists of those~$u\in\A_s$ for which $B_n=0$ for all~$n\in\ZZ$. The subspace~$\A_s^0$ is characterized by the additional property that also~$A_0=0$.

\section{Principal series}\label{sect-psa} 
\markright{6. PRINCIPAL SERIES}
For the transfer operators, defined in Section~\ref{sect-to} below, as well as the modules of the cohomology spaces, defined in Chapter~\ref{part:cohom} below, the spaces and the action of the principal series representations of~$\PSL_2(\RR)$ are crucial. For these, various models are known. (See, e.\,g., \cite[Section~2.1]{BLZ13}) We present the model that we will use for our applications, the line model.

Within the line model, the spaces of the principal series representations of the Lie group~$\PSL_2(\RR)$ are certain spaces of functions on~$\proj\RR$. In order to provide explicit formulas, we use two coordinate charts for the (one-dimensional, smooth) manifold~$\proj\RR=\RR\cup\{\infty\}$, namely $(\RR,\id)$ and $(\RR,S)$, where~$S\colon t\mapsto -1/t$ is identified with the action of the element~$S=\textbmat{0}{1}{-1}{0}$ on $\proj\RR$. Thus, the image of the chart $(\RR,\id)$ is $\RR=\proj\RR\smallsetminus\{\infty\}$, and the image of the chart $(\RR,S)$ is $\proj\RR\smallsetminus\{0\}$. The change of charts is given by $t\mapsto -1/t$, thus by the action of~$S$. With respect to this manifold atlas, each function on~$\proj\RR$ is given by a pair~$(f, f_\infty)$ of functions 
\[
 f,f_\infty\colon \RR\to \CC\,.
\]
To be element of a representation space of the principal series of~$\PSL_2(\RR)$ \emph{with spectral parameter~$s\in\CC$}, one of the conditions the pair $(f,f_\infty)$ needs to satisfy is the relation 
\[
 f(t) = |t|^{-2s} f_\infty\bigl(-1/t\bigr)\quad\text{for $t\in\RR\smallsetminus\{0\}$\,.}
\]
In addition we will either require regularity properties that will allow to determine the value of~$f_\infty$ at~$0$ from the knowledge of the behavior of $f_\infty$ in a punctured neighborhood of~$0$, and hence from knowing the function~$f$, or we will not be interested in the precise value of $f_\infty$ at~$0$. We refer to the discussions below for further details. Therefore, by slight abuse of notation, we identify $f$ with the pair~$(f,f_\infty)$ and consider~$f$ as a function on all of~$\proj\RR$ that is `well-behaved' for $|t|\to\infty$ .

The action of~$\PSL_2(\RR)$ on the spaces of the principal series representations  with spectral parameter~$s$ is given by\index[symbols]{T@$\tau_s(g)$}
\begin{equation}\label{taudef} 
\left(\tau_{s}\big(g^{-1}\big) f\right)(t) \coloneqq  f\vert_{2s}g(t) \coloneqq  |ct+d|^{-2s} \, f\left( \frac{at+b}{ct+d}\right)\,,
\end{equation}
where 
\[
 g=\bmat abcd \, \in \PSL_2(\RR)\,.
\]
(We note that we assume that $f\colon \proj\RR\to\CC$ is sufficiently well-behaved for $|t|\to\infty$ such that~\eqref{taudef} is well-defined for~$t=-d/c$.) Throughout we will work with right modules for the representation spaces, which causes the inverse~$g^{-1}$ in~\eqref{taudef}. We refer to~\cite{BLZ13} and~\cite[\S2]{BLZm} for more details and other models.

\subsection{Regularity at infinity}\label{sec:regularity}

Let~$s\in\CC$, let~$I\subseteq\proj\RR$ be an open interval with~$\infty\in I$, and let~$f\colon I\to\CC$ be a function. We say that~$f$ is \emph{real-analytic at~$\infty$ for the spectral parameter~$s$}\index[defs]{real-analytic at~$\infty$} if and only if there exists~$g\in\PSL_2(\RR)$ such that $gI\subseteq\RR$ and the map~$\tau_s(g)f$ is real-analytic at~$g\infty$ (that is, $\tau_s(g)f$ is real-analytic in a neighborhood of~$g\infty$). A standard choice for~$g$ is
\begin{equation}\label{standardchoice}
 g = S = \bmat 0{-1}10\,,
\end{equation}
resulting in the characterization that $f$ is real-analytic at~$\infty$ if and only if for all sufficiently large~$|t|$, the function~$f$ is given by a power series in~$-1/t$ times~$|t|^{-2s}$, thus
\begin{equation}
 f(t) = |t|^{-2s} \sum_{n=0}^\infty c_nt^{-n}\,.
\end{equation}
Analogously, we define the notions of \emph{smoothness ($C^\infty$) at~$\infty$}\index[defs]{smooth at~$\infty$}\index[defs]{$C^\infty$ at~$\infty$} or any other type of regularity at~$\infty$ (e.\,g., $C^p$, $p\in\RR_{\geq 0}$, etc.).\index[defs]{regularity at~$\infty$} In particular, smoothness of~$f$ at~$\infty$ is characterized by the existence of an asymptotic expansion
\begin{equation}
 f(t) \sim |t|^{-2s} \sum_{n=0}^\infty c_nt^{-n} \qquad\text{as $|t|\to\infty$\,.}
\end{equation}
We remark that the notion of real-analyticity, smoothness, etc., at~$\infty$ does not depend on the choice of~$g\in\PSL_2(\RR)$ with~$g\infty\in\RR$. 

We further remark that for any open subset~$U\subseteq \proj\RR$, any point~$x\in U$ and any element $g\in\PSL_2(\RR)$, the action~$\tau_s(g)$ preserves real-analyticity and smoothness at~$x$. That is, for any function~$h\colon U\to\CC$, the function~$h$ is real-analytic or smooth at~$x$ if and only if~$\tau_s(g)h$ is real-analytic or smooth at~$gx$, respectively.

In Section~\ref{sec:PGL} we will define the extension of~$\tau_s$ to~$\PGL_2(\RR)$. In anticipation of this definition we remark that all of the definitions and characterizations in the present section extend without changes to the action by elements in~$\PGL_2(\RR)$.

\subsection{Presheaves and sheaves}\label{sec:sheaves}

Let~$s\in\CC$. For any open subset~$I\subseteq \proj\RR$ and different types of conditions~\tcond on functions on~$I$ (e.\,g., regularity, growth, local behavior) we set\index[symbols]{Vaaaa@$\V s \cond$}\index[symbols]{Vaaab@$\V s \om$}
\begin{equation}
 \V s \cond (I) \coloneqq  \left\{ f\colon I \to \CC \setmid \text{$f$ satisfies~\tcond} \right\}\,.
\end{equation}
We refer to Chapter~\ref{part:cohom} for the list of conditions that we will use. Since all of these conditions are local and $(\Gamma,\tau_s)$-equivariant for any Hecke triangle group~$\Gamma$ (even for any Fuchsian group~$\Gamma$), we obtain $(\Gamma,\tau_s)$-equivariant presheaves or sheaves~$\V s \cond$ with bijective linear transformation maps
\[
 \tau_s(g)\colon \V s \cond (I) \to \V s \cond (gI)\qquad (g\in\Gamma)\,.
\]

\subsection{Holomorphic extensions}\label{sec:holomext}

For the notion of complex period functions, we will consider real-analytic functions~$f$ on some open subset~$I\subseteq\RR$ (or~$\proj\RR$) and apply the definition in~\eqref{taudef} to their extension to a holomorphic function on some open neighborhood of~$I$ in~$\CC$ (or~$\proj\CC$). See, e.\,g., Section~\ref{sect-to}. This extension of the definition in~\eqref{taudef} can only be applied with caution as explained in what follows.

Let~$s\in\CC$, let~$\Gamma$ be a Hecke triangle group, and let~$I\subseteq\proj\RR$ be an open subset. Any real-analytic function~$f\colon I\to\CC$ has an holomorphic extension to some open neighborhood of~$I$ in~$\proj\CC$. For every~$g\in\Gamma$, the function~$\tau_{s}(g)f$ has a holomorphic extension as well. However, in whatever way we choose the holomorphic extension of the automorphy factor~$|ct+\nobreak d|^{-2s}$ in~\eqref{taudef}, in general, the representation property 
\begin{equation}\label{nearly_repr}
 \tau_s(g_1)\tau_s(g_2)f = \tau_s(g_1g_2)f
\end{equation}
does not extend to outside of~$\proj\RR$. In other words, $\tau_s$ does not define a representation of~$\Gamma$ on spaces of holomorphic functions. 

However, in all situation where we will make use of the relation~\eqref{nearly_repr} for holomorphic functions~$f$, we will only consider a certain semi-subgroup of~$\Gamma$. For this semi-subgroup, a common choice of the holomorphic extension of the automorphy factor in~\eqref{taudef} is indeed possible. We refer to Section~\ref{sec:real_complex} below, in particular to Proposition~\ref{prop-FE-RC}, for a prominent example of such a situation.

\section{Transfer operators and period functions}\label{sect-to}
\markright{7. TRANSFER OPERATORS AND PERIOD FUNCTIONS}

Throughout let~$\Gamma$ be a Hecke triangle group. The families of slow and fast transfer operators for~$\Gamma$ that we will use in Theorems~\ref{thmA_new} and~\ref{thmB_new} have been developed in~\cite{Pohl_hecke_infinite}. More precisely, the family~$\big(\tro s\big)_{s\in\CC}$ of slow transfer operators is exactly the one from~\cite{Pohl_hecke_infinite}; we will review it in Section~\ref{sec:slowTO}. The family~$\big(\ftro s\big)_{s\in\CC}$ of fast transfer operators arises by a reduction and simplification of the one in~\cite{Pohl_hecke_infinite}, capturing the essential spectral properties of the latter. The transfer operators~$\ftro s$ can also be derived directly from the slow transfer operators~$\tro s$. In Section~\ref{sec:fastTO} we will discuss both ways of their construction in more details. All these transfer operators arise directly from the discretizations of the geodesic flow on the unit sphere bundle of~$\Gm\backslash \uhp$ that were developed in~\cite{Pohl_Symdyn2d, Pohl_hecke_infinite}, or are closely related to those. We will briefly review these discretizations in Section~\ref{sec:discr}.

As already mentioned in the Introduction, period functions are eigenfunctions of the transfer operators with eigenvalue~$1$. Precise definitions will be given in Section~\ref{sec:periodfunctions}, and the relation between the space of real period functions and the space of complex period functions will be discussed in Section~\ref{sec:real_complex}.

\subsection{Discretizations and transfer operators}\label{sec:discr}

For any sufficiently nice discrete dynamical system~$F\colon D\to D$ (e.\,g., $D$ being an open subset\footnote{In the construction of the transfer operators below, $D$ will be a subset of a disjointified union of open sets of~$\RR$. The transfer operator itself acts initially on a space of functions with domain~$D$. The transfer operator is then extended in a continuous way to a space of functions with an open set as domain, which in a certain sense, is an open hull of~$D$. For details we refer to the detailed construction in~\cite{Pohl_hecke_infinite}. } of~$\RR$, and $F$ being a differentiable function with at most countable preimages) the associated transfer operator~$\TO_{s,F}$ with parameter~$s\in\CC$ is defined by
\begin{equation}
 \TO_{s,F}f(x) \coloneqq  \sum_{y\in F^{-1}(x)} \big| F'(y)\big|^{-s} f(y)\,,
\end{equation}
acting on appropriate spaces of functions~$f$ on~$D$. The choice of these spaces depends on the applications.  The transfer operators in~\cite{Pohl_hecke_infinite} are associated to a discrete dynamical system that arises from the discretization of the geodesic flow on $\Gamma\backslash\uhp$ as provided in \cite{Pohl_Symdyn2d, Pohl_hecke_infinite}. We briefly recall the rough structure of such discretizations and refer to \cite{Pohl_Symdyn2d, Pohl_hecke_infinite} for details. 

Throughout, for any vector~$v$ in the unit tangent bundle~$\shp$ of~$\uhp$, we let $\gamma_v$ denote the geodesic on~$\uhp$ determined by~$v$. Thus, $\gamma_v$ is the unique geodesic satisfying 
\begin{equation}
 \gamma_v'(0) = v\,.
\end{equation}
Likewise, for any vector~$\hat v\in\Gamma\backslash \shp$ of the unit tangent bundle of~$\Gamma\backslash\uhp$, we let $\hat\gamma_v$ denote the geodesic on~$\Gamma\backslash\uhp$ determined by~$\hat v$. Thus
\begin{equation}
 \hat\gamma_v\!'(0) = \hat v\,.
\end{equation}
A \emph{(strong) cross section}\index[defs]{cross section} for the geodesic flow on~$\Gamma\backslash\uhp$ is a subset~$\hat C$ of~$\Gamma\backslash \shp$ such that 
\begin{itemize}
\item every geodesic~$\hat\gamma$ on~$\Gamma\backslash\uhp$ intersects~$\hat C$ in a discrete sequence of times, which is either empty or `future-infinite.' More precisely, if we consider the derivative of~$\hat\gamma$ as a curve in~$\Gamma\backslash \shp$, thus
 \[
  {\hat\gamma}\,'\colon \RR\to\Gamma\backslash \shp\,,
 \]
then the set 
\[
 \left\{ t\in\RR_{\geq 0} \setmid \hat\gamma'(t) \in \hat C \right\}
\]
is required to be discrete in~$\RR$ and either empty or infinite. This set may and does depend on the specific geodesic~$\hat\gamma$. 
\item every periodic geodesic of~$\Gamma\backslash\uhp$ intersects~$\hat C$.
\end{itemize}
The \emph{first return map}~$\hat R\colon\hat C\to \hat C$ \index[defs]{first return map} of a cross section~$\hat C$ is the map which assigns to~$\hat v\in \hat C$ the vector~$\hat\gamma_v\!'(t_0)\in\hat C$, where 
\begin{equation}\label{eq:def_firstreturn}
 t_0 \coloneqq  \min\left\{ t>0 \setmid \hat\gamma_v\!'(t) \in \hat C\right\}
\end{equation}
is the \emph{first return time} of~$\hat v$.\index[defs]{first return time} The first return map is a discretization of the geodesic flow on~$\Gamma\backslash\uhp$.

In~\cite{Pohl_Symdyn2d, Pohl_hecke_infinite} a construction of a cross section~$\hat C$ is proposed for which there exists a subset~$C^*$ of~$\shp$ that contains exactly one representative of each element in~$\hat C$ and which decomposes into the two sets 
\begin{itemize}
\item $C^*_1$, which consists of all unit tangent vectors~$v\in \shp$ that are based on the geodesic on~$\uhp$ connecting~$-1$ and~$\infty$, and for which the geodesic~$\gamma_v$ ends in the interval~$(-1,\infty)_c$ (that is, $\gamma_v(\infty)\in (-1,\infty)_c$) but not in a funnel interval or a cuspidal point, and 
\item $C^*_2$, which consists of all unit tangent vectors~$v\in \shp$ that are based on the geodesic on~$\uhp$ connecting~$1$ and~$\infty$, and for which the geodesic~$\gamma_v$ ends in the interval~$(\infty,1)_c$  but not in a funnel interval or a cuspidal point.
\end{itemize}
The sets~$C^*_1$ and~$C^*_2$ are visualized in Figure~\ref{fig:slowTO}, the directions of the unit tangent vectors belonging to~$C^*_1$ and~$C^*_2$ are indicated in grey.

\begin{figure}
\centering
\includegraphics[width=0.95\linewidth]{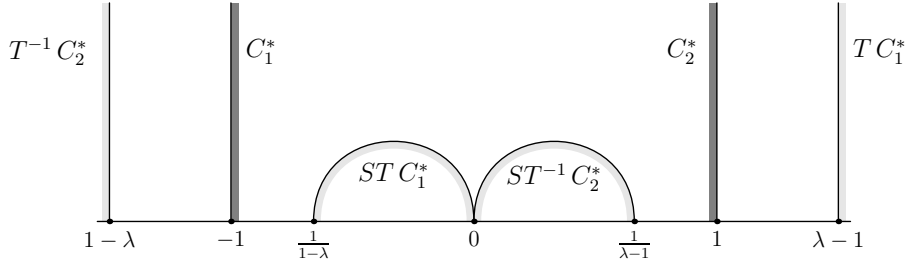}
\caption{Cross section for slow transfer operator.}\label{fig:slowTO}
\end{figure}

The set~$C^*$ determines a map~$F$ on a certain subset~$D$ of
\[
\bigl((-1,\infty)_c\times\{1\}\bigr) \cup \bigl((\infty,1)_c\times\{2\}\bigr)
\]
as follows: For any~$\hat v\in \hat C$ let~$v$ denote the (unique) representative of~$\hat v$ in~$C^*$. Set (`coding bit of~$\hat v$')
\begin{equation}
 b(\hat v) \coloneqq
 \begin{cases}
  1 & \text{if $v\in C^*_1$\,,}
  \\
  2 & \text{if $v\in C^*_2$\,.}
 \end{cases}
\end{equation}
Let 
\[
 \pr\colon \hat C\to \RR\times\{1,2\},\quad \hat v\mapsto \bigl(\gamma_v(\infty),b(\hat v)\bigr)
\]
denote the map which assigns to~$\hat v\in \hat C$ the future endpoint of the geodesic on~$\uhp$ determined by the representative~$v\in C^*$ of~$\hat v$ and its coding bit~$b(\hat v)$. Then~$F$ is the unique map for which the diagram
\[
\xymatrix{
\hat C \ar[r]^{\hat R} \ar[d]_{\pr} & \hat C  \ar[d]^{\pr}
\\
D \ar[r]^F & D
}
\]
commutes. 

The set~$D$ consists of the endpoints~$\gamma_v(\infty)$ of the geodesics~$\gamma_v$ determined by the elements~$v\in C^*$, and a symbol that captures if $\gamma_v$ intersects $C^*_1$ or $C^*_2$ (in other words, if $v\in C^*_1$ or $v\in C^*_2$). Thus,
\[
 D \;=\; \Big(\bigl( (-1,\infty)_c\smallsetminus( \Omega(\Gamma) \cup \Gamma\infty) \bigr) \times \{1\}\Big) \cup \Big(\bigl( (\infty,1)_c\smallsetminus( \Omega(\Gamma) \cup \Gamma\infty) \bigr) \times \{2\}\Big)\,.
\]

The map~$F$ can be read off from Figure~\ref{fig:slowTO} in the following way: For~$(x,j)\in\nobreak D$ we pick a vector~$v\in C^*_j$ such that~$\gamma_v(\infty)=x$. Let~$t_0>0$ be the first return time of~$v$, that is (cf.~\eqref{eq:def_firstreturn})
\begin{equation}
 t_0 \coloneqq \min\{ t>0 \setmid \gamma_v'(t)\in \Gm\,C^* \}\,,
\end{equation}
and suppose that $\gamma_v'(t_0)\in g\,C^*_k$ ($g\in\Gm$). (We note that~$g$ and~$k$ are uniquely determined.) Then
\begin{equation}
 F\bigl(x,j\bigr) = \bigl(g^{-1}\,x, k\bigr)\,.
\end{equation}
One easily sees that~$g$ and~$k$ do not depend on the choice of~$v$. We refer to~\cite{Pohl_Symdyn2d, Pohl_hecke_infinite} for more details. 

The discrete dynamical system~$(D,F)$ gives rise to the \emph{slow transfer operators}, defined in Section~\ref{sec:slowTO} below. An induction of the system~$(D,F)$ on the parabolic elements then gives rise to the \emph{fast transfer operators} from~\cite{Pohl_hecke_infinite}. This induction process is equivalent to work with a certain sub-cross section of~$\hat C$; it has the effect that the arising discrete dynamical system is uniformly expanding. 

The motivation to call these transfer operators \emph{slow} or \emph{fast} originates in the properties of the underlying discrete dynamical systems. In both cases, the discrete dynamical systems are piecewise given by fractional linear transformation of certain elements in~$\Gamma$. For slow transfer operators, the discrete dynamical system is finitely branched. More precisely, it decomposes into finitely many pieces only, and hence is reminiscent of a slow continued fraction algorithm (equivalently, a Farey algorithm). Whereas for fast transfer operators, the discrete dynamical system is infinitely branched and reminiscent of a fast continued fraction algorithm.

\subsection{Slow transfer operators}\label{sec:slowTO}

As mentioned in Section~\ref{sec:discr}, the \emph{slow transfer operator}~$\tro s$ arises from the discretization of the geodesic flow on~$\Gamma\backslash \shp$ that is provided in~\cite{Pohl_hecke_infinite} (a simplication of the discretization in~\cite{Pohl_Symdyn2d}). The operator~$\tro s$ is (see \cite{Pohl_hecke_infinite})\index[defs]{slow transfer operator}\index[defs]{transfer operator!slow}\index[symbols]{L@$\tro s$}
\begin{equation}\label{stroGm} 
\tro s \ceqq 
\begin{pmatrix} \tau_s(T^{-1}S) + \tau_s(T^{-1}) & \tau_s(T^{-1}S) \\
\tau_s(TS) & \tau_s(TS) + \tau_s(T)
\end{pmatrix},
\end{equation}
acting on functions vectors 
\[
 \begin{pmatrix} f_1\colon (-1,\infty)_c\to\CC \\ f_2\colon (\infty,1)_c\to\CC \end{pmatrix}\,.
\]
Well-definedness of this operator follows directly from its geometric construction, for details we refer to~\cite{Pohl_hecke_infinite}. It can also easily be checked by a straightforward calculation. The transfer operator~$\tro s$ can be considered to act on various spaces. In what follows we present the spaces that we will use.

To simplify notation we use the following conventions: For any sets~$D_1, D_2$ we let 
\[
 D_1 \uplus D_2
\]
denote the \emph{disjointified union} of~$D_1$ and~$D_2$. \index[defs]{disjointified union}\index[symbols]{$\uplus$} If the sets~$D_1$ and $D_2$ are \emph{disjoint}, then their disjointified union equals their disjoint union, which we also denote by\index[symbols]{$\sqcup$}
\[
 D_1 \sqcup D_2\,.
\]
If the sets~$D_1, D_2$ are \emph{not disjoint initially}, then forming their disjointified union includes to \emph{disjointify}~$D_1$ and $D_2$ \index[defs]{disjointification} (i.\,e., to consider them as being disjoint) by, e.\,g., using the identifications $D_1\cong D_1\times\{1\}$ and $D_2\nobreak\cong\nobreak D_2\nobreak\times\nobreak\{2\}$. The disjointified union is then identified with the disjoint union of the disjointified sets, thus
\[
  D_1 \uplus D_2 \cong \bigl(D_1\times\{1\}\bigr) \sqcup \bigl(D_2\times\{2\}\bigr)\,.
\]
In what follows we will always suppress the identification of~$D_j$ with $D_j\times\{j\}$ ($j=1,2$) from the notation. Further, for open subsets~$D_1, D_2\subseteq\nobreak\RR$ or $D_1,D_2\subseteq\CC$ and $p\in \ZZ_{\geq 0}\cup\{\infty, \omega\}$ we will use throughout the natural identification\index[symbols]{C@$C^p(D_1\uplus D_2)$}
\[
 C^p(D_1\uplus D_2) \cong C^p(D_1) \times C^p(D_2)\,,\qquad f \leftrightarrow \big(f\vert_{D_1}, f\vert_{D_2}\big)\,.
\]

From~\eqref{stroGm} we immediately see that the transfer operator~$\tro s$ (for any~$s\in\CC$) acts as a linear map on any of the spaces
\[
C^p\big((-1,\infty)_c \uplus (\infty,1)_c\big)\,,\quad p\in\ZZ_{\geq 0}\cup\{\infty,\omega\}\,.
\]
We set \index[symbols]{D@$D_\RR$}
\begin{equation}\label{eq:def_DR}
 D_\RR \coloneqq  (-1,\infty)_c \uplus (\infty,1)_c
\end{equation}
and will make particular use of the space $C^\omega(D_\RR)$.\index[symbols]{C@$C^\omega(D_\RR)$}

We will also need domains of definition for~$\tro s$ that consists of extensions of the spaces~$C^p(D_\RR)$ to spaces of functions with complex domains. Working out the expression for the element~$\tro s f$ in~\eqref{stroGm} with~$f=(f_1,f_2)$ we find
\begin{align}\label{TO1}
 (\tro s f)_1(t) & =  |t+\lambda|^{-2s}f_1\left( \frac{-1}{t+\lambda} \right) + f_1(t+\lambda) + |t+\lambda|^{-2s} f_2\left(\frac{-1}{t+\lambda}\right)
 \intertext{for $t\in (-1,\infty)_c$, and}
 \label{TO2}
 (\tro s f)_2(t) & = |t-\lambda|^{-2s}f_1\left(\frac{1}{\lambda-t}\right) + |t-\lambda|^{-2s}f_2\left(\frac{1}{\lambda-t}\right) + f_2(t-\lambda)
\end{align}
for $t\in (\infty,1)_c$. We extend the automorphy factor~$|t+\lambda|^{-2s}$ in~\eqref{TO1} holomorphically to~$\CC\smallsetminus (-\infty,-\lambda]$ by 
\[
z\mapsto (z+\lambda)^{-2s}\,,
\]
and we extend the automorphy factor~$|t-\lambda|^{-2s}$ in~\eqref{TO2} holomorphically to the domain~$\CC\smallsetminus [\lambda,\infty)$ by 
\[
 z\mapsto (\lambda-z)^{-2s}\,.
\]
With these extensions, one easily sees that the transfer operator~$\tro s$ also acts on the spaces~$C^p(D_\CC)$, $p\in\RR_{\geq 0}\cup\{\infty,\omega\}$, where\footnote{Our discussion shows that we may be able to use larger domains in the function spaces. We refer to Section~\ref{sec:real_complex} below for an extended discussion. For our applications, $D_\CC$ is sufficient.} \index[symbols]{D@$D_\CC$}
\[
 D_\CC \coloneqq  \big(\CC\smallsetminus (-\infty,-1]\big)\uplus \big(\CC\smallsetminus [1,\infty)\big)\,.
\]
We will mainly use~$C^\omega(D_\CC)$. \index[symbols]{C@$C^\omega(D_\CC)$} (We recall that for~$C^\omega(D_\RR)$ the~$\omega$ refers to real-analytic functions, but for~$C^\omega(D_\CC)$ to holomorphic functions.) 

Another natural choice for the holomorphic extensions of the automorphy factors would be $\big((z\pm\lambda)^2\big)^{-s}$, which is holomorphic on $\CC\smallsetminus \bigl( \mp \lambda+i\RR\bigr)$. The difference between these two choices of holomorphic extensions is only the maximal domain to which the functions can be extended, not the values on the common domain.

We call eigenfunctions of~$\tro s$ with eigenvalue~$1$ also \emph{$1$-eigenfunctions}.\index[defs]{$1$-eigenfunction} The usage of~`$1$-' here is not related to the one for $1$-cohomology spaces. 

\subsection{Period functions}\label{sec:periodfunctions}
Let~$s\in\CC$ and let~$K\in \{\RR,\CC\}$. We say that a function~$f\in\nobreak C^\om(D_K)$ is a \emph{$K$-period function} for the spectral parameter~$s$ if it satisfies that $\tro s f=f$. If~$K=\RR$, then we call a $K$-period function also a \index[defs]{real period function}\index[defs]{period function!real}\emph{real period function}. Likewise, $\CC$-period functions are also called \index[defs]{complex period function}\index[defs]{period function!complex}\emph{complex period functions}. We let \index[symbols]{FE@$\FE_s^\om$}
\begin{equation}
\FE_s^\om(K)\coloneqq  \Bigl\{ f\in C^\omega(D_K)\setmid \tro s f = f \Bigr\}
\end{equation}
denote the space of all $K$-period functions. 

If $b\in C^\om(K)$, then $(-b,b)\in \FE_s^\om(K)$ if and only if $\tau_s(T)b=b$.
We define the subspace of \emph{boundary $K$-period functions}\index[defs]{boundary period function}\index[defs]{boundary $K$-period function}\index[defs]{period function!boundary} by \index[symbols]{BFE@$\BFE_s^\om$}
\begin{equation}
 \BFE_s^\om(K) \coloneqq  \Bigl\{ (-b,b) \setmid b\in C^\omega(K),\ \tau_s(T)b=b\Bigr\} \,.
\end{equation}
The spaces~$\FE_s^\omega(K)$ and~$\BFE_s^\om(K)$ are vector spaces. 

The space~$\FE_s^\om(\CC)$ can also be characterized as
\begin{equation}
 \FE_s^\om(\CC) = \left\{ f\in \FE_s^\om(\RR) \setmid \text{$f$ extends to an element in $C^\omega(D_\CC)$}\right\}\,,
\end{equation}
and hence can be understood as a subspace of~$\FE_s^\om(\RR)$. Indeed, every complex period function obviously restricts to a real period function, and hence
\[
 \FE_s^\om(\CC) \subseteq \left\{ f\in \FE_s^\om(\RR) \setmid \text{$f$ extends to an element in $C^\omega(D_\CC)$}\right\}\,.
\]
Conversely, if $f=(f_1,f_2)\in\FE_s^\om(\RR)$ is a real period function such that $f_1$ has a holomorphic extension~$\tilde f_1$ to~$\CC\smallsetminus(-\infty,-1]$ and $f_2$ has a holomorphic extension~$\tilde f_2$ to~$\CC\smallsetminus [1,\infty)$, then the identity theorem of complex analysis implies that the functional equation 
\[
 \begin{pmatrix} f_1 \\ f_2 \end{pmatrix} = \begin{pmatrix} \tau_s(T^{-1}S)f_1 + \tau_s(T^{-1})f_1 + \tau_s(T^{-1}S)f_2 \\ \tau_s(TS)f_1 + \tau_s(TS)f_2 + \tau_s(T)f_2 \end{pmatrix} = \tro s\begin{pmatrix} f_1 \\ f_2 \end{pmatrix}
\]
remains valid for the holomorphic extension~$\tilde f=(\tilde f_1,\tilde f_2)$ of~$f$. Thus $\tilde f=\tro s \tilde f$, and hence   
\[
 \FE_s^\om(\CC) \supseteq \left\{ f\in \FE_s^\om(\RR) \setmid \text{$f$ extends to an element in $C^\omega(D_\CC)$}\right\}\,.
\]

\subsection{Real and complex period functions}\label{sec:real_complex}

We now investigate the extendability properties of real period functions. In Proposition~\ref{prop-re} we will show that real period functions~$f\in \FE_s^\omega(\RR)$ uniquely extend to holomorphic functions~$\tilde f$ on certain larger domains in~$\RR$ preserving the property of being a $1$-eigenfunction of~$\tro s$. Moreover, we will show, in Proposition~\ref{prop-FE-RC}, that under certain conditions the domain of holomorphy can be chosen so large that the extension~$\tilde f$ of $f$ is a complex period function, thus an element of~$\FE_s^\omega(\CC)$. We start by discussing the limiting obstacles for such extensions. 

Each period function~$f=(f_1,f_2) \in \FE^\om_s(\RR)$ is, by definition, a $1$-eigenfunction of the transfer operator~$\tro s$. Written out, the relation $\tro sf = f$ becomes (see also~\eqref{TO1} and~\eqref{TO2})
\begin{align}
\begin{aligned}\label{FEeqa}
f_1(x) &= 
\tau_s(T^{-1}S) f_1(x) + \tau_s(T^{-1})f_1(x) + \tau_s(T^{-1}S)f_2(x)
\\
& = 
(\lambda+x)^{-2s}\,f_1\left( \frac{-1}{\lambda+x}\right) + f_1(x+\lambda) +
(\lambda+x)^{-2s}\,f_2\left( \frac{-1}{\lambda+x}\right)\,,
\end{aligned}
\intertext{and}
\begin{aligned}\label{FEeqb}
f_2(x) &= \tau_s(TS)f_1(x)  + \tau_s(T)f_2(x)+\tau_s(TS)f_2(x)
\\
&= (\lambda-x)^{-2s} f_1\left( \frac{1}{\lambda-x}\right) + f_2(x-\lambda)
 +(\lambda-x)^{-2s} f_2\left( \frac{1}{\lambda-x}\right)\,,
\end{aligned}
\end{align}
where $x\in(-1,\infty)$ in the first two equations~\eqref{FEeqa}, and $x\in(-\infty,1)$ in the latter two equations~\eqref{FEeqb}.

There are two types of obstacles limiting the domain to which a generic period function~$f$ extends real-analytically or holomorphically. The first one is the extendability of the automorphy factors in~\eqref{FEeqa}-\eqref{FEeqb} deriving from the~$\tau_s$-action, thus the domain of the complex logarithm, which we already alluded at in Section~\ref{sec:slowTO}.

Due to the factor~$(\lambda + x)^{-2s}$ in~\eqref{FEeqa}, the domain of extendability of~$f_1$ is contained in~$\CC\smallsetminus (-\infty, -\lambda]$. (We choose throughout the principal value for the logarithm, thus the logarithm is holomorphic on $\CC\smallsetminus (-\infty,0]$.) Analogously, the domain of any extension of~$f_2$ is contained in~$\CC\smallsetminus [\lambda,\infty)$. 

The second obstacle are the repelling fixed points of the hyperbolic elements acting in~\eqref{FEeqa}-\eqref{FEeqb} as well as the paths points in~$\CC$ take towards the attracting fixed points of the hyperbolic and parabolic elements in~\eqref{FEeqa}-\eqref{FEeqb}. The latter is an issue for extensions into the complex plane only. An extended discussion is provided right before  Proposition~\ref{prop-FE-RC}. We provide more details regarding the former issue. For simplicity of exposition, the discussion is restricted to real domains.

The image of the domain of~$f_1$ under the hyperbolic element~$ST$ acting in~\eqref{FEeqa} is compactly contained in the domain of~$f_1$ as well as of~$f_2$, and analogously for the image of the domain of~$f_2$ under the hyperbolic element~$ST^{-1}$ acting in~\eqref{FEeqb}:
\[
 \overline{ST(-1,\infty)},\ \overline{ST^{-1}(-\infty,1)} \subseteq (-1,\infty)\cap (-\infty,1)\,.
\]
Moreover, the parabolic elements~$T$ and~$T^{-1}$ acting in~\eqref{FEeqa} and~\eqref{FEeqb} pull the domain of~$f_1$ and~$f_2$ towards their attracting fixed point~$\infty$ through the domain of~$f_1$ and~$f_2$, respectively. We note that $\infty$ is a boundary point of both domains of definition. Therefore, using~\eqref{FEeqa}-\eqref{FEeqb} for bootstrapping (and ignoring for the moment possible restrictions from automorphy factors) allows us to simultaneously extend the domain of~$f_2$ until the repelling fixed point
\[
 \fixTS^+ = \frac{\lambda+\sqrt{\lambda^2-4}}2
\]
of~$ST^{-1}$ (which is the attracting fixed point of~$TS$), and the domain of~$f_1$ until the repelling fixed point
\[
 -\fixTS^+ = \frac{-\lambda-\sqrt{\lambda^2-4}}2
\]
of~$ST$ such that the extended functions remain real-analytic and still satisfy~\eqref{FEeqa} and~\eqref{FEeqb} on all of~$(-\fixTS^+,\infty)\times (-\infty,\fixTS^+)$. Beyond~$\pm\fixTS^+$ the actions of~$ST$ and~$ST^{-1}$, respectively, are expanding which yields that in general no further extension is possible.

In Proposition~\ref{prop-re} below we discuss a more general situation, considering regularities weaker than real-analyticity as well.

\begin{prop}\label{prop-re}
Let $p\in \NN_0\cup\{\infty,\omega\}$, let 
\begin{align*}
f_1&\colon (-1,\infty)\to \CC\,,
\\
f_2 & \colon (-\infty, 1)\to \CC
\end{align*}
be $p$~times continuously differentiable functions, and suppose that $(f_1,f_2)$ satisfies~\eqref{FEeqa} and~\eqref{FEeqb} on~$(-1,\infty)\times (-\infty,1)$. Then there are unique $p$~times continuously differentiable extensions
\begin{align*}
 \tilde f_1 &\colon (-\fixTS^+,\infty)\to\CC\,,
 \\
 \tilde f_2 &\colon (-\infty,\fixTS^+)\to\CC
\end{align*}
of~$f_1$ and~$f_2$, respectively, such that $(\tilde f_1,\tilde f_2)$ satisfies relations \eqref{FEeqa} and~\eqref{FEeqb} on~$(-\fixTS^+,\infty)\times(-\infty,\fixTS^+)$.
\end{prop}

\begin{proof}
Since $\fixTS^+$ is the attracting fixed point of~$TS$, and $-\fixTS^+$ is the attracting fixed point of~$T^{-1}S$, we have
\[
 \big(T^{-1}S\big)^n(-1) \searrow -\fixTS^+\quad\text{and}\quad \big(TS\big)^n1 \nearrow \fixTS^+
\]
as~$n\to\infty$.  We claim that for every~$n\in\NN$ there exist unique $C^p$~extensions 
\begin{align*}
f_{1,n}&\colon \left( \big(T^{-1}S)^n(-1),\infty\right) \to \CC\,,
\\
f_{2,n}&\colon \left( -\infty, \big(TS)^n1\right) \to\CC
\end{align*}
of~$f_1,f_2$, respectively, such that $(f_{1,n},f_{2,n})$ satisfies~\eqref{FEeqa}-\eqref{FEeqb} on 
\[
 \left( \big(T^{-1}S)^n(-1),\infty\right) \times \left( -\infty, \big(TS)^n1\right)\,.
\]
The limiting case~$n\to\infty$ then establishes the existence of~$(\tilde f_1, \tilde f_2)$ with the properties claimed in the statement of the proposition. We set~$f_{1,0}\coloneqq  f_1$, $f_{2,0}\coloneqq  f_2$, and iteratively (cf.~\eqref{FEeqa}-\eqref{FEeqb}) for~$n=1,2,3,\ldots$,
\begin{align*}
f_{1,n} &\coloneqq \tau_s( T^{-1}S ) f_{1,n-1} + \tau_s(T^{-1})f_{1,n-1} + \tau_s(T^{-1}S)f_{2,n-1}
\\
f_{2,n} &\coloneqq \tau_s(TS)f_{1,n-1} + \tau_s(TS)f_{2,n-1} + \tau_s(T)f_{2,n-1}\,.
\end{align*}
Since $(f_1,f_2)$ satisfy~\eqref{FEeqa}-\eqref{FEeqb} on~$(-1,\infty)\times (-\infty,1)$ by assumption, it immediately follows that $f_{1,1}$ and $f_{2,1}$ are well-defined and~$C^p$ on 
\[
 \left( \big(T^{-1}S\big)(-1),\infty\right)= (1-\lambda,\infty) \quad\text{and}\quad \left( -\infty, \big(TS)1\right)=(-\infty, \lambda-1)\,,
\]
respectively, and satisfy~\eqref{FEeqa}-\eqref{FEeqb} on these domains. Note that all automorphy factors in~\eqref{FEeqa} are well-defined and real-analytic on~$(-\fixTS^+,\infty)$, and all automorphy factors in~\eqref{FEeqb} are well-defined and real-analytic on~$(-\infty,\fixTS^+)$. 
Since we have, for each~$n\in\NN$,
\begin{align*}
 ST\left( \big(T^{-1}S\big)^n(-1),\infty\right) & = \left( \big(T^{-1}S\big)^{n-1}(-1),ST.\infty\right) \subseteq (1-\lambda,0)
 \\
 & \subseteq (1-\lambda,\infty)\cap (-\infty, \lambda-1)\,,
 \\
 T\left( \big(T^{-1}S)\big)^n(-1),\infty\right) & \subseteq (-1,\infty)\,,
 \\
 ST^{-1}\left(-\infty, \big(TS\big)^n1 \right) & \subseteq (1-\lambda,\infty)\cap (-\infty, \lambda-1)\,,
 \\
 T^{-1} \left(-\infty, \big(TS\big)^n1 \right) & \subseteq (-\infty, 1)\,,
\end{align*}
induction over~$n$ shows the claim. This completes the proof.
\end{proof}

For any period function~$f=(f_1,f_2)\in \FE_s^\om(\RR)$, the component functions~$f_1$ and~$f_2$ are
real-analytic, and hence have holomorphic extensions to some complex neighborhoods~$U_1$
of~$(-1,\infty)$ and~$U_2$ of~$(-\infty,1)$, respectively. Obviously, these holomorphic extensions satisfy~\eqref{FEeqa}-\eqref{FEeqb} as long as both sides of all equalities in~\eqref{FEeqa}-\eqref{FEeqb} are well-defined on the considered points.  

In general, there is no enforcement that the functions~$f_1$ and~$f_2$ have holomorphic extensions to all of~$\CC \smallsetminus(-\infty,-1]$ and~$\CC\smallsetminus [1,\infty)$, respectively. In turn, the real period function~$f\in\nobreak\FE_s^\om(\RR)$ does not necessarily extend to a complex period function~$\tilde f\in\nobreak\FE_s^\om(\CC)$. Bootstrapping using~\eqref{FEeqa}-\eqref{FEeqb} does not need to yield a holomorphic extension of~$f$ to all of~$\big(\CC \smallsetminus(-\infty,-1]\big)\times \big(\CC\smallsetminus [1,\infty)\big)$ because the actions of~$T$ (in~\eqref{FEeqa}) and of~$T^{-1}$ (in~\eqref{FEeqb}) do not contract towards the real axis, only towards~$\infty + i\RR$ (in contrast, the other elements~$ST$ and~$ST^{-1}$ acting in~\eqref{FEeqa} and~\eqref{FEeqb}, respectively, contract towards the real axis). Hence, in general, one cannot establish contraction into a neighborhood onto which holomorphic extendability is already secured. However, as soon as the neighborhoods~$U_1$ and~$U_2$ are left- and right-rounded at~$\infty$, respectively, bootstrapping is possible.

\begin{prop}\label{prop-FE-RC} 
Let $f=(f_1,f_2)\in \FE_s^\om(\RR)$. Suppose that~$f_1$ has a holomorphic extension to a complex neighborhood~$U_1$ of~$(-1,\infty)$ that is left-rounded at~$\infty$ and that $f_2$ has a holomorphic extension to a complex neighborhood~$U_2$ of~$(-\infty,1)$ that is right-rounded at~$\infty$. Then the map~$f$ extends uniquely to a complex period function~$\tilde f \in \FE^\om_s(\CC)$.
\end{prop}

For a proof of Proposition~\ref{prop-FE-RC} one can proceed as in~\cite[Proof of Proposition~3.7]{Adam_Pohl} (see also~\cite[Section~3.4]{Adam_Pohl}), where a geometric variant of bootstrapping is used. Alternatively, a proof can be provided by means of a bootstrapping analogously to~\cite[Chapter III.4]{LZ01} taking advantage of the explicit formulas for the acting elements in~\eqref{FEeqa}-\eqref{FEeqb}. For both variants it is important to note that during the bootstrapping procedure only certain products of the operators~$\tau_s(P)$ with $P \in \{ T^{-1}S, TS, T^{-1},T\}$ are applied to the pairs of holomorphic functions that arise in the process. In other words, only for a certain selected subset of elements~$g\in\Gamma$, the operators~$\tau_s(g)$ need to be able to act on holomorphic functions with certain domains. It is rather easy to see that for these elements, a joint holomorphic extension of the automorphy factors is possible. (Compare the discussion in Section~\ref{sec:holomext}.) It can be taken in agreement with~\eqref{FEeqa}-\eqref{FEeqb}.

\subsection{Fast transfer operators}\label{sec:fastTO}

The \emph{fast transfer operator}~$\ftro s$ \index[defs]{fast transfer operator}\index[defs]{transfer operator!fast}\index[symbols]{L@$\ftro s$} that we use here is given (initially only formally) as 
\begin{equation}\label{ftro}
 \ftro s = \begin{pmatrix}
 \sum\limits_{n \geq 1} \tau_s(T^{-n}S) & \sum\limits_{n \geq 1} \tau_s(T^{-n}S)
 \\[3mm]
 \sum\limits_{n\geq 1} \tau_s(T^n S) & \sum\limits_{n\geq 1} \tau_s(T^n S)
\end{pmatrix}\,.
\end{equation}
This definition differs from the fast transfer operator as developed in~\cite[\S3.3]{Pohl_hecke_infinite}. Therefore, before we specify the realm of the spectral parameter~$s$ and the domains on which~$\ftro s$ is an actual operator, we provide two ways of constructing~$\ftro s$. These two approaches are related but nevertheless provide different insights. 

At the current stage of exposition, both deductions should be understood as calculations on a symbolic level only. Taking into account the discussions of convergence and meromorphic extensions that are provided in Sections~\ref{sect-osav} and~\ref{sec:fastconv}, these calculations can easily be converted into proper ones. Both deductions of~$\ftro s$ stress the role of $1$-eigenfunctions. We will discuss the motivation in Section~\ref{sec:essential}.

\medskip

\subsubsection{Construction based on the fast transfer operator family in~\cite{Pohl_hecke_infinite}}\label{sec:fastTOP} 

As mentioned in Section~\ref{sec:discr}, the discrete dynamical system that gives rise to the fast transfer operator~$\ftroP s$ as developed in~\cite[\S3.3]{Pohl_hecke_infinite} arises from the discrete dynamical system~$(D,F)$ that gives rise to the slow transfer operator~$\tro s$ by an induction process on parabolic elements. For details we refer to~\cite{Pohl_hecke_infinite}. 

The fast transfer operator~$\ftroP s$ is given by  \index[symbols]{L@$\ftroP s$}
\begin{equation}\label{eq:ftroP}
 \ftroP s = 
\begin{pmatrix}
\tau_s(T^{-1}S) & \sum\limits_{n\in\NN}\tau_s(T^{-n}) & 0 & \tau_s(T^{-1}S)
\\[1mm]
\tau_s(T^{-1}S) & 0 & 0 & \tau_s(T^{-1}S)
\\[1mm]
\tau_s(TS) & 0 & 0 & \tau_s(TS)
\\[1mm]
\tau_s(TS) & 0 & \sum\limits_{n\in\NN}\tau_s(T^n) & \tau_s(TS)
\end{pmatrix}\,,
\end{equation}
acting on function vectors
\[
\begin{pmatrix}
f_1 \colon (-1,1)_c\to\CC
\\[1mm]
f_2 \colon (-1,\infty)_c\to\CC
\\[1mm]
f_3 \colon (\infty,1)_c\to\CC
\\[1mm]
f_4 \colon (-1,1)_c\to\CC
\end{pmatrix}\,.
\]
A straightforward calculation shows that the $1$-eigenspaces of~$\ftroP s$ and~$\ftro s$ are in bijection via the maps
\[
 (f_1,f_2,f_3,f_4) \mapsto  (f_1,f_4)
\] 
and
\[
 \big(f_1, \tau_s(T^{-1}S)(f_1+f_4), \tau_s(TS)(f_1+f_4), f_4\big)  \mapsfrom  (f_1,f_4)\,.
\]

\medskip

\subsubsection{Constructing the fast transfer operator from the slow transfer operator}\label{sec:constrRoelof}

The equation~$\tro sf=f$ for the period function $f=(f_1,f_2)$ can be written as
\begin{equation}\label{prto} 
\begin{cases} \left(1-\tau_s(T^{-1})\right) f_1 & \ceqq 
\tau_s(T^{-1}S)\, \left( f_1+f_2\right)\,,\\
\left( 1-\tau_s(T) \right) f_2 &\ceqq  \tau_s(TS) \,\left(f_1+f_2\right)\,.
\end{cases}
\end{equation}
Formally, the sum
\[
 \sum_{n\geq 0} \tau_s(T^{\pm n})
\]
is an inverse of $1-\tau_s(T^{\pm 1})$, respectively. For the period function $(f_1,f_2) $ this gives (again only formally)
\begin{align*}
\begin{pmatrix} f_1\\f_2\end{pmatrix}
&\stackrel{\hphantom{\eqref{prto}}}{=} \begin{pmatrix} \sum\limits_{n\geq 0}  \tau_s(T^{-n}) \, \bigl( 1-\tau_s(T^{-1})\bigr) f_1
\\[2mm]
  \sum\limits_{n\geq 0} \tau_s(T^n) \, \bigl( 1-\tau_s(T)\bigr) f_2
\end{pmatrix} \\
&
\stackrel{\eqref{prto}}= 
\begin{pmatrix}  \sum\limits_{n\geq 0} \tau_s(T^{-n}) \tau_s(T^{-1}S) (f_1+f_2)
\\[2mm]
 \sum\limits_{n\geq 0} \tau_s(T^{n}) \tau_s(TS) (f_1+f_2)
\end{pmatrix}\,,
\end{align*}
thus
\[
 f=\ftro s f\,.
\]
We stress that these calculations are only on a formal basis. Taking into account questions of convergence, we will see in Chapter~\ref{part:TO} that not all period functions are also $1$-eigenfunctions of~$\ftro s$.

\medskip

\subsubsection{Essential properties of the fast transfer operator}\label{sec:essential}

As shown in~\cite{Pohl_hecke_infinite}, there exists a Banach space~$\mc B$ on which, for~$\Rea s >\tfrac12$, the transfer operator~$\ftroP s$ acts as a nuclear operator of order zero. The map~$s\to\ftroP s$ extends meromorphically to all of~$\CC$. The Fredholm determinant of this transfer operator family equals the Selberg zeta function~$Z_X$ of~$\Gamma\backslash\uhp$, thus
\begin{equation}\label{szf_to}
 Z_X(s) = \det\left(1-\ftroP s\right)\,.
\end{equation}
Due to the relation between the zeros of~$Z_X$ and the resonances of the Laplacian on~$X$, the $1$-eigenfunctions of~$\ftroP s$ are of particular interest. 

In addition, a straightforward application of~\cite{FP_szf} in combination with~\eqref{szf_to} shows that, on a suitable Banach space, 
\begin{equation}
 Z_X(s) = \det\left(1-\ftro s\right)\,.
\end{equation}
See Section~\ref{sec:fredholm} for details. These two facts indicate that with regard to investigations of Laplace eigenfunctions, the transfer operators~$\ftroP s$ and~$\ftro s$ are interchangeable.

Furthermore, in~\cite{Adam_Pohl}, it is shown that the spaces of $1$-eigenfunctions of~$\ftroP s$ are isomorphic to certain spaces of $1$-eigenfunctions of~$\tro s$. The isomorphism in~\cite{Adam_Pohl} and the construction in Section~\ref{sec:constrRoelof} are closely related. This again indicates that for the investigation between automorphic forms and transfer operator eigenfunctions, the eigenfunctions with eigenvalue $1$ are the ones (possibly even the only ones) of interest, and the transfer operators~$\ftroP s$ and~$\ftro s$ are equally suitable.

\medskip

\subsubsection{Comments on convergence and meromorphic extension}

One easily sees that for~$s\in\CC$, $\Rea s > \frac12$, and $f=(f_1,f_2)\in C^\omega(D_\RR)$, the infinite sums in~\eqref{ftro} converge. Proceeding from this, and showing meromorphic continuation of the map~$s\mapsto \ftro s$ is by now standard. The key step is to relate the infinite sums in~\eqref{ftro} to the Hurwitz zeta function and to take advantage of the meromorphic continuation of the Hurwitz zeta function.

There is some flexibility for the function spaces on which the fast transfer operators~$\ftro s$ become actual operators. For various applications of transfer operators, e.\,g., for representing the Selberg zeta function as in~\eqref{szf_to}, the precise spaces are rather unimportant as long as they contain functions of sufficient regularity and sufficiently many $1$-eigenfunctions. For our applications, however, the choice of the spaces is of utmost importance. 

Furthermore, the so-called one-sided averages (see Section~\ref{sect-osav}) which govern the convergence and meromorphic continuation of the fast transfer operators are convenient for other purposes as well. Therefore, we will provide, in Sections~\ref{sect-osav} and~\ref{sec:fastconv}, a rather detailed discussion of the spaces of definition for~$\ftro s$, the convergence and meromorphic continuation.

\subsection{One-sided averages}\label{sect-osav}

For~$s\in\CC$ we set (initially only formally) \index[defs]{average!one-sided}\index[defs]{one-sided average}\index[symbols]{A@$\av{s,T}^\pm$}
\begin{equation}\label{def_averages}
 \av{s,T}^+ \coloneqq  \sum_{n\geq 0}\tau_s(T^{-n})\quad\text{and}\quad \av{s,T}^- \coloneqq  -\sum_{n\leq -1}\tau_s(T^{-n}) = - \sum_{n\geq 1}\tau_s(T^n)\,.
\end{equation}
As in~\cite{BLZm}, we call these operators `averages' although, due to the omission of any actual averaging, this terminology is not completely correct. We recall several results from~\cite[Section~4.2]{BLZm}.

Let $(\alpha,\beta)_c$ be an interval in~$\proj\RR$ with $\infty\in (\alpha,\beta)_c$. Suppose that $\varphi\in \V s \om((\alpha,\beta)_c)$. Then the map
\begin{equation}
 h(t) \coloneqq  \tau_s(S)\varphi(t) = |t|^{-2s}\varphi\left(-\tfrac1t\right)
\end{equation}
is real-analytic in~$0$. Hence, in a neighborhood of~$0$, the map $h$ has a (convergent) power series expansion
\begin{equation}\label{h_exp}
 h(t)=\sum_{m\geq 0} a_m t^m\,.
\end{equation}
Applying the one-sided averages to~$\varphi$ gives
\begin{align}
\av{s,T}^+\varphi(t)  &=  \sum_{n\geq 0} \varphi(t+n\lambda) = \sum_{n\geq 0} \frac{1}{|t+n\lambda|^{2s}} h\left( - \frac{1}{t+n\lambda}\right) \label{conv+}
\intertext{and}
\av{s,T}^-\varphi(t) & = -\sum_{n\geq 1} \varphi(t-n\lambda) = -\sum_{n\geq 1} \frac{1}{|t-n\lambda|^{2s}} h\left( -\frac{1}{t-n\lambda}\right)\,. \label{conv-}
\end{align}
Since~$h$ is bounded in a neighborhood of~$0$, the infinite series in~\eqref{conv+} and~\eqref{conv-} are compactly convergent for~$\Rea s > \tfrac12$. Further, since $h$ is real-analytic in a neighborhood of~$0$, it follows that for~$\Rea s>\tfrac12$,  
\begin{equation}\label{eq:Av_ra}
\av{s,T}^+ \varphi \in \V s \om\left( ( \alpha,\infty)_c\right)\,,\quad
\av{s,T}^- \varphi \in \V s \om \left( ( \infty, \beta+\lambda)_c\right)\,. 
\end{equation}
Moreover, if we extend the automorphy factors~$|t+n\lambda|^{-2s}$ in~\eqref{conv+} holomorphically to~$\CC\smallsetminus (-\infty,0]$ by~$(z+n\lambda)^{-2s}$, and extend the automorphy factors~$|t-n\lambda|^{-2s}$ in~\eqref{conv-} holomorphically to~$\CC\smallsetminus[\lambda,\infty)$ by~$(z-n\lambda)^{-2s}$, then the map~$\av{s,T}^+\varphi$ is holomorphic on a right half plane, and~ $\av{s,T}^-\varphi$ is holomorphic on a left half plane for $\Rea s>\tfrac12$.  

The functions in~\eqref{eq:Av_ra} are not necessarily real-analytic in~$\infty$. However, we have the asymptotic expansions \index[defs]{average!asymptotic expansion}
\begin{equation}\label{avasexp} 
\begin{aligned}
\av{s,T}^+\varphi(t) &\sim |t|^{-2s} \sum_{m\geq -1} c_m\,t^m && \quad\text{as $t\uparrow\infty$}\,,
\\
\av{s,T}^-\varphi(t) &\sim |t|^{-2s} \sum_{m\geq -1} c_m \, t^m && \quad\text{as $t\downarrow-\infty$}\,, 
\end{aligned} 
\end{equation}
\emph{with the same coefficients}~$c_m$ in both cases. The dependence of
the coefficients on the~$a_m$ is provided by~\cite[(4.12)]{BLZm}. In
particular~$c_{-1}$ is a multiple of the coefficient~$a_0$ in~\eqref{h_exp}.

If the coefficient~$a_0$ in~\eqref{h_exp} vanishes, and hence $h(t)=O(t)$ as~$t\to 0$, then~\eqref{conv+} and~\eqref{conv-} converge for all~$\Rea s>0$, and all of the above remains valid for any~$s\in\CC$ with~$\Rea s>0$. If~$a_0\not=0$ then we can define~$\av{s,T}^\pm\varphi$ by treating the term~$|t|^{-2s} a_0$ separately with the help of the Hurwitz zeta function. This gives a first order singularity at~$s=\frac12$, and defines 
\[
s\mapsto \av{s,T}^\pm \varphi = \av{s,T}^\pm\tau_s(S)h
\]
as a meromorphic function of~$s$ on the region~$\Rea s>0$ (for any fixed~$h$). We continue to denote this meromorphic continuation by~$\av{s,T}^\pm$. The asymptotic expansions in~\ref{avasexp} remain valid for the meromorphic continuations.

For any~$\Rea s>0$, we have
\begin{equation}\label{avrel}
 \av{s,T}^\pm \left( 1-\tau_s(T^{-1} )\right) \varphi \ceqq 
 \left(1-\tau_s(T^{-1})\right) \av {s,T}^\pm \varphi \ceqq  \varphi\,.
\end{equation}

\subsection{Convergence and meromorphic extension of fast transfer operators}\label{sec:fastconv}

The considerations on one-sided averages in Section~\ref{sect-osav} allow us to almost immediately establish the fast transfer operators as actual operators on certain function spaces. As in Section~\ref{sect-osav}, the infinite sums in~\eqref{ftro} (and hence~$\ftro s$) converge for~$s\in\CC$ with~$\Rea s$ sufficiently large when applied to elements in the considered function spaces, and then admit a meromorphic continuation. As above, meromorphic continuation of the map $s\mapsto\ftro s$ means that for any function vector~$f=(f_1,f_2)$ in the considered function space and any points~$x_1\in (-1,\infty)_c$, $x_2\in (\infty,1)_c$, the map
\[
 s\mapsto  \begin{pmatrix} \sum\limits_{n\geq 1} \tau_s(T^{-n}S)\big(f_1+f_2\big)(x_1) \\[2mm] \sum\limits_{n\geq 1}\tau_s(T^nS)\big(f_1+f_2)(x_2) \end{pmatrix}
\]
extends meromorphically. (We note that for $\Rea s\gg 1$, this is just the definition of~$\ftro s$, see~\eqref{ftro}.)

Here, we are only interested in spectral parameters~$s\in\CC$ with~$\Rea s \in\nobreak (0,1)$, and we therefore restrict all considerations to this domain. We set \index[symbols]{C@$C^{\omega,0}(D_\RR)$}
\begin{equation}
 C^{\omega,0}(D_\RR) \coloneqq  \left\{\, (f_1,f_2)\in C^\omega(D_\RR) \setmid f_1(0)+f_2(0) = 0 \,\right\}
\end{equation}
and let  \index[symbols]{O@$\Op$}
\[
 \Op\big(C^\omega(D_\RR), C^\omega(D_\RR) \big)
\]
denote the vector space of the linear operators $C^\om(D_\RR)\to C^\om(D_\RR)$.

\begin{prop}\label{fast_merom}
\begin{enumerate}[{\rm (i)}]
\item\label{fmi} For~$s\in\CC, \Rea s>\frac12$, Equation~\eqref{ftro} defines~$\ftro s$ as a linear operator on~$C^\omega(D_\RR)$.
\item\label{fmii} The map 
\[
\CC_{\Rea >\frac12} \to \Op\big(C^\omega(D_\RR), C^\omega(D_\RR) \big)\,,\quad s\mapsto \ftro s
\]
extends meromorphically to~$\CC_{\Rea > 0}$ with a pole at~$s=\frac12$ of order at most~$1$. 
\item\label{fmiii} For~$s\in\CC, \Rea s>0$, $s\not=\frac12$, the $1$-eigenfunctions of~$\ftro s$ in~$C^\omega(D_\RR)$ are in~$\FE_s^\om(\CC)$.
\item\label{fmiv} If the domain of the transfer operators~$\ftro s$ is restricted to~$C^{\omega,0}(D_\RR)$, then the map~$s\mapsto \ftro s$ is holomorphic on~$\CC_{\Rea > 0}$.
\item\label{fmv} For~$s\in\CC, \Rea s>0$, the $1$-eigenfunctions of~$\ftro s$ in~$C^{\omega,0}(D_\RR)$ are in~$\FE_s^\om(\CC)$.
\item\label{fmvi} For~$s\in\CC, \Rea s > 0$, the space~$\BFE_s^\om(\CC)$ is contained in the kernel of~$\ftro s$.
\end{enumerate}
\end{prop}

\begin{proof}
Let~$f=(f_1,f_2)\in C^\om(D_\RR)$. Then $\tau_s(S) \left(f_1+f_2\right) \in \V s \om \left((1,-1)_c \right)$ and $\tau_s(T^{-1}S) \left(f_1+f_2\right) \in \V s \om \left((1-\lambda,-1-\lambda)_c \right)$. From the considerations in Section~\ref{sect-osav} it follows that in the region of absolute
convergence we have
\[
\ftro s \begin{pmatrix} f_1\\ f_2 \end{pmatrix}
\ceqq  \begin{pmatrix} \av{s,T}^+ \tau_s(T^{-1}S)(f_1+f_2) \\[1mm]
-\av {s,T}^- \tau_s(S)(f_1+f_2)
\end{pmatrix}\,.
\]
The resulting vector~$(\tilde f_1,\tilde f_2)$ satisfies 
\[  
\tilde f_1\in C^\om \left((1-\lambda, \infty)_c\right)\,,\quad \tilde f_2\in C^\om\left(
(\infty,-1+\lambda)_c\right)\,. 
\]
This proves~\eqref{fmi} and~\eqref{fmiv}. Restricting appropriately to~$(-1,\infty)_c$ or to~$(\infty,1)_c$ gives a meromorphic family of operators in~$C^\om(D_\RR)$. This implies~\eqref{fmii}.

The function~$t\mapsto \tilde f_1(t)$ has a holomorphic extension to a right half-plane, and the map~$t\mapsto \tilde f_2(t)$ extends holomorphically to a left half-plane. Therefore, if $(\tilde f_1,\tilde f_2)$ happens to be equal to~$(f_1,f_2)$, then  Proposition~\ref{prop-FE-RC} yields $(f_1,f_2) \in \FE_s^\om(\CC)$. From this~\eqref{fmiii} and~\eqref{fmv} follow. If~$f\in \BFE_s^\om(\CC)$, then $f_1+f_2=0$, and hence~$\ftro s f =0$. This shows~\eqref{fmvi}.
\end{proof}

\subsection{Spaces of complex period functions}\label{sec:def_cplxperiod}

In Section~\ref{sec:periodfunctions} we defined the space of complex period functions with spectral parameter~$s\in\CC$ to be 
\[
 \FE_s^\om(\CC)  = \left\{\,f\in C^\om(D_\CC) \setmid f=\TO_s^\slow f\,\right\}\,.
\]
In the Introduction we further defined the subspaces~$\FE_s^{\om,1}(\CC)$ and~$\FE_s^{\om,0}(\CC)$ of the space~$\FE_s^\om(\CC)$ which will be needed for a transfer-operator based interpretation of resonant funnel forms and cuspidal funnel forms. Since the definitions of both subspaces involve the transfer operator~$\TO_s^\fast$, Proposition~\ref{fast_merom} indicates that the domains in~$\CC$ of the parameter~$s$ for which these definitions are used need to be chosen with care. We set \index[symbols]{FE@$\FE_s^{\om,1}(\CC)$}
\begin{equation}
 \FE_s^{\om,1}(\CC) \coloneqq \bigl\{\, f\in\FE_s^\om(\CC) \setmid f=\TO_s^\fast f\,\bigr\}
\end{equation}
for~$s\in\CC$, $\Rea s\in (0,1)$, $s\not=\frac12$, and \index[symbols]{FE@$\FE_s^{\om,0}(\CC)$}
\begin{equation}
 \FE_s^{\om,0}(\CC) \coloneqq \bigl\{\, f = (f_1,f_2)\in\FE_s^{\om,1}(\CC) \setmid f_1(0) = -f_2(0)\,\bigr\}
\end{equation}
for~$s\in\CC$, $\Rea s\in (0,1)$. It follows immediately from Proposition~\ref{fast_merom} that these spaces are well-defined.

\section{An intuition and some insights}\label{sec:heuristic}
\markright{8. AN INTUITION}

In the Introduction, right after the statement of Theorem~\ref{thmA_new}, we alluded at the fact that the relation between $1$-eigenfunctions of the transfer operator~$\tro s$ and Laplace eigenfunctions with spectral parameter~$s$ is essentially given by a certain integral transform, up to sophistifications due to problems of convergence. In this section we briefly present an intuition about the relation between $1$-eigen\-func\-tions of the slow transfer operator~$\tro s$, $1$-cocycles with suitable coefficient spaces, and Laplace eigenfunctions. We hope that the explanation illuminates the steps and choices in Chapters~\ref{part:cohom}-\ref{part:TO}, in particular the definitions in~\eqref{c1short}-\eqref{c2short}, and the approach of Proposition~\ref{prop:cocycle}. The isomorphisms in~\cite{BLZm, Moeller_Pohl, Pohl_mcf_Gamma0p, Pohl_mcf_general, Pohl_spectral_hecke} between Maass cusp forms for cofinite Fuchsian groups, parabolic $1$-cocycles and $1$-eigenfunctions of slow transfer operators (equivalently, period functions for Maass cusp forms) have the same structure as the isomorphisms proposed here, and can be based on the same intuition. We refer to~\cite{Pohl_Zagier} for an informal presentation of the isomorphisms in the situation of cofinite Fuchsian groups, elaborated for the example of the modular group~$\PSL_2(\ZZ)$. In what follows we also discuss the additional influence from the funnel of the Hecke triangle surfaces.

We recall from Figure~\ref{fig:slowTO} (on p.~\pageref{fig:slowTO}) the set~$C^*=C^*_1\cup C^*_2$ of representatives for a cross section~$\wh C$ of the geodesic flow on~$\Gamma\backslash\uhp$ which gives rise to the slow transfer operator family~$\bigl(\tro s\bigr)_s$. We refer to~\eqref{stroGm} for its explicit formula. We recall further that for any unit tangent vector~$v\in \shp$ we let~$\gamma_v$ denote the geodesic on~$\uhp$ determined by~$v$.

Let~$f=(f_1,f_2)$ be a $1$-eigenfunction of~$\tro s$, let~$u$ be a Laplace eigenfunction with spectral parameter~$s$, and let~$c$ denote a cocycle in a $1$-cohomology space to be specified in Chapter~\ref{part:cohom}-\ref{part:TO}. For the presentation of the intuition we consider these $1$-cohomology spaces to be some rather abstract vector spaces, and think about a $1$-cocycle to be a map~$c\colon \Xi\times\Xi\to V$ defined on a set~$\Xi$ and a $\Gm$-module and vector space~$V$ (both to be determined) such that~$c$ satisfies certain compatibility properties; the intuition helps to clarify all necessary sets, modules and properties. Suppose that there are isomorphisms (somewhat similar to those for cofinite Fuchsian groups) between period functions, the $1$-cohomology spaces, and the Laplace eigenfunctions under which~$f$, $u$, and~$c$ are isomorphic. 

The linking pin between~$f=(f_1,f_2)$, $u$ and~$c$ is the set~$C^*$ of representatives for the cross section~$\wh C$, and its~$\Gamma$-translates.
For~$j\in\{1,2\}$, we call the set
\begin{equation}
 S(C^*_j)\coloneqq  \left\{\, \gamma_v(\infty) \setmid v\in C^*_j\,\right\} \subseteq \proj\RR
\end{equation}
the \emph{shadow}\index[defs]{shadow} of~$C^*_j$ in~$\proj\RR$, and the set 
\begin{equation}
 B(C^*_j)\coloneqq  \left\{\, \gamma_v(0) \setmid v\in C^*_j \,\right\} \subseteq \uhp
\end{equation}
the \emph{base}\index[defs]{base} of~$C^*_j$. Thus, 
\begin{align*}
 S(C^*_1) &= (-1,\infty)\,, &B(C^*_1) &= -1 + i\RR_{>0}
 \intertext{and}
 S(C^*_2) & = (-\infty,1)\,, &B(C^*_2)&= 1 + i\RR_{>0}\,.
\end{align*}
The Laplace eigenfunction~$u$ is (essentially) identified with the family of integrals 
\begin{equation}
 t\mapsto \int_{gB(C^*_j)} \omega_s(u,t)\,,\qquad (g\in\Gamma,\ j\in\{1,2\})\,,
\end{equation}
where~$\omega_s$ is a one-form that will be defined in Section~\ref{sect-coaief}. One of the achievements in~\cite{BLZm} was to establish such an identification for cofinite Fuchsian groups.

For $j\in\{1,2\}$, the functions~$f_j$ are associated to~$C^*_j$. We may imagine~$f_j$ to be defined on~$C^*_j$ and to satisfy so many invariances that the actual domain of definition of~$f_j$ is the shadow~$S(C^*_j)$. 

The cocycle~$c$ is associated to the family 
\[
\Gamma C^* = \bigcup_{g\in\Gamma} gC^*
\]
in the following way: For any two endpoints~$\xi,\eta$ of~$\Gamma C^*$ we imagine $c(\xi,\eta)$ to `live' on the geodesic connecting~$\xi$ and~$\eta$, more precisely to be defined on the set of tangent vectors to this geodesic, and to obey so many invariances that $c(\xi,\eta)$ descends to a function on~$\proj\RR$. This idea is identical to the one above for~$f=(f_1,f_2)$. We make it more rigorous. 

Let 
\[
 \Xi \coloneqq  \Gamma 1 \cup \Gamma\infty
\]
denote the set of endpoints of~$\Gamma C^*$ in~$\partial\uhp$, and let $V=V_s$ denote a $\Gamma$-sheaf of functions on~$\proj\RR$ in the sense of Section~\ref{sec:sheaves}. The properties of these functions are specified in Chapters~\ref{part:cohom}-\ref{part:TO}. Following the idea above, we imagine~$c(-1,\infty)\vert_{(-1,\infty)}$ to be defined on~$C^*_1$, and we identify~$C^*_1$ with its base~$B(C^*_1)$ and all unit tangent vectors~$v$ based in~$B(C^*_1)$ such that~$\gamma_v$ ends in~$S(C^*_1)$. The base~$B(C^*_1)$ corresponds to the geodesic from~$-1$ to~$\infty$ and hence to apply the cocycle~$c$ to~$(-1,\infty)$, and $S(C^*_1)$ is the domain of~$c(-1,\infty)\vert_{(-1,\infty)}$. Combining this with the intuition about~$f=(f_1,f_2)$ from above, we use the identification
\begin{equation}\label{def_heur_f1}
 c(-1,\infty)\vert_{(-1,\infty)} = f_1\,.
\end{equation}
Analogous argumentation suggests
\begin{equation}\label{def_heur_f2}
 c(1,\infty)\vert_{(-\infty,1)} = -f_2\,.
\end{equation}
The `$-$' is caused by the orientation of~$C^*_2$, which is opposite to the one of~$C^*_1$.

If we want to find a definition of~$c(-1,\infty)$ on~$(-\infty,-1)$ that is consistent with this intuition, then we are guided to ask for the set of unit tangent vectors that are based on~$B(C^*_1)$ but are `on the side other than~$C^*_1$'. As indicated by Figure~\ref{fig:cocycle}, this set is not contained in the $\Gamma$-translates of~$C^*$. 
\begin{figure}
\centering 
\includegraphics[width=0.95\linewidth]{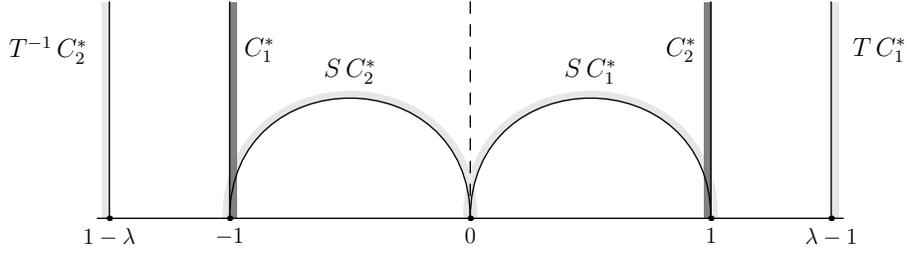}
\caption{The set~$C^*$ of representatives of cross section, and relevant \mbox{$\Gamma$-}translates.}\label{fig:cocycle}
\end{figure}
For that reason we take the next available set, namely~$T^{-1}C^*_2$. Hence
\begin{equation}\label{def_heur1}
 c(-1,\infty)\vert_{(-\infty,-1)} = -\tau_s(T^{-1})f_2\,.
\end{equation}
Again, the `$-$' is caused by the orientation of~$T^{-1}C^*_2$. Analogously, we are guided to set
\begin{equation}
 c(1,\infty)\vert_{(1,\infty)} = \tau_s(T)f_1\,.
\end{equation}
The relation in~\eqref{def_heur1} can also be argued in the following way: Instead of thinking of~$c(-1,\infty)$ to be related to the geodesic from~$-1$ and~$\infty$, we can use any path connecting~$-1$ and~$\infty$. In particular, we can use the geodesic from~$-1$ to~$1-\lambda$ and then the geodesic from~$1-\lambda$ to~$\infty$. Following the ideas above we find
\begin{equation}\label{vanish1}
 c(1-\lambda,\infty)\vert_{(-\infty,1-\lambda)} = -\tau_s(T^{-1})f_2
\end{equation}
and
\begin{equation}\label{vanish2}
 c(-1,1-\lambda)\vert_{(-\infty, 1-\lambda)\cup (-1,\infty)} = 0\,.
\end{equation}
The domain~$(1-\lambda,-1)$ is contained in a funnel interval, and hence representatives of periodic geodesics on~$\Gamma\backslash\uhp$ cannot end here. For that reason, the precise definition of cocycles on this interval always has to be adapted to the applications considered. Relations~\eqref{vanish1} and~\eqref{vanish2} give, on~$(-\infty,1-\lambda)$, the equality
\[
 c(-1,\infty) = c(-1,1-\lambda) + c(1-\lambda,\infty)\,.
\]

The leading idea is that we are allowed to exchange tangent vectors by `earlier' tangent vectors. This means that if the tangent vector~$v\in \shp$ is given, and $w$ is tangent to~$\gamma_v$ at some time~$t<0$, we are allowed to use~$w$ instead of~$v$. For example, for each vector on~$TC_1^*$ there is exactly one `earlier' vector on 
\[
 SC_1^* \cup SC_2^* \cup C_1^*\,.
\]
Thus, on~$(\lambda-1,\infty)$ we have
\begin{equation}\label{heur_cycle}
 \tau_s(T)c(-1,\infty) = \tau_s(S)c(-1,\infty) + \tau_s(S)c(1,\infty) + c(-1,\infty)\,.
\end{equation}
Alternatively, we find on~$(\lambda-1,\infty)$ the relation
\[
 c(\lambda-1,\infty) = c(1,0)+c(0,-1)+c(-1,\infty)\,.
\]
Using~\eqref{def_heur_f1} and~\eqref{def_heur_f2} in~\eqref{heur_cycle} we get
\[
 \tau_s(T)f_1 = \tau_s(S)f_1 + \tau_s(S)f_2 + f_1\,,
\]
which is recovering part of the properties of~$f$ being a $1$-eigenfunction of~$\tro s$. The remaining equality follows analogously.

We end this section with the example of how to find a formula for~$c(0,\infty)$. To that end, we consider the dashed line in Figure~\ref{fig:cocycle}. Constructing the tangent space to~$0+i\RR_{>0}$ from `earlier' vectors that we find on $\Gamma$-translates of~$C_1^*$ and~$C_2^*$ we get (compare with~\eqref{def_grpcoc} below)
\begin{equation}
 c(0,\infty) = 
 \begin{cases}
  \tau_s(S)f_2 + f_1 & \text{on $(0,\infty)$}
  \\
  -f_2 -\tau_s(S)f_1 & \text{on $(-\infty,0)$\,.}
 \end{cases}
\end{equation}

A major part of the discussions in Chapters~\ref{part:cohom}--\ref{part:TO} will be devoted to turn this intuition into rigorous proofs, to identify which modules~$V_s$ are needed and which properties have to be asked for the period functions to characterize the funnel forms, the resonant funnel forms, and the cuspidal funnel forms.

%% file: Mem-BP-partII.tex

\setcounter{sectold}{\arabic{section}}

\markboth{II. SEMI-ANALYTIC COHOMOLOGY}{II. SEMI-ANALYTIC COHOMOLOGY}

\chapter{Semi-analytic cohomology}\label{part:cohom}

\setcounter{section}{\arabic{sectold}}

For proving the isomorphisms claimed in 
Theorems~\ref{thmA_new} and \ref{thmB_new} between spaces of funnel forms and 
spaces of $1$-eigenfunctions of the slow transfer operators we use cohomology 
spaces as intermediate objects. In this chapter, we present definitions of these 
cohomology spaces. 

In Chapter~\ref{part:autom}, we will then establish linear bijections between 
spaces of funnel forms and these cohomology spaces, and, in Chapter~\ref{part:TO}, linear bijections between spaces of period functions and the same 
cohomology spaces. 

The cohomology theory we will use is a generalization of the standard group cohomology. The 
latter can be described with homogeneous cocycles on a set on which the group 
acts freely. As alluded at in Section~\ref{sec:heuristic}, we will use a variant of group cohomology based on subsets of~$\proj\RR$ on which the Hecke triangle group~$\Gm$ does \emph{not} act freely. We will provide a detailed construction in 
Section~\ref{sec:cohom_abstract}. A special case of this construction gives the parabolic cohomology spaces, which were
essential for~\cite{BLZm} and which were used in the transfer-operator-based characterizations of Laplace 
eigenfunctions in, e.\,g.,~\cite{BM09, Moeller_Pohl, Pohl_mcf_general}.

In Section~\ref{sec:modules} we will discuss the modules that we will use as values for the cocycle classes. These modules consist of (subspaces of the complex vector space of) \emph{semi-analytic} functions on the projective 
line~$\proj\RR$. These are real-analytic functions on~$\proj\RR$ that are 
allowed to have finitely many singularities. Depending on whether we seek to 
characterize funnel forms, resonant funnel forms or cuspidal funnel forms, we 
require additional properties modelled such that the cohomology classes 
correspond to the kind of funnel forms under investigation. 

We need the spaces of first cohomology only, for which reason we 
restrict major parts of the discussion to these. Throughout let~$\Gamma=\Gamma_\lambda$ be a Hecke triangle group. We remark that the constructions in Sections~\ref{sec:groupcohom}-\ref{sec:setcohom} apply without changes to arbitrary groups.

\section{Abstract cohomology spaces}\label{sec:cohom_abstract} 
\markright{9. COHOMOLOGY SPACES}

Let~$M$ be a~$\Gamma$-module that is a vector space over~$\CC$. Throughout we work with right $\Gamma$-modules, with the action denoted by~$v \mapsto v|g$ for~$v\in M$, $g\in\Gamma$. We refer to~\cite{Brown} for a general reference on group cohomology. 

\subsection{Standard group cohomology}\label{sec:groupcohom}

We use the standard description of the first cohomology space~$H^1(\Gm;M)$ of standard 
group cohomology of~$\Gm$ with values in the module~$M$ in terms of \emph{inhomogeneous} cocycles.\index[defs]{cocycle!inhomogeneous} The space of inhomogeneous $1$-cocycles is \index[symbols]{Z@$Z^1(\Gamma;M)$}
\begin{equation}
 Z^1(\Gm;M) \ceqq  \Bigl\{ \psi \setmid \Gm \rightarrow M \ \left\vert\  
\forall\,{g,h\in \Gm}\colon\psi_{gh} = \psi_g|h + \psi_h \right.\Bigr\}\,,
\end{equation}
the space of inhomogeneous $1$-coboundaries is \index[symbols]{B@$B^1(\Gamma;M)$}
\begin{equation}
B^1(\Gm;M) \ceqq  \Bigl\{ \psi\in Z^1(\Gm;M) \ \left\vert\ \exists\,{a\in M}\ 
\forall\,{g\in
\Gm}\colon \psi_g = a|(1-g) \right.\Bigr\}\,,
\end{equation}
and the first cohomology space is the quotient space \index[symbols]{H@$H^1(\Gamma;M)$}\index[defs]{cohomology!group}
\begin{equation}
H^1(\Gm;M) \ceqq  Z^1(\Gm;M) /B^1(\Gm;M)\,.
\end{equation}

\subsection{Cohomology on an invariant set}\label{sec:setcohom}

We construct cohomology spaces with cocycles on a set with a~$\Gamma$-action that is \emph{not necessarily free} by extending the construction of standard group cohomology with homogeneous cocycles on a set with a free~$\Gamma$-action. We use the discussion in~\cite[Section~5.1, Section~6]{BLZm} as a base, and adapt it to non-free~$\Gamma$-actions.

Let~$\Xi$ be a set with a~$\Gm$-action that does not need to be free. For 
traditional reasons we suppose that~$\Gm$ acts on~$\Xi$ from the left, and turn 
it into a right~$\Gamma$-action by taking inverses. 

\medskip

\subsubsection{Chain complex}

For~$i\in \NN_0$ let~$\CC[\Xi^{i+1}]$ \index[symbols]{C@$\CC[\Xi^{i+1}]$} be the
vector space of finite $\CC$-linear combinations of the form
\[ 
\sum_{(\xi_0,\ldots,\xi_i)\in \Xi^{i+1}} c_{\xi_0,\ldots,\xi_i}
 (\xi_0,\ldots,\xi_i)  
\]
such that $c_{\xi_0,\ldots,\xi_i}\in\CC$ for all~$(\xi_0,\ldots,\xi_i)\in 
\Xi^{i+1}$ and $c_{\xi_0,\ldots,\xi_i}=0$ for all but finitely many~$(\xi_0,\ldots,\xi_i)\in \Xi^{i+1}$. We endow~$\CC[\Xi^{i+1}]$ with the right~$\Gm$-action induced by \index[symbols]{$(\xi_0,\ldots,\xi_i)\vert g$}
\begin{equation}
(\xi_0,\ldots,\xi_i)|g = (g^{-1}\xi_0,\ldots,g^{-1}\xi_i)\,,
\end{equation}
where~$(\xi_0,\ldots,\xi_i)\in \Xi^{i+1}$, $g\in\Gm$.
Let 
\[
\partial_i\colon\CC[\Xi^{i+1}]\rightarrow \CC[\Xi^i] 
\]
denote the~$\CC$-linear $\Gm$-equivariant \emph{boundary map} \index[defs]{boundary map}
induced by
\begin{equation}\label{eq:partial} \partial_i(\xi_0,\ldots,\xi_i) \coloneqq  \sum_{j=0}^i (-1)^j
(\xi_0,\ldots, \widehat{\xi_j},\ldots, \xi_i)
\,,\end{equation}
and let~$\eps\colon\CC[\Xi]\rightarrow \CC$ be the \emph{augmentation map} \index[defs]{augmentation map} induced by
$\eps(\xi)=1$. In~\eqref{eq:partial}, the symbol $(\xi_0,\ldots, \widehat{\xi_j},\ldots, \xi_i)$ denotes the 
$i$-tuple that arises by omitting~$\xi_j$ from the \mbox{$(i+1)$-}tuple~$(\xi_0,\ldots, 
\xi_j,\ldots, \xi_i)$. Then
\[
\ldots \stackrel{\partial_3}{\longrightarrow} \CC[\Xi^3]
 \stackrel{\partial_2}{\longrightarrow} \CC[\Xi^2]
 \stackrel{\partial_1}{\longrightarrow} \CC[\Xi]
 \stackrel{\eps}{\longrightarrow} \CC \longrightarrow 0 
\]
is a chain complex.

\medskip

\subsubsection{Cohomology spaces}

We consider the induced complex \index[defs]{cochain complex}
\begin{align*}
& 0 \longrightarrow \Hom_{\CC[\Gm]}(\CC[\Xi];M)
\stackrel{d^0}{\longrightarrow} \Hom_{\CC[\Gm]}(\CC[\Xi^2];M)
\stackrel{d^1}{\longrightarrow} \ldots\,,
\end{align*}
where the coboundary maps~$d^\bullet$ are induced by the boundary maps~$\partial_\bullet$.
For~$i\in\NN_0$ we let \index[defs]{space of cocycles} \index[symbols]{Z@$Z^i_\Xi(\Gm;M)$}
\begin{equation}
 Z^i_\Xi(\Gm;M) \coloneqq  \ker d^i
\end{equation}
denote the space of \emph{$i$-cocycles of the cohomology on~$\Xi$}, \index[defs]{cocycle} and we let \index[defs]{space of coboundaries}\index[symbols]{B@$B^i_\Xi(\Gm;M)$}
\begin{equation}
 B^i_\Xi(\Gm;M) \coloneqq  \image d^{i-1}
\end{equation}
denote the space of \emph{$i$-coboundaries} (with the definition $B^0_\Xi(\Gm;M) \coloneqq  \{0\}$). \index[defs]{coboundary} The $i$-th cohomology space of the cohomology on~$\Xi$ is then the quotient space \index[defs]{cohomology space!on $\Gm$-invariant set} \index[symbols]{H@$H^i_\Xi(\Gm;M)$}\index[defs]{cohomology!on $\Gm$-invariant set}
\begin{equation}
H^i_\Xi(\Gm;M) = Z^i_\Xi(\Gm;M)/B^i_\Xi(\Gm;M)\,.
\end{equation}
Throughout, we let $[c]$ denote the element in~$H^i_\Xi(\Gm;M)$ \index[symbols]{$[c]$} that is represented by the cocycle~$c\in Z^i_\Xi(\Gm;M)$.

All $i$-cocycles are determined by maps~$\Xi^{i+1}\rightarrow M$. In the case~$i=1$ (the only case we will use), we can and shall identify the $1$-cocycles with the maps \index[symbols]{C@$c$}
\[
c\colon\Xi\times\Xi\to M
\]
that satisfy the \emph{cocycle relation} \index[defs]{cocycle relation}
\begin{equation}
 c(\xi,\eta)+c(\eta,\zeta) = c(\xi,\zeta)  
\end{equation}
for all~$\xi,\eta,\zeta\in\Xi$, and the 
\emph{$\Gm$-equivariance} \index[defs]{$\Gamma$-equivariance!cocycle}\index[defs]{cocycle!$\Gamma$-equivariance}
\begin{equation}
 c(\xi,\eta)|g = c(g^{-1}\xi,g^{-1}\eta)
\end{equation}
for all~$\xi,\eta\in\Xi$, $g\in \Gm$.

The $1$-coboundaries can be identified with those maps~$c\colon\Xi\times \Xi\to M$ for which there exists a \emph{$\Gm$-equivariant} \index[defs]{$\Gamma$-equivariant function}\index[defs]{$\Gamma$-equivariance!function} function~$f\colon \Xi\to M$ (i.\,e., $f(g^{-1}\xi) = f(\xi)|g$ for~$g\in\Gm$, $\xi\in\Xi$) such that
\begin{equation}
 c(\xi,\eta) = f(\xi) - f(\eta)
\end{equation}
for all~$(\xi,\eta)\in\Xi^2$.

\medskip

\subsubsection{Potentials, and relation to standard group cohomology}\label{sec:potentials}
Let $c$ be a cocycle in $ Z^1_\Xi(\Gm;M)$. A \emph{potential}~$p$ \index[defs]{potential} \index[symbols]{P@$p$} of~$c$ is a
map~$p\colon \Xi\rightarrow M$ such that~$c=dp$, hence
\[
c(\xi_1,\xi_2) = dp(\xi_1,\xi_2) = p(\xi_1)-p(\xi_2)
\]
for all~$\xi_1,\xi_2\in\Xi$. One easily checks the following properties:

\begin{enumerate}[{\rm (i)}]
\item\label{pot_shift} The set of potentials of~$c$ is nonempty and parametrized by~$M$. Indeed, for any choice~$\xi_0\in \Xi$, the map
\begin{equation}\label{pick_p}
 p\colon\Xi\to M\,,\quad p(\xi) \coloneqq  c(\xi,\xi_0)
\end{equation}
is a potential of~$c$ (satisfying the additional property~$p(\xi_0)=0$). If~$p$ is a potential of~$c$ then, for each~$b\in M$, the map 
\[
 p+b\colon \Xi \to M,\quad \xi\mapsto p(\xi) + b\,,
\]
is also a potential of~$c$. Conversely, the difference between any two potentials of~$c$ is constant, and given by an element in~$M$. 

\item The cocycle~$c$ has a $\Gamma$-equivariant potential if and only if~$c$ is a coboundary.

\item For any potential~$p$ of~$c$ and any~$g\in\Gm$, the quantity \index[symbols]{P@$\psi_g$} \index[defs]{group cocycle associated to potential}
\begin{equation}\label{cocpot} 
\psi_g \coloneqq  p(g^{-1} \xi) - p(\xi)|g
\end{equation}
does not depend on the choice~$\xi\in \Xi$. It defines an inhomogeneous group cocycle~$\psi \in Z^1(\Gm;M)$. Choosing another potential of~$c$ results in changing~$\psi$ by a group coboundary. In turn, the map \index[symbols]{P@$\Phi_\Xi$}
\begin{equation}\label{map_standcohom}
\Phi_\Xi\colon H^1_\Xi(\Gamma;M) \to H^1(\Gamma;M)\,,\quad [c] \mapsto [\psi]\,,
\end{equation}
where~$\psi\in Z^1(\Gamma;M)$ is formed by picking any representative~$c$ of~$[c]$ and any potential~$p$ of~$c$, is well-defined, linear and natural. 
\item The image of~$\Phi_\Xi$ consists of those group cohomology classes in~$H^1(\Gamma;M)$ that have for each~$\xi\in\Xi$ a representative~$\psi=\psi(\xi)$ which vanishes on the elements of the stabilizer group
\[
 \Gamma_\xi\coloneqq  \left\{ g\in\Gamma\setmid g^{-1}\xi=\xi \right\}
\]
of~$\xi$ in~$\Gamma$. Thus, for all~$g\in\Gamma_\xi$ we have $\psi_g=0$. We note that the group cocycle~$\psi$ may depend on~$\xi$.
\item An equivalent way to describe the assignment of~$[\psi]$ to~$[c]$ results from combining~\eqref{cocpot} with~\eqref{pick_p}. Then 
\begin{equation}
 \psi \colon\Gamma\to M\,,\quad \psi_g \coloneqq  c\big(g^{-1}\xi_0,\xi_0\big)
\end{equation}
for some~$\xi_0\in\Xi$. We remark that~$\psi$ depends on the choice of~$c$ and~$\xi_0$, its group cohomology class however does not.
\end{enumerate}

The map~$\Phi_\Xi$ provides a close relation between cohomology on $\Gm$-invariant sets and standard group cohomology. The following lemma shows that the relation is deeper than seen directly from~\eqref{map_standcohom}.

\begin{lem}\label{lem:cocgrp_inj}
The map~$\Phi_\Xi$ in~\eqref{map_standcohom} is injective. 
\end{lem}

\begin{proof}
Let~$c\in Z^1_\Xi(\Gamma;M)$, let~$p\colon\Xi\to M$ be a potential of~$c$, and suppose that the associated group cocycle~$\psi\colon\Gamma\to M$ is a coboundary. To show injectivity of~$\Phi_\Xi$ it suffices to show that~$c\in B^1_\Xi(\Gamma;M)$. 

Since $\psi\in B^1(\Gamma;M)$ there exists~$m\in M$ such that 
\[
\psi_g = m|(1-g) 
\]
for all~$g\in\Gamma$. We set
\[
 f\coloneqq  p-m\colon \Xi\to M\,,\quad f(\xi) = p(\xi)-m\,,
\]
and claim that~$f$ is a~$\Gamma$-equivariant potential of~$c$. Recalling~\eqref{pot_shift} from above, it suffices to show the~$\Gamma$-equivariance. Thus, for all~$g\in\Gamma$, $\xi\in\Xi$ we have
\[
 p(g^{-1}\xi) - p(\xi)|g = \psi_g = m|(1 - g)\,,
\]
and hence
\[
 f(g^{-1}\xi) = p(g^{-1}\xi) - m = \big( p(\xi) - m \big)|g = f(\xi)|g\,.
\]
This completes the proof.
\end{proof}

Lemma~\ref{lem:potentials} below provides a characterization of potentials, which we will take advantage of in the cohomological interpretation of automorphic forms. For a cocycle~$c\in Z^1_\Xi(\Gamma;M)$ we call any pair~$(p,\psi)$ consisting of a potential~$p$ of~$c$ and a group cocycle~$\psi$ defined as in~\eqref{cocpot} using~$p$ a \emph{pgc-pair associated to~$c$} (`\textbf{p}otential--\textbf{g}roup \textbf{c}ocycle pair'). \index[defs]{pgc-pair}

\begin{lem}\label{lem:potentials}
Let~$R$ be a set of representatives of~$\Gm\backslash \Xi$, i.\,e., $\Xi = \bigsqcup_{r\in R} \Gm r$. Further let~$\psi\in\nobreak Z^1(\Gamma;M)$ be a group cocycle. 
\begin{enumerate}[{\rm (i)}]
\item\label{lem:poti}
For each~$r\in R$ let~$p_r\in M$ be such that 
\begin{equation}\label{pot_well}
 \forall\, g\in\Gamma_r\colon p_r = p_r|g + \psi_{g}\,.
\end{equation}
Then
\begin{enumerate}[{\rm (a)}]
\item there exists a unique map~$p\colon \Xi\to M$ such that 
\begin{enumerate}[{\rm (1)}]
\item for all~$r\in R$ we have~$p(r) = p_r$,
\item for all~$\gamma\in\Gamma$, $\xi\in\Xi$ we have $p(\gamma^{-1}\xi) = p(\xi)|\gamma + \psi_{\gamma}$.
\end{enumerate}
\item there exists a (unique) cocycle in~$Z^1_\Xi(\Gamma;M)$ with potential~$p$.
\end{enumerate}
\item\label{lem:potii} If to $c\in Z^1_\Xi(\Gamma;M)$ is associated the pgc-pair~$(p,\psi)$, then for every~$r\in R$ the relation~\eqref{pot_well} holds with~$p_r\coloneqq  p(r)$.
\end{enumerate}
\end{lem}

\subsection{Relation to parabolic cohomology spaces}

The cohomological interpretation of automorphic forms for cofinite Fuchsian groups in~\cite{BLZm}, and their relation to eigenfunctions of transfer operators in~\cite{Moeller_Pohl, Pohl_mcf_general, Pohl_mcf_Gamma0p, BM09} takes advantage of parabolic cohomology spaces. We refer to~\cite{BLZm} for the definition of parabolic cohomology spaces, and will now briefly explain (for readers familiar with~\cite{BLZm}) the relation between parabolic cohomology spaces and cohomology spaces on sets with a~$\Gamma$-action.

The definitions and statements in Sections~\ref{sec:groupcohom}-\ref{sec:setcohom} apply to any group~$\Gamma$, even though we considered only Hecke triangle groups. In this section~$\Gamma$ necessarily is a Hecke triangle group. Further we let $\Xi = \Gm\infty \sqcup \Gm 1$. In Chapters~\ref{part:autom}-\ref{part:TO} we will use the cohomology spaces with this~$\Gamma$-invariant set.

The set~$\Xi$ is the disjoint union of the orbit of the cusp of~$\Gm$ and the orbit of one ordinary point. The subgroup~$\Gm_1$ fixing~$1$ is trivial, whereas the subgroup~$\Gm_\infty$ fixing~$\infty$ is generated by~$T$. The discussion in Section~\ref{sec:potentials} shows that the subspace of~$H^1(\Gamma;M)$ corresponding to~$H^1_\Xi(\Gm;M)$ under the map~$\Phi_\Xi$ from~\eqref{map_standcohom} consists of those group cohomology classes that contain a cocycle satisfying~$\psi_T=0$. Thus, $H^1_\Xi(\Gm;M)$ is canonically isomorphic to the \emph{parabolic cohomology space}~$H^1_{\textnormal{par}}(\Gm;M)$.\index[defs]{parabolic cohomology space}\index[defs]{cohomology space!parabolic}

\section{Modules}\label{sec:modules}\markright{10. MODULES}

For the cohomological interpretation of automorphic forms we will use the cohomology spaces defined in Section~\ref{sec:setcohom} with specific modules of \emph{semi-analytic functions} on~$\proj\RR$. We will provide detailed definitions in Sections~\ref{sect-sars}-\ref{sec:submodules}. In addition we will need cohomology spaces whose elements satisfy additional properties which cannot be modelled as properties of the modules. These properties will be discussed in Sections~\ref{sect-cococ}-\ref{sec:sing}. Throughout let~$\Gamma=\Gamma_\lambda$ be a Hecke triangle group and set \index[symbols]{X@$\Xi$}
\[
 \Xi \coloneqq  \Gamma1 \sqcup \Gamma\infty\,.
\]

\subsection{Modules of semi-analytic functions}\label{sect-sars}

We use a notion of semi-analyt\-icity on~$\Xi$ that is analogous to~\cite[Definition~10.2]{BLZm} where semi-analyticity was defined to mean that singularities may occur anywhere in $\proj\RR$ or were restricted to cuspidal points. However, the funnel present in our situation requires an adaptation, that we provide in what follows.

Recall the sheaf~$\V s \om$ of real-analytic functions from Section~\ref{sect-psa}. For any finite subset~$F\subset\proj\RR$, we let (as in~\cite[(2.21--22)]{BLZm}) \index[symbols]{Vaaah@$\V s \om[F]$} 
\begin{equation}
\V s \om[F] \coloneqq  \V s \om \bigl( \proj\RR \smallsetminus F \bigr)
\end{equation}
denote the linear space of real-analytic functions on~$\proj\RR\smallsetminus F$. We endow  
\[
 \mc I\coloneqq  \Bigl\{ \text{$F\subseteq \proj\RR$ finite}\Bigr\}
\]
with the structure of a directed set by setting $F_1\leq F_2$ if~$F_1\subseteq 
F_2$, we identify~$\V s \om[F_1]$ with its image in~$\V s \om[F_2]$ under the 
canonical embedding if $F_1\leq F_2$, and we let \index[symbols]{Vaaad@$\V s \fs(\proj\RR)$}
\begin{equation}
\V s \fs\!(\proj\RR) \coloneqq  \varinjlim \V s \om [F]\,
\end{equation}
denote the direct limit of the direct system~$(\V s \om[F])_{F\in\mc I}$. Since for any~$g\in\PSL_2(\RR)$ and~$F\in\mc I$ we have
\[
\tau_s(g) \V s \om[F] \ceqq  \V s \om[gF]\,,
\]
the space~$\V s \fs\!(\proj\RR)$ is a~$\PSL_2(\RR)$-module for the action~$\tau_s$. We 
remark that the space~$\V s \om[F]$ itself typically is \emph{not} a~$\PSL_2(\RR)$-module, not even a~$\Gamma$-module. We call~$\V s \fs\!(\proj\RR)$ a \emph{module of semi-analytic functions}. \index[defs]{module of semi-analytic functions}

For~$f\in \V s \fs\!(\proj\RR)$ we denote by \index[symbols]{B@$\bsing$}\index[defs]{singularities!boundary}\index[defs]{boundary singularities}
\begin{equation}
\bsing f \coloneqq  \bigcap \big\{ F\in \mc I \setmid f\in \V s \om[F]\big\}
\end{equation}
the minimal set~$F\in \mc I$ such that~$f\in \V s \om[F]$. We call $\bsing f$ the \emph{set of boundary singularities} of~$f$. The space of elements~$f\in \V s \fs\!(\proj\RR)$ with $\bsing f = \emptyset$ is identical to the~$\PSL_2(\RR)$-module~$\V s \om(\proj\RR)$, and hence $\V s \om(\proj\RR)$ is seen to be a submodule of~$\V s \fs\!(\proj\RR)$. For elements $f_1,f_2$ of~$\V s \fs\!(\proj\RR)$ we write \index[symbols]{$\equiv$}
\begin{equation}\label{eq:equiv_fns}
f_1\equiv f_2 \quad\text{if}\quad f_1-f_2\in \V s \om(\proj\RR)\,.
\end{equation}
This means that for any representatives~$\tilde f_j\in \V s \om[F_j]$ of~$f_j$, $j\in\{1,2\}$, the element
\[
\tilde f_1-\tilde f_2 \in \V s \om[F_1\cup F_2]
\]
extends to an element of~$\V s \om(\proj\RR)$.

We let \index[symbols]{Vaaad@$\V s \fxi(\proj\RR)$}
\begin{equation}
\V s \fxi(\proj\RR) = \varinjlim \V s \om [F]\qquad\text{ with }F \subseteq \Xi = \Gm\,\infty \cup\Gamma\, 1\,
\end{equation}
denote the~$\Gamma$-submodule of~$\V s \fs\!(\proj\RR)$ consisting of the elements~$f\in \V s \fs\!(\proj\RR)$ with 
\[
\bsing f\subseteq \Xi\,.
\]
The~$\Gamma$-module~$\V s \fxi(\proj\RR)$ will be crucial for the cohomological characterization of automorphic forms for the non-cofinite Hecke triangle group~$\Gm$. (See Chapter~\ref{part:autom} below.) We note that~$\V s \fxi(\proj\RR)$ is a~$\Gamma$-module, but not a~$\PSL_2(\RR)$-module. 

For the isomorphisms between cohomology spaces and eigenfunctions of transfer operators we will also need spaces of elements in the injective limit that are not necessarily defined on all of~$\proj\RR$. (See Chapter~\ref{part:TO} below.) For any open subset~$I\subseteq\nobreak\proj\RR$ we let \index[symbols]{Vaaad@$\V s \fxi(I)$}
\begin{equation}
 \V s \fxi(I) \coloneqq \varinjlim \V s \om \bigl( I \smallsetminus F \bigr)\qquad\text{ with }F \subseteq \Xi = \Gm\,\infty \cup\Gamma\, 1
\end{equation}
denote the space of restrictions of the elements in~$\V s \fxi(\proj\RR)$ to~$I$. We use the equivalence relation~$\equiv$ defined in~\eqref{eq:equiv_fns} also for these restricted elements but only requesting that $f_1-f_2\in\V s \om(I)$. We remark that typically the space~$\V s \fxi(I)$ is not a $\Gamma$-mod\-ule. (The assignment $I \rightarrow V^{\om(\Xi)}_s(I)$ for open $I\subset \proj\RR$ determines a presheaf, a property we will not use in this article.) Further, we will refer to elements of~$\V s \fxi(I)$ as `functions,' even though it is a slight abuse of the concept of an injective limit.

\subsection{Submodules of semi-analytic vectors}\label{sec:submodules}

We now define several submodules 
\[
\V s {\fxi;\cond_1;\cond_2}(\proj\RR)
\]
of~$\V s \fxi(\proj\RR)$ by imposing conditions on the type of singularities the 
considered functions may have at the points in~$\Xi$. The conditions~$\tcond_1$ are requirements on the type of singularities at the cuspidal points~$\Gamma\infty$, whereas the conditions~$\tcond_2$ are requirements on the singularities at the points of the orbit~$\Gm\, 1$. Throughout we use `smooth' to mean~$C^\infty$.

\begin{defn}\label{def:conditions}
Let $x_0\in\proj\RR$ and~$f\in \V s \fxi(\proj\RR)$.
\begin{enumerate}[{\rm (i)}]
\item We say that~$f$ has a \emph{simple singularity} \index[defs]{simple singularity} at~$x_0$ if the map
\[
 \proj\RR\smallsetminus\bsing f  \to \proj\RR\,,\quad x\mapsto m(x)f(x)
\]
with 
\[ 
m(x) \coloneqq  
\begin{cases}
(x-x_0) &\text{ if } x_0\in \RR\,,
\\
x^{-1}&\text{ if }x_0=\infty\,.
\end{cases}
\]
extends to a smooth (i.\,e., $C^\infty$)~function in a neighborhood of~$x_0$ in~$\proj\RR$.
\item We say that~$f$ is \emph{smooth} \index[defs]{smooth} at~$x_0$ if~$f$ has a smooth extension to 
a neighborhood of~$x_0$ in~$\proj\RR$.
\item We say that~$f$ has an \emph{analytic jump} \index[defs]{analytic jump} at~$x_0$ if there exist points~$\alpha,\beta\in\proj\RR$ such that~$x_0\in (\alpha,\beta)_c$ and the restrictions 
of~$f$ to~$(\alpha,x_0)_c$ and~$(x_0,\beta)_c$ are real-analytic, hence
\[
 f_\ell\coloneqq  f\vert_{(\alpha,x_0)_c} \in \V s \om\big( (\alpha,x_0)_c 
\big),\quad f_r \coloneqq  f\vert_{(x_0,\beta)_c} \in \V s \om\big( (x_0,\beta)_c 
\big)\,,
\]
and both functions~$f_\ell$ and~$f_r$ have a real-analytic continuation to a 
neighborhood of~$x_0$ in~$\proj\RR$. See Figure~\ref{fig:anjump}.
\begin{figure}
\centering
\includegraphics[width=0.6\linewidth]{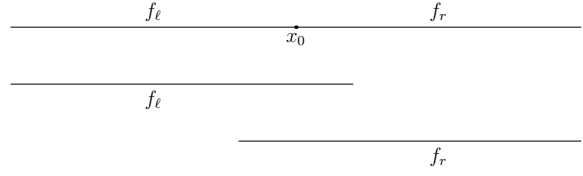}
\caption{Analytic jump.}\label{fig:anjump}
\end{figure}

\item We say that~$f$ satisfies the condition~\texc \index[defs]{condition!$\exc$}\index[defs]{exc@$\exc$}\index[symbols]{E@$\exc$} at~$x_0$ if~$f$ admits a 
holomorphic extension to an open subset of~$\proj\CC$ that is rounded at~$x_0$.
\end{enumerate}
\end{defn}

\begin{defn}\label{def:modulescond}
\begin{enumerate}[{\rm (i)}]
\item We let $\V s {\fxi;\smp;}(\proj\RR)$ \index[symbols]{Vaaaz@$\V s {\fxi;\smp;}(\proj\RR)$} \index[symbols]{S@$\smp$}\index[defs]{smp@$\smp$}\index[defs]{condition!$\smp$} denote the module consisting of the functions~$f\in\V s \fxi(\proj\RR)$ which, for all~$\xi\in\bsing(f) \cap \Gamma\infty$, have a simple 
singularity at~$\xi$.
\item We let $\V s {\fxi;\infty;}(\proj\RR)$ \index[symbols]{Vaaaz@$\V s {\fxi;\infty;}(\proj\RR)$} \index[defs]{condition!$\infty$} denote the module consisting of those functions~$f$ in $\V s \fxi(\proj\RR)$ which are smooth at all~$\xi\in\bsing(f)\cap\Gamma\infty$.
\item We let $\V s {\fxi;\exc;}(\proj\RR)$ \index[symbols]{Vaaaz@$\V s {\fxi;\exc;}(\proj\RR)$} denote the module consisting of those~$f\in\nobreak\V s \fxi(\proj\RR)$ which, for all~$\xi\in\bsing(f)\cap\Gamma\infty$, satisfy the 
condition~\texc at~$\xi$.
\item We let $\V s {\fxi;-;\aj}(\proj\RR)$ \index[symbols]{Vaaaz@$\V s {\fxi;-;\aj}(\proj\RR)$} \index[defs]{aj@$\aj$}\index[defs]{condition!$\aj$}\index[symbols]{A@$\aj$} denote the module consisting of those~$f\in\V s \fxi(\proj\RR)$ which, for all~$\xi\in\bsing(f)\cap\Gamma1$, have an analytic jump 
at~$\xi$.
\end{enumerate}
\end{defn}

\begin{rmk}
\begin{enumerate}[{\rm (i)}]
\item Modules formed of semi-analytic functions with simple singularities or satisfying smoothness or the condition~\texc are of utmost importance in~\cite{BLZm} and~\cite{BCD} as well. In~\cite[Section~9.5]{BLZm}, rounded neighborhoods are called \emph{excised} \index[defs]{excised} which explains the acronym for the condition~\texc.
\item Let $f\in \V s \fxi(\proj\RR)$. Since the set~$\bsing(f)$ of points at which~$f$ is not real-analytic is discrete, and $f\in \V s \omega[\bsing(f)]$, the function~$f$ satisfies all the properties defined in 
Definition~\ref{def:conditions} at all points~$x_0\notin\bsing(f)$. In turn, in the statement of Definition~\ref{def:modulescond}, the restriction to the set~$\bsing(f)$ can be omitted at all places. 
\item The properties defined in Definition~\ref{def:conditions} and the conditions in Definition~\ref{def:modulescond} are preserved under the action~$\tau_s(g)$ for each~$g\in\Gamma$. Due to this stability, the spaces in Definition~\ref{def:modulescond} are indeed $\Gamma$-modules. All of them are submodules of~$\V s \fxi(\proj\RR)$, and all contain~$\V s \om(\proj\RR)$.

A slight caveat when considering stability under the~$\Gamma$-action is due for the condition~\texc. In case that we investigate the validity of the condition~\texc at~$\infty$ and use the element~$S$ (see 
\eqref{standardchoice}) for the transformation into $\RR$, resulting in the 
transformed function 
\[
 t\mapsto |t|^{-2s} f\left(-\tfrac1t\right)\,,
\]
we need to use~$z\mapsto (z^2)^{-s}$ for the holomorphic extension of~$t\mapsto 
|t|^{-2s}$, which has its discontinuities at~$i\RR$ and hence does not interfere 
with testing the satisfiability of~\texc.
\end{enumerate}
\end{rmk}

\subsection{Conditions on cocycles }\label{sect-cococ}

We now define some properties of cocycles that we will need for the cohomological interpretation of automorphic forms for~$\Gm$ but that are not modelled as properties of a well-chosen~$\Gm$-module. We refer to Section~\ref{sec:heuristic}, where an intuition about these conditions is explained.  

\begin{defn}
Let $W\subseteq \V s \fxi(\proj\RR)$ be a~$\Gamma$-submodule. 
\begin{enumerate}[{\rm (i)}] 
\item By~$Z^1_\Xi(\Gm;W)_\sic$ \index[symbols]{Z@$Z^1_\Xi(\Gm;W)_\sic$}  we denote the subspace of~$Z^1_\Xi(\Gm;W)$ consisting of the cocycles~$c$ that satisfy \index[defs]{condition!$\sic$}\index[defs]{singularity condition}\index[defs]{condition!singularity} \index[defs]{sic@$\sic$}\index[symbols]{S@$\sic$}
\[
\bsing c(\xi,\eta) \;\subseteq\; \{\xi,\eta\}\qquad\text{ for all
}\xi,\eta\in \Xi\,. 
\]
Further, we let~$B^1_\Xi(\Gm;W)_\sic$ \index[symbols]{B@$B^1_\Xi(\Gm;W)_\sic$} be the space of coboundaries~$df \in B^1_\Xi(\Gm;W)$ given by $\Gm$-equivariant maps~$f:\Xi \rightarrow W$ that satisfy
\[
\bsing f(\xi) \subseteq \{\xi\}\qquad\text{ for all } \xi\in \Xi\,.
\]
We define the semi-analytic cohomology space with \emph{singularity
condition} as \index[symbols]{H@$H^1_\Xi(\Gm;W)_\sic$} 
\[
H^1_\Xi(\Gm;W)_\sic \coloneqq  Z^1_\Xi(\Gm;W)_\sic  /
B^1_\Xi(\Gm;W)_\sic\,.
\]

\item We let~$Z^1_\Xi(\Gm;W)^\van$ \index[symbols]{Z@$Z^1_\Xi(\Gm;W)^\van$} be the subspace of~$Z^1_\Xi(\Gm;W)$ consisting of those cocycles~$c$ that satisfy the \emph{vanishing condition} \index[defs]{vanishing condition} \index[defs]{condition!vanishing}\index[defs]{condition!$\van$}\index[defs]{van@$\van$}\index[symbols]{V@$\van$}
\begin{equation}\label{def_tvan}
c(1,\lambda-1)\vert_{(\lambda-1,1)_c}=0\,.
\end{equation}
Further, we let \index[symbols]{B@$B^1_\Xi(\Gm;W)^\van$}
\[
B^1_\Xi(\Gm;W)^\van\coloneqq  B^1_\Xi(\Gm;W) \cap Z^1_\Xi(\Gm;W)^\van
\]
and  \index[symbols]{H@$H^1_\Xi(\Gm;W)^\van$}
\[
 H^1_\Xi(\Gm;W)^\van \coloneqq  Z^1_\Xi(\Gm;W)^\van  / B^1_\Xi(\Gm;W)^\van\,.
\]
\item We let \index[symbols]{Z@$Z^1_\Xi(\Gm;W)^\van_\sic$} \index[symbols]{B@$B^1_\Xi(\Gm;W)^\van_\sic$} \index[symbols]{H@$H^1_\Xi(\Gm;W)^\van_\sic$}
\begin{align*}
Z^1_\Xi(\Gm;W)^\van_\sic &\coloneqq  Z^1_\Xi(\Gm;W)_\sic \cap Z^1_\Xi(\Gm;W)^\van\,,
\\
B^1_\Xi(\Gm;W)^\van_\sic &\coloneqq  B^1_\Xi(\Gm;W)_\sic \cap B^1_\Xi(\Gm;W)^\van
\intertext{and}
H^1_\Xi(\Gm;W)^\van_\sic & \coloneqq  Z^1_\Xi(\Gm;W)^\van_\sic  / B^1_\Xi(\Gm;W)^\van_\sic\,.
\end{align*}
\end{enumerate}
\end{defn}

The definition of~$B^1_\Xi(\Gm;W)^\van$ implies that the space~$H^1_\Xi(\Gm;W)^\van$ is a subspace of~$H^1_\Xi(\Gm;W)$, namely the one consisting of those cohomology classes that have a representative in~$Z^1_\Xi(\Gm;W)^\van$. Proposition~\ref{prop:char_sic} below shows that the analogous statement holds for~$H^1_\Xi\bigl(\Gm;\V s \fxi(\proj\RR)\bigr)_\sic$, and then also for~$H^1_\Xi\bigl(\Gm;\V s \fxi(\proj\RR)\bigr)^\van_\sic$.

\begin{rmk}
For any cocycle~$c\in Z^1_\Xi(\Gm;W)$, the vanishing condition~\tvan as defined in~\eqref{def_tvan} is equivalent to
\[
 c(1,\infty) = \tau_s(T)c(-1,\infty) \qquad\text{on $(\lambda-1,1)_c$\,.}
\]
We will not make use of this alternative characterization here, but we note that it consistent with the intuition as presented in Section~\ref{sec:heuristic} indicating that $c(1,\infty)$ and $\tau_s(T) c(-1,\infty)$ can be interchanged outside of the interval~$(1,\lambda-1)$.
\end{rmk}

\begin{prop}\label{prop:char_sic}
Let $W$ be a~$\Gamma$-submodule of~$\V s \fxi(\proj\RR)$. Then $H^1_\Xi(\Gm;W)_\sic$ is a subspace of~$H^1_\Xi\big(\Gm;W\big)$. It consists of the cohomology classes in~$H^1_\Xi(\Gm;W)$ that have a representative in~$Z^1_\Xi(\Gm;W)_\sic$.
\end{prop}

\begin{proof}
It suffices to show that
\[
B^1_\Xi(\Gm;W)_\sic \ceqq  Z^1_\Xi(\Gm;W)_\sic \cap B^1(\Gm;W)\,,
\]
of which the inclusion~$\subseteq$ is obvious. To establish the converse inclusion, let 
\[
 c\in Z^1_\Xi(\Gm;W)_\sic \cap B^1(\Gm;W)
\]
and fix a~$\Gamma$-equivariant function~$f\colon\Xi\to W$ such that~$c=df$. Let~$\xi\in\Xi$, and set 
\[
 S(\xi) \coloneqq  \bsing f(\xi)\,.
\]
For any~$\gamma\in\Gamma$ we have
\[
 c(\xi,\gamma\xi) = f(\xi) - f(\gamma\xi) = f(\xi) - \tau_s(\gamma)f(\xi)\,.
\]
From~$c\in Z^1_\Xi(\Gm;W)_\sic$, and hence 
\[
 c(\xi,\gamma\xi) \subseteq \{\xi,\gamma\xi\},
\]
it follows that
\[
 S(\xi) \subseteq S(\xi) \cap \big( \{\xi,\gamma\xi\} \cup \gamma S(\xi)\big) = S(\xi) \cap \Big( \{\xi\} \cup \gamma\big( S(\xi) \cup\{\xi\} \big)\Big)\,.
\]
Thus
\begin{equation}\label{xi_sing}
 S(\xi) \subseteq S(\xi) \cap \bigcap_{\gamma\in\Gamma} \Big( \{\xi\} \cup \gamma\big( S(\xi) \cup\{\xi\} \big)\Big)\,.
\end{equation}
Since~$S(\xi)\cup\{\xi\}$ is finite, we find~$\gamma\in\Gamma$ such that 
\[
 S(\xi) \cap \gamma\big( S(\xi) \cup \{\xi\} \big) = \emptyset\,,
\]
which, together with~\eqref{xi_sing}, implies
\[
 S(\xi) \subseteq \{\xi\}\,.
\]
This completes the proof.
\end{proof}

\subsection{Cohomological interpretation of the singularity condition}\label{sec:sing}

We now show that for certain parameters~$s\in\CC$, the two  spaces~$H^1_\Xi\bigl(\Gm;\V s \fxi(\proj\RR)\bigr)_\sic$ and $H^1\bigl(\Gm;\V s \om(\proj\RR)\bigr)$ are (naturally) isomorphic. This constitutes a cohomological interpretation of the singularity condition.

\begin{prop}\label{prop-sic} 
Let $s\in\CC$, $\Rea s\in (0,1)$, $s\neq \frac12$. Then the natural map \index[symbols]{P@$\Psi_\Xi$}
\[
\Psi_\Xi\colon H^1\bigl(\Gm;\V s \om(\proj\RR)\bigr) \to H^1_\Xi\bigl(\Gm;\V s \fxi(\proj\RR)\bigr)
\]
is injective. Its image is the space~$H^1_\Xi\bigl(\Gm;\V s \fxi(\proj\RR)\bigr)_\sic$.
\end{prop}

This characterization of~$H^1_\Xi\bigl(\Gm;\V s \fxi(\proj\RR)\bigr)_\sic$ will be used in the cohomological interpretation of automorphic forms in Chapter~\ref{part:autom}. In addition, it establishes a relation between the spaces of standard group cohomology  and certain spaces of cohomology on $\Gm$-invariant sets. A similar statement for certain cohomology spaces for the modular group~$\PSL_2(\ZZ)$ is contained in~\cite[Theorem 2.3]{BM09}.

Preparatory to the proof of Proposition~\ref{prop-sic} we provide a detailed definition of the map~$\Psi_\Xi$.

\medskip

\subsubsection{The natural map from analytic to semi-analytic cohomology}\label{def_natmap}

To simplify notation we set 
\[ 
 W_0 \coloneqq  \V s \om(\proj\RR)\quad\text{and}\quad W_1 \coloneqq  \V s \fxi(\proj\RR)\,.
\] 
Let~$\psi\in Z^1(\Gamma;W_0)$. We use Lemma~\ref{lem:potentials} to assign to~$\psi$ a cocycle in $Z^1_\Xi(\Gamma;W_1)$ as explained in what follows. 

Since $\Xi = \Gamma\infty \cup \Gamma1$, the set~$\{\infty,1\}$ serves as a set of representatives for~$\Gamma\backslash\Xi$. We set
\begin{equation}\label{repr_p}
 p_\infty\coloneqq \av{s,T}^+\psi_T \quad\text{and}\quad p_1\coloneqq  0\,,
\end{equation}
where~$\av{s,T}^+$ is the one-sided average from~\eqref{def_averages}. The discussion in Section~\ref{sect-osav} shows that~$p_\infty\in\V s \om[\infty]$ and 
\[
 \psi_T = \left( 1 - \tau_s\big(T^{-1}\big) \right)p_\infty\,.
\]
From the latter it easily follows that 
\[
 \psi_g = \left( 1-\tau_s\big(g^{-1}\big)\right) p_\infty
\]
for all~$g\in\Gamma_\infty= \{ T^n \setmid n\in\ZZ\}$. 

By Lemma~\ref{lem:potentials}, there is a unique cocycle~$c=c(\psi)\in Z^1_\Xi(\Gamma;W_1)$ with (unique) potential~$p\colon \Xi\to W_1$ such that~$p(\infty) = p_\infty$ and $p(1)=p_1$ (we remark that \emph{a fortiori} $\psi\in\nobreak Z^1(\Gamma;W_1)$). This construction provides a linear map 
\begin{equation}\label{map_cocycles}
 Z^1(\Gamma;W_0) \to Z^1_\Xi(\Gamma;W_1)\,.
\end{equation}
Further, if $\psi$ is a coboundary, then there exists $b\in W_1$ such that 
\[
 \Xi\to W_1\,,\quad \xi\mapsto p(\xi)-b\,,
\]
is~$\Gamma$-invariant, and hence $c=dp$ is a coboundary. Thus, the map in~\eqref{map_cocycles} restricts to a linear map
\begin{equation}\label{map_coboundaries}
 B^1(\Gamma;W_0) \to B^1_\Xi(\Gamma;W_1)\,.
\end{equation}
In turn, the maps in~\eqref{map_cocycles} and~\eqref{map_coboundaries} induce a (unique) linear map
\begin{equation}\label{map_spaces}
 \Psi_\Xi \colon H^1(\Gamma;W_0) \to H^1_\Xi(\Gamma;W_1)\,,\quad  [\psi] \mapsto [c(\psi)]\,.
\end{equation}
If we pick any elements~$q_\infty,q_1\in\V s \fxi(\proj\RR)$ such that 
\[
 \psi_T = \left( 1-\tau_s\big(T^{-1}\big)\right)q_\infty\,,
\]
then the potential~$q$ resulting from Lemma~\ref{lem:potentials} differs from~$p$ by a~$\Gm$-equivariant map, and hence the cocycle~$\tilde c\in Z^1_\Xi(\Gamma;W_1)$ associated to~$q$ and~$\psi$ is in the same cohomology class as~$c$. Thus, the map~$\Psi_\Xi$ in~\eqref{map_spaces} is indeed natural. 

\medskip

\subsubsection{Proof of Proposition 10.7}

For the proof of Proposition~\ref{prop-sic} we take advantage of the following result, 
provided by~\cite[Proposition~13.1]{BLZm}, on the possibility to separate the 
singularities of semi-analytic functions. It is a consequence of a refinement 
of Mittag-Leffler's Theorem on the simultaneous lifts of finite families of 
meromorphic functions, see, e.\,g.,~\cite[Theorem~1.4.5]{Ho}. 

\begin{prop}[{\cite[Proposition~13.1]{BLZm}}; separation of singularities]\label{prop-sepsing} 
For any function~$f\in \V s \om \bigl[\xi_1,\ldots,\xi_n\bigr]$ there are~$A_1,\ldots,A_n\in \V s \fs\!(\proj\RR)$ with
\[
\sum_{j=1}^n A_j \ceqq  f\,,\qquad \text{and}\qquad A_j \in \V s \om[\xi_j]
\quad\text{ for }j=1,\ldots,n\,.
\]
\end{prop}

\begin{proof}[Proof of Proposition~\ref{prop-sic}]
We set $W_0 \coloneqq  \V s \om(\proj\RR)$ and $W_1 \coloneqq  \V s \fxi(\proj\RR)$. We start by showing that the image of~$\Psi_\Xi$ is contained in~$H^1_\Xi(\Gamma;W_1)_\sic$, and that~$\Psi_\Xi$ is injective. For both claims let $[\psi]\in H^1(\Gamma;W_0)$, and let $c\in Z^1_\Xi(\Gamma;W_1)$ be the representative of~$\Psi_\Xi([\psi])$ that is constructed as in Section~\ref{def_natmap} from a representative~$\psi\in Z^1(\Gamma;W_0)$ for~$[\psi]$ and which has a potential~$p\colon\Xi\to W_1$ satisfying 
\[
p(\infty)=\av{s,T}^+\psi_T \quad\text{and}\quad p(1) = 0\,,
\]
see~\eqref{repr_p}. As seen in Section~\ref{def_natmap}, $p(\infty)\in \V s \om[\xi]$ for all~$\xi\in\Xi$. Therefore the cocycle~$c$ satisfies the singularity condition~\tsic. It now follows from Proposition~\ref{prop:char_sic} that $\Psi_\Xi([\psi]) \in H^1_\Xi(\Gamma;W_1)_\sic$.

To show the injectivity of the map~$\Psi_\Xi$, it suffices to prove that $\psi\in B^1(\Gamma;W_0)$ if $c \in B^1_\Xi(\Gamma;W_1)$. To that end suppose that~$c = dp$ is a coboundary. We note that the potential~$p$ is not necessarily $\Gm$-equivariant. However we find a $\Gm$-equivariant map~$f\colon\Xi \rightarrow W_1$ with~$c=df$. For all~$\xi\in\Xi$ we then have
\begin{equation}\label{eq:b}
 f(1) = f(\xi) - p(\xi)\,.
\end{equation}
We set $b\coloneqq -f(1)$. By Lemma~\ref{lem:potentials} we have for all~$g\in\Gamma$ and any~$\xi\in\Xi$:
\begin{align*}
 \psi_g & = p(g^{-1}\xi) - \tau_s(g^{-1})p(\xi)  
 \\
 & = f(g^{-1}\xi) - \tau_s\big(g^{-1}\big)f(\xi) + b - \tau_s\big(g^{-1}\big)b
 \\
 & =  \left( 1 - \tau_s\big(g^{-1}\big)\right)b\,.
\end{align*}
For establishing that~$\psi$ is a group coboundary, it remains to show that~$b\in W_0$. To that end we fix~$g\in\Gm$ such that 
\[
 g\,\bsing f(1) \cap \bsing f(1) = \emptyset\,.
\]
Since we have $\bsing p(\xi) \subseteq  \{\xi\}$ for all~$\xi\in\Xi$, it follows with~\eqref{eq:b} that
\begin{align*}
 \bsing f(1) & \subseteq \bsing p(\xi) \cup \bsing f(\xi) 
 \\
 & \subseteq \{\xi\} \cup \bsing f(\xi)
\end{align*}
for all~$\xi\in\Xi$. Thus
\begin{equation}\label{eq:xi_inter}
 \bsing f(1) \subseteq \bigcap_{\xi\in\Xi} \Bigl( \{\xi\}\cup\bsing f(\xi) \Bigr)\,.
\end{equation}
We now show that $\bsing f(1)=\emptyset$ by showing that the intersection in~\eqref{eq:xi_inter} is empty. To set end we set, for  each~$\xi\in\Xi$,
\[
 F(\xi) \coloneqq \{\xi\}\cup\bsing f(\xi)\,.
\]
The set~$F(\xi)$ a finite and contained in~$\Xi$ since $f$ maps into $W_1 = \V s \fxi(\proj\RR)$. Further, the $\Gamma$-equivariance of~$f$ yields that 
\begin{equation}\label{eq:Fg}
 F(g\xi) = gF(\xi)
\end{equation}
for each~$\xi\in\Xi$, $g\in\Gamma$. We pick $\xi_0\in\Xi$. Without loss of generality (using the finiteness of $F(\xi)$, $\xi\in\Xi$, and~\eqref{eq:Fg}) we may assume that $\infty\notin F(\xi_0)$. Hence $F(\xi_0)$ is a bounded subset of~$\RR$, and we find $n\in\ZZ$ such that 
\[
 T^n F(\xi_0) \cap F(\xi_0) = \emptyset\,.
\]
Using again~\eqref{eq:Fg}, we get $F(T^n\xi_0)\cap F(\xi_0) = \emptyset$, which implies
\[
 \bsing f(1) = \emptyset
\]
by~\eqref{eq:xi_inter}. Therefore $f(1)\in W_0$. Since $b=-f(1)$, it now follows that~$\psi$ is a coboundary, and hence~$\Psi_\Xi$ is injective.
 
To complete the proof of Proposition~\ref{prop-sic} it remains to show that the image of~$\Psi_\Xi$ exhausts~$H^1_\Xi(\Gamma;W_1)_\sic$. Let $c\in Z^1_\Xi(\Gamma;W_1)_\sic$. In what follows we will construct a group cocycle~$\psi\in Z^1(\Gamma;W_0)$ such that the construction from Section~\ref{def_natmap} gives a cocycle~$\tilde c\in Z^1_\Xi(\Gamma;W_1)_\sic$ in the same cohomology class as~$c$, which then finishes the proof. 
 
We recall that for each pair~$(\xi,\eta)\in \Xi^2$, the singularities of~$c(\xi,\eta)$ are contained in~$\{\xi,\eta\}$. By Proposition~\ref{prop-sepsing} we find and fix elements~$A_\xi^\eta\in \V s \om[\xi]$ and $A_\eta^\xi \in \V s \om[\eta]$ such that
\[
 c(\xi,\eta) = A^\eta_\xi - A_\eta^\xi\,.
\]
We choose~$A_\xi^\xi = 0$. The map~$q\colon\Xi\to W_1$, 
\[
 q(\xi) \coloneqq  c(\xi,\infty) + A_\infty^0
\]
is a potential of~$c$. The associated group cocycle~$\psi\in Z^1(\Gamma;W_1)$ is given by
\[
 \psi_{g^{-1}} = q(g\xi) - \tau_s(g)q(\xi)\,,
\]
which is independent of~$\xi\in\Xi$. 

We now show that~$\psi\in Z^1(\Gamma;W_0)$. For any three pairwise different elements~$\xi,\eta,\zeta\in \Xi$ we
rewrite the cocycle property of~$c$ as
\[ 
0 \ceqq  \bigl( A^\eta_\xi-A^\zeta_\xi\bigr) + \bigl(A_\eta^\zeta-A_\eta^\xi\bigr) + \bigl( A_\zeta^\xi-A_\zeta^\eta)\,.\]
The three terms on the right hand side have their singularities in~$\{\xi\}$, $\{\eta\}$,
$\{\zeta\}$, respectively. By an argument analogous to the one used above for establishing that~$b\in W_0$ we find 
\[
A_\xi^\eta\;\equiv\; A_\xi^\zeta\,,
\]
where~$\equiv$ denotes equality modulo~$W_0$ (see Section~\ref{sect-sars}). Analogously, we obtain 
\[
\tau_s(g) A_\xi^\eta \;\equiv\; A_{g\xi}^{g\eta}
\]
for all~$g\in\Gm$, all~$\xi,\eta\in\Xi$. It follows that for each~$\xi\in\Xi$,
\[
 q(\xi) = A_\xi^\infty- A_\infty^\xi+A_\infty^0 \equiv A_\xi^\infty\in \V s \om[\xi]\,.
\]
Further, $q(\infty)= A_\infty^0$ and $q(0)= A_0^\infty$. Therefore,
\begin{align*}
\psi_T &\ceqq  q(\infty) - \tau_s\big(T^{-1}\big)q(\infty) = A_\infty^0 - \tau_s(T^{-1})A_\infty^0 \;\equiv\;
A_\infty^0 - A_\infty^{-\lambda} \;\equiv\; 0\,,\\
\psi_S &\ceqq q(0)- \tau_s(S) q(\infty) \;\equiv\; A_0^\infty - A_0^\infty
\ceqq 0\,.
\end{align*}
Hence~$\psi\in Z^1(\Gm;W_0)$.

Applying now the construction from Section~\ref{def_natmap} to~$\psi$ and choosing~$p_\infty=q(\infty)$, $p_1=0$, yields a cocycle~$\tilde c\in Z^1_\Xi(\Gm;W_1)$ in the same cohomology class as~$c$. This completes the proof.
\end{proof}

%% file: Mem-BP-partIII.tex

\setcounter{sectold}{\arabic{section}}
\chapter{Automorphic forms and cohomology}\label{part:autom}

\setcounter{section}{\arabic{sectold}}

\markboth{III. AUTOMORPHIC FORMS AND COHOMOLOGY}{III. AUTOMORPHIC FORMS AND COHOMOLOGY}
  
In this chapter we will establish the following cohomological interpretation of funnel forms, resonant funnel forms and cuspidal funnel forms.

\begin{mainthm}\label{thm:cohominter}
For $s\in\CC$, $\Rea s\in (0,1)$, $s\not=\frac12$, 
\[
 \A_s \cong H^1_\Xi\bigl(\Gm; \V s {\fxi;\exc;\aj}(\proj\RR)\bigr)_\sic^\van
\]
and
\[
 \A^1_s \cong H^1_\Xi\bigl(\Gm; \V s {\fxi;\exc,\smp;\aj}(\proj\RR)\bigr)_\sic^\van\,.
\]
For $s\in\CC$, $\Rea s\in(0,1)$,
\[
 \A_s^0 \cong H^1_\Xi\bigl(\Gm; \V s {\fxi;\exc,\infty;\aj}(\proj\RR)\bigr)_\sic^\van\,.
\]
\end{mainthm}

The isomorphisms in Theorem~\ref{thm:cohominter} are not merely statements of existence or dimension equalities; we will provide explicit maps that realize these isomorphisms and provide further insights into them.  

For \emph{cofinite} Fuchsian groups, a cohomological interpretation of automorphic forms is provided by Bruggeman--Lewis--Zagier~\cite{BLZm}. We base our approach as much as possible on their work and put the emphasis of our discussion on the extension to funnels. To keep our paper to a reasonable length, we refer to results in~\cite{BLZm} whenever possible. However, we will always explain the general ideas of their results and hope that in this way the exposition is understandable also to readers unfamiliar with~\cite{BLZm}.

Our starting point is the cohomological interpretation, provided by~\cite{BLZm}, of $\Gamma$-invariant eigenfunctions with spectral parameter~$s$ of the Laplacian~$\Delta$ as cocycle classes in the first group cohomology space of~$\Gamma$ with values in~$\V s \om(\proj\RR)$. This interpretation, which we will recall in Section~\ref{sect-coaief}, is given by an explicit and injective integral transform that assigns to each element~$u\in\E_s^\Gamma$ a $1$-cocycle of the form 
\[
 \psi^u_g(t) = \int_{g^{-1}z_0}^{z_0} \omega_s(u,t) \qquad (g\in\Gamma,\ t\in\RR)\,,
\]
where~$\omega_s(u,\cdot)$ is a certain closed $1$-form on~$\uhp$ (evaluating~$u$ in Green's form against a Poisson-like kernel function) and where for the integration we fix an (arbitrary) base point~$z_0\in\nobreak\uhp$ and pick any path in~$\uhp$ from~$g^{-1}z_0$ to~$z_0$. A different choice of~$z_0$ changes~$\psi^u$ by a $1$-coboundary, and hence the cocycle~$\psi^u$ defines a cocycle class in~$H^1\bigl(\Gamma,\V s \om(\proj\RR)\bigr)$. 

We would like to determine the subsets of~$H^1\bigl(\Gamma,\V s \om(\proj\RR)\bigr)$ that correspond to the subsets~$\A_s^\ast$ ($\ast\in\{\underline{\ \ }, 0, 1\}$) of~$\E_s^\Gamma$ under this integral transform. For that we would need to characterize the behavior of the elements in~$\A_s^\ast$ at the ends of~$\Gamma\backslash\uhp$ (i.\,e., the cusp and the funnel) in terms of properties of the cocycle classes. 
For any hyperbolic surface~$N\backslash\uhp$, these two types of ends are caused by fundamentally different properties of~$N$: cusps of~$N\backslash\uhp$ are caused by the \emph{presence} of certain elements in~$N$ (namely, by parabolic elements), whereas funnels of~$N\backslash\uhp$ are caused by the \emph{lack} of sufficiently many elements in~$N$.

One of the major achievements of~\cite{BLZm} is to show that the behavior of \mbox{$N$-}in\-vari\-ant Laplace eigenfunctions at cusps can indeed be completely and constructively characterized by properties of the group cocycle classes at the parabolic elements responsible for the cusps. Even though the results of \cite{BLZm} are for hyperbolic surfaces of \emph{finite area}, we will see below that cusps of hyperbolic surfaces of \emph{infinite area} can be handled in the same way. 

In contrast, finding a characterization of the behavior of~$N$-invariant Laplace eigenfunctions at the funnels of~$N\backslash\uhp$ purely in terms of properties of group cocycles seems to be impossible due to the impossibility to describe funnels by the presence of specific elements in~$N$. The group cohomology of~$N$ is too coarse and too rigid for this purpose. 

The mixed cohomology spaces in~\cite{BLZm} turned out to be a flexible tool. We here use an extension of it that we will present in Section~\ref{sec:cohomtess} (restricted to Hecke triangle groups~$\Gamma$). This mixed cohomology is based on a tesselation of~$\uhp$ that is closely related to a suitable fundamental domain for~$\Gamma\backslash\uhp$. It allows us to treat the cusp, the funnel and the remaining compact part of~$\Gamma\backslash\uhp$ separately for all cohomological considerations. The behavior of Laplace eigenfunctions at the cusp and the funnel can be characterized by appropriate choices of modules in a mixed tesselation cohomology and some additional properties on the cocycle classes. See Section~\ref{sec:inj}-\ref{sect-coc-ff}. 

As a final step to complete the proof of Theorem~\ref{thm:cohominter} we will then show that these mixed tesselation cohomology spaces are naturally isomorphic to the cohomology spaces on the $\Gamma$-invariant set~$\Xi=\Gm 1\cup\Gm\infty$. See Section~\ref{sec:relacohom}.

\section{Invariant eigenfunctions via a group cohomology}
\label{sect-coaief}
\markright{11. INVARIANT EIGENFUNCTIONS VIA GROUP COHOMOLOGY}

In this section we recall the explicit map of the space
\[
\mc E_s^\Gamma = \bigl\{ \text{$f\colon \uhp\to\CC$ $\Gm$-invariant}\setmid \Delta f = s(1-s)f \bigr\}
\]
of $\Gamma$-invariant eigenfunctions with spectral parameter~$s$ of the Laplacian~$\Delta$ to the cohomology space~$H^1(\Gamma; \V s \om(\proj\RR))$ from~\cite{BLZm}. The basis tools to define the one-form~$\omega_s(u,t)$ in the introduction to this chapter are a Green's form and a Poisson-like kernel. 

For any functions~$u,v\in C^\infty(\uhp)$ we define their \emph{Green's
form} \index[defs]{Green's form}\index[symbols]{$\{\cdot,\cdot\}$} to be the smooth ($C^\infty$) differential form 
\begin{equation}\label{Gfacc2}
\begin{aligned}
 \bigl\{u,v\bigr\} & =  i \left( \frac{\partial u}{\partial
 z } \, v- u \frac{\partial
 v}{\partial z }\right)\, dz
 + i \left( u \,\frac{\partial v}{\partial
 \bar z }- \frac{\partial
 u}{\partial \bar z }\, v\right)\, d\bar z\\
 & = \left(\frac{\partial u}{\partial y}\,v - u\,\frac{\partial v}{\partial y}\right)\, dx 
 + \left(u\,\frac{\partial v}{\partial x} - \frac{\partial u}{\partial x}\,v \right)\,dy
 \,.
\end{aligned}
\end{equation}
The Green's form is $\PSL_2(\RR)$-equivariant, i.\,e., 
\[
\{u\circ g,v\circ g\} \ceqq  \{u,v\}\circ g\qquad (g\in\PSL_2(\RR))\,. 
\]
Moreover, if~$u$ and~$v$ are eigenfunctions of~$\Delta$ with the same
eigenvalue, then the form~$\{u,v\}$ is closed. The (parameter-free) \emph{Poisson-like kernel map} \index[defs]{Poisson-like kernel map} \index[symbols]{R@$R$} is given by 
\begin{equation}
 R\colon \RR\times\uhp\to\CC\,,\quad R(t;z) = \Ima \frac{1}{t-z}\,.
\end{equation}
For any~$s\in\CC$ and any fixed~$t\in\RR$ the map~$R(t;\cdot)^s$ is a~$\Delta$-eigenfunction with spectral parameter~$s$. For any fixed~$z\in\CC$, the function~$R(\cdot;z)^s$ determines an element of~$\V s \om(\proj\RR)$. Further, $R(\cdot;\cdot)^s$ is $\PSL_2(\RR)$-equivariant, which here means
\begin{equation}
 \tau_s(g)\big( R(\cdot;z)^s\big)(t) = R(t;gz)^s\,.
\end{equation}
See~\cite[(1.6), (1.7), (2.25)]{BLZm}. These properties imply that for~$u\in \E_s^\Gm$, $t\in\RR$ and~$z_1,z_2\in\uhp$ the integral 
\[
 \int_{z=z_1}^{z_2} \bigl\{ u, R(t;\cdot)^s\bigr\}
\]
does not depend on the choice of the path in~$\uhp$ from~$z_1$ to~$z_2$. The map
\begin{equation}\label{eq:integralc}
 \RR\to \CC\,,\quad t\mapsto \int_{z=z_1}^{z_2} \bigl\{ u, R(t;\cdot)^s\bigr\}
\end{equation}
determines an element~$c_u(z_1,z_2)\in \V s \om(\proj\RR)$ that satisfies 
\[
 c_u(g z_1,g z_2) \ceqq  \tau_s(g) \,
c_u(z_1,z_2)\qquad (g\in \Gamma)\,.
\]
Thus, for any choice of a base point~$z_0\in \uhp$ the map 
\begin{equation}
 \psi^{u,z_0} \colon \Gamma\to \V s \om(\proj\RR)\,,\qquad g\mapsto c_u(g^{-1}z_0,z_0)
\end{equation}
is an element of~$Z^1\bigl(\Gm;\V s \om(\proj\RR)\bigr)$. Its cohomology class~$[\psi^{u,z_0}]\in H^1\bigl(\Gm;\V s \om(\proj\RR)\bigr)$ does not depend on~$z_0$ and will therefore be denoted by~$[\psi^u]$. 

\begin{prop}[{\cite[Proposition~5.1]{BLZm}}]\label{prop:isoEcohom}
The map 
 \[
  \E^\Gm_s \rightarrow H^1\bigl(\Gm;\V s \om(\proj\RR)\bigr)\,,\quad u\mapsto [\psi^u]
 \]
is injective.
\end{prop}

\begin{rmk}\label{rmk:transition}
In~\cite{BLZm} the Green's form~$[\cdot,\cdot]$ \index[symbols]{$[\cdot,\cdot]$} is used, which is given by
\[
 [u,v] = \frac{\partial u}{\partial z}\, v\, dz + u\, \frac{\partial v}{\partial\overline z}\, d\overline z\,.
\]
It is related to~$\{\cdot,\cdot\}$ by
\[
 \bigl\{u,v\bigr\}= 2i \bigl[u,v\bigr]- i d(uv)\,.
\]
See~\cite[\S1.3]{BLZm}. Therefore, the cocycle class~$[\psi^u]$ associated to~$u\in\E_s^\Gamma$ in Proposition~\ref{prop:isoEcohom} differs to the one associated to~$u$ by~\cite[Proposition~5.1]{BLZm} by a multiplicative factor of~$2i$. 
\end{rmk}

\section{Tesselation cohomology}\label{sec:cohomtess}
\markright{12. TESSELATION COHOMOLOGY}

We follow~\cite[\S6,\S11]{BLZm} for the definition of cohomology spaces and cocycles that are based on a tesselation of $\uhp$ adapted to the action of~$\Gamma$. For brevity, we refer to these cohomology spaces as \emph{tesselation cohomology spaces}. Our choice of the tesselation (see below) separates the influence of the cusp from the influence of the funnel and therefore allows us to discuss separately the cohomological interpretation of properties of automorphic forms at cusps and funnels.

\subsection{Choice of a tesselation, and cohomology}\label{sec:tesselation}

We use the fundamental domain 
\begin{equation}
 \fd_1 \coloneqq \{ z\in\uhp \setmid \Rea z \in (0,\lambda),\ |z|>1,\ |z-\lambda|>1\}
\end{equation}
for~$\Gamma$ in~$\uhp$ (see Figure~\ref{fig-fd1}; see also Figure~\ref{fig-fd}, p.~\pageref{fig-fd}) for the construction of the tesselation. It has the advantage over choosing~$\fd_0$ (see Figure~\ref{fig-fd}) that~$\fd_1$ represents the set of ordinary points by a \emph{connected} subset of~$\RR$, namely by $[1,\lambda-1]$. We fix~$Y>1$ and divide~$\fd_1$ into three parts by cutting~$\fd_1$ along the euclidean lines from~$i$ to~$\lambda+i$ and from~$iY$ to~$\lambda+iY$ (see Figure~\ref{fig-fd1}).
\begin{figure}
\centering
\includegraphics[height=8cm]{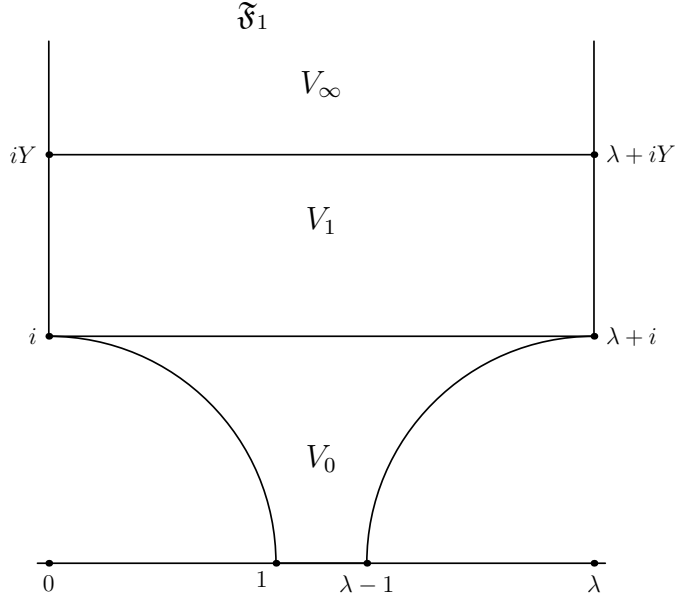}
\caption{Fundamental domain~$\fd_1$ used for the tesselation.
We consider the free side~$[1,\lambda-\nobreak 1]\subseteq\RR$ as part of
the boundary of the fundamental domain. To build the tesselation we
divide the fundamental domain into three parts by the euclidean lines
from~$i$ to~$\lambda+i$ and from~$iY$ to~$\lambda+iY$ for some (fixed)~$Y>1$. 
In this way we can separate the influence of the singularities at the
cusps from the influence of the singularities at the ordinary points.}
\label{fig-fd1}
\end{figure}
We denote the closures of these three parts in~$\proj\CC$ by~$V_0$, $V_1$ and~$V_\infty$ as indicated in Figure~\ref{fig-gen} (i.\,e., $V_0$ contains~$[1,\lambda-1]$, $V_\infty$ contains~$\infty$, and $V_1$ is the middle part). We extend the $\Gamma$-action on~$\uhp\cup\proj\RR$ to a $\Gamma$-action on the power set of~$\uhp\cup\proj\RR$ by considering the action of~$\Gamma$ as motions.  We use the tesselation\index[defs]{tesselation}\index[symbols]{T@$\tess$}
\begin{equation}
 \tess \coloneqq \bigl\{ gV_j \setmid  g\in\Gamma,\ j\in\{0,1,\infty\} \bigr\}
\end{equation}
of the \emph{extended upper half plane}\index[defs]{extended upper half plane}\index[defs]{upper half plane!extended}\index[symbols]{Had@$\uhp^*$} 
\[
 \uhp^* \coloneqq \uhp \cup \Gamma\infty \cup \Gamma\,[1,\lambda-1]\,.
\]
The tesselation cohomology is formed in analogy to simplicial cohomology, using the tesselation objects instead of simplices. In what follows we present the details.

For~$j\in\NN_0$ we let~$X^\tess_j$ denote the set of~$j$-dimensional objects associated to this tesselation, endowed with orientations as specified below. For~$j\geq 3$ we have $X^\tess_j=\emptyset$. The set~$X^\tess_0$ of \emph{vertices}\index[defs]{vertex of a tesselation}\index[defs]{tesselation!vertex}\index[symbols]{X@$X^\tess_0$} of~$\tess$ is 
\begin{equation}
\begin{aligned}
 X^\tess_0 &= \Gamma\,\{1,i,iY,\infty\}
 \\
 & = \Gamma\,1 \sqcup \Gamma\,i \sqcup \Gamma\,iY \sqcup \Gamma\,\infty\,. 
\end{aligned}
\end{equation}
The set~$X^\tess_1$ of \emph{edges}\index[defs]{edge of a tesselation}\index[defs]{tesselation!edge}\index[symbols]{X@$X^\tess_1$} of~$\tess$ is generated by the edges~$e_1,e_2,e_\infty, f_0,f_1,f_\infty$ indicated in Figure~\ref{fig-gen}. Thus,
\begin{equation}
 X^\tess_1 = \Gamma\,\{e_1,e_2,e_\infty,f_0,f_1,f_\infty\}\,.
\end{equation}
\begin{figure}
\centering
\includegraphics[height=8cm]{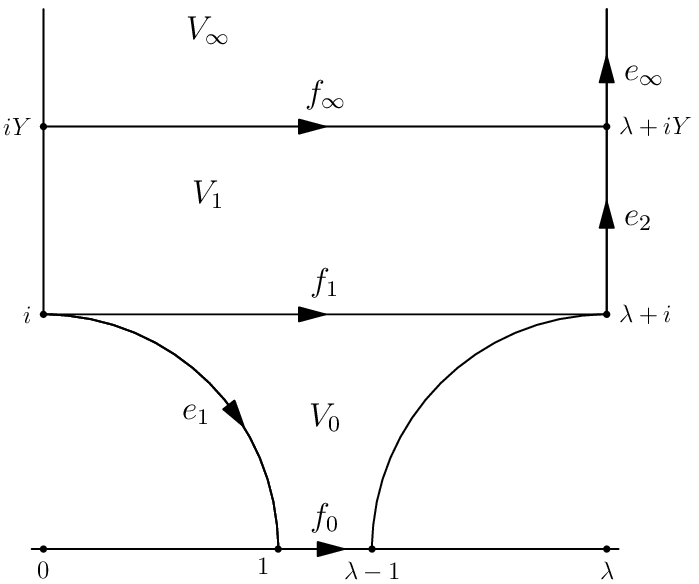}
\caption{Generating elements for the tesselation~$\tess$.} \label{fig-gen}
\end{figure}
For these basic edges we will use the orientations indicated in Figure~\ref{fig-gen}, and inherit it to all other elements of~$X^\tess_1$ via the~$\Gamma$-action. The choice of the specific orientation of the elements of~$X^\tess_1$ does not influence the qualitative structure of any of the following results.
Finally, the set~$X^\tess_2$ of \emph{faces}\index[defs]{face of a tesselation}\index[defs]{tesselation!face}\index[symbols]{X@$X^\tess_2$} is 
\begin{equation}
 X^\tess_2 = \Gamma\,\{V_0,V_1,V_\infty\}\,.
\end{equation}
We orientate the boundary of these basic faces~$V_0, V_1, V_\infty$ counterclockwise, and push this orientation to all other faces by the~$\Gamma$-action.
We stress that each of the sets~$X^\tess_j$, $j\in\NN_0$, consists of finitely many~$\Gm$-orbits and that the set~$\Xi$ used in Theorem~\ref{thm:cohominter} equals~$X^\tess_0\cap\proj\RR$.

We define the boundary maps~$\partial_j\colon \CC[X_j^\tess]\to\CC[X_{j-1}^\tess]$, $j\in\NN_0$, as well as the augmentation map as in simplicial (co-)homology, and we let $F^\tess_\bullet = \CC[X^\tess_\bullet]$ \index[symbols]{F@$F^\tess_\bullet$} denote the associated resolution. Throughout we will need only first cohomology spaces. We denote the space of coycles of~$F^\tess_\bullet$ with values in the~$\Gamma$-module~$M$ by~$Z^1(F^\tess_\bullet;M)$, \index[symbols]{Z@$Z^1(F^\tess_\bullet;M)$} the coboundary space by~$B^1(F^\tess_\bullet;M)$, \index[symbols]{B@$B^1(F^\tess_\bullet;M)$} and the cohomology space by~$H^1(F^\tess_\bullet;M)$. \index[symbols]{H@$H^1(F^\tess_\bullet;M)$}

We collect a few properties of~$F^\tess_\bullet$ and of the coycles in a tesselation cohomology. We recall that any element~$e\in X^\tess_1$ is an oriented edge between two points in~$X^\tess_0$, namely from its \emph{tail}~$t(e)$ \index[defs]{tail of an edge}\index[defs]{edge of a tesselation!tail}\index[symbols]{T@$t(e)$} to its \emph{head}~$h(e)$,\index[defs]{head of an edge}\index[defs]{edge of a tesselation!head}\index[symbols]{H@$h(e)$} which we shall denote also by~$e_{t(e),h(e)}$. \index[symbols]{E@$e_{t(e),h(e)}$} The element~$-e_{t(e),h(e)} \eqqcolon e_{h(e),t(e)}$ of~$\ZZ[X^\tess_1]$ is the edge with opposite orientation. We let $\CC_1[X^\tess_1]$ \index[symbols]{C@$\CC_1[X^\tess_1]$} denote the set of \emph{balanced} paths in~$X^\tess_1$, \index[defs]{balanced path} that is, the set of elements~$p\in\nobreak\CC[X^\tess_1]$ such that 
\[
 p = \sum_{j=1}^N \alpha_j b_j
\]
with~$N\in\NN$, $\alpha_j\in \CC$, $b_j\in X^\tess_1$ for all~$j\in\{1,\ldots, N\}$, $\alpha_1=1=\alpha_N$, and for all~$z\in X^\tess_0$
\[
 \sum_{\substack{j=1 \\ h(b_j)=z}}^{N-1} \alpha_jh(b_j) - \sum_{\substack{k=2 \\ t(b_k)=z}}^N \alpha_kt(b_k) = 0\,.
\]
In other words, $p$ can be understood as a path from~$t(b_1)$ to~$h(b_N)$ such that the weight of the path at its endpoints is~$1$ and at all inner points the sum of the weights of the incoming edges equals the sum of the weights of the outgoing edges. We set $t(p) \coloneqq t(b_1)$ and $h(p)\coloneqq t(b_N)$. A  cocycle in~$Z^1(F^\tess_\bullet;M)$ is completely determined by its values on~$\CC_1[X^\tess_1]$. Moreover, it is constant on all subsets 
\[
 \bigl\{ p\in \CC_1[X^\tess_1] \setmid \vphantom{\CC_1[X^\tess_1]} \big(t(p),h(p)\big) = (z_1,z_2)\bigr\}\qquad (z_1,z_2\in X^\tess_0)\,.
\]
These properties allow us to provide two alternative characterizations of the elements of~$Z^1(F^\tess_\bullet;M)$. The essence of both characterizations is to take advantage of the rigidity of the properties of cocycles and coboundaries to determine them completely by their properties on a subset of~$F^\tess_\bullet$ or a set closely related to~$F^\tess_\bullet$. The two characterizations are closely related but focus on slightly different properties. The first characterization indicates a close relation of tesselation cohomology to cohomology on~$\Gamma$-invariant sets (cf.~Section~\ref{sec:setcohom}).

\begin{enumerate}[1)]
\item Any map~$c\colon X^\tess_0\times X^\tess_0\to M$ that satisfies the cocycle relation 
\begin{equation}\label{cocrel2}
 c(z_1,z_2) + c(z_2,z_3) = c(z_1,z_3) \qquad (z_1,z_2,z_3\in X^\tess_0)
\end{equation}
and the~$\Gamma$-equivariance
\begin{equation}\label{cocequiv2}
 c(g\,z_1, g\,z_2) = g\,c(z_1,z_2) \qquad (z_1,z_2\in X^\tess_0,\ g\in\Gamma) 
\end{equation}
determines a unique cocycle~$\tilde c\in Z^1(F^\tess_\bullet;M)$ via
\begin{equation}\label{char2}
 \tilde c(p) = c(t(p),h(p))\qquad (p\in \CC_1[X^\tess_1])\,.
\end{equation}
The cocycle~$\tilde c$ from~\eqref{char2} is a coboundary if and only if there exists a \mbox{$\Gamma$-}equi\-variant map~$f\colon X^\tess_0\to\nobreak M$ such that for all~$z_1,z_2\in X^\tess_0$ we have
\begin{equation}\label{cobchar2}
 c(z_1,z_2) = f(z_1) - f(z_2)\,.
\end{equation}
Vice versa, every cocycle~$\tilde c$ defines a unique map~$c\colon X^\tess_0\times X^\tess_0\to M$ satisfying~\eqref{cocrel2} and~\eqref{cocequiv2}, by~\eqref{char2}. Moreover, if~$\tilde c$ is a coboundary, then~$c$ satisfies~\eqref{cobchar2}.
\item Alternatively, the elements of~$Z^1(F^\tess_\bullet;M)$ can be naturally identified with the \mbox{$\Gamma$-}equi\-variant  maps~$c\colon X^\tess_1\to M$ that satisfy $dc(V_j)=0$ for~$j\in\{0,1,\infty\}$. Coboundaries are identified with those maps~$c$ for which there exists a~$\Gamma$-equi\-variant map~$f\colon X^\tess_0\to M$ such that 
\begin{equation}
 c(e) = f(t(e))- f(h(e))
\end{equation}
for all~$e\in X^\tess_1$.
\end{enumerate}
In the following sections we will use these two characterizations of the elements in~$Z^1(F^\tess_\bullet;M)$ and $B^1(F^\tess_\bullet;M)$ without further discussions. As in Section~\ref{sec:setcohom} we associate to each cocycle~$c\in Z^1(F^\tess_\bullet;M)$ and each choice of~$z_0\in X^\tess_0$ a group cocycle~$\psi^{c,z_0}\in Z^1(\Gamma;M)$ by \index[symbols]{P@$\psi^c_g$}
\begin{equation}
 \psi^{c,z_0}_{g^{-1}} \coloneqq c(gz_0,z_0)  \qquad (g\in\Gamma)\,.
\end{equation}
(This definition may also be expressed as 
\[
\psi^{c,z_0}_{g^{-1}} = c(p_{gz_0,z_0})\,,
\]
where $p_{gz_0,z_0}$ is any element in~$\CC_1[X^\tess_1]$ with tail~$gz_0$ and head~$z_0$.) If~$c$ is a coboundary, then $\psi^{c,z_0}$ is a group coboundary. Clearly, the cocycle class of~$\psi^{c,z_0}$ does not depend on the choice of~$z_0$, for which reason we denote it also by~$[\psi^c]$, thus
\[
 [\psi^c] \coloneqq [\psi^{c,z_0}]\,.
\]
In turn, the map
\begin{equation}
 H^1(F^\tess_\bullet;M) \to H^1(\Gamma;M),\quad [c]\mapsto [\psi^c]
\end{equation}
is well-defined.

The resolution~$F^\tess_\bullet$ has a natural subresolution~$F^{\tess,Y}_\bullet$, \index[symbols]{F@$F^{\tess,Y}_\bullet$} built on the part of the tesselation that is completely contained in~$\uhp$: \index[symbols]{X@$X_0^{\tess,Y}$}\index[symbols]{X@$X_1^{\tess,Y}$}\index[symbols]{X@$X_2^{\tess,Y}$}
\begin{equation}
X_0^{\tess,Y} \coloneqq \Gm\; \{ i,i\nobreak Y\},\quad
X_1^{\tess,Y} \coloneqq \Gm\;\{e_1,e_2,f_1,f_\infty\},\quad\text{and}\quad
X_2^{\tess,Y} \coloneqq \Gm \;V_1.
\end{equation}
We denote the associated cocycle spaces and cohomology spaces by~$Z^1(F^{\tess,Y}_\bullet;M)$, $B^1(F^{\tess,Y};M)$ and~$H^1(F^{\tess,Y}_\bullet;M)$. \index[symbols]{Z@$Z^1(F^{\tess,Y}_\bullet;M)$}\index[symbols]{B@$B^1(F^{\tess,Y}_\bullet;M)$}\index[symbols]{H@$H^1(F^{\tess,Y}_\bullet;M)$} The elements of these classes can be characterized as above, restricting everything to~$F^{\tess,Y}_\bullet$, $X^{\tess,Y}_1$ and~$X^{\tess,Y}_0$.

\subsection{Relation to group cohomology}

Since~$F^\tess_\bullet$ is not a projective resolution, the cohomology spaces associated to~$F^\tess_\bullet$ are not (isomorphic to) the standard group cohomology spaces. However, the subresolution~$F^{\tess,Y}_\bullet$ is a projective resolution and hence the cohomology space~$H^1\bigl(F_\bullet^{\tess,Y};\V s \om(\proj\RR)\bigr)$ is naturally isomorphic to the space~$H^1\bigl(\Gamma;\V s \om(\proj\RR)\bigr)$. See \cite[Section~11.2]{BLZm}. 

We provide here a hands-on proof of the isomorphism  $H^1\bigl(F_\bullet^{\tess,Y};\V s \om(\proj\RR)\bigr) \cong \allowbreak
 H^1\bigl(\Gamma;\V s \om(\proj\RR)\bigr)$, using an adapted version of Lemma~\ref{lem:potentials}. We recall from Section~\ref{sec:tesselation} that $[\psi^c]$ denotes the group cocycle class associated to the cocycle class~$[c]$ in~$H^1(F^{\tess,Y}_\bullet;M)$ for any $\Gamma$-module~$M$.

\begin{lem}\label{lem:iso_cohom}
Let~$s\in\CC$. Then the map 
\[
H^1\bigl(F_\bullet^{\tess,Y};\V s \om(\proj\RR)\bigr) \to H^1\bigl(\Gamma;\V s \om(\proj\RR)\bigr)\,,\quad [c]\mapsto [\psi^c]
\]
is bijective.
\end{lem}

\begin{proof}
To establish surjectivity, let $[\psi]\in H^1\bigl(\Gamma;\V s \om(\proj\RR)\bigr)$ and pick a representative~$\psi\in Z^1\bigl(\Gamma;\V s \om(\proj\RR)\bigr)$. We construct a cocycle~$c\in Z^1\bigl(F_\bullet^{\tess,Y};\V s \om(\proj\RR)\bigr)$ such that $\psi^c=\psi$. To that end we recall that $X_0^{\tess,Y} = \Gamma\{i,iY\}$. We set 
\[
q_i \coloneqq \frac12\psi_S \quad\text{and}\quad q_{iY} \coloneqq 0 \in \V s \om(\proj\RR)\,. 
\]
Then the necessary analogue of~\eqref{pot_well} is satisfied: For all~$r\in\{i, iY\}$ and all~$g\in\Gamma_r$ we have
\begin{equation}\label{q_well}
 q_r = \tau_s(g)q_r + \psi_g\,.
\end{equation}
(We note that $\Gamma_i = \{\id, S\}$ and $\Gamma_{iY} = \{\id\}$.) We define $q\colon X_0^{\tess,Y} \to \V s \om(\proj\RR)$ by 
\begin{align*}
 q(gi) &\coloneqq \psi_{g^{-1}} + \tau_s(g)q_i
 \\
 q(giY) & \coloneqq \psi_{g^{-1}} + \tau_s(g)q_{iY}
\end{align*}
for all~$g\in\Gamma$. Using~\eqref{q_well} it follows that~$q$ is well-defined. Then~$c\coloneqq dq$ is an element of~$Z^1(F_\bullet^{\tess,Y}; \V s \om(\proj\RR))$, and $\psi^c=\psi$. 

A straightforward calculation shows that other choices for the representative~$\psi$ of~$[\psi]$ or for~$q_i$ and~$q_{iY}$ (obeying~\eqref{q_well}) lead to cocycles in the same cocycle class as~$c$.  Thus, the constructed cocycle class~$[c]\in H^1\bigl(F_\bullet^{\tess,Y};\V s \om(\proj\RR)\bigr)$ is independent of all choices, and is a preimage of~$[\psi]$. This shows surjectivity.

For injectivity it suffices to show that whenever we have $c\in Z^1\bigl(F_\bullet^{\tess,Y};\V s \om(\proj\RR)\bigr)$ and $z_0\in X_0^{\tess,Y}$ such that $\psi\coloneqq  \psi^{c,z_0}\in B^1(\Gamma;\V s \om(\proj\RR))$, then $c\in B^1(F_\bullet^{\tess,Y}; \V s \om(\proj\RR))$. Suppose that $c$, $z_0$ and~$\psi$ are of this form. Consider the potential
\[
 q\coloneqq  X_0^{\tess,Y} \to \V s \om(\proj\RR),\quad q(x) \coloneqq c(e_{x,z_0})
\]
of~$c$, where $e=e_{x,z_0}$ is any element in~$\CC_1[X^{\tess,Y}_1]$ with $\big(t(e),h(e)\big) = (x,z_0)$, and recall that 
\begin{equation}\label{eq:how_grpcoc}
 \psi_{g^{-1}} = q(gx) - \tau_s(g)q(x)
\end{equation}
for all~$g\in\Gamma$, $x\in X_0^{\tess,Y}$. Since $\psi$ is a coboundary, we find~$a\in \V s \om(\proj\RR)$ such that 
\begin{equation}\label{eq:grpcoc_cob}
 \psi_{g^{-1}} = \big( 1 - \tau_s(g)\big) a
\end{equation}
for all~$g\in\Gamma$. Let
\[
 p\coloneqq q - a \colon X_0^{\tess,Y} \to \V s \om(\proj\RR)\,. 
\]
Combining~\eqref{eq:how_grpcoc} and~\eqref{eq:grpcoc_cob} shows that for each~$g\in\Gamma$, $x\in X_0^{\tess,Y}$
\begin{align*}
 p(gx) - \tau_s(g)p(x) & = q(gx)- a - \tau_s(g)\big( q(x) - a\big) 
 \\
 & = q(gx) - \tau_s(g)q(x) - \big(1-\tau_s(g)\big) a = 0\,.
\end{align*}
Thus, $p$ is a $\Gamma$-equivariant potential of~$c$, and hence $c$ is a coboundary. 
\end{proof}

In Section~\ref{sect-coaief} we gave the interpretation of~$\E_s^\Gamma$ in terms of standard group cohomology. In terms of tesselation cohomology, the integral 
\[
c_u(z_1,z_2) = \int_{z_1}^{z_2} \bigl\{ u, R(t;\cdot)^s\bigr\}
\]
assigned in~\eqref{eq:integralc} to~$u\in\E_s^\Gamma$ and $z_1,z_2\in\uhp$ corresponds to the cocycle~$c_u$ in $Z^1\bigl(F^{\tess,Y}_\bullet;\V s \om(\proj\RR)\bigr)$ that is defined as \index[symbols]{C@$c_u$}
\begin{equation}\label{cu} 
c_u(x)(t) \ceqq  \int_{x} \bigl\{ u,
R(t;\cdot)^s\bigr\}\qquad (x\in \CC_1[X_1^{\tess,Y}])\,.
\end{equation}
Proposition~\ref{prop:isoEcohom} (which is just~\cite[Proposition~5.1]{BLZm}) induces the following tesselation cohomological interpretation of~$\E_s^\Gamma$.

\begin{prop}\label{prop:map_step1}\index[symbols]{R@$\coh_s$}
Let~$s\in\CC$, $\Rea s\in (0,1)$. The map 
\[
 \coh_s\colon \mc E_s^\Gamma \to H^1\bigl(F_\bullet^{\tess,Y};\V s \om(\proj\RR)\bigr)\,,\quad u\mapsto [c_u]\,,
\]
is injective.
\end{prop}

\begin{proof}
The map 
\[
 \mc E_s^\Gamma \to H^1\bigl(\Gamma;\V s \om(\proj\RR)\bigr)\,,\quad u\mapsto [\psi^{c_u}]\,,
\]
is injective by~\cite[Proposition~5.1]{BLZm}. An application of Lemma~\ref{lem:iso_cohom} finishes the proof. 
\end{proof}

\subsection{Mixed cohomology spaces}\label{sec:def_mixed}

The tesselation~$\tess$ allows us to consider the Hecke triangle surface~$\Gamma\backslash\uhp$ as split into three parts: 
\begin{itemize}
\item a neighborhood of the cusp, represented by~$\Gamma\infty  \subseteq X^\tess_0$, 
\item a neighborhood of the funnel, represented by~$\Gamma 1\subseteq X^\tess_0$ and~$\Gamma f_0 \subseteq X^\tess_1$, and
\item the `inner part', represented by~$X^{\tess,Y}_0$, $X^{\tess,Y}_1$ and~$X^{\tess,Y}_2$.
\end{itemize}
The elements in~$\Gamma e_\infty$ and~$\Gamma V_\infty$ connect the cusp to the inner part, and the elements in~$\Gamma e_1$ and~$\Gamma V_0$ connect the funnel to the inner part. We remark that the tesselation contains no direct connections between the cusp and the funnel, and that the sets~$X^{\tess,Y}_j$, $j\in\{0,1,2\}$, that are purely related to the inner part, are $\Gamma$-invariant. Therefore the resolution~$F^\tess_\bullet$ admits a natural definition of mixed cohomology spaces that allow us to describe separately the behavior of cocycles (in particular, of their singularities) at the cusp and the funnel as well as their regularity in the inner part. We refer to~\cite[Definition~11.4]{BLZm} for a complete definition of mixed tesselation cohomology, and restrict our exposition to the elements that we will need for the cohomological interpretation of (resonant, cuspidal) funnel forms. 

Let~$V,W$ be $\Gamma$-modules (and complex vector spaces) such that~$V\subseteq W$. We denote the first mixed cohomology space by~$H^1(F^\tess_\bullet; V,W)$, \index[defs]{mixed cohomology spaces}\index[defs]{cohomology!mixed}\index[symbols]{H@$H^1(F^\tess_\bullet; V,W)$} the space of $1$-cocycles of mixed cohomology by~$Z^1(F^\tess_\bullet; V,W)$ \index[symbols]{Z@$Z^1(F^\tess_\bullet; V,W)$}, and its subspace of $1$-coboundaries of mixed cohomology by~$B^1(F^\tess_\bullet; V,W)$.\index[symbols]{B@$B^1(F^\tess_\bullet; V,W)$}   The $1$-cocycles~$c\in Z^1(F^\tess_\bullet;V,W)$ are determined by \mbox{$\Gm$-}equivariant maps~$c\colon X^\tess_1\rightarrow W$ such that $c(e)\in V$ if~$e\in X^{\tess,Y}_1$ and $dc(V_j)\nobreak=\nobreak0$ for~$j\in\{0,1,\infty\}$. Such a cocycle is a coboundary if there exists a \mbox{$\Gm$-}equivariant map~$f\colon X^\tess_0\rightarrow W$ such that $c(e) = f(t(e))-f(h(e))$ for all~$e\in X_1^\tess$ and $f(z)\in V$ if~$z\in\nobreak X^{\tess,Y}_0$. 

Each cocycle in~$Z^1(F^\tess_\bullet;V,W)$ restricts to a cocycle in~$Z^1(F^{\tess,Y}_\bullet;V)$. For `good' combinations of~$V$ and~$W$ (as we will have below) each cocycle in~$Z^1(F^{\tess,Y}_\bullet;V)$ extends (in at least one way) to a cocycle in the mixed cocycle space~$Z^1(F^\tess_\bullet;V,W)$. In this case the elements in the cohomology space~$H^1(F^\tess_\bullet;V,W)$ may be understood as refinements of the cocycle classes in~$H^1(F^{\tess,Y}_\bullet;V)$. Nevertheless, in the passage from $H^1(F^{\tess,Y}_\bullet;V)$ to $H^1(F^\tess;V,W)$ cocycle classes do not split, which is a property rather specific for first cohomology spaces. More precisely, for $j\in\{0,1\}$, the cohomology space~$H^j(F^\tess_\bullet;V,W)$ embeds into the space~$H^j(F^{\tess,Y}_\bullet;V)$; for $j\geq 2$ this is typically not true. In view of Proposition~\ref{prop:map_step1} such embedding properties are useful for characterizing the image of subsets of~$\mc E_s^\Gamma$ under~$\coh_s$. In Proposition~\ref{prop-cohAH} we will prove an even stronger embedding statement of which we will take advantage for the characterization of the $\coh_s$-images of the spaces of funnel forms, resonant funnel forms and cuspidal funnel forms. We refer to the introduction of Section~\ref{sec:inj} for more explanations.

The cocycles in~$Z^1(F^\tess_\bullet; V,W)$ can also be naturally identified (see also Section~\ref{sec:tesselation}, in particular~\eqref{char2}) with the maps 
\[
 c\colon X^\tess_0\times X^\tess_0\to W
\]
satisfying the cocycle relation and $\Gamma$-equivariance analogous to~\eqref{cocrel2} and~\eqref{cocequiv2}, as well as the additional property
\[
 c\big(X^{\tess,Y}_0 \times X^{\tess,Y}_0\big) \subseteq V\,.
\]
For $W\subseteq \V s \om(\Xi)$ we define singularity and vanishing conditions in analogy to those in Section~\ref{sect-cococ}: We let $Z^1(F^\tess_\bullet;V,W)_\sic$ \index[symbols]{Z@$Z^1(F^\tess_\bullet;V,W)_\sic$} denote the space of those cocycles~$c\in\nobreak Z^1(F^\tess_\bullet;V,W)$ that satisfy 
\[
 \bsing c(e) \subset \proj\RR\cap \{t(e),h(e)\} \qquad\text{for all $e\in X^\tess_1$\,.}
\]
Further, we let $Z^1(F^\tess_\bullet;V,W)^\van$ \index[symbols]{Z@$Z^1(F^\tess_\bullet;V,W)^\van$} be the space of the cocycles~$c\in Z^1(F^\tess_\bullet;V,W)$ that satisfy 
\[
c(f_0)\vert_{(\lambda-\nobreak 1,1)_c} = 0\,.
\]
We use the obvious definitions for the spaces~$Z^1(F^\tess_\bullet; V,W)_{\cond_1}^{\cond_2}$, $B^1(F^\tess_\bullet; V,W)_{\cond_1}^{\cond_2}$, $H^1(F^\tess_\bullet; V,W)_{\cond_1}^{\cond_2}$, where $\cond_1\in\{\underline{\ \ },\sic\}$, $\cond_2\in\{\underline{\ \ },\van\}$. \index[symbols]{Z@$Z^1(F^\tess_\bullet;V,W)^\van_\sic$}\index[symbols]{B@$B^1(F^\tess_\bullet;V,W)^\van_\sic$}\index[symbols]{H@$H^1(F^\tess_\bullet;V,W)^\van_\sic$}

\section{Extension of cocycles}\label{sec:inj}
\markright{13. EXTENSION OF COCYCLES}

By virtue of the map~$\coh_s$ from Proposition~\ref{prop:map_step1}, the space~$\A_s$ of funnel forms is linear isomorphic to a subspace of~$H^1\bigl(F^{\tess,Y}_\bullet; \V s \om(\proj\RR)\bigr)$. We aim at describing this subspace in cohomological terms. However, the  space~$H^1\bigl(F^{\tess,Y}_\bullet; \V s \om(\proj\RR)\bigr)$ is built using only the part of the tesselation~$\tess$ that is completely contained \emph{in} the upper half plane~$\uhp$. Hence  the properties of funnel forms at the \emph{ends} of~$\Gamma\backslash\uhp$ (i.\,e., at the cusp and at the funnel) cannot be characterized in terms of intrinsic properties of cocycle classes in~$H^1\bigl(F^{\tess,Y}_\bullet;\V s \om(\proj\RR)\bigr)$. In other words, the space~$\coh_s(\A_s)$ cannot be described using only the subresolution~$F^{\tess,Y}_\bullet$. In order to find a better description of~$\coh_s(\A_s)$, a natural approach is to extend the cocycle classes in~$\coh_s(\A_s)$ to cocycle classes in the mixed cohomology space $H^1\bigl(F^\tess_\bullet;\V s \om(\proj\RR), M\bigr)$ on the full resolution~$F^\tess_\bullet$. This approach immediately raises the questions of suitable choices for $\Gamma$-modules~$M$ and the uniqueness of such extensions.

The main result of this section is Proposition~\ref{prop-cohAH}, which settles a part of these questions. It will be complemented by Theorem~\ref{thm-alcoh} showing the surjectivity of the lower three horizontal embeddings in Proposition~\ref{prop-cohAH} and thereby establishing a complete cohomological characterization of funnel forms, resonant funnel forms and cuspidal funnel forms. Recall the definitions of the singularity condition~\tsic and of the vanishing condition~\tvan from Section~\ref{sec:def_mixed}.

\begin{prop}\label{prop-cohAH}\index[symbols]{R@$\coh_s$}
Let $s\in\CC$, $\Rea s\in (0,1)$. The map~$\coh _s$ from Proposition~\ref{prop:map_step1} 
induces the system of injective linear maps
\[
\xymatrix@C=2cm{
\mc E_s^\Gamma \ar@{^{(}->}[r]^(.3){\coh _s} & H^1\bigl(F_\bullet^{\tess,Y}; \V s \om(\proj\RR)\bigr)
\\
\A_s \ar@{^{(}->}[r]^(.3){\coh _s} \ar@{^{(}->}[u]
& H^1\bigl(F^\tess_\bullet;\V s \om(\proj\RR),\V s
{\fxi;\exc;\aj}(\proj\RR)\bigr)^\van_\sic \ar@{^{(}->}[u] \\
\A^1_s \ar@{^{(}->}[r]^(.3){\coh _s} \ar@{^{(}->}[u]
& H^1\bigl(F^\tess_\bullet;\V s \om(\proj\RR),\V s {\fxi;\exc,\smp;\aj}(\proj\RR)\bigr)^\van_\sic \ar@{^{(}->}[u]\\
\A^0_s \ar@{^{(}->}[r]^(.3){\coh _s} \ar@{^{(}->}[u]
& H^1\bigl(F^\tess_\bullet;\V s \om(\proj\RR),\V s
{\fxi;\exc,\infty;\aj}(\proj\RR)\bigr)^\van_\sic \ar@{^{(}->}[u] }
\]
where the restriction~$s\not=\frac12$ is imposed for the two middle instances of the map~$\coh_s$.
\end{prop}

A key step for the proof of Proposition~\ref{prop-cohAH} is to show that the cocycles in the space~$Z^1\bigl(F^{\tess,Y}_\bullet; \V s \om(\proj\RR)\bigr)$ that are associated to funnel forms by~\eqref{cu} extend to mixed cocycles in~$Z^1(F^\tess_\bullet; \V s \om(\proj\RR), M)^\van_\sic$, where the module~$M$ is chosen as indicated in Proposition~\ref{prop-cohAH}, depending on whether we consider the space~$\A_s$ of all funnel forms or we restrict to the space~$\A_s^1$ or to the space~$\A_s^0$. We discuss these extendabilities in the following proposition.

\begin{prop}\label{prop-af-co}
Let $s\in\CC$, $\Rea s \in (0,1)$.
\begin{enumerate}[{\rm (i)}]
\item\label{afcoi} For each cuspidal funnel form~$u\in \A_s^0$ the cocycle~$c_u \in Z^1\bigl( F^{\tess,Y}_\bullet;\V s \om(\proj\RR) \bigr)$ has a
unique extension to a cocycle in
\[ 
Z^1\bigl( F^\tess_\bullet;\V s \om(\proj\RR),\V s {\fxi;\exc,\infty;\aj}(\proj\RR)\bigr)^\van_\sic \,.
\]
\item\label{afcoii} Let $s\neq \frac12$. For each resonant funnel form~$u\in \A_s^1$ the associated cocycle~$c_u$ in $Z^1\bigl( F^{\tess,Y}_\bullet;\V s \om(\proj\RR) \bigr)$ has a unique extension to a cocycle in
 \[
 Z^1\bigl( F^\tess_\bullet;\V s \om(\proj\RR),\V s {\fxi;\exc,\smp;\aj}(\proj\RR)\bigr)^\van_\sic\,.
 \]
\item\label{afcoiii} Let $s\neq \frac12$. For each funnel form~$u\in \A_s$ the cocycle~$c_u \in Z^1\bigl( F^{\tess,Y}_\bullet;\V s \om(\proj\RR) \bigr)$ has
 an extension to a cocycle in
\[ 
Z^1\bigl( F^\tess_\bullet;\V s \om(\proj\RR),\V s {\fxi;\exc;\aj}(\proj\RR)\bigr)^\van_\sic\,,
\]
which is not unique. Any two choices for the extension of~$c_u$ differ by a coboundary.
\end{enumerate}
\end{prop}

We prepare the proof of Proposition~\ref{prop-af-co} (and eventually of Proposition~\ref{prop-cohAH}) with a series of lemmas. Lemma~\ref{lem-io} implies that for any funnel form~$u$ the integral in the definition in~\eqref{cu} of the cocycle~$c_u$ remains well-defined if the integration is performed along any element in~$\CC_1[X^\tess_1]$ that stays away from the cusp but may approach the funnel. It will help us to establish the extendability of the cocycles and to understand the structure of their set of singularities at ordinary points. We refer to Section~\ref{sec:def_ff} and in particular to~\eqref{def_ff} for the notion of $s$-analytic boundary behavior and real-analytic cores.

\begin{lem}\label{lem-io}
Let $s\in\CC$, $\Rea s > 0$, and let $J\subseteq \RR$ be an open interval. Suppose that $u\in C^2(\uhp)$
has $s$-analytic boundary behavior near~$J$ and satisfies $\Delta u = s(1-\nobreak s) \, u$.
\begin{enumerate}[{\rm (i)}]
\item\label{ioi} For all~$\xi,\eta\in J\cup\uhp$ and $t\in\RR\smallsetminus\{\xi,\eta\}$ and any path~$p$ in~$\uhp\cup\RR$ from~$\xi$ to~$\eta$ with at most its endpoints in~$\RR$ the integral \index[symbols]{I@$I_u$}
\begin{equation}\label{eq:def_int_trafo}
I_u(\xi,\eta)(t) \coloneqq \int_\xi^\eta \bigl\{u, R(t;\cdot)^s\bigr\}\ceqq \int_{p} \bigl\{u, R(t;\cdot)^s\bigr\}
\end{equation}
converges and its value is independent of the choice of~$p$.
\begin{center}
\includegraphics[width=8cm]{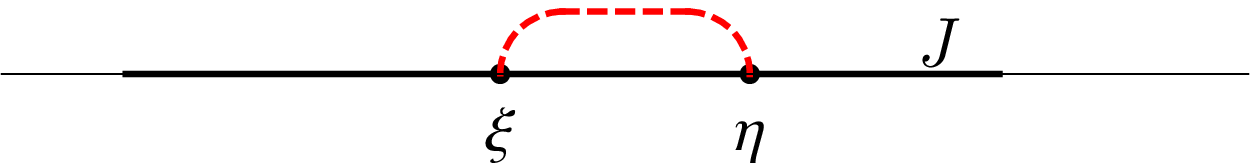}
\end{center}
The map 
\[
 \RR\to \CC,\quad t\mapsto I_u(\xi,\eta)(t)
\]
extends to an element of $\V s \om[\xi,\eta]$.
\item\label{ioii} Let $\xi,\eta\in J$ with $\xi<\eta$ and suppose that $A$ denotes a real-analytic core of~$u$ near~$J$. Then  
\begin{equation} \label{eq:Iu}
I_u(\xi,\eta)(t) \ceqq  \begin{cases}
0 &\text{ if } t\in (\eta,\xi)_c\,,\\
2\sqrt{\pi}\, \dfrac{\Gf(s+1/2)}{\Gf(s)}\, A(t)
&\text{ if } t\in (\xi,\eta)_c\,.
\end{cases}
\end{equation}
\item\label{ioiii} If $\xi\in J\cap\Gamma 1$ and $z\in\uhp$, then $I_u(\xi,z) \in \V s {\om;-;\aj}[\xi]$.
\item\label{ioiv} If $\xi,\eta\in J\cap\Gamma 1$, then $I_u(\xi,\eta)\in \V s {\om;-;\aj}[\xi,\eta]$.
\end{enumerate}
\end{lem}

\begin{proof}
Let $\xi,\eta\in J\cup\uhp$ and suppose that $p\colon [0,1]\to \uhp\cup\RR$ is a path from~$\xi$ to~$\eta$ as in~\eqref{ioi}. For the proof of~\eqref{ioi} and~\eqref{ioii} we evaluate the one-form in the integrand. To that end let $A\colon U\to \CC$ be a real-analytic core of~$u$ near~$J$. We may and shall assume that $\uhp\subseteq U$. Thus, 
\[
u(z) = y^s A(z)
\]
for all~$z\in \uhp$. On~$\uhp$ and for all~$t\in\RR$ we have
\begin{equation}\label{acc-expr} 
\begin{aligned} 
\bigl\{ u, R(t;\cdot)^s \bigr\} &= \frac{i\, y^{2s}}{(t-z)^s (t-\bar z)^s}\, \biggl( A_z\,dz - A_{\bar z}\, d\bar z + s A\Bigl(\frac{ d\bar z}{t-\bar z} - \frac{dz}{t-z} \Bigr)\biggr)
\\
&= \frac{y^{2s}}{|t-z|^{2s}}\, \biggl( \Bigl( A_y + 2s \frac y{|t-z|^2} A\Bigr)\, dx 
\\
 & \hphantom{\frac{y^{2s}}{|t-z|^{2s}}\, A_y + 2s {|t-z|}} + \Bigl( - A_x + 2s \frac{t-x}{|t-z|^2}A \Bigr)\, dy \biggr)\,. 
\end{aligned}
\end{equation}
Let $t\in\RR\smallsetminus\{\xi,\eta\}$ (as in~\eqref{ioi}). Since the path~$p$ from~$\xi$ to~$\eta$ is bounded away from~$t$, the value of~$|t-z|$ is bounded away from~$0$ on the path~$p$. From this, the requirement that $\Rea s > 0$, and the fact that the image of the path~$p$ is contained in $U$, it follows that the integral
\[
 \int_{p} \bigl\{ u, R(t;\cdot)^s\bigr\}
\]
converges, which is just the latter integral in~\eqref{eq:def_int_trafo}.

Since the one-form~$\{u,R(t;\cdot)^s\}$ is closed, the value of the integral in~\eqref{eq:def_int_trafo} neither depends on the choice of the path~$p$ nor on the choice of the real-analytic core~$A$. 

Further, $I_u(\xi,\eta)$ is real-analytic on~$\RR\smallsetminus\{\xi,\eta\}$ as a parameter integral with real-analytic integrand. By definition, $I_u(\xi,\eta)$ extends to an element of~$\V s \om[\xi,\eta]$ if $\tau_s(S)I_u(\xi,\eta)$ extends real-analytically to~$0$. For $t\in\RR$, $t\not=0$, we have
\begin{align*}
 \tau_s(S) I_u(\xi,\eta)(t) & = \int_\xi^\eta |t|^{-2s} \bigl\{ u, R\big(-\tfrac1t;\cdot\big)^s \bigr\}\,.
\end{align*}
A straightforward calculation shows that the integrand extends real-analytically to~$t=0$, and the argumentation as above yields the convergence of the integral for~$t=0$ as well. This completes the proof of~\eqref{ioi}.

For establishing~\eqref{ioii} we first suppose that $t\in \RR\smallsetminus [\xi,\eta]$. Then we can deform the path~$p$ in the integral~\eqref{eq:def_int_trafo} to one which has the interval~$[\xi,\eta]$ as image, which is completely contained in~$\uhp$. Then~\eqref{acc-expr} shows that the integrand is zero, and hence $I_u(\xi,\eta)(t)=0$. Now real-analyticity yields that $I_u(\xi,\eta)(t)=0$ also for~$t=\infty\in (\eta,\xi)_c$.

Suppose now that $t\in (\xi,\eta)$.  Then we deform the path~$p$ to one as indicated in the following figure:
\begin{center}
\includegraphics[width=8cm]{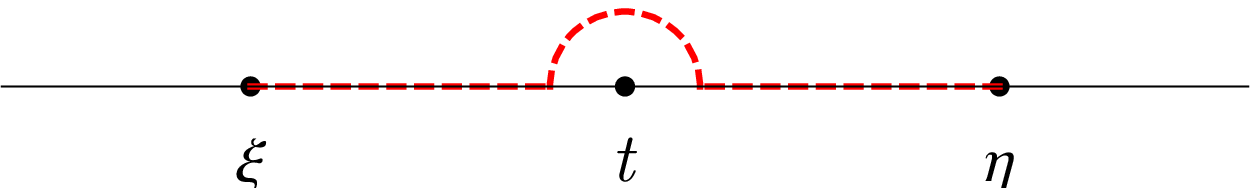}
\end{center}
It splits into the three subpaths 
\begin{align*}
 p_{-,\delta} & \colon (\xi, t-\delta) \to (\xi, t-\delta)\,,&\quad x&\mapsto x\,,
 \\
 p_{t,\delta} & \colon (0,\pi) \to \uhp\cup\RR\,,&\quad \theta &\mapsto t + \delta e^{i\theta}\,,
 \\
 p_{+,\delta} & \colon (t+\delta,\eta) \to (t+\delta,\eta)\,,&\quad x&\mapsto x\,,
\end{align*}
for some (small)~$\delta > 0$ (of which we will take the limit~$\delta\searrow 0$ further below). From~\eqref{acc-expr} we see that the integrand of~\eqref{eq:def_int_trafo} vanishes along~$p_{-,\delta}$ and~$p_{+,\delta}$. Along~$p_{t,\delta}$ we find 
\begin{align*}
\int_{p_{t,\delta}} \bigl\{ u, R(t;\cdot)^s \bigr\} 
& = \int_{\theta=0}^\pi \delta \big(\sin\theta\big)^{2s} \Big[ A_z\big(p_{t,\delta}(\theta)\big) e^{i\theta} + A_{\bar z}\big(p_{t,\delta}(\theta)\big) e^{-i\theta} \Big] \\
&\qquad\hbox{}
+ 2s \big(\sin\theta\big)^{2s} A\big(p_{t,\delta}(\theta)\big) \, d\theta\,.
\end{align*}
Since each term of the integrand is bounded and the real-analytic core~$A$ is continuous, we find
\begin{align*}
 I_u(\xi,\eta)(t) & = \lim_{\delta\searrow0} \int_{p_{t,\delta}}  \bigl\{ u, R(t;\cdot)^s \bigr\} 
 \\
 & = 2s A(t) \int_0^\pi \big(\sin\theta\big)^{2s}\, d\theta = 2 \sqrt{\pi} \frac{\Gamma\big(s+\tfrac12\big)}{\Gamma(s)} A(t)\,.
\end{align*}
This completes the proof of~\eqref{ioii}.

For~\eqref{ioiv} we note that since $A$ is real-analytic on~$J$, and $(\xi,\eta)_c\subseteq J$, the expression in~\eqref{eq:Iu} shows that $I_u(\xi,\eta)$ has analytic jumps at~$\xi$ and~$\eta$. For~\eqref{ioiii} we fix $\eta\in J$, $\eta\not=\xi$. Then 
\[
 I_u(\xi,z) = I_u(\xi,\eta) + I_u(\eta,z)\,.
\]
By~\eqref{ioiv}, $I_u(\xi,\eta)$ has an analytic jump at~$\xi$, and by~\eqref{ioi}, $I_u(\eta,z)$ is real-analytic at~$\xi$. In turn, $I_u(\xi,z)$ has an analytic jump at~$\xi$.
\end{proof}

Let $s\in\CC$, $\Rea s > 0$, and $u\in \mc E_s$. Then the integral  \index[symbols]{I@$I_u$}
\begin{equation}\label{eq:extended}
 I_u(\xi,\eta)(t) \coloneqq  \int_\xi^\eta \big\{ u, R(t;\cdot)^s\big\} = \int_p\big\{ u, R(t;\cdot)^s\big\}
\end{equation}
in~\eqref{eq:def_int_trafo} is also well-defined for any pair~$(\xi,\eta)\in \big(\Omega\cup\uhp\big)^2$. (Recall from Section~\ref{sec:ordinary} that $\Omega$ denotes the ordinary set of $\Gm$, i.\,e., $\Omega=\proj\RR\smallsetminus\Lambda(\Gm)$.) In Lemma~\ref{lem-io} we considered only the case that if $\xi,\eta$ are in~$\Omega$, then they are contained in a \emph{connected} subset of~$\proj\RR$ near which $u$ has $s$-analytic boundary behavior. However, by splitting the integral into several paths, one easily sees that the convergence and well-definedness of the integrals in~\eqref{eq:extended} are valid in this larger generality.

\begin{lem}\label{lem:int_trans}
Let $s\in\CC$, $\Rea s>0$, and $u\in\mc E_s$. For all~$\xi,\eta\in\Omega\cup\uhp$ and all~$g\in\Gamma$ we have
 \[
 \tau_s(g)I_u(\xi,\eta) = I_u(g\xi,g\eta)\,.
\]
\end{lem}

\begin{proof}
Let $g\in\Gamma$ and $\xi,\eta\in\Omega\cup\uhp$.  By~\cite[(1.10), (2.25)]{BLZm} we have 
\[
 \tau_s(g)\int_\xi^\eta \big\{ u, R(t;\cdot)^s\big\} = \int_{g\xi}^{g\eta} \big\{ u, R(t;\cdot)^s\big\}\,.
\]
\end{proof}

The following three lemmas will be needed for the discussion of the relation between different extensions of the cocycles and, in case of resonant funnel forms, for showing the uniqueness of these extensions. The first lemma is a generalization of~\cite[Proposition~4.1]{BLZm}.

\begin{lem}\label{lem:hyp_inv}
Let $s\in\CC\smallsetminus(-\NN_0)$. Let $\eta\in\Gamma$ be hyperbolic and let $I\subseteq \proj\RR$ be an open interval that contains at least one fixed point of~$\eta$. Suppose that $f\in\V s \om(I)$ satisfies
\[
 \tau_s(\eta)f= f \quad\text{on\ \ $I\cap\eta I$\,.}
\]
Then $f=0$.
\end{lem}

\begin{proof}
We proceed by contraposition. To that end let $s\in\CC$, and suppose that $f\in \V s \om(I)$ satisfies all hypothesis and does not vanish identically. We need to show that $s\in -\NN_0$. Without loss of generality we may assume that 
\[
 \eta = \bmat{r^{\frac12}}{0}{0}{r^{-\frac12}}
\]
with~$r>0$. The fixed points of~$\eta$ are~$0$ and~$\infty$. 

Suppose first that~$0\in I$. By real-analyticity, 
\[
 f(t) = \sum_{n\in\NN_0} c_n t^n
\]
in a neighborhood of~$0$ in~$\RR$, for suitable~$c_n\in\CC$, $n\in\NN_0$. Thus, in a (possibly smaller) neighborhood of~$0$ we have
\[
 f(t) = \tau_s(\eta)f(t) = \sum_{n\in\NN_0} c_n r^{s+n} t^n\,.
\]
Since $f\not=0$, there exists~$n\in\NN_0$ such that $s+n=0$, and hence $s\in -\NN_0$. 

Suppose now that $\infty\in I$. Then 
\[
 h\coloneqq  \tau_s(S)f
\]
is real-analytic in a neighborhood of~$0$ in~$\RR$, and is $\tau_s(S\eta S)$-invariant near~$0$. Thus, near~$0$, we have
\[
 h(t) = \sum_{n\in\NN_0} c_n t^n
\]
for suitable~$c_n\in\CC$, $n\in\NN_0$, and hence
\[
 h(t) = \tau_s(S\eta S)h(t) =  \sum_{n\in\NN_0} c_n r^{-s-n}t^n\,.
\]
From $h\not=0$ it follows that $s\in -\NN_0$. This completes the proof.
\end{proof}

\begin{lem}\label{lem:vanish_cob}
Let $s\in\CC$, and let $V,W$ be $\Gamma$-modules. Suppose that the cocycle $c\in Z^1(F^\tess_\bullet; V, W)$ vanishes on $F^{\tess,Y}_\bullet$ (i.\,e., $c=0$ on~$F^{\tess,Y}_\bullet$). Then $c\in B^1(F^\tess_\bullet; V, W)$.
\end{lem}

\begin{proof}
We fix $x_0\in X_0^{\tess,Y}$ and consider the potential
\[
 q\colon X_0^\tess \to W\,, \quad q(x)\coloneqq c(e_{x,x_0})\,,
\]
of~$c$. The associated group cocycle~$\psi\colon\Gamma\to W$ vanishes everywhere because
\[
 \psi_{g^{-1}} = q(gx_0) - \tau_s(g)q(x_0) = 0
\]
for all $g\in\Gamma$. Thus, $c$ is a coboundary.
\end{proof}

\begin{lem}\label{lem:coc_unique}
Let $s\in\CC$, $\Rea s\in (0,1)$. Suppose that the cocycle
\[
c\in Z^1\bigl(F^\tess_\bullet; \V s \om(\proj\RR), \V s {\fxi;\smp;}(\proj\RR)\bigr)_\sic^\van
\]
vanishes on~$F^{\tess,Y}_\bullet$. Then:
\begin{enumerate}[{\rm (i)}]
\item If $s\not=\frac12$, then $c=0$. 
\item If $c(e_\infty)\in \V s {\fxi;\infty;}$, then $c=0$ also for $s=\frac12$.
\end{enumerate}
\end{lem}

\begin{proof}
By the $\Gamma$-equivariance of~$c$, the definition of~$F^\tess_\bullet$ as the $\CC$-module with basis~$X^\tess_\bullet$, and the vanishing of~$c$ on~$F^{\tess,Y}_\bullet$ it suffices to show that~$c$ vanishes on 
\[
 X^\tess_1 \smallsetminus X^{\tess,Y}_1 = \{ e_1, e_\infty, f_0 \}\,.
\]
Since $c(f_\infty) = 0$, it follows that 
\[
 0=dc(V_\infty) = \big( 1 -\tau_s(T^{-1})\big)c(e_\infty) + c(f_\infty) = \big(1-\tau_s(T^{-1})\big)c(e_\infty)\,,
\]
showing that 
\[
 c(e_\infty) \in \bigl( \V s {\fxi;\smp;}(\proj\RR) \bigr)^T\,.
\]
For $s\not=\frac12$, \cite[Proposition~9.13]{BLZm} yields 
\[
 c(e_\infty) = 0\,.
\]
In the case that $c(e_\infty)\in \V s {\fxi;\infty;}(\proj\RR)$, \cite[Proposition~9.9]{BLZm} shows $c(e_\infty) = 0$ also for~$s=\frac12$.

To show that $c$ also vanishes on~$e_1$ and~$f_0$, we note that $c$ is a coboundary by Lemma~\ref{lem:vanish_cob}. Thus,
we find a $\Gamma$-equivariant potential 
\[
 p\colon X^\tess_0 \to \V s {\fxi;\smp;}(\proj\RR)
\]
of~$c$. From
\[
 c(f_0) = p(1) - p(\lambda-1) = \big(1 - \tau_s(TS)\big) p(1)
\]
and the vanishing property~\tvan of~$c$ it follows that 
\[
 p(1) = \tau_s(TS)p(1) \quad\text{on $(\lambda-1,1)_c$\,.}
\]
The singularity property~\tsic of~$c$ implies
\[
 p(1) \in \V s \om[1]\,.
\]
Thus, $p(1) = 0$ on~$(\lambda-1,1)_c$ by Lemma~\ref{lem:hyp_inv} (choose in Lemma~\ref{lem:hyp_inv} first $I=(\lambda-1,\fixTS^-)_c$ and $\eta=TS$, and then $I=(\fixTS^+,1)_c$ and $\eta=(TS)^{-1}$). Then real-analyticity yields $p(1)=0$ on $\proj\RR\smallsetminus\{1\}$, and hence $p(1)=0$ as element of the sheaf~$\V s \fxi$. It follows that
\[
 c(f_0)=0\,.
\]
Further, 
\[
 0 = dc(V_0) = \big( 1 - \tau_s(TS)\big) c(e_1) + c(f_0) - c(f_1) = \big( 1 - \tau_s(TS)\big) c(e_1)\,.
\]
Since~$c(e_1)\in \V s \om[1]$, Lemma~\ref{lem:hyp_inv} yields
\[
 c(e_1) = 0\,.
\]
This completes the proof.
\end{proof}

\begin{proof}[Proof of Proposition~\ref{prop-af-co}]
Let $s\in\CC$, $\Rea s\in (0,1)$, and let $u\in\A_s$. In case that $u\notin\A_s^0$, the additional restriction $s\not=\frac12$ applies. Lemma~\ref{lem:vanish_cob} shows that any two extensions of any cocycle in~$Z^1\bigl(F^{\tess,Y}_\bullet; \V s \om(\proj\RR)\bigr)$ differ by a coboundary in the mixed cohomology spaces. In case that $u\in\A_s^1$, uniqueness of the claimed extensions follows from Lemma~\ref{lem:coc_unique}. It remains to show the existence of the extensions.

Let $c_u\in Z^1\bigl( F^{\tess,Y}_\bullet;\V s \om(\proj\RR) \bigr)$ be the cocycle associated to~$u$. In what follows we construct an extension of~$c_u$ to a cocycle on~$F^\tess_\bullet$ with values in the~$\Gamma$-modules and properties as claimed. We denote this extension by~$c_u$ as well even though for~$u\notin\A_s^1$ it does not need to be unique. In order to define this extension we first extend it to 
\[
 \{ e_\infty, e_1, f_0\} \subseteq X^\tess_1\smallsetminus X^{\tess,Y}_1\,,
\]
and then continue it $\CC[\Gm]$-linearly to all of~$F^\tess_\bullet$, using that $\Gm$ acts freely on~$X_1^\tess$. For the extension to~$e_\infty$ we take advantage of~\cite[Proposition 12.1]{BLZm}, for the definition on~$e_1$ and~$f_0$ we use Lemma~\ref{lem-io}. We set 
\[
 c_u(e_\infty) \coloneqq  -\av{s,T}^+ c_u(f_\infty)\,, \quad
 c_u(e_1) \coloneqq  I_u(i,1) \quad\text{and}\quad c_u(f_0) \coloneqq  I_u(1,\lambda-1)\,.
\]
To show the cocycle property, we observe that
\begin{align*}
 dc_u(V_\infty) &= \big(1-\tau_s(T^{-1})\big)c_u(e_\infty) + c_u(f_\infty) 
 \\
 & = -\big(1-\tau_s(T^{-1})\big)\av{s,T}^+ c_u(f_\infty) + c_u(f_\infty) 
 \\
 & = -c_u(f_\infty) + c_u(f_\infty) = 0
\end{align*}
by~\eqref{avrel}, and that
\begin{align*}
 dc_u(V_0) & = c_u(e_1) + c_u(f_0) - \tau_s(T)c_u(e_1) - c_u(f_1)  
 \\
 & = I_u(i,1) + I_u(1,\lambda-1) - I_u(\lambda+i,\lambda-1) - I_u(i,\lambda+i) = 0
\end{align*}
by Lemma~\ref{lem:int_trans}. By $\CC[\Gm]$-linearity, $c_u$ satisfies the cocycle condition on all of~$F^\tess_\bullet$. 
By Lemma~\ref{lem-io}, 
\[
 c_u(f_0)\vert_{(\lambda-1,1)_c} = 0\,,
\]
and hence $c_u$ satisfies the vanishing condition~\tvan. Further
\[
 c_u(e_1) \in \V s {\om;-;\aj}[1] \quad\text{and}\quad c_u(f_0)\in \V s {\om;-;\aj}[1,\lambda-1]\,.
\]
As in~\cite[Proposition~12.1]{BLZm} one proves that 
\begin{align*}
 c_u(e_\infty) & \in \V s {\om;\exc,\infty;}(\proj\RR) && \text{if $u\in \A_s^0$\,,}
 \\
 c_u(e_\infty) & \in \V s {\om;\exc,\smp;}[\infty] && \text{if $u\in \A_s^1$\,,}
 \\
 c_u(e_\infty) & \in \V s {\om;\exc;}[\infty] && \text{if $u\in \A_s$\,.}
\end{align*}
Therefore, $c_u$ satisfies the singularity condition~\tsic, and $c_u$ has values in the module as claimed in the statement of this proposition.
\end{proof}

\begin{rmk}
For $u\in\A^0_s$ the extension of the cocycle~$c_u$ in~\eqref{cu} to a cocycle in~$Z^1\bigl(F^\tess_\bullet; \V s \om(\proj\RR),\V s {\fxi;\exc,\infty;\aj}(\proj\RR) \bigr)^\van_\sic$ is given by
\[
 c_u(e_\infty)(t) \coloneqq \int_{e_\infty} \{ u, R(t;\cdot)^s\}\,.
\]
The exponential decay of~$u$ towards~$\infty$ causes the integral to converge.
\end{rmk}

The following lemma not only implies that the vertical maps on the right hand side of the diagram in Proposition~\ref{prop-cohAH} are embeddings, it will also be needed to establish the injectivity of the three lower horizontal maps.

\begin{lem}\label{lem:mixed_inj}
Let $s\in\CC$ and suppose that $W$ is a $\Gamma$-module that satisfies \[\V s \fxi(\proj\RR) \subseteq W \subseteq \V s \om(\proj\RR)\,.\] Then the map\footnote{We recall that we identify cocycles in $Z^1\bigl(F_\bullet^{\tess,Y}; \V s \om(\proj\RR)\bigr)$ with maps $c\colon X_1^{\tess,Y}\to \V s \om(\proj\RR)$ that obey certain relations. See Section~\ref{sec:tesselation}.}
\[
 H^1\bigl(F_\bullet^\tess; \V s \om(\proj\RR), W\bigr)_\sic^\van \to H^1\bigl(F_\bullet^{\tess,Y}; \V s \om(\proj\RR)\bigr)\,,\quad [c]\mapsto \bigl[c\vert_{X_1^{\tess,Y}}\bigr]\,,
\]
is injective.
\end{lem}

\begin{proof}
It suffices to show that each cocycle $c\in Z^1(F_\bullet^\tess;\V s \om(\proj\RR), W)_\sic^\van$ that satisfies $c\vert_{X_1^{\tess,Y}}\in B^1(F_\bullet^{\tess,Y};\V s \om(\proj\RR))$ is contained in~$B^1(F_\bullet^\tess; \V s \om(\proj\RR), W)_\sic^\van$.

Let $c\in Z^1(F_\bullet^\tess;\V s \om(\proj\RR), W)_\sic^\van$ be such a cocycle and set~$\tilde c\coloneqq  c\vert_{X_1^{\tess,Y}}$. By hypothesis, $\tilde c\in B^1(F_\bullet^{\tess,Y}; \V s \om(\proj\RR))$. Thus we find a $\Gamma$-equivariant potential 
\[
p\colon X_0^{\tess,Y} \to \V s \om(\proj\RR)
\]
of~$\tilde c$. We claim that~$p$ extends to a $\Gamma$-equivariant potential~$q\colon X_0^\tess \to W$ of $c$. To that end, we let
\begin{align}
 q_1 &\coloneqq p(i) - c(e_1) \label{def_q1}
 \\
 q_\infty & \coloneqq p(\lambda + iY) - c(e_\infty)\,. \label{def_qinfty}
\end{align}
We define the map~$q$ by
\[
 q\vert_{X_0^{\tess,Y}} \coloneqq p
\]
and
\[
 q(g1)\coloneqq \tau_s(g)q_1\,,\qquad q(g\infty)\coloneqq \tau_s(g)q_\infty
\]
for all~$g\in\Gamma$. This is indeed well-defined: We have 
\begin{align*}
\tau_s(T^{-1})q_\infty & = \tau_s(T^{-1})p(\lambda+iY) - \tau_s(T^{-1})c(e_\infty) = p(iY) - c(T^{-1}e_\infty) 
\\
& = p(iY) - c(f_\infty) - c(e_\infty)  = p(\lambda + iY) - c(e_\infty) 
\\
& = q_\infty\,,
\end{align*}
where we used the $\Gm$-equivariance of~$p$ for the second equality, and the identity $0=dc(V_\infty) = -c(T^{-1}e_\infty) + c(f_\infty)+c(e_\infty)$ for the third equality, and the property $c(f_\infty) = p(iY) - p(\lambda+iY)$ for the forth equality. Thus, for all~$g\in\Gm_\infty$, 
\[
 q(g\infty) = \tau_s(g)q_\infty = q_\infty\,.
\]
The stabilizer group~$\Gm1$ of~$1$ is trivial. Thus, $q$ is a $\Gamma$-equivariant extension of~$p$ to~$X_0^\tess$ with values in~$W$. 

It remains to show that $q$ is a potential of~$c$. Due to $\Gamma$-equivariance, it suffices to show that $dq(e) = c(e)$ for $e\in\{e_0,e_2,e_\infty, f_0,f_1,f_\infty\}$. Since $q$ extends~$p$, which is a potential of $c\vert_{X_1^{\tess,Y}}$, the identity $dq(e)=p(e)$ is clearly satisfied for $e\in \{f_1,f_\infty\}$. Further, for $e=e_0$ the definition of~$q$ and $q_1$ from~\eqref{def_q1} yields
\[
 c(e_1) = p(i) - q_1 = q(i) - q(1) = dq(e_0)\,.
\]
Analogously, using~\eqref{def_qinfty} we find
\[
 c(e_\infty) = p(\lambda + iY) -q_\infty = q(\lambda+iY) - q(\infty) = dq(e_\infty)\,.
\]
Finally, taking advantage of 
\[
0=dc(V_0)=c(e_1) + c(f_0) -\tau_s(TS)c(e_1) - c(f_1)
\]
we find that 
\begin{align*}
dq(f_0) & = q(1) - q(\lambda-1) = q(1)  - \tau_s(TS)q(1) 
\\
& = p(i) - c(e_1) - \tau_s(TS)\big( p(i) - c(e_1)\big) 
\\
& = p(i) - \tau_s(TS)p(i) - c(e_1) + \tau_s(TS)c(e_1)
\\
& = p(i) - p(\lambda+i) - c(e_1) + \tau_s(TS) c(e_1)
\\
& = c(f_1) + \tau_s(TS)c(e_1) -c(e_1)
\\
& = dc(V_0) + c(f_0)
\\
& = c(f_0)\,.
\end{align*}
Then $\Gamma$-equivariance implies that $c=dq$ on all of~$X_1^\tess$. This completes the proof.
\end{proof}

Proposition~\ref{prop-cohAH} now follows from a combination of Propositions~\ref{prop:map_step1}, \ref{prop-af-co} and Lemma \ref{lem:mixed_inj}. We provide a few more details.

\begin{proof}[Proof of Proposition~\ref{prop-cohAH}]
We only show the claimed properties for the top rectangle in the diagram of the statement, that is, the claims in relation with the full space~$\A_s$ of funnel forms. The properties of the remaining maps can be shown analogously. 

The injectivity of the maps 
\[
 \coh_s\colon\E_s^\Gamma \to H^1\bigl(F^{\tess,Y}_\bullet;\V s \om(\proj\RR)\bigr)
\]
and 
\[
 \varrho_s^Y\colon H^1\bigl(F^\tess_\bullet;\V s \om(\proj\RR), \V s {\fxi;\exc;\aj}(\proj\RR)\bigr)_\sic^\van \to H^1\bigl(F^{\tess,Y}_\bullet; \V s \om(\proj\RR)\bigr)
\]
are the statements of Proposition~\ref{prop:map_step1} and Lemma~\ref{lem:mixed_inj}. The injectivity of the map
\[
 \A_s \to \E_s^\Gamma\,,\quad u\mapsto u
\]
is clear. Let $u\in\A_s$ and let $c_u\in Z^1\bigl(F^{\tess,Y}_\bullet;\V s \om(\proj\RR)\bigr)$ be the representative of the cocycle class~$\coh_s(u)\in H^1\bigl(F^{\tess,Y}_\bullet;\V s \om(\proj\RR)\bigr)$ from~\eqref{cu}. By Proposition~\ref{prop-af-co}, the cocycle~$c_u$ extends to an element~$\tilde c_u\in Z^1\bigl(F^\tess_\bullet;\V s \om(\proj\RR), \V s {\fxi;\exc;\aj}(\proj\RR)\bigr)_\sic^\van$ that is unique up to a coboundary in~$B^1\bigl(F^\tess_\bullet;\V s \om(\proj\RR),\V s {\fxi;\exc;\aj}(\proj\RR)\bigr)_\sic^\van$. Thus, $c_u$ (and hence $\coh_s(u)$) induces a unique element in~$H^1\bigl(F^\tess_\bullet;\V s \om(\proj\RR),\V s {\fxi;\exc;\aj}(\proj\RR)\bigr)_\sic^\van$. We denote the map
\begin{equation}\label{mapcohsfunnel}
 \A_s \to H^1\bigl(F^\tess_\bullet;\V s \om(\proj\RR), \V s {\fxi;\exc;\aj}(\proj\RR)\bigr)_\sic^\van\,,\quad u\mapsto [\tilde c_u]
\end{equation}
by~$\coh_s$ as well. By construction, 
\[
 \coh_s\vert_{\A_s} = \varrho_s^Y\circ\coh_s\,.
\]
Then the injectivity of~$\varrho_s^Y$ implies the injectivity of~$\coh_s$ in~\eqref{mapcohsfunnel}.
\end{proof}

\section{Surjectivity I: Boundary germs}
\markright{14.  BOUNDARY GERMS}

This section and Section~\ref{sect-coc-ff} below are devoted to the proof of the surjectivity of the three lower horizontal maps~$\coh_s$ in the diagram in Proposition~\ref{prop-cohAH}. In rough terms, we will construct funnel forms, which are Laplace eigenfunctions defined on \emph{all} of the upper half plane, from a family (a cocycle) in~$\V s \fxi(\proj\RR)$, which induces a family of holomorphic functions in a (usually rather \emph{small}) complex neighborhood of certain subsets of~$\proj\RR$. In the context of \emph{cofinite} Fuchsian groups, discussed in~\cite{BLZm}, it turned out to be helpful to substitute the sheaf~$\V s \om$ by a sheaf~$\W s \om$ of $s$-analytic boundary germs. We apply the same approach here.
In this section we define and study the sheaf~$\W s \om$. In particular, for any choices of~$\tcond_1\in\{ \underline{\ \ }, \smp,\infty\}$ and $\tcond_2\in\{\underline{\ \ },\sic\}$, $\tcond_3\in\{\underline{\ \ },\van\}$ we will find a natural isomorphism
between the two spaces
\[
H^1\bigl(F^\tess_\bullet; \V s \om(\proj\RR), \V s {\fxi;\exc,\cond_1;\aj}(\proj\RR)\bigr)_{\cond_2}^{\cond_3}
\]
and
\[
H^1\bigl(F^\tess_\bullet; \W s \om(\proj\RR), \W s {\fxi;\exc,\cond_1;\aj}(\proj\RR)\bigr)_{\cond_2}^{\cond_3}\,.
\]
For results we will often refer to~\cite{BLZ13, BLZm} and only indicate where changes and modifications are needed.

\subsection{Analytic boundary germs and semi-analytic modules}\label{sect-abg}

For any open subset~$Z\subseteq\uhp$ we set \index[symbols]{Eac@$\E_s(Z)$}
\begin{equation}
 \E_s(Z) \coloneqq\bigl\{\,u\in C^2(Z) \setmid \text{$\Delta u = s(1-\nobreak s) u$
on~$Z$}\,\vphantom{C^2(Z)}\bigr\}\,.
\end{equation}
In this notation, $\E_s^\Gamma=\{u\in\E_s(\uhp)\setmid \forall\, g\in \Gm\colon u\circ g=u\}$.

Suppose that $I\subseteq\proj\RR$ is an open subset, $U$ an open neighborhood of $I$ in~$\proj\CC$ and \mbox{$u\in\E_s(U\cap\uhp)$}. In extension of the definition in Section~\ref{sec:def_ff} we say that the map~$u$ has \emph{$s$-analytic boundary behavior near~$I$} \index[defs]{$s$-analytic boundary behavior} if 
\begin{equation}
 U\cap\uhp\to\CC\,,\quad z\mapsto y^{-s}u(z)
\end{equation}
extends to a real-analytic function on a neighborhood of~$I$ in~$\proj\CC$. Here, $y=y(z)=\Ima z$ (as defined in Section~\ref{sec:notation}). We recall from Section~\ref{sec:regularity} that real-analyticity in~$\infty$ for the spectral parameter~$s$ means that the map
\[
 z\mapsto \left(\frac{y}{|z|^2}\right)^{-s} u\left(-\frac1z\right)
\]
extends to a real-analytic function on a complex neighborhood of~$0$. As in Section~\ref{sec:def_ff} we call any such real-analytic extension a \emph{real-analytic core of~$u$ near~$I$}. \index[defs]{real-analytic core} We set\index[symbols]{B@$\mc B_s(I,U)$}
\[
 \mc B_s(I,U) \coloneqq \bigl\{\, u\in\mc E_s(U\cap\uhp)   \setmid  \text{$u$ has $s$-analytic boundary behavior near~$I$}\, \bigr\}\,.
\]
If $U_1\subseteq U_2$ are open neighborhoods of~$I$ in~$\proj\CC$, then $\mc B_s(I,U_2)\subseteq \mc B_s(I,U_1)$. We use the inclusion of sets as a direction on the set~$\mc U_o(I)$ of all open neighborhoods of~$I$ in~$\proj\CC$ and endow $\big(\mc B_s(I,U)\big)_{U\in\mc U_o(I)}$ with the structure of a directed system using the natural embeddings. Let \index[symbols]{Waa@$\W s \om$}
\begin{equation}
 \W{s}\om(I) \coloneqq \varinjlim \mc B_s(I,U)
\end{equation}
denote the direct limit of this directed system. The family of all spaces~$\W s \om(I)$, \mbox{$I\subseteq \proj\RR$}~open, forms the sheaf~$\W s \om$ of \emph{$s$-analytic boundary germs}.\index[defs]{$s$-analytic boundary germ}\index[defs]{boundary germ} The sheaf~$\W s \om$ is naturally isomorphic to the sheaf~$\V s \om$ as explained in what follows.

For any ~$u\in \mc B_s(I,U)$ and any real-analytic core~$A$ of~$u$ near~$I$, the restriction~$A\vert_I$ is an element of~$\V s \om(I)$. The assignment~$u\mapsto A\vert_I$ induces a \emph{restriction map} \index[defs]{restriction map} \index[symbols]{R@$\rho_s$}
\begin{equation}
\rho_{s,I}\colon \W s \om(I)\to \V s \om(I)
\end{equation}
and further the \emph{restriction map} (of sheaves)
\[
 \rho_s\colon\W {s}\om \to \V{s} \om\,.
\]
Moreover, the action \index[symbols]{T@$\tau$}
\[
 \tau\colon\Gamma\times\mc B_s \to \mc B_s\,,\quad \tau(g^{-1})u \coloneqq u\circ g\,,
\]
where \index[symbols]{B@$\mc B_s$}
\[
 \mc B_s \coloneqq \bigcup_{\substack{I\subseteq \proj\RR\\ \text{open}}} \bigcup_{U\in\mc U_o(I)} \mc B_s(I,U)\,,
\]
descends to a $\Gamma$-action on~$\W s \om$, which we also denote by~$\tau$, turning $\W s\om$ into a \mbox{$\Gamma$-}equivariant sheaf. By~\cite[(5.10)]{BLZ13}, the restriction map~$\rho_s$ intertwines $\tau$ on~$\W s \om$ with the action of~$\tau_s$ on~$\V s \om$. 

The following result, which follows immediately from~\cite[Theorem 5.6]{BLZ13}, allows us to use the \mbox{$\Gamma$-}equi\-variant sheaf~$\W s \om$ instead of~$\V s \om$ for the proof of the surjectivity of the maps~$\coh_s$.

\begin{prop}[\cite{BLZ13}]\label{thm-resiso}
\begin{enumerate}[{\rm (i)}]
\item For each open subset~$I\subseteq\proj\RR$ the restriction map $\rho_{s,I}\colon\W s \om(I)\to\V s \om(I)$ is an isomorphism.
\item The restriction map induces an isomorphism $\rho_s\colon \W s\om\to\V s\om$  of $\Gm$-equivariant sheaves.
\end{enumerate}
\end{prop}

The map~$\rho_s$ being an isomorphism allows us to formulate the definition of concepts involving singularities in terms of $s$-analytic boundary germs, as it is done in~\cite{BLZm}: For any finite subset~$F\subseteq\proj\RR$ we let (as in Section~\ref{sect-sars}) \index[symbols]{Wad@$\W s \om[F]$}\index[symbols]{Wab@$\W s \fs(\proj\RR)$}
\begin{equation}
 \W s \om[F] \coloneqq \W s \om(\proj\RR\smallsetminus F)\,, \qquad  \W s \fs\!(\proj\RR) \coloneqq \varinjlim_F \W s \om[F]\,.
\end{equation}
and\index[symbols]{Wac@$\W s \fxi(\proj\RR)$}
\begin{equation}
 \W s \fxi(\proj\RR) \coloneqq \varinjlim_{F\subseteq \Xi} \W s \om[F]\,.
\end{equation}
The restriction map~$\rho_s\colon \W s \om \to \V s \om$ obviously extends to a $\Gamma$-equivariant isomorphism \index[symbols]{R@$\rho_s$}
\begin{equation}
 \rho_s\colon \W s \fs\!(\proj\RR) \to \V s \fs\!(\proj\RR)\,;
\end{equation}
the sets of singularities of sections of~$\V s \fs\!(\proj\RR)$ and $\W s \fs\!(\proj\RR)$ correspond to each other under~$\rho_s$. 
For~$w\in\W s \fs\!(\proj\RR)$ we set\index[symbols]{B@$\bsing$}\index[defs]{boundary singularities}\index[defs]{singularities!boundary}
\begin{equation}
 \bsing w \coloneqq \bsing \rho_s(w)\,.
\end{equation}
Likewise, conditions on the singularities can be formulated for elements in the space~$\W s \fxi(\proj\RR)$. We say that an element~$w$ of~$\W s \fxi(\proj\RR)$ satisfies the conditions~\tsmp, $\infty$, \taj or~\texc if and only if $\rho_s(w) \in \V s \fxi(\proj\RR)$ 
has the considered property. More explicitly, for $\tcond\in\{\smp,\infty,\aj,\exc\}$, an element~$w\in\W s \fxi(\proj\RR)$ satisfies the condition~\tcond if and only if there exist a finite set~$F\subseteq \proj\RR$, an open neighborhood~$U$ of~$\proj\RR\smallsetminus F$ in~$\proj\CC$ and a map~$A\colon U\to \CC$ such that 
\begin{enumerate}[{\rm (i)}]
 \item $A$ is real-analytic on $U$\,,
 \item $A(z) = y^{-s}w(z)$ for all $z\in U\cap\uhp$\,,
 \item $A\vert_{\proj\RR} \in \V s \fxi$ satisfies~\tcond. For \tcond=\,\texc we require in addition that $U$ is rounded at all points in~$F\cap \Gamma\infty$ and that~$A\vert_{\proj\RR\smallsetminus F}$ extends holomorphically to~$U$.
\end{enumerate}

For completeness we remark that our terminology is not completely identical to the one in~\cite{BLZ13,BLZm}. However, the essence is preserved. In~\cite[\S5.1]{BLZ13} and~\cite[\S3.1]{BLZm} the disk model of the hyperbolic plane is used for the definition of $s$-analytic boundary germs, and in~\cite{BLZ13} they are called eigenfunction germs.

\subsection{Cohomology classes attached to funnel forms}\label{sect-ccaffabg}

The restriction map $\rho_s$ induces a restriction map
\[ 
H^1\bigl( F^\tess_\bullet; \W{s}\om(\proj\RR),\W{s}{\fxi;\exc,\cond;\aj}(\proj\RR)\bigr)
\stackrel{\rho_s}\longrightarrow H^1\bigl( F^\tess_\bullet; \V{s}\om(\proj\RR),\V{s}{\fxi;\exc,\cond;\aj}(\proj\RR)\bigr)\,, 
\]
which, by Proposition~\ref{thm-resiso}, is an isomorphism for any choice of~$\tcond\in\{\underline{\ \ },\smp,\infty\}$. Set $V\coloneqq\W s \om(\proj\RR)$ and $W\coloneqq \W s {\fxi;\exc,\cond;\aj}(\proj\RR)$. We let (cf.\@ Section~\ref{sec:def_mixed})
\begin{align*}
 Z^1(F^\tess_\bullet; V,W)_\sic \coloneqq \Big\{ c\in Z^1(F^\tess_\bullet; & V, W) \setmid
 \\
 & \forall\, e\in X_1^\tess\colon \bsing c(e) \subseteq \proj\RR \cap \{t(e),h(e)\} \Big\}
\end{align*}
be the space of cocycles with singularity condition, and
\[
 Z^1(F^\tess_\bullet; V,W)^\van \coloneqq \left\{ c\in Z^1(F^\tess_\bullet; V,W) \setmid \rho_s(c(f_0))\vert_{(\lambda-1,1)_c} = 0 \right\}
\]
the space of cocycles  with vanishing condition in~$Z^1(F^\tess_\bullet; V,W)$. The definitions of the spaces~$Z^1(F^\tess_\bullet; V,W)_{\cond_1}^{\cond_2}$, $B^1(F^\tess_\bullet; V,W)_{\cond_1}^{\cond_2}$ and  $H^1(F^\tess_\bullet; V,W)_{\cond_1}^{\cond_2}$ for the conditions~$\tcond_1\in\{\underline{\ \ }, \sic\}$, $\tcond_2\in\{\underline{\ \ },\van\}$ are then immediate. The restriction map~$\rho_s$ from above descends to an isomorphism 
\begin{equation}\label{res-coh}
\begin{aligned}
 H^1\bigl( F^\tess_\bullet; \W{s}\om(\proj\RR),&\W{s}{\fxi;\exc,\cond;\aj}(\proj\RR)\bigr)_{\cond_1}^{\cond_2}
\\ &
 \qquad\stackrel{\rho_s}\longrightarrow H^1\bigl( F^\tess_\bullet; \V{s}\om(\proj\RR),\V{s}{\fxi;\exc,\cond;\aj}(\proj\RR)\bigr)_{\cond_1}^{\cond_2}
\end{aligned}
\end{equation}
for all choices of~$\cond\in\{\underline{\ \ },\smp,\infty\}$, $\cond_1\in\{\underline{\ \ }, \sic\}$ and $\cond_2\in\{\underline{\ \ },\van\}$, which we continue to call~$\rho_s$. Let \index[symbols]{Q@$\qcoh_s$}
\begin{equation}\label{qcoh}
 \qcoh_{s} \coloneqq \rho_s^{-1}\circ \coh_s \colon \A_s^* \to H^1\bigl( F^\tess_\bullet;
\W{s}\om(\proj\RR),\W{s}{\fxi;\exc,\cond(*);\aj}(\proj\RR)\bigr)^\van_\sic\,,
\end{equation}
where~$\coh_s$ is the map from Proposition~\ref{prop-cohAH}, $*\in\{\underline{\ \ }, 0, 1\}$ and 
\[
\cond(\underline{\ \ }) = \underline{\ \ },\quad \cond(1) = \smp,\quad \cond(0) = \infty\,.
\]
Proposition~\ref{prop-cohAH} yields that $\qcoh_s$ is injective (for appropriate values of~$s$); in the following section we will show that $\qcoh_s$ is also surjective. 

Similar to the map~$\coh_s$, also the map~$\qcoh_{s}$ can be described explicitly with a kernel function:
We replace the Poisson kernel~$R(t;z)^s$ used for~$\coh_s$ by the kernel function (see~\cite[(3.8), (A.8), (A.9)]{BLZ13})
\begin{equation}\label{def:kernel_qcoh}
q_s(z,z') \coloneqq Q_{s-1}\bigl( \cosh d(z,z')\bigr),
\end{equation}
where $d(z,z')$ denotes the hyperbolic distance between~$z,z'\in\uhp$, and $Q_{s-1}$ denotes the $Q$-Legendre function with parameter~$s-1$. Formula~\eqref{def:kernel_qcoh} defines the kernel function for~$s\in \CC\smallsetminus \ZZ_{\leq 0}$ and on~$\uhp^2\smallsetminus\text{(diagonal)}$. 
For any fixed~$z'\in\uhp$, the map~$q_s(\cdot, z')$ represents an element of~$\W s \om\bigl(\proj\RR)$. Its image under the restriction map~$\rho_s$ is closely related to the kernel function of~$\coh_s$: 
\begin{align*}
\rho_s q_s(\,\cdot\,,z') &\ceqq  b(s) \, R(\,\cdot\,;z')^s\,,\\
b(s) &\ceqq  \frac{\Gf(s)\,\Gf(1/2)}{\Gf(s+1/2)}\,. 
\end{align*}

For any funnel form~$u$, the cohomology class~$\qcoh_{s} u$ can be represented by a cocycle~$\psi_u $ on $F^\tess_1$. For edges~$e\in X^{\tess,Y}_1$ we have
\begin{equation}\label{quY} 
\psi_u(e)(z) \ceqq  \int_{e} \bigl\{ u, q_s(\,\cdot\,,z)
\bigr\}
\end{equation}
and
\[
\rho_s \psi_u(e) = b(s)\, c_u(e)\,,
\]
where~$c_u\in Z^1(F_\bullet^\tess; \V s \om(\proj\RR), \V s \fxi)$ is a cocycle associated to~$u$ by Proposition~\ref{prop-af-co}. For edges that involve the boundary~$\proj\RR$ of~$\uhp$ we set
\begin{equation}
\psi_u(e) \coloneqq b(s)^{-1}\, \rho_s^{-1} c_u(e)\,,
\end{equation}
avoiding the normalization process for the boundary germs.

An important property of the kernel function~$q_s$ is the following theorem, which is similar to the Cauchy theorem for holomorphic
functions.

\begin{thm}[Theorem 3.1 in \cite{BLZ13}]\label{thm-psC}
Let $s\in \CC \smallsetminus\frac12\ZZ_{\leq 0}$, let $C$ be a piecewise smooth positively oriented simple closed curve in~$\uhp$, and let $U\subset \uhp$ be an open set containing~$C$ and the region enclosed by~$C$. Suppose that $u\in \E_s(U)$. Then for each~$z\in \uhp\smallsetminus C$
\[ 
\frac2\pi \int_C \bigl\{ u,q_s(\cdot,z )\bigr\} 
\ceqq  
\begin{cases} 
u(z) &\text{ if $z$ is inside the region encircled by~$C$}\,,\\
0 &\text{ if $z$ is outside the region encircled by~$C$}\,.
\end{cases}
\]
\end{thm}

In~\cite[Theorem 3.1]{BLZ13} this theorem is proved for the disk model of hyperbolic space and the Green's form~$[u,q_s]$. The transition to the Green's form~$\{u, q_s\}$ used here requires the multiplicative factor in front of the integral is~$\frac 2\pi$ instead of~$\frac1{2\pi}$ as in~\cite{BLZ13} (see Remark~\ref{rmk:transition}).  For a closed path the exact contribution~$-i\,d\bigl(uq_s(\cdot,z)\bigr)$ (by which the two Green's forms differ) does not matter, and for the cocycles on~$F^\tess_1$ it amounts to the addition of a coboundary.

\subsection{Representatives of boundary germs}

Our aim is to construct a funnel form from a cocycle with values in the boundary germs. This requires to go from Laplace eigenfunctions defined only near the boundary of~$\uhp$ to Laplace eigenfunctions on the whole of~$\uhp$. The following proposition (proved in~\cite{BLZ13}) shows that non-trivial boundary germs cannot be represented by elements of~$\E_s(\uhp)$.

\begin{prop}[Proposition~5.3 in \cite{BLZ13}]\label{prop:trivrep}
For~$s\in \CC\smallsetminus \ZZ_{\leq 0}$, the zero element is the only element of~$\W{s}\om(\proj\RR)$ that can be represented by an element of~$\E_s(\uhp)$.
\end{prop}

For the construction of funnel forms, we use therefore representatives of $s$-analytic boundary germs that are \textit{a priori} functions on all of~$\uhp$ but not necessarily Laplace eigenfunctions.  

For any open subset~$I\subset \proj\RR$ we let $\G{s}\om(I)$ \index[symbols]{G@$\G s \om$} denote the vector space of functions in~$C^2(\uhp)$ that represent a section of~$\W{s}\om(I)$. Thus, an element~$f\in \G{s}\om(I)$ is a \mbox{$C^2$-}function on~$\uhp$ for which there exists an open neighborhood~$U$ of~$I$ in~$\proj\CC$ such that the restriction of~$f$ to~$U\cap\uhp$ is in~$\E_s(U\cap\uhp)$, and that the function
\[
U\cap\uhp\to\CC\,,\quad z\mapsto y^{-s}\, f(z)\,,
\]
extends to a real-analytic function on~$U$. In particular, elements of~$\G{s}\om(\proj\RR)$ represent $s$-analytic boundary germs on~$\proj\RR$. 

From any open neighborhood~$U$ of~$I$ in~$\proj\CC$ and any representative of an $s$-analytic boundary germ in~$\mc E_s(U\cap\uhp)$ we can construct an element of~$\G s \om(I)$ representing the same $s$-analytic boundary germ by using a $C^2$ cut-off function. 
In this way we get a short exact sequence of sheaves 
\[
0 \longrightarrow \N{s}\om \rightarrow \G{s}\om \rightarrow \W{s}\om
\rightarrow 0\,,
\]
where $\N{s}\om$ is defined as the kernel of the map $\G s \om \to \W s \om$ described above. \index[symbols]{N@$\N s \om$} A section in~$\N{s}\om(I)$ is a function in~$C^2(\uhp)$ that vanishes on~$U \cap \uhp$ for some neighborhood~$U$ of~$I$ in~$\CC$. In particular, $\N{s}\om(\proj\RR)=C^2_c(\uhp)$.

We extend these definitions, in a way analogous to the procedures in Sections~\ref{sect-sars} and \ref{sect-abg}, to form the spaces~$\G{s}\om[F] \supseteq \N{s}\om[F]$ for finite subsets~$F\subseteq\proj\RR$, and the $\Gm$-modules~$\G{s}\fxi(\proj\RR) \supseteq \N{s}\fxi(\proj\RR)$, etc. \index[symbols]{G@$\G s \fxi$}\index[symbols]{N@$\N s \fxi$} For~$f\in C^2(\uhp)$ we call the smallest closed subset~$Z\subset \uhp$ such that $f$ restricts to~$\E_s(\uhp\smallsetminus Z)$ its set of \emph{singularities}~$\sing s(f)$. \index[defs]{set of singularities}\index[defs]{singularities}\index[symbols]{S@$\sing s$} We remark that this concept of singularities depends on the spectral parameter~$s$. Indeed, if $u \in \E_{s_1}(\uhp)$ such that $u\neq 0$ and if $s\in\CC$ such that~$s_1(1-\nobreak s_1) \neq s(1-\nobreak s)$, then $\sing s(u) = \uhp$. We remark further that if $f$ represents the $s$-analytic boundary germ~$w$, then the sets~$\sing s(f)$ and $\bsing(w)$ are always disjoint. While $\sing s(f)$ is a subset of~$\uhp$, $\bsing(w)$ is a subset of~$\proj\RR$.

\section{Surjectivity II: From cocycles to funnel forms}\label{sect-coc-ff}
\markright{15. FROM COCYCLES TO FUNNEL FORMS}

In this section we construct an isomorphism \index[symbols]{T@$\alphu_s$}
\[ 
\alphu_s\colon H^1\bigl( F^\tess_\bullet; \W s \om(\proj\RR), \W s {\fxi;\exc;\aj}(\proj\RR)\bigr)^\van_\sic \rightarrow \A_s
\]
that descends to isomorphisms 
\begin{align*}
\alphu_s&\colon H^1\bigl( F^\tess_\bullet;\W s \om(\proj\RR), \W s {\fxi;\exc,\smp;\aj}(\proj\RR)\bigr)^\van_\sic \rightarrow \A_s^1
\intertext{and}
\alphu_s&\colon H^1\bigl( F^\tess_\bullet; \W s \om(\proj\RR), \W s {\fxi;\exc,\infty;\aj}(\proj\RR)\bigr)^\van_\sic \rightarrow \A_s^0\,,
\end{align*}
each one inverting the corresponding instance of the map~$\qcoh_s$ (see~\eqref{qcoh}). As already mentioned above, we follow largely the approach in~\cite[\S7, \S12.2--4]{BLZm}, where similar results were obtained for \emph{cofinite} Fuchsian groups, and put the main emphasis of the discussion on the extension to funnels. 
We have to carry out a number of tasks:
\begin{enumerate}[(T1)]
\item\label{task1} For a given cocycle~$\psi$ construct an invariant eigenfunction~$u_\psi$, and show that $u_\psi$ does not depend on the choice of the cocycle in its cohomology class or on other choices made in the construction, and yields zero if $\psi$ is a coboundary.
\item\label{task2} Show that $u_\psi$ is a funnel form, and is a resonant or cuspidal funnel form if the cocycle~$\psi$ has values in the appropriate submodule.
\item\label{task3} Show that if the cocycle~$\psi$ represents a funnel form~$u$ (i.\,e., if $\qcoh_s(u)=\psi$), then $u_\psi = u$.
\item\label{task4} Show that if $u_\psi =0$, then $\psi$ is a coboundary.
\end{enumerate}

We start with proposing a construction of $\Gamma$-invariant eigenfunctions from cocycles in~$Z^1\bigl(F^\tess_\bullet;\W s \om(\proj\RR), \W s {\fxi;\exc;\aj}(\proj\RR)\bigr)^\van_\sic$.

\subsection{From a cocycle to an invariant eigenfunction}\label{sect-coiei}

In this section we discuss how to associate to a cocycle $\psi \in Z^1\bigl(F^\tess_\bullet;\W s \om(\proj\RR), \W s {\fxi;\exc;\aj}(\proj\RR)\bigr)^\van_\sic$ a $\Gamma$-invar\-iant eigenfunction~$u_\psi$ of $\Delta$ with spectral parameter~$s$. We will first assign to any cocycle~$\psi$ and any $\eps>0$ a cochain~$\tilde\psi_\eps\in C^1(F^\tess_\bullet;\G s \om(\proj\RR), \G s {\fxi;\exc;\aj}(\proj\RR)\bigr)$. One of the essential properties of the cochain~$\tilde\psi_\eps$ will be that for each edge~$e\in X^\tess_1$, the function~$\tilde\psi_\eps(e)\in\G s \fxi(\proj\RR)$ is a representative of~$\psi(e)\in \W s \fxi(\proj\RR)$, and that the set~$\sing s \tilde\psi_\eps(e)$ is `sufficiently close' to~$e$. Evaluating this cochain along well-chosen cycles around points in~$\uhp$ will then define an element~$u_\psi\in\mc E_s^\Gamma$. 

To state the properties of the associated cochain~$\tilde\psi_\eps$ more precisely we need a few definitions. For $R>0$ and $e\in X^{\tess,Y}_1$ we let $N_R(e)$ denote the closed hyperbolic $R$-neighborhood of~$e$, that is, the set of points in~$\uhp$ with hyperbolic distance to~$e$ at most~$R$. Further we set \index[symbols]{N@$N_R$}
\begin{equation}
 N_R(\infty) \coloneqq \bigl\{ z\in\uhp \setmid |\Rea z| \leq R,\ |\Ima z| \geq \tfrac1R \bigr\}\,,
\end{equation}
and for~$g\in\Gamma$, 
\begin{align}
 N_R(ge_\infty) & \coloneqq gT\,N_R(\infty)\,,
 \\
 N_R(ge_1) & \coloneqq gh\,N_R(\infty)\,,
\end{align}
where
\[
 h \coloneqq \frac{1}{\sqrt{2Y}}\bmat{1}{-Y}{1}{Y}
\]
is the element in $\PSL_2(\RR)$ that maps the geodesic path from~$iY$ to~$\infty$ to the edge~$e_1$. For any $R>Y$ the set~$N_R(e_\infty)$ is a neighborhood of~$e_\infty$, and the set~$N_R(e_1)$ is a neighborhood of~$e_1$. For any $\eps>0$ we let
\[
 E_\eps(0) \coloneqq \{ z\in \uhp\setmid \Ima z < \eps\}
\]
denote the set of point in~$\uhp$ that are $\eps$-near to~$\RR$. For $g\in\Gamma$ we set
\begin{equation}
 E_\eps(gf_0) \coloneqq gE_\eps(0)\,,
\end{equation}
which we may understand as the set of points in~$\uhp$ that are $\eps$-near to $gf_0$ in the imaginary direction.

\begin{prop}\label{prop:eps_cochain}
Let $s\in\CC$, $\Rea s>0$. Suppose that 
\[
\psi\in Z^1\bigl(F^\tess_\bullet; \W s \om(\proj\RR), \W s {\fxi;\exc;\aj}(\proj\RR)\bigr)_\sic^\van
\]
and $\eps>0$. Then there exists a cochain
\[
\tilde\psi_\eps\in C^1\bigl(F^\tess_\bullet;\G s \om(\proj\RR), \G s {\fxi;\exc;\aj}(\proj\RR)\bigr)
\]
with the following properties:
\begin{enumerate}[{\rm (i)}]
 \item For all~$e\in X^\tess_1$, the function~$\tilde\psi_\eps(e)\in \G s \fxi(\proj\RR)$ is a representative of the element~$\psi(e)\in\W s \fxi(\proj\RR)$.
 \item There exists~$R>0$ such that for all~$e\in X^\tess_1\smallsetminus(\Gamma f_0)$ we have 
 \[\sing s \tilde\psi_\eps(e)\subseteq N_R(e)\,.\]
 \item\label{prop:cochainiii} There exist~$R>0$ such that for all~$g\in\Gamma$ we have
 \[
  \sing s \tilde\psi_\eps(gf_0) \subseteq \bigl( N_R(ge_1) \cup N_R(gTSe_1) \bigr) \cap E_\eps(gf_0)\,.
 \]
\end{enumerate}
\end{prop}

For showing the existence of the cochain~$\tilde\psi_\eps$ in Proposition~\ref{prop:eps_cochain} we will prescribe it on the six edges in $\{e_1, e_2, e_\infty, f_0,f_1,f_\infty\}$ (see Figure~\ref{fig-gen}, p.~\pageref{fig-gen}) which form a $\Gamma$-basis of~$X^\tess_1$, and then extend it to all of~$X^\tess_1$ in the unique way such that~$\tilde\psi_\eps$ is $\Gamma$-equivariant. The values of $\tilde\psi_\eps$ on these six edges are independent of each other. The value of~$\eps$ has no influence on the value of $\tilde\psi_\eps$ on the five edges in $\{e_1,e_2,e_\infty, f_1,f_\infty\}$, it only influences the choice of~$\tilde\psi_\eps(f_0)$. This property can already be expected from the statement of Proposition~\ref{prop:eps_cochain} considering that only for $\tilde\psi_\eps(f_0)$ (and its $\Gamma$-translates) the bounding domain for the set of singularities depends on~$\eps$. There is a lot of freedom in the choice of~$\tilde\psi_\eps$ (see the proof below of Proposition~\ref{prop:eps_cochain}). Nevertheless  we cannot expect to find a choice of~$\tilde\psi_\eps$ satisfying the cocycle relation, and hence, up to some exceptional cases, $\tilde\psi_\eps$ is only a cochain, not a cocycle. 

For cofinite Fuchsian groups such a result is provided in~\cite[Section~12.2]{BLZm}, of which we will take advantage here. The presence of ordinary points in our setup makes the construction more complicated. In particular, compared to~\cite{BLZm}, the treatment of the edge~$f_0$ requires a new approach.

\begin{proof}[Proof of Proposition~\ref{prop:eps_cochain}]
For $e\in \{e_2, f_1, f_\infty\}$ the function~$\psi(e)$ is 
an element of~$\W s \om(\proj\RR)$ since $e\in X^{\tess,Y}_1$. For $\tilde\psi_\eps(e)$ we fix any lift of~$\psi(e)$ to~$\G s \om(\proj\RR)$, not depending on~$\eps$. Since we find, by definition of~$\G s \om(\proj\RR)$, an open neighborhood~$U$ of~$\proj\RR$ in~$\proj\CC$ such that the restriction of~$\tilde\psi_\eps(e)$ to~$U\cap\uhp$ is in~$\mc E_s(U\cap\uhp)$, the set~$\sing s \tilde\psi_\eps(e)$ of singularities of~$\tilde\psi_\eps$ is in a compact subset of~$\uhp$. Thus, for all sufficiently large~$R>0$,
\[
\sing s \tilde\psi_\eps(e)\subseteq N_R(e)\,.
\]

For $e=e_\infty$, we rely on~\cite[\S~12.2]{BLZm} for the choice of~$\tilde\psi_\eps(e_\infty)$. In this situation, the condition~\texc at the singularities of~$\psi(e_\infty)$ yields that we can find a lift~$\tilde\psi_\eps(e_\infty)$ of~$\psi(e_\infty)$ in~$\G s {\fxi;\exc;\aj}(\proj\RR)$ that satisfies $\sing s \tilde\psi_\eps(e_\infty) \subseteq  N_R(e_\infty)$ for all sufficiently large~$R>0$.

For $e=e_1$ we have $\psi(e_1) \in \W s{\om;\exc;\aj}[1]$. Thus, $\psi(e_1)(z) = y^s A(z)$ on~$U \cap \uhp$ for some open neighborhood~$U$ of~$\proj\RR\smallsetminus \{1\} $ in~$\CC$ and a real-analytic map~$A\colon U\to\CC$. The condition~\taj (`analytic jump') implies that there is a real-analytic function~$A_+$ on a neighborhood~$U_+$ of~$(1-\nobreak \eps,\infty)$ in~$\CC$ such that $A_+$ coincides with~$A$ on and near~$(1,\infty)$, and similarly there is a real-analytic function~$A_-$ on a complex neighborhood of~$ (-\infty,1+\nobreak\eps)$ that coincides with~$A$ on and near~$(-\infty,1)$. This means that near the point~$1$ the element~$\psi(e_1)$ is represented by an $s(1-\nobreak s)$-eigenfunction of~$\Delta$ with a singularity on the points of~$e_1$. Therefore we find (and fix) a lift~$\tilde\psi_\eps(e_1)\in \G s {\fxi;\exc;\aj}(\proj\RR)$ of~$\psi(e_1)$ with $\sing s \tilde\psi_\eps(e_1) \subseteq N_R(e_1)$ for all sufficiently large~$R>0$.

For $e=f_0$, an analogous discussion shows that the condition~\taj makes it possible to lift $\psi(f_0)\in\W s {\om;\exc;\aj}[1,\lambda-1]$ to a map~$\Psi\in\G s {\fxi;\exc;\aj}(\proj\RR)$ with
\[
\sing s \Psi  \subset N_R(e_1) \cup N_R(TS e_1)
\]
for all sufficiently large~$R>0$. 

We now show that we can choose $\Psi$ such that $\sing s \Psi$ is, in addition, contained in~$E_\eps(0) = \{ z\in \uhp\setmid \Ima z < \eps\}$. Let~$A$ be a real-analytic core of~$\psi(f_0)$ (or of~$\Psi$). For an arbitrary choice of~$\Psi$, its set~$\sing s \Psi$ of singularities is contained in a domain as indicated in Figure~\ref{fig-f0}.
\begin{figure}
\centering
\[\setlength\unitlength{1.2cm}
\begin{picture}(7,1.9)(0,-.1)
\put(0,0){\line(1,0){6}}
\put(1.5,0){\circle*{.1}}
\put(5,0){\circle*{.1}}
\put(1.4,-.3){$1$}
\put(4.7,-.3){$\lambda\!-\!1$}
\put(0.5,.7){$0$}
\put(5.9,.9){$0$}
\put(2.9,.16){$y^s A(z)$}
\put(2.4,.8){singularities}
\thicklines
\qbezier(1.5,0)(1.5,.5)(1,.6)
\qbezier(5,0)(5,.5)(5.7,.6)
\qbezier(1,.6)(1.2,1.1)(1.7,1.4)
\qbezier(5.7,.6)(5.4,1.1)(4.8,1.4)
\put(1.7,1.4){\line(1,0){3.1}}
\put(3.25,0){\oval(3.5,.9)[t]}
\end{picture}
\]
\caption{Lift of~$\psi(f_0)$} \label{fig-f0}
\end{figure}
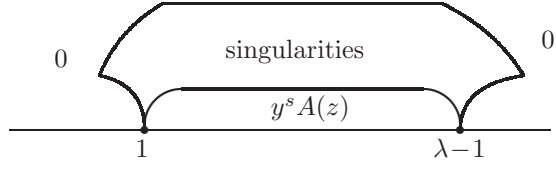
The condition~\tvan on the cocycle~$\psi$ shows that $\psi(f_0)$ vanishes identically near~$0$. Since $0$ has an obvious real-analytic extension to~$\uhp$, we can put the transition zone between~$0$ and~$y^s \, A(z)$ as near to the interval~$[1,\lambda-1]$ as is convenient. In particular, we can choose $\Psi$ such that $\Psi(z) = 0$ for~$\Ima z \geq \eps$, and hence $\sing s \Psi$ is $\eps$-near to~$\RR$. We set $\tilde\psi_\eps(f_0)$ to be any such choice of~$\Psi$. 

Extending the choices for $\tilde\psi_\eps(e)$ with $e\in\{e_1,e_2,e_\infty, f_0,f_1,f_\infty\}$ to a $\Gamma$-equiv\-ari\-ant function~$\tilde\psi_\eps$ on all of $X^\tess_0$ completes the proof.
\end{proof}

We now discuss how to associate to $\psi\in Z^1\bigl(F^\tess_\bullet;\W s \om(\proj\RR), \W s {\fxi;\exc;\aj}(\proj\RR) \bigr)$ an element $u_\psi\in \E_s^\Gamma$. In view of Theorem~\ref{thm-psC} and the integral transform in~\eqref{quY} our definition is rather natural; we proceed as in~\cite[\S12.2]{BLZm}. To that end let $Z$ be a compact subset of~$\uhp$ (which one should imagine as being rather large; we will define the eigenfunction first on~$Z$ and then enlarge~$Z$). We choose~$\eps>0$, a lift $\tilde\psi_\eps\in C^1(F^\tess_\bullet;\G s \om(\proj\RR), \G s {\fxi;\exc;\aj}(\proj\RR)\bigr)$ of $\psi$ (as provided by Proposition~\ref{prop:eps_cochain}) and a (wide) path~$C$ along edges of~$\tess$ such that 
\begin{itemize}
\item $C$ encircles~$Z$ once in the positive direction, and
\item for each edge~$e\in X^\tess_1$ that is part of~$C$, the set~$Z$ is contained in a component of~$\uhp\smallsetminus\sing s \tilde\psi_\eps(e)$ that has a part of~$\proj\RR$ in its boundary. 
\end{itemize}
The path~$C$ may have to go through cusps and along edges in~$\Gm f_0$. The difference with the cofinite case in~\cite{BLZm} is that if~$C$ passes through~$f_0$ or a $\Gm$-translate of~$f_0$, then we cannot always make that path~$C$ any wider at that place. The dependence on~$\eps$ of the sets of singularities of~$\tilde\psi_\eps(g\,f_0)$ (for~$g\in\Gm$) ensures that the hyperbolic distance of~$\sing s \tilde\psi_\eps(g\,f_0)$ to~$Z$ can be increased as much as we want by decreasing~$\eps$. Therefore, the choice of a pair~$(\eps, C)$ with the claimed properties is indeed possible. 

For $z\in Z$ we set 
\begin{equation}\label{eq:def_u}
u_{\psi,Z}(z) \coloneqq \frac2\pi\, \tilde\psi_\eps(C)(z)\,.
\end{equation}
In the same way as in~\cite[\S7,\S12]{BLZm} we deduce that the value of the right hand side of~\eqref{eq:def_u} does not depend on the choice of the compact set~$Z$ that contains~$z$, and not on the choices of~$\eps$, $C$ and~$\tilde\psi_\eps$ as long as these objects obey the requirements imposed above on their choices. Therefore, $u_{\psi,Z}$ is indeed only depending on~$\psi$ and~$Z$. Further, for any compact set~$Z'\subseteq\uhp$ with $Z'\supseteq Z$, the map~$u_{\psi,Z}$ is the restriction of~$u_{\psi,Z'}$ to~$Z$. Thus, exhausting $\uhp$ with increasing sequences of compact subsets, we get in the limit a function \index[symbols]{U@$u_\psi$}
\begin{equation}
u_\psi\colon\uhp\to\CC
\end{equation}
only depending on~$\psi$. This function is a $\Gamma$-invariant Laplace eigenfunction with eigenvalue~$s(1-s)$. In total, we constructed a map
\begin{equation}\label{eq:coc_u}
Z^1\bigl(F^\tess_\bullet;\W s \om(\proj\RR), \W s {\fxi;\exc;\aj}(\proj\RR)\bigr)^\van_\sic\to \mc E_s^\Gamma\,,\qquad \psi\mapsto u_\psi\,.
\end{equation}

\begin{prop}\label{prop-aldef}
Let $s\in \CC$. 
\begin{enumerate}[{\rm (i)}]
\item The map in~\eqref{eq:coc_u} induces a linear map \index[symbols]{T@$\alphu_s$}
\begin{equation}\label{al-def} 
\alphu_s\colon H^1\bigl(F^\tess_\bullet; \W s \om(\proj\RR),\W
s{\fxi;\exc;\aj}(\proj\RR)\bigr)^\van_\sic \rightarrow \E_s^\Gm\,.
\end{equation}
\item Suppose that $\Rea s\in (0,1)$. If $u\in \A_s$ and $s\neq \frac12$, or if $u\in \A_s^0$, then
\begin{equation}\label{al-coh}
\alphu_s \, \qcoh_s u = u\,.
\end{equation}
\end{enumerate}
\end{prop}

\begin{proof}
These statements follow as in~\cite[\S7, \S12]{BLZm}: If $\psi$ in~\eqref{eq:coc_u} is a coboundary, then $u_\psi=0$, and hence the map in~\eqref{eq:coc_u} descends to cocycle classes. 

For the second group of statements suppose that $u\in\A_s$ and $\Rea s\in (0,1)$. If $s\not=\frac12$ or $u\in\A_s^0$, then Proposition~\ref{prop-cohAH} in combination with~\eqref{res-coh} shows that 
\[\qcoh_s u \in H^1\bigl(F^\tess_\bullet; \W s \om(\proj\RR),\W
s{\fxi;\exc;\aj}(\proj\RR)\bigr)^\van_\sic\,.\]
If now $\psi$ represents~$\qcoh_s u$, then a reasoning as in~\cite[\S12.2]{BLZm} (using Theorem~\ref{thm-psC} on the boundary of the face~$V_1\in X^\tess_2$; see Figure~\ref{fig-gen} on p.~\pageref{fig-gen}) implies that $u_\psi=u$.
\end{proof}

\begin{rmk}
\begin{enumerate}[{\rm (i)}]
\item The existence of the map~$\alphu_s$ does not require any assumptions on the spectral parameter~$s$. The restrictions on~$s$ in the second group of statements in Proposition~\ref{prop-aldef} are needed to guarantee that $\qcoh_s u$ is in the cohomology space on which~$\alphu_s$ is defined.
\item Proposition~\ref{prop-aldef} completes Tasks~\ref{task1} and~\ref{task3} (from the list on p.~\pageref{task1}).
\item In~\cite[\S7.2, (12.5)]{BLZm} the invariant eigenfunction~$u_\psi$ is represented by a sum of $\Gm$-translates of a
function (averaging operators). We do not know whether in the context of a discrete group with infinite covolume such a representation is useful or easy to handle.
\end{enumerate}
\end{rmk}

\subsection{A cocycle on an orbit of ordinary points}

In this section we complete Task~\ref{task4} and the first part of Task~\ref{task2} (from the list on p.~\pageref{task4}). For the latter, i.\,e., showing that $u_\psi$ is a funnel form if $\psi\in Z^1\bigl(F^\tess_\bullet; \W s \om(\proj\RR), \W s{\fxi;\exc;\aj}(\proj\RR)\bigr)^\van_\sic$, we need to understand the behavior of~$u_\psi$ near the set~$\Omega$ of ordinary points. In view of~\eqref{eq:def_u} this means that we need to understand the behavior of~$\psi$ on paths connecting points in~$\Gamma\,1$. For this reason we start with studying the $1$-cocycle on~$\Xi$ induced by~$\psi$, and `good' representatives for it in~$\G s \fxi$. 

Throughout let $\psi \in Z^1\bigl(F^\tess_\bullet; \W s \om(\proj\RR), \W s{\fxi;\exc;\aj}(\proj\RR)\bigr)^\van_\sic$. We recall from Section~\ref{sec:tesselation} that $\psi$ determines a $1$-cocycle~$c$ on~$\Xi\subseteq X^\tess_0$ by 
\begin{equation}\label{eq:coc_path}
 c(\xi,\eta) = \psi(p)\,,
\end{equation}
where $p\in\CC_1[X^\tess_1]$ is any balanced path with $\big(t(p),h(p)\big) = (\xi,\eta)$. The cocycle class~$[\psi]$ of the cocycle~$\psi$ on~$F^\tess_\bullet$ is completely determined by the cocycle~$c$ on~$\Xi$. 

For any~$\xi,\eta\in\Xi$, the element~$c(\xi,\eta)$ is in~$\W s {\fxi;\exc;\aj}(\proj\RR)$ with boundary singularities satisfying $\bsing c(\xi,\eta)\subseteq\{\xi,\eta\}$. In what follows we will find a function~$\tilde c(\xi,\eta)$ in~$\mc E_s(\uhp)$ that represents $c(\xi,\eta)$ on~$(\eta,\xi)_c$. If $\xi=\eta$, then $c(\xi,\eta)=0$. In this case we set $\tilde c(\xi,\eta) \coloneqq 0$. 

From now on we suppose that $\xi\not=\eta$. In the definition of~$c(\xi,\eta)$ in~\eqref{eq:coc_path} we will use only those paths~$p\in\CC_1[X^\tess_1]$ from~$\xi$ to $\eta$ that do not intersect the interval $(\eta,\xi)_c$, i.\,e., none of the edges that is part of~$p$ has an endpoint in~$(\eta,\xi)_c$. We call such paths \emph{$(\xi,\eta)$-reduced}. We assign to~$\psi$ a family~$(\tilde\psi_\eps)_{\eps>0}$ of cochains in~$C^1(F^\tess_\bullet;\G s \om(\proj\RR),\G s{\fxi;\exc;\aj}(\proj\RR)\bigr)$ as provided by Proposition~\ref{prop:eps_cochain}.

\begin{figure}
\centering
\includegraphics[height=8cm]{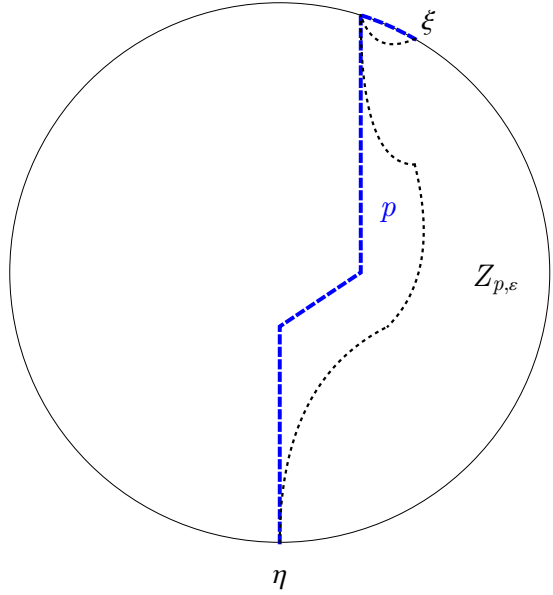}
\caption{A $(\xi,\eta)$-reduced path~$p$ from~$\xi$ to~$\eta$ depicted in the disk model of
the hyperbolic plane, and an indication of a choice of~$Z_{p,\eps}$.} \label{fig-pld0}
\end{figure}
Let $\eps>0$ and $p\in\CC_1[X^\tess_1]$ be a balanced $(\xi,\eta)$-reduced path from~$\xi$ to~$\eta$. Then the map~$\tilde \psi_\eps(p)\in \G s {\fxi;\exc;\aj}(\proj\RR)$ represents $c(\xi,\eta)$. Further, between the path~$p$ and the interval~$[\eta,\xi]_c$ there is a region~$Z_{p,\eps} \subseteq\uhp$ such that
\begin{enumerate}[(Z1)]
\item\label{propZ1} $\sing s \tilde\psi_\eps(p) \cap Z_{p,\eps} = \emptyset$,
\item\label{propZ2} $Z_{p,\eps}$ is connected, 
\item\label{propZ3} $(\eta,\xi)_c$ is in the closure of~$Z_{p,\eps}$ in~$\proj\CC$, and there exists an open neighborhood~$U$ of~$(\eta,\xi)_c$ in $\proj\CC$ such that $Z_{p,\eps}$ contains $U\cap\uhp$.
\end{enumerate}
The restricted function~$\tilde\psi_\eps(p)\vert_{Z_{p,\eps}}$ is in~$\E _s \bigl( \mathring Z_{p,\eps} \bigr)$ and represents $c(\xi,\eta)$ on~$(\eta,\xi)_c$. Therefore it has $s$-analytic boundary behavior on~$(\eta,\xi)_c$. In what follows we always assume that $Z_{p,\eps}$ is chosen to be the maximal set with  Properties~\ref{propZ1}-\ref{propZ3}. 

If we pick another balanced path from~$\xi$ to~$\eta$, say $p'\in\CC_1[X^\tess_1]$, and any~$\eps'>0$ then 
\begin{equation}
 \tilde Z \coloneqq Z_{p',\eps'} \cap Z_{p,\eps} \not= \emptyset\,.
\end{equation}
Since $Z_{p',\eps'}$ and $Z_{p,\eps}$ satisfy~\ref{propZ1}-\ref{propZ3}, it immediately follows that the closure of~$\tilde Z$  contains~$(\eta,\xi)_c$, and there is an open neighborhood~$U$ of~$(\eta,\xi)_c$ in~$\proj\CC$ such that $U\cap\uhp$ is contained in~$\tilde Z$. Let $Z$ be the connected component of~$\tilde Z$ that contains~$(\eta,\xi)_c$. 
\begin{figure}
\centering
\includegraphics[height=8cm]{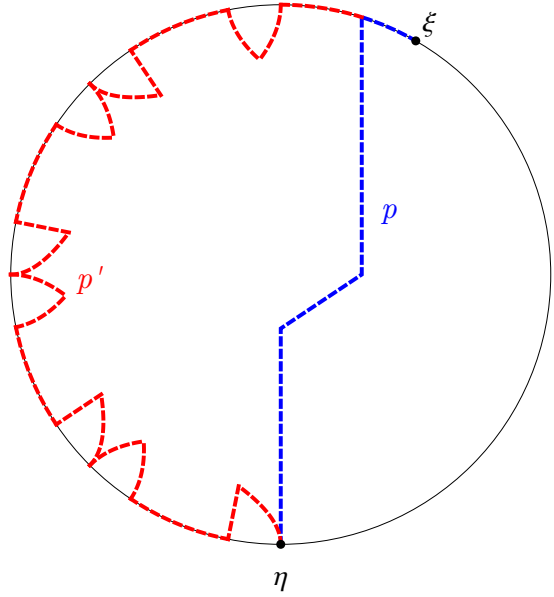}
\caption{Two $(\xi,\eta)$-reduced paths from~$\xi$ to~$\eta$ depicted in the disk model of
the hyperbolic plane. The path~$p'$ is nearer to the interval~$[\xi,\eta]_c$ in~$\proj\RR$ than~$p$.} \label{fig-pld}
\end{figure}
Since $\tilde\psi_\eps(p)\vert_Z$ and $\tilde\psi_{\eps'}(p')\vert_Z$ both represent $c(\xi,\eta)$ on~$(\eta,\xi)_c$, we may choose $U$ so small such that 
\[
\tilde\psi_\eps(p)\vert_{U\cap\uhp} = \tilde\psi_{\eps'}(p')\vert_{U\cap\uhp} \quad\eqqcolon\Psi
\]
and
\[
  \Psi\in \mc E_s(U\cap\uhp)\,.
\]
Because $\tilde\psi_\eps(p)\vert_Z, \tilde\psi_{\eps'}(p')\vert_Z\in\mc E_s(\mathring Z)$ and $Z$ is connected, it now follows that 
\[
 \tilde\psi_\eps(p)\vert_Z = \tilde\psi_{\eps'}(p')\vert_Z\,.
\]
(The maps~$\tilde\psi_\eps(p)$ and $\tilde\psi_{\eps'}(p')$ might differ on larger domains.) We can find a sequence~$(p_j)_{j\in\NN}$ of $(\xi,\eta)$-reduced paths in~$\CC_1[X^\tess_1]$ from~$\xi$ to~$\eta$ that `moves towards' the interval~$[\xi,\eta]_c$ as~$j\to\infty$ and a sequence~$(\eps_j)_j$ in~$\RR_{>0}$ with $\eps_j\to 0$ as~$j\to \infty$ such that $\big(Z_{p_j,\eps_j}\big)_j$ is an increasing sequence of sets exhausting~$\uhp$. (We refer for the possibility of finding these paths to the explanation after the proof of Proposition~\ref{prop:eps_cochain}.) Fixing such sequences, we define $\tilde c(\xi,\eta)\colon\uhp\to \CC$ by 
\begin{equation}\label{eq:seq_paths}
 \tilde c(\xi,\eta)(z) \coloneqq \tilde\psi_{\eps_j}(p_j)(z) \qquad\text{for $z\in Z_{p_j,\eps_j}$, $j\in\NN$\,.}
\end{equation}
Then $\tilde c(\xi,\eta)$ is indeed well-defined, independent of the chosen sequences, and is an element in~$\E_s(\uhp)$ that represents $c(\xi,\eta)$ on~$(\eta,\xi)_c$.

Even though the map 
\begin{equation}\label{eq:laplace_ef}
 \tilde c\colon \Xi\times\Xi \to \E_s(\uhp)\,,\quad (\xi,\eta) \mapsto \tilde c(\xi,\eta)
\end{equation}
does not necessarily inherit the anti-symmetry $c(\xi,\eta)=-c(\eta,\xi)$ from $c$, the elements~$\tilde c(\xi,\eta)$ and~$\tilde c(\eta,\xi)$ are not unrelated. 

\begin{lem}\label{lem:ps-decomp}
For $\xi,\eta\in\Xi$, $\xi\not=\eta$, 
\[
 \tilde c(\xi,\eta) + \tilde c(\eta,\xi) \ceqq  \frac \pi 2 \, u_\psi\,.
\]
\end{lem}

\begin{proof}
For any~$\eps>0$ and any $(\xi,\eta)$-reduced path $p'\in\CC_1[X^\tess_1]$ from~$\xi$ to~$\eta$ and any $(\eta,\xi)$-reduced path $p''\in\CC_1[X^\tess_1]$ from~$\eta$ to~$\xi$ we have
\[
 \tilde \psi_\eps(p') + \tilde\psi_\eps(p'') = \frac\pi2 u_\psi\quad\text{on\ \ $Z_{p',\eps}^{(\xi,\eta)}\cap Z_{p'',\eps}^{(\eta,\xi)}$}
\]
by~\eqref{eq:def_u}. Here, $Z_{p',\eta}^{(\xi,\eta)}$ denotes the region used in the construction of~$\tilde c(\xi,\eta)$ (denoted above by~$Z_{p',\eta}$), and $Z_{p'',\eta}^{(\eta,\xi)}$ denotes the region used in the construction of~$\tilde c(\eta,\xi)$. Picking sequences of paths~$(p'_j)_j$, $(p''_j)_j$ and a null sequence~$(\eps_j)_j$ such that the sequence  
\[
 \Big(Z_{p'_j,\eps_j}^{(\xi,\eta)} \cap Z_{p''_j,\eps_j}^{(\eta,\xi)}\Big)_j
\]
increases and exhausts~$\uhp$ yields the claimed statement.
\end{proof}

Restricted to the ordinary points~$\Gamma\,1$, the cocycle~$c$ enjoys additional properties, in particular certain vanishing properties. In the following lemma we present a first result, in Lemma~\ref{lem:u_definite} a second one. Both results are essentially consequences of the fact that between nearby points in $\Gamma\,1$, we can choose a path along edges of~$X^\tess_1$ that are completely contained in~$\proj\RR$.

\begin{lem}\label{lem:1_boundary}
Let $\xi,\eta\in\Gamma\,1$, $\xi\not=\eta$, and suppose that $(\xi,\eta)_c\subseteq\Omega$. Then $\tilde c(\xi,\eta)=0$ and $u_\psi= \frac{\pi}2\,\tilde c(\eta,\xi)$. 
\end{lem}

\begin{proof}
As soon as $\tilde c(\xi,\eta) = 0$ is shown, Lemma~\ref{lem:ps-decomp} yields 
\[
 u_\psi = \frac{\pi}2 \bigl( \tilde c(\xi,\eta) + \tilde c(\eta,\xi)\bigr) = \frac{\pi}2 \tilde c(\eta,\xi)\,.
\]
So it remains to show that $\tilde c(\xi,\eta)=0$ for $\xi\not=\eta$. To that end, we fix for each $\eps>0$ a lift 
\[
\tilde\psi_\eps\in C^1\bigl(F^\tess_\bullet; \G s \om(\proj\RR), \G s {\fxi;\exc;\aj}(\proj\RR)\bigr)
\]
of~$\psi$ as provided by Proposition~\ref{prop:eps_cochain}. In what follows we evaluate these lifts on an, in a certain sense, optimal path from~$\xi$ to~$\eta$ to determine~$\tilde c(\xi,\eta)$.

Since $\xi,\eta\in\Gamma\,1$ and since the interval~$(\xi,\eta)_c$ is contained in the set~$\Omega$ of ordinary points, there exists a path~$p$ from~$\xi$ to~$\eta$ of the form 
\[
 p = \sum_{j=1}^m g_jf_0
\]
for suitable $m\in\NN$ and $g_j\in\Gamma$, $j\in\{1,\ldots,m\}$ such that 
\[
t(g_1f_0) = \xi\,,\quad h(g_mf_0) = \eta
\]
and
\[
h(g_jf_0) = t(g_{j+1}f_0)\ \text{for $j\in\{1,\ldots, m-1\}$}\,.
\]
(In other words, this path is the direct connection from~$\xi$ to~$\eta$ along~$\proj\RR$.) To simplify notation we set $\xi_0 \coloneqq \xi$ and 
\[
 \xi_j \coloneqq h(g_jf_0) \quad\text{for $j\in\{1,\ldots,m\}$}\,.
\]
We note that $\xi_m=\eta$. For all $j\in\{1,\ldots,m \}$ we have (using \eqref{eq:coc_path})
\[
 c(\xi_{j-1},\xi_j) = \psi( g_jf_0 ) = 0 \quad\text{on $(\xi_j,\xi_{j-1})_c$} 
\]
by the vanishing condition~\tvan. Thus, 
\[
 c(\xi,\eta) = \sum_{j=1}^m c(\xi_{j-1},\xi_j) = 0 \quad\text{on $(\eta,\xi)_c$\,.} 
\]
This implies that for each $\eps>0$, 
\[
 \tilde\psi_\eps(p) = 0 \quad\text{on $Z_{p,\eps}$\,,}
\]
and hence (using \eqref{eq:seq_paths})
\[
 \tilde c(\xi,\eta) = 0 \quad\text{on $Z_{p,\eps}$\,.}
\]
Taking the limit $\eps\to 0$ shows $\tilde c(\xi,\eta) = 0$ on all of~$\uhp$. This completes the proof.
\end{proof}

With these preparations we can carry out the first part of Task~\ref{task2} (from the list on p.~\pageref{task2}).

\begin{prop}\label{prop-alff}
For all $s\in \CC$,
\[ 
\alphu_s H^1\bigl( F^\tess_\bullet; \W s \om(\proj\RR), \W s
{\fxi;\exc;\aj}(\proj\RR)\bigr)^\van_\sic \;\subseteq\; \A_s\,. 
\]
\end{prop}

\begin{proof}
Let $\psi\in Z^1(F^\tess_\bullet;\W s \om(\proj\RR), \W s {\fxi;\exc;\aj}(\proj\RR))^\van_\sic$. We use the notation from above.  It suffices to show that the $\Gamma$-invariant eigenfunction~$u_\psi$ has $s$-analytic boundary behavior near the interval
\[
\Bigl(\frac1{\lambda-1}, \lambda-1 \Bigr) = \bigl(ST^{-1}\, 1, TS\, 1\bigr)
\]
because all ordinary points of~$\Gm$ are in the union of the $\Gm$-translates of this interval.

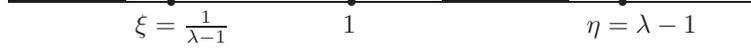
\begin{figure}
\centering
\[\setlength\unitlength{1.2cm}
\begin{picture}(8,.7)(-3.5,-.2)
\put(-3.5,0){\line(1,0){8}}
\put(-2,0){\circle*{.1}}
\put(0,0){\circle*{.1}}
\put(3,0){\circle*{.1}}
\put(-2.4,-.35){$\xi = \frac{1}{\lambda-1}$}
\put(-0.1,-.35){$1$}
\put(2.6,-.35){$\eta = \lambda-1$}
\thicklines
\put(-3.8,0){\line(1,0){1.3}}
\put(1,0){\line(1,0){1.4}}
\end{picture}
\]
\caption{Relative positions of $\xi$, $1$ and $\eta$.}
\label{fig:extension}
\end{figure}

We set 
\[
 \xi\coloneqq \frac{1}{\lambda-1}\quad\text{and}\quad \eta \coloneqq \lambda-1\,.
\]
Lemma~\ref{lem:1_boundary} shows that  
\[
 u_\psi = \frac{\pi}{2}\, \tilde c(\eta,\xi)\,.
\]
By construction, $\tilde c(\eta,\xi) \in \G s {\fxi}(\proj\RR)$ represents $c(\eta,\xi)\in \W s {\fxi}(\proj\RR)$ on $(\xi,\eta)$. Further, since $c$ satisfies the singularity condition~\tsic, the boundary singularities of $c(\eta,\xi)$ are contained in $\{\xi,\eta\}$. In particular, $c(\eta,\xi)$ has no singularities in $(\xi,\eta)$, and hence $c(\eta,\xi)\in \W s \om[\xi,\eta]$. Therefore, $\tilde c(\eta,\xi)$, and hence $u_\psi$, has $s$-analytic boundary behavior near~$(\xi,\eta)$.
\end{proof}

Complementary to Lemma~\ref{lem:1_boundary} is the following lemma, which we will use to complete Task~\ref{task4}. (See Proposition~\ref{prop-alinj}.)

\begin{lem}\label{lem:u_definite}
If $u_\psi =0$, then the cocycle $c$ vanishes on $\Gm\,1\times\Gm\,1$.
\end{lem}

\begin{proof}
On the diagonal of $\Gm1\times\Gm1$, the cocycle $c$ vanishes independently of any requirements on~$u_\psi$. To show that $c$ also vanishes off the diagonal, we will first consider the case that the two elements in $\Gm1$ are in the same connected component of the set~$\Omega$ of ordinary points, and then we will use this first result to discuss the general case. Throughout we will consider $c$ as a cocycle on $\Gm\,1$ instead of on all of~$\Xi$.

\begin{enumerate}[{\rm (i)}]
\item\label{defi_step1} Let $\xi,\eta\in\Gm\,1$, $\xi\not=\eta$ and suppose that $(\xi,\eta)_c\subseteq\Omega$.  Lemma~\ref{lem:1_boundary} shows that $\tilde c(\xi,\eta)=0$. Further, since $u_\psi=0$ by hypothesis, Lemma~\ref{lem:1_boundary} yields $\tilde c(\eta,\xi) = 0$. 
Now $\tilde c(\xi,\eta)$ represents $c(\xi,\eta)$ on~$(\eta,\xi)_c$ by construction, which implies that $c(\xi,\eta)=0$ on~$(\eta,\xi)_c$. Since $\tilde c(\eta,\xi)$ represents $c(\eta,\xi) = - c(\xi,\eta)$ on~$(\xi,\eta)_c$, we find that $c(\xi,\eta) = 0$ on~$(\xi,\eta)_c$. Thus, $c(\xi,\eta) = 0$ on $\proj\RR\smallsetminus\{\xi,\eta\}$.  Since $c(\xi,\eta)$ satisfies the condition~\taj (`analytic jump'), it extends real-analytically to~$\xi$ and~$\eta$ by the null function from \emph{both} sides. From this it follows that $c(\xi,\eta) = 0$ on all of $\proj\RR$. (Alternatively, one can use the argument that $c(\xi,\eta)$ is represented by $0$ on $\proj\RR\smallsetminus\{\xi,\eta\}$. Thus, as an element of $\W s \fxi(\proj\RR)$, it is equal to the null function.)
\item Let now $\nu,\mu\in\Gamma\,1$ be such that $\nu\not=\mu$ (and not necessarily in the same connected component of~$\Omega$). We find (and fix) $g_\nu,g_\mu\in\Gamma$  such that the path~$g_\nu f_0$ starts at~$\nu$ and the path~$g_\mu f_0$ ends at~$\mu$. We set 
\[
 \nu'\coloneqq h\bigl(g_\nu f_0)\quad\text{and}\quad \mu'\coloneqq t\bigl(g_\mu f_0\bigr)\,.
\]
Then the intervals~$(\nu,\nu')_c$ and $(\mu',\mu)_c$ are contained in~$\Omega$. Without loss of generality we may assume that the points $\nu,\nu',\mu,\mu'$ are pairwise distinct, and that $(\mu,\nu)_c\subseteq (\mu',\nu')_c$. By~\eqref{defi_step1}, $c(\nu,\nu')=0$ and hence
\begin{equation}\label{eq:cocyclesstep}
 c(\nu,\mu) = c(\nu,\nu') + c(\nu',\mu) = c(\nu',\mu)\,.
\end{equation}
Since $c$ satisfies the singularity condition~\tsic, we have
\[
 \bsing c(\nu,\mu) \subseteq \{\nu,\mu\}\quad\text{and}\quad \bsing c(\nu',\mu)\subseteq\{\nu',\mu\}\,.
\]
Therefore, the equality in~\eqref{eq:cocyclesstep} implies
\begin{equation}\label{eq:wheresing1}
 \bsing c(\nu,\mu) \subseteq\{\mu\}\,.
\end{equation}
Analogously, we get from 
\[
 c(\nu,\mu) = c(\nu,\mu) + c(\mu,\mu') = c(\nu,\mu')
\]
that 
\begin{equation}\label{eq:wheresing2}
 \bsing c(\nu,\mu) \subseteq \{\nu\}\,.
\end{equation}
Combining~\eqref{eq:wheresing1} and~\eqref{eq:wheresing2} shows that $c(\nu,\mu)$ does not have singularities. Thus, $c(\nu,\mu)\in \W s \om(\proj\RR)$. We now show that $c(\nu,\mu)$ is represented by an element $\varphi\in\E_s(\uhp)$ on all of~$\proj\RR$. Since $c(\nu,\nu') = 0$ by the discussion in~\eqref{defi_step1}, the definition in~\eqref{eq:coc_path} yields that 
\[
 \psi(g_\nu f_0) = c(\nu,\nu') = 0\,,
\]
and hence
\[
\psi(g f_0) = 0 \quad\text{for all~$g\in\Gamma$}
\]
by $\Gamma$-equivariance. Therefore we may and shall choose the lift~$\tilde\psi_\eps$ of $\psi$ such that $\tilde\psi_\eps(g_\nu\,f_0)=0$. (The construction in the proof of Proposition~\ref{prop:eps_cochain} shows that this choice is indeed possible.) From the definition in~\eqref{eq:seq_paths} it follows that 
\[
 \tilde c(\nu,\nu') = \tilde\psi_\eps(g_\nu f_0) = 0\,.
\]
Analogous reasoning shows that 
\[
 \tilde c(\mu',\mu) = \tilde\psi_\eps(g_\mu f_0) = 0\,.
\]
Therefore
\begin{align*}
 \tilde c(\nu,\mu) = \tilde c(\nu,\nu') + \tilde c(\nu',\mu') + \tilde c(\mu',\mu) = \tilde c(\nu',\mu')\,.
\end{align*}
We set $\varphi\coloneqq \tilde c(\nu,\mu)$. By~\eqref{eq:laplace_ef}, $\varphi\in\E_s(\uhp)$.  By construction,
\[
  \varphi\in \G s \om\bigl( \proj\RR\smallsetminus\{\nu,\mu\}\bigr) \cap \G s \om\bigl( \proj\RR\smallsetminus\{\nu',\mu'\}\bigr)\,.
\]
Since the points~$\nu,\nu',\mu,\mu'$ are pairwise distinct, it follows that
\[
  \varphi\in \G s \om\bigl( \proj\RR \bigr)\,.
\]
Again by construction, the map~$\varphi=\tilde c(\nu',\mu')$ represents $c(\nu,\mu) = c(\nu',\mu')$ on~$(\mu',\nu')_c$. Further, since $\tilde c(\nu,\mu) = -\tilde c(\mu,\nu)$ because of $u_\psi=0$ and the result from Lemma~\ref{lem:u_definite}, the map $\varphi = -\tilde c(\mu,\nu)$ represents $c(\nu,\mu) = -c(\mu,\nu)$ also on~$(\nu,\mu)_c$. Therefore, $\varphi$ represents $c(\nu,\mu)$ on
\[
 (\mu',\nu')_c \cup (\nu,\mu)_c = \proj\RR\,,
\]
and hence $\varphi$ is a map in~$\E_s(\uhp)$ that represents~$c(\nu,\mu)\in \W s \om(\proj\RR)$ on all of~$\proj\RR$. From~\cite[Proposition~5.3]{BLZm} (see also Proposition~\ref{prop:trivrep}) it follows that $c(\nu,\mu) = 0$.
\end{enumerate}
This completes the proof.
\end{proof}

We are now able to complete Task~\ref{task4} (from the list on p.~\pageref{task4}).

\begin{prop}\label{prop-alinj}
The linear map $\alphu_s$ in~\eqref{al-def} is injective for all $s\in\CC$.
\end{prop}

\begin{proof}
Suppose that $\psi\in Z^1\bigl(F^\tess_\bullet;\W s \om(\proj\RR), \W s {\fxi;\exc;\aj}(\proj\RR)\bigr)^\van_\sic$ is such that $u_\psi=0$. It suffices to show that $\psi$ is a coboundary. We use the notation from above.

Lemma~\ref{lem:u_definite} shows that $c = 0$ on $\Gm\,1\times\Gm\,1$. Let $\chi\in Z^1\bigl(\Gamma; \W s \fxi(\proj\RR)\bigr)$ denote a group cocycle associated to~$c$ via~\eqref{pick_p} and~\eqref{cocpot} (see also Section~\ref{sec:tesselation}). If we use $\xi_0\in\Gm\,1$ in~\eqref{pick_p} and  $\xi\in\Gm\,1$ in~\eqref{cocpot}, then we find $\chi=0$. In turn, the potential~$p$ of~$c$ defined by~\eqref{pick_p} is $\Gm$-equivariant and can serve as the map~$f$ in~\eqref{cobchar2}. Thus, $c$, and hence $\psi$, is a coboundary.
\end{proof}

\begin{rmk}
Here we use an orbit of ordinary points of~$\Gm$ to prove the injectivity of~$\alphu_s$. In the cofinite case with cusps the approach
in~\cite[\S12.3]{BLZm} uses the set of cusps for this purpose. It is based on the long exact sequence for mixed parabolic cohomology. See~\cite[Proposition 11.9]{BLZm}. We could have used that approach here as well, however the present approach seems simpler.

In the cocompact case these two possibilities are absent. In~\cite[\S7.3]{BLZm} the orbit of a hyperbolic fixed point is added to the set $X^\tess_0$, and used for the proof of the injectivity.
\end{rmk}

\subsection{Isomorphisms}

In Proposition~\ref{prop-alff} we have seen that the image of the map~$\alphu_s$ in~\eqref{al-def} consists of funnel forms. The proof of
this result is part of Task~\ref{task2} (from the list on p.~\pageref{task2}). We still have to consider the image of~$\alphu_s$ when we restrict it to the mixed cohomology spaces with values in~$\W s {\fxi;\exc,\smp;\aj}(\proj\RR)$ and~$\W s {\fxi;\exc,\infty;\aj}(\proj\RR)$. Once that is completed we have obtained
the final result of this section.

\begin{thm}\label{thm-alcoh}
The linear maps~$\alphu_s$ in~\eqref{al-def} and $\qcoh_s$ in~\eqref{qcoh} are inverse isomorphisms in the following
situations:

For $s\in\CC$ with $\Rea s\in (0,1)$ and $s\neq \frac12:$
\begin{align} \label{qcoh-al}
\qcoh_s&:\quad \A_s \stackrel\cong\longleftrightarrow H^1\bigl(
F^\tess_\bullet;\W s \om(\proj\RR), \W s{\fxi;\exc;\aj}(\proj\RR)\bigr)^\van
_\sic\quad : \alphu_s\,,\\
\label{qcoh-alq}
 \qcoh_s&:\quad \A_s^1 \stackrel\cong\longleftrightarrow H^1\bigl(
F^\tess_\bullet;\W s \om(\proj\RR), \W s{\fxi;\exc,\smp;\aj}(\proj\RR)\bigr)^\van
_\sic\quad : \alphu_s\,.
\end{align}
For $s\in\CC$ with $\Rea s \in (0,1):$
\begin{equation}\label{qcoh-al0}
 \qcoh_s:\quad \A_s^0 \stackrel\cong\longleftrightarrow H^1\bigl(
F^\tess_\bullet;\W s \om(\proj\RR), \W
s{\fxi;\exc,\infty;\aj}(\proj\RR)\bigr)^\van _\sic\quad :\alphu_s\,. 
\end{equation}
\end{thm}

\begin{proof}
Proposition~\ref{prop-cohAH} together with the inverse of
the restriction map in \eqref{res-coh} shows that $\qcoh_s $ has its
image in the mixed cohomology spaces on the right.

Proposition~\ref{prop-aldef} shows that $\alphu_s$ is left inverse to~$\qcoh_s\vert_{\A_s}$.
By Proposition~\ref{prop-alff} the image of $\alphu_s$ is in~$\A_s$. Therefore
Proposition~\ref{prop-alinj} yields that $\alphu_s$ is indeed inverse to~$\qcoh_s$, not only left-inverse. This establishes~\eqref{qcoh-al}.

For the two other cases we have to show that the image under~$\alphu_s$ is indeed in the space of resonant funnel forms~$\A_s^1$ and cuspidal funnel forms~$\A_s^0$, respectively. To that end let $\psi\in Z^1\bigl( F^\tess_\bullet; \W s \om(\proj\RR),\W s {\fxi;\exc,\cond;\aj}(\proj\RR)\bigr)^\van_\sic$, where \tcond denotes either $\infty$ or~\tsmp. In the latter case we assume $s\not=\frac12$.  Set $u\coloneqq u_\psi$, which is in~$\A_s$ by Proposition~\ref{prop-alff}. By~\eqref{al-coh} and~\eqref{quY} we can change $\psi$ in its cohomology class in~$H^1\bigl( F^\tess_\bullet; \W s \om(\proj\RR),\W s {\fxi;\exc,\cond;\aj}(\proj\RR)\bigr)^\van_\sic$ such that $\psi(e)(z) = \int_e \bigl\{ u,q_s(\cdot,z)\bigr\}$ for~$e\in X_1^{\tess,Y}$. So there exists $h\in \W s {\om;\cond}[\infty]$ (e.\,g.\@ $h=\psi(e_\infty)$) such that $\psi(f_\infty) = h - \tau_s(T^{-1}) h$. Using the argumentation from~\cite[Propositions~9.11 and~9.15]{BLZm} we conclude that the $1$-periodic function~$v(z)= u(z/\lambda)$ satisfies 
\[ 
v(z) \ceqq 
\begin{cases} p\,y^{1-s} + \oh( e^{-\eps y}) &\text{for some $p\in\CC$ and $\eps>0$},\text{ in
the case $\tcond=\tsmp$}\,,\\
\oh(e^{-\eps y}) &\text{ for some }\eps>0,\text{ in the case $\tcond=\infty$}
\end{cases}
\]
as $y\uparrow \infty$. Thus, $u\in \A^1_s$ for $\tcond=\tsmp$ and $u\in \A^0_s$ for $\tcond=\infty$.
\end{proof}

\section{Relation between cohomology spaces}\label{sec:relacohom}
\markright{16. RELATION BETWEEN COHOMOLOGY SPACES}

We used the mixed cohomology spaces in Section~\ref{sect-coc-ff} to prove the surjectivity of the map from cocycle classes to funnel forms. As a final step for the preparation of the proof of Theorem~\ref{thm:cohominter} (see p.~\pageref{thm:cohominter}) we will show, in Proposition~\ref{prop-coh-Tess-Xi}, that the mixed tesselation cohomology spaces~$H^1\bigl(F_\bullet^\tess;\V s \om(\proj\RR), W\bigr)^\van_\sic$ (with $W$ being any submodule of $\V s \fxi(\proj\RR)$) are naturally isomorphic to the cohomology spaces~$H^1_\Xi(\Gamma;W)$. We start with two preparatory lemmas. The first one is complementary to Lemma~\ref{lem:cocgrp_inj} and Proposition~\ref{prop-sic}.

\begin{lem}\label{lem:analytic_repr}
Let $s\in\CC$, and let $W$ be a $\Gamma$-module that satisfies the relation $\V s \om(\proj\RR)\subseteq W \subseteq \V s \fxi(\proj\RR)$. For each $c\in Z^1_\Xi(\Gm;W)_\sic$ the associated group cocycle class $[\psi]\in H^1(\Gm;W)$ contains a representative $\psi^c\in Z^1\bigl(\Gm;\V s \om(\proj\RR)\bigr)$. The induced map
\begin{equation}\label{eq:analytic_well}
 H^1_\Xi(\Gm;W)_\sic \to H^1\bigl(\Gm; \V s \om(\proj\RR)\bigr)\,,\quad [c] \mapsto [\psi^c]
\end{equation}
is injective (and for $W=\V s \fxi(\proj\RR)$ inverse to the map~$\Psi_H$ from Proposition~\ref{prop-sic}).
\end{lem}

\begin{proof}
We set $W_0\coloneqq  \V s \om(\proj\RR)$. Let $c\in Z^1_\Xi(\Gm;W)_\sic$. For each pair $(\xi,\eta)\in \Xi^2$, the singularities of $c(\xi,\eta)$ are contained in $\{\xi,\eta\}$. By Proposition~\ref{prop-sepsing} we find and fix 
elements $A_\xi^\eta\in \V s \om[\xi]$ and $A_\eta^\xi \in \V s \om[\eta]$ such that
\[
c(\xi,\eta) = A^\eta_\xi - A_\eta^\xi\,.
\]
We choose $A_\xi^\xi = 0$. The map $q\colon\Xi\to W$, 
\[
 q(\xi) \coloneqq  c(\xi,\infty) + A_\infty^0
\]
is a potential for $c$. The associated group cocycle $\psi\in Z^1(\Gamma;W)$ is given by
\[
 \psi_{g^{-1}} = q(g\xi) - \tau_s(g)q(\xi)\,,
\]
which is independent of $\xi\in\Xi$. We show that $\psi\in Z^1(\Gamma;W_0)$. 

For any three pairwise different elements $\xi,\eta,\zeta\in \Xi$ we
rewrite the cocycle property of $c$ as
\[ 0 \ceqq  \bigl( A^\eta_\xi-A^\zeta_\xi\bigr) + \bigl(A_\eta^\zeta-A_\eta^\xi\bigr)
+ \bigl( A_\zeta^\xi-A_\zeta^\eta)\,.\]
The three terms on the right hand side have their singularities in $\{\xi\}$, $\{\eta\}$,
$\{\zeta\}$, respectively. Therefore
\[
A_\xi^\eta\;\equiv\; A_\xi^\zeta\,,
\]
where $\equiv$ denotes equality modulo $W_0$ (see Section~\ref{sect-sars}). Analogously, we obtain 
\[
\tau_s(g) A_\xi^\eta \;\equiv\;
A_{g\xi}^{g\eta}
\]
for all $g\in\Gm$, all $\xi,\eta\in\Xi$. 
Further, $q(\infty)= A_\infty^0$ and $q(0)= A_0^\infty$. Therefore,
\begin{align*}
\psi_T &\ceqq  q(\infty) - \tau_s\big(T^{-1}\big)q(\infty) = A_\infty^0 - \tau_s(T^{-1})A_\infty^0 \;\equiv\;
A_\infty^0 - A_\infty^{-\lambda} \;\equiv\; 0\,,\\
\psi_S &\ceqq q(0)- \tau_s(S) q(\infty) \;\equiv\; A_0^\infty - A_0^\infty
\ceqq 0\,.
\end{align*}
Hence $\psi\in Z^1(\Gm;W_0)$.

To show that the assignment in~\eqref{eq:analytic_well} is well-defined it suffices to show that if $c$ is a coboundary then $\psi$ is so. Suppose that $c\in B^1_\Xi(\Gm;W)_\sic$. Then there exists a $\Gm$-equivariant potential $f\colon \Xi\to W$ of $c$. Then the associated group cocycle $\psi$ vanishes, and hence is in $B^1(\Gm;W_0)$.

Since $H^1_\Xi(\Gm;W)_\sic$ is a subspace of $H^1_\Xi(\Gm;W)$ by Proposition~\ref{prop:char_sic}, the injectivity of the map in~\eqref{eq:analytic_well} follows immediately from the injectivity of the map
\[
 H^1_\Xi(\Gm;W) \to  H^1(\Gm; W)\,, \quad [c] \mapsto [\psi^c]
\]
which is established in Lemma~\ref{lem:cocgrp_inj}.
\end{proof}

\begin{lem}\label{lem:submod_inj}
Let $s\in\CC$, and let $W$ be a submodule of $\V s \fxi(\proj\RR)$ that contains $\V s \om(\proj\RR)$. Then the map 
\begin{equation}\label{eq:submod_inj}
 H^1_\Xi(\Gm;W) \to H^1_\Xi\bigl(\Gm,\V s \fxi(\proj\RR)\bigr)\,,\quad [c] \mapsto [c]
\end{equation}
is injective.
\end{lem}

\begin{proof}
Clearly, the map in \eqref{eq:submod_inj} is well-defined. To show that this map is injective for any submodule $W\subseteq \V s \fxi(\proj\RR)$ with $\V s \om(\proj\RR)\subseteq W$, it suffices to show it for the case $W=\V s \om(\proj\RR)$. To that end let 
\[
c\in Z^1_\Xi(\Gm;\V s \om(\proj\RR)) \cap B^1_\Xi\bigl(\Gm;\V s \fxi(\proj\RR)\bigr)\,.
\]

Let $q\colon\Xi\to \V s \om(\proj\RR)$ be a potential of $c$ (considered as element of $Z^1_\Xi\bigl(\Gm;\V s \om(\proj\RR)\bigr)$. Since the cocycle~$c$ is in the coboundary space~$B^1_\Xi\bigl(\Gm;\V s \fxi(\proj\RR)\bigr)$, there exists also a $\Gm$-equivariant potential $p\colon \Xi \to \V s \fxi(\proj\RR)$ of $c$. It suffices to show that $p$ maps into the space~$\V s \om(\proj\RR)$. For all $\xi\in \Xi$ and all $g\in\Gm$ we have 
\[
 p(g\xi)-p(\xi) = c(g\xi,\xi) \in \V s \om(\proj\RR)\,,
\]
and hence 
\[
\bsing p(g\xi) = \bsing p(\xi)\,.
\]
However, for every $\xi\in\Xi$ there exists $g\in\Gm$ such that 
\[
 g\,\bsing p(\xi) \cap \bsing p(\xi)  = \emptyset\,.
\]
From $g\,\bsing p(\xi) = \bsing p(g\xi)$ it follows that $\bsing p(\xi) = \emptyset$, and therefore $p(\xi)\in \V s \om(\proj\RR)$.
\end{proof}

\begin{prop}\label{prop-coh-Tess-Xi}
Let $s\in\CC$, and let $W$ be a submodule of $\V s \fxi(\proj\RR)$ that contains $\V s \om(\proj\RR)$. Then there is a natural
isomorphism
\[ 
H^1(F^\tess_\bullet;\V s \om(\proj\RR),W)^\van_\sic \;\cong\; H^1_\Xi
(\Gm;W)^\van_\sic \,,
\]
realized by the map
\[
 \pr_\Xi\colon H^1(F^\tess_\bullet;\V s \om(\proj\RR),W)^\van_\sic \to H^1_\Xi
(\Gm;W)^\van_\sic\,,\qquad  [c] \mapsto [c\vert_{\Xi\times\Xi}]\,.
\]
\end{prop}

\begin{proof}
Let $V\coloneqq  \V s \om(\proj\RR)$ and let $W\subseteq \V s \fxi(\proj\RR)$ be any submodule containing $V$. We identify elements in $Z^1(F^\tess_\bullet; V,W)$ with maps 
\[
 c\colon X^\tess_0\times X^\tess_0\to W
\]
obeying the conditions stated above in Section~\ref{sec:def_mixed}. Since $\Xi\subseteq X^\tess_0$, the map 
\[
 H^1(F^\tess_\bullet; V,W) \to H^1_\Xi(\Gm; W),\quad [c] \mapsto [c\vert_{\Xi\times\Xi}],
\]
is well-defined and natural. The conditions $\tsic$ and $\tvan$ being properties involving only points or pairs of points in $\Xi$ yields that this map descends to maps
\[
 H^1(F^\tess_\bullet; V,W)_{\cond_1}^{\cond_2} \to H^1_\Xi(\Gm; W)_{\cond_1}^{\cond_2},
\]
where $\cond_1$ is $\tsic$ or void, and $\cond_2$ is $\tvan$ or void. It remains to show that each cocycle class $[c]\in H^1_\Xi(\Gm;W)_\sic^{\cond_2}$ has a representative $c\in Z^1_\Xi(\Gm;W)_\sic^{\cond_2}$, 
\[
 c\colon\Xi\times\Xi\to W,
\]
which extends to an element~$\tilde c\in Z^1(F^\tess_\bullet;V,W)_\sic^{\cond_2}$, and that the cocycle class~$[\tilde c]$ (an element in  $H^1(F^\tess_\bullet; V,W)_\sic^{\cond_2}$) does not depend on the choices of~$c$ and its extension~$\tilde c$. 

Let $[c]\in H^1_\Xi(\Gm;W)_\sic^{\cond_2}$. Since the cohomology space~$H^1_\Xi(\Gm;W)_\sic^{\cond_2}$ embeds into $H^1_\Xi\bigl(\Gm; \V s \fxi(\proj\RR)\bigr)_\sic$ by Lemma~\ref{lem:submod_inj}, the combination of Lemmas~\ref{lem:submod_inj} and~\ref{lem:analytic_repr} shows that there exists a representative $c$ of $[c]$ (considered as element of the space $H^1_\Xi\bigl(\Gm;\V s \fxi(\proj\RR)\bigr)_\sic$) in $Z^1_\Xi\bigl(\Gamma,\V s \fxi(\proj\RR)\bigr)_\sic$ with associated pgc-pair $(p,\psi)$ such that $\psi\in Z^1(\Gm;V)$. In what follows we take advantage of Lemma~\ref{lem:potentials} to establish the existence of $\tilde c\in Z^1_\Xi(\Gm;V,W)_\sic^{\cond_2}$ with the requested properties. To that end we set 
\[
 \tilde p_{iY}\coloneqq  0\,,\quad \tilde p_i\coloneqq  \frac12\psi_S\,,\quad \tilde p_\xi\coloneqq  p(\xi) \quad\text{for $\xi\in\Xi$\,.}
\]
To check that $\{\tilde p_r \setmid r\in \Xi\cup\{i,iY\}\}$ satisfies the condition in~\eqref{pot_well} (see Lemma~\ref{lem:potentials}) we first note that $(p,\psi)$ being a pgc-pair for $c$ implies that for all~$\xi\in\Xi$ and all $g\in\Gamma_\xi$ we have
\[
 \tilde p_\xi = \tau_s\big(g^{-1}\big)\tilde p_\xi + \psi_g
\]
by Lemma~\ref{lem:potentials}\eqref{lem:potii}. For $r=iY$ the condition in~\eqref{pot_well} is satisfied since $\Gamma_{iY} = \{\id\}$. For $r=i$ we have $\Gamma_{i} = \{\id,S\}$ and 
\[
 \tau_s(S)\tilde p_i + \psi_S = \frac12\tau_s(S)\psi_S + \psi_S = -\frac12\psi_S + \psi_S = \frac12\psi_S = \tilde p_i\,.
\]
Thus, \eqref{pot_well} is indeed satisfied for all $r\in\Xi\cup\{i,iY\}$, which is a generating set of~$X^\tess_0$ under the action of~$\Gamma$. Therefore Lemma~\ref{lem:potentials}\eqref{lem:poti} now shows that there exists a unique cocycle $\tilde c\in Z^1_\Xi(\Gm;V,W)_\sic^{\cond_2}$ with potential 
\[
 \tilde p\colon X^\tess_0\to W
\]
satisfying $\tilde p(i) = \tilde p_i$, $\tilde p(iY) = \tilde p_{iY}$ and
\[
 \tilde p\vert_\Xi = p\quad\text{and}\quad \tilde p(X^\tess_0) \subseteq V\,.
\]
In particular, $\tilde c$ is an extension of $c$. Lemmas~\ref{lem:submod_inj} and~\ref{lem:analytic_repr} yield that making other choices in this construction produces a cocycle in $Z^1_\Xi(\Gm;V,W)_\sic^{\cond_2}$ in the same cohomology class. 
\end{proof}

\section{Proof of Theorem D}\label{sec:proof_cohominter}
\markright{17. PROOF OF THEOREM D}

The combination of Theorem~\ref{thm-alcoh} with the constructions from Sections~\ref{sect-abg}-\ref{sect-ccaffabg}, in particular the isomorphism in~\ref{res-coh}, and with Proposition~\ref{prop-coh-Tess-Xi} yields a proof of Theorem~\ref{thm:cohominter} (from p.~\pageref{thm:cohominter}), including explicit isomorphisms between the considered spaces of funnel forms and the cohomology spaces. In what follows we first briefly recapitulate the maps developed in the previous sections that assign cocycle classes to funnel forms, and vice versa. Further below we will then give a more detailed proof of Theorem~\ref{thm:cohominter}.

\subsection*{From funnel forms to cocycle classes on the invariant set}

Given a funnel form~$u\in\A_s$ we define a cocycle~$c^u\in Z^1\bigl(F^\tess_\bullet;\V s \om(\proj\RR), \V s {\fxi;\exc;\aj}(\proj\RR)\bigr)^\van_\sic$ by setting
\begin{align}
 c^u(e)(t) &\coloneqq \int_{t(e)}^{h(e)} \left\{u, R(t;\cdot)^s \right\} \qquad \text{for $e\in X^\tess_1\smallsetminus\Gamma\,e_\infty$}
 \label{cu_inside}
 \intertext{and}
 c^u(e_\infty)(t) & \coloneqq -\av {s,T}^+ c^u(f_\infty)\,.
 \label{cu_cusp}
\end{align}
We recall that $t(e)$ and $h(e)$ denote the tail and head of the edge~$e\in X^\tess_1$. The integration in~\eqref{cu_inside} may be performed along any path in~$\uhp\cup\Gamma\,1$ from~$t(e)$ to~$h(e)$ with at most finitely many points in~$\Gamma\,1$. For~$e\in X^{\tess,Y}_1$, the integral in~\eqref{cu_inside} is equal to 
\[
 \int_e  \left\{u, R(t;\cdot)^s \right\}\,.
\]
For~$e\in\Gamma\,f_0$, the element~$e$ (considered as a path) is completely contained in~$\proj\RR$, and hence the integration in~\eqref{cu_inside} cannot be performed along~$e$, but along any path in~$\uhp\cup\Omega$ from~$t(e)$ to~$h(e)$ with at most finitely many points in~$\Omega$. (See Lemma~\ref{lem-io} and the paragraph following its proof.) The definition of~$c^u(e_\infty)$ in~\eqref{cu_cusp} with the help of the averaging operator~$\av {s,T}^+$  is a regularization of the typically non-convergent integral
\begin{equation}\label{wish_int}
 \int_{e_\infty} \left\{u, R(t;\cdot)^s \right\}\,.
\end{equation}
Taking advantage of the $\Gamma$-equivariance of~$c^u$ we have, at least for~$\Rea s \gg 1$,
\begin{align*}
 c^u(e_\infty)(t) & = -\av {s,T}^+ c^u(f_\infty) = - \sum_{n\geq 0} \tau_s(T^{-n})c^u(f_\infty) = -\sum_{n\geq 0} c^u(T^{-n}\,f_\infty)
 \\
 & = \sum_{n\geq 0} \int_{iY +\lambda -n\lambda}^{iY-n\lambda} \left\{u, R(t;\cdot)^s \right\}
 \\
 & = \lim_{n\to\infty} \int_{iY+\lambda}^{iY-n\lambda} \left\{u, R(t;\cdot)^s \right\}\,.
\end{align*}
For~$s\in\CC$ with `small' real part, the averaging operator is defined via meromorphic continuation. (See Section~\ref{sect-osav}.) For cuspidal funnel forms~$u\in\A_s^0$, the integral in~\eqref{wish_int} converges for all~$s\in\CC$ with~$\Rea s>0$, and hence in this case
\[
 c^u(e_\infty)(t) = \int_{e_\infty} \left\{u, R(t;\cdot)^s \right\}\,.
\]
The definition of~$c^u$ on the set~$\{e_\infty\}\cup X^\tess_1\smallsetminus\Gamma\,e_\infty$ (even only on the finite subset $\{e_1,e_2,e_\infty,f_0,f_1,f_\infty\}$) indeed determines the cocycle uniquely on all of~$F^\tess_\bullet$ by additivity (i.\,e., by the cocycle relation) and $\Gamma$-equivariance.  

The cocycle~$c^u$ descends to a cocycle in~$Z^1_\Xi\bigl(\Gamma;\V s {\fxi;\exc;\aj}(\proj\RR)\bigr)^\van_\sic$ by restriction to~$\Xi\times\Xi$, i.\,e.,
\[
 c^u\vert_{\Xi\times\Xi}(\xi,\eta) \coloneqq c^u(p)\,,
\]
where $p$ is any path in~$\CC_1[X^\tess_1]$ with $\big(t(p),h(p)\big) = (\xi,\eta)$. In total, we have the map \index[symbols]{C@$\cocu_s$}
\begin{equation}\label{def_cocu}
 \cocu_s\colon \A_s \to H^1_\Xi\bigl(\Gamma;\V s {\fxi;\exc;\aj}(\proj\RR)\bigr)^\van_\sic\,,\qquad u\mapsto [c^u\vert_{\Xi\times\Xi}]
\end{equation}
from funnel forms to cocycle classes on~$\Xi$. 

\subsection*{From cocycle classes on to funnel forms} 
Let  $[c] $ be any cocycle class in $ H^1_\Xi\bigl(\Gamma;\V s {\fxi;\exc;\aj}(\proj\RR)\bigr)^\van_\sic$. 
Let $c\in Z^1_\Xi\bigl(\Gamma;\V s {\fxi;\exc;\aj}(\proj\RR)\bigr)^\van_\sic$
be a representing cocycle. This cocycle has a unique extension to a cocycle on~$F^\tess_1$, which we here denote also by~$c$. We consider the latter as a map
\[
c\colon F^\tess_\bullet\to \V s {\fxi;\exc;\aj}(\proj\RR)\,.
\]
Now each element in~$\V s {\fxi;\exc;\aj}(\proj\RR)$ shall be identified with a boundary germ in~$\W s {\fxi;\exc;\aj}(\proj\RR)$. In a certain sense, this identification `thickens' elements in $\V s {\fxi;\exc;\aj}(\proj\RR)$ to suitable Laplace eigenfunctions in a neighborhood of~$\proj\RR$. Under this identification, the cocycle becomes a map
\[
c\colon F^\tess_\bullet\to \W s {\fxi;\exc;\aj}(\proj\RR)\,.
\]
We pick a representative~$\psi\colon F^\tess_\bullet\to \G s \fxi(\proj\RR)$ of~$c$. This means that for each edge~$e\in X^\tess_1$, we fix an element~$\psi(e)\in C^2(\uhp)$ which represents~$c(e)\in \W s \fxi(\proj\RR)$ near~$\proj\RR$. In this way, we stretch $c(e)$ to a function on all of~$\uhp$ which near (essential parts of)~$\proj\RR$ is a Laplace eigenfunction with spectral parameter~$s$. More precisely, we pick a family~$(\psi_\eps)_{\eps>0}$ of such representatives with the property that for all edges~$e\in X^\tess_1$, the set of singularities of~$\psi_\eps(e)$ (that is, the subset of~$\uhp$ on which $\psi_\eps(e)$ is not a Laplace eigenfunction) is close to~$e$, and $\eps$-close to~$f_0$ if $e=f_0$. 

We associate to~$[c]$ the function~$u_c\colon \uhp \to \CC$ given by
\[
 u_{[c]}(z) \coloneqq \frac2\pi \psi_\eps(C)(z)\,,
\]
where $\eps>0$ is sufficiently small and $C$ is a path along edges in~$X^\tess_1$ that winds around~$z$ once in counterclockwise direction at great distance and that stays away from all singularity sets of~$\psi_\eps(e)$ for all edges~$e$ contained in the path. The function~$u_{[c]}$ is well-defined and independent of all choices made in its construction. This yields the map \index[symbols]{U@$\ucoc_s$}
\begin{equation}\label{def_ucoc}
 \ucoc_s\colon H^1_\Xi\bigl(\Gamma;\V s {\fxi;\exc;\aj}(\proj\RR)\bigr)^\van_\sic\to \A_s\,,\qquad [c]\mapsto u_{[c]}\,.
\end{equation}

\subsection*{Proof of Theorem D}\label{proof:cohominter}
We have $\cocu_s = \pr_\Xi\circ\coh_s$ and $\ucoc_s = \alphu_s\circ\rho_s\circ\pr_\Xi^{-1}$ with $\pr_\Xi$ from Proposition~\ref{prop-coh-Tess-Xi}, $\coh_s$ from Proposition~\ref{prop-cohAH}, $\alphu_s$ from Proposition~\ref{prop-aldef} and $\rho_s$ from~\eqref{res-coh}. Then Theorem~\ref{thm-alcoh}, the isomorphism in~\eqref{res-coh} and Proposition~\ref{prop-coh-Tess-Xi} show that $\cocu_s$ and $\ucoc_s$ are inverse isomorphisms in the following situations:
\begin{enumerate}[(i)]
 \item For any $s\in\CC$, $\Rea s\in (0,1)$: 
 \[
   \cocu_s:\quad \A_s^0 \stackrel{\cong}{\longleftrightarrow} H^1_\Xi\bigl(\Gamma;\V s {\fxi;\exc,\infty;\aj}(\proj\RR)\bigr)^\van_\sic \quad :\ucoc_s\,.
 \]
\item For any $s\in\CC$, $\Rea s\in (0,1)$, $s\not=\frac12$:
\begin{align*}
 \cocu_s:\quad &\A_s^1 \stackrel{\cong}{\longleftrightarrow} H^1_\Xi\bigl(\Gamma;\V s {\fxi;\exc,\smp;\aj}(\proj\RR)\bigr)^\van_\sic \quad :\ucoc_s\,,
 \intertext{and}
 \cocu_s:\quad &\A_s \stackrel{\cong}{\longleftrightarrow} H^1_\Xi\bigl(\Gamma;\V s {\fxi;\exc;\aj}(\proj\RR)\bigr)^\van_\sic \quad :\ucoc_s\,.
\end{align*}
\end{enumerate}
This completes the proof. \qed

%% file: Mem-BP-partIV.tex

\setcounter{sectold}{\arabic{section}}
\markboth{IV. TRANSFER OPERATORS AND COHOMOLOGY}{IV. TRANSFER OPERATORS AND COHOMOLOGY}
\chapter{Transfer operators and cohomology}\label{part:TO}

\setcounter{section}{\arabic{sectold}}
\setcounter{equation}{0}
\renewcommand\theequation{\Roman{chapter}.\arabic{equation}}

In this chapter we will provide an interpretation of the cohomology spaces in Theorem~\ref{thm:cohominter} (and hence in Theorem~\ref{thm-alcoh}) in terms of complex period functions. We recall from Section~\ref{sect-sars} the $\Gamma$-module~$\V s \fxi(\proj\RR)$ of elements in the injective limit that are represented by maps that are real-analytic on~$\proj\RR$ up to finitely many singularities. We further recall the sets~$\V s \fxi(I)$ for open subsets~$I\subseteq\proj\RR$ consisting of elements that are represented by maps that are real-analytic on~$I$ up to finitely many singularities. Finally, we recall that we refer to elements of~$\V s \fxi(I)$ as `functions.'

In Section~\ref{sec:heuristic} we alluded at the existence of a relation between (real) period functions~$f=(f_1,f_2)$ and cocycle classes~$[c]\in H^1_\Xi\bigl(\Gm;\V s \fxi(\proj\RR) \bigr)$ satisfying the well-motivated identifications 
\begin{align}
 c(1,\infty)\vert_{(\infty,1)_c}  &= -f_2 
 \label{c1short}
 \\
 c(-1,\infty)\vert_{(-1,\infty)_c} & = f_1\,.
 \label{c2short}
\end{align}
We will now show that each cocycle~$c\in Z^1_\Xi\bigl(\Gamma;\V s \fxi(\proj\RR)\bigr)$ is uniquely determined by its two values 
\[
 c(1,\infty)\vert_{(\infty,1)_c} \quad\text{and}\quad c(-1,\infty)\vert_{(-1,\infty)_c}\,,
\]
and that these values can attain every function in the spaces~$\V s \fxi\big( (\infty,1)_c\big)$ and $\V s \fxi\big( (-1,\infty)_c\big)$, respectively. This gives us a linear and bijective map \index[symbols]{P@$\pc$}
\begin{equation}\label{map_pc_first}
 \pc\colon \V s \fxi\big( (-1,\infty)_c\big) \times \V s \fxi\big( (\infty,1)_c\big) \to Z^1_\Xi\bigl(\Gamma;\V s \fxi(\proj\RR)\bigr)
\end{equation}
that is uniquely determined by~\eqref{c1short} and~\eqref{c2short}.

The main result of this chapter is then the following theorem, which together with Theorem~\ref{thm-alcoh} (Theorem~\ref{thm:cohominter}, p.~\pageref{thm:cohominter}) yields a proof of Theorems~\ref{thmA_new} and~\ref{thmB_new} (from p.~\pageref{thmA_new}). For the definitions of the various spaces of complex period functions we refer to Sections~\ref{sec:periodfunctions} and \ref{sec:def_cplxperiod}.

\begin{mainthm}\label{thm-FE-coh}
Let $s\in\CC$.
\begin{enumerate}[{\rm (i)}]
\item The map~$\pc$ induces an isomorphism
\[
\FE^\om_s(\CC)  / \BFE^\om_s(\CC) \longrightarrow H^1_\Xi\bigl(\Gm;\V
s {\fxi;\exc;\aj}(\proj\RR)\bigr)^\van_\sic\,. 
\]
\item If $\Rea s\in (0,1)$, $s\not=\frac12$, then the map~$\pc$ induces an isomorphism
\[
 \FE_s^{\om,1}(\CC) \to H^1_\Xi\bigl(\Gm;\V s {\fxi;\exc,\smp;\aj}(\proj\RR)\bigr)^\van_\sic\,.
\]
\item If $\Rea s\in (0,1)$, then the map~$\pc$ induces an isomorphism 
\[
 \FE_s^{\om,0}(\CC) \to H^1_\Xi\bigl(\Gm;\V s {\fxi;\exc,\infty;\aj}(\proj\RR)\bigr)^\van_\sic\,.
\]
\end{enumerate}
\end{mainthm}

After having established $\pc$ as an isomorphism with domain and range as in~\eqref{map_pc_first} (see Section~\ref{sec:pc}), we will show, in Section~\ref{sec:pc_real}, that $\pc$ descends to an isomorphism
\[
 \FE_s^\om(\RR) \to Z^1_\Xi\bigl(\Gm; \V s {\fxi;-;\aj}(\proj\RR)\bigr)^\van_\sic.
\]
For suitable values of~$s$, this map induces isomorphisms
\begin{align*}
 \FE_s^\om(\CC) & \to Z^1_\Xi\bigl(\Gm; \V s {\fxi;\exc;\aj}(\proj\RR)\bigr)^\van_\sic,
 \\
 \FE_s^{\om,1}(\CC) & \to Z^1_\Xi\bigl(\Gm; \V s {\fxi;\exc,\smp;\aj}(\proj\RR)\bigr)^\van_\sic
 \intertext{and}
 \FE_s^{\om,0}(\CC) & \to Z^1_\Xi\bigl(\Gm; \V s {\fxi;\exc,\infty;\aj}(\proj\RR)\bigr)^\van_\sic,
\end{align*}
as we will show in Section~\ref{sec:pc_complex}. Identifying the preimages of coboundaries under these maps is then the final ingredient for the proof of Theorem~\ref{thm-FE-coh}, which we will provide in Section~\ref{sec:proof_thm-FE-coh}.

\renewcommand\theequation{\arabic{section}.\arabic{equation}}

\section{The map from functions to cocycles}\label{sec:pc}
\markright{18. FROM FUNCTIONS TO COCYCLES}

In Section~\ref{sec:heuristic} we explained insights for a constructive and geometrically motivated isomorphism between eigenfunctions of transfer operators and cocycles. In this section we will show that, guided by these insights, we can even construct an isomorphism between the function space~$\V s \fxi\big( (-1,\infty)_c\big) \times \V s \fxi\big( (\infty,1)_c\big)$ and the cocycle space~$Z^1_\Xi\bigl(\Gamma;\V s \fxi(\proj\RR)\bigr)$ for every $s\in\CC$. We recall from~\eqref{eq:def_DR} that 
\[
D_\RR = (-1,\infty)_c \uplus (\infty,1)_c\,.
\]
In analogy to the conventions in Section~\ref{sec:slowTO} we set \index[symbols]{Vaaaf@$\V s \fxi(D_\RR)$}
\begin{equation}
 \V {s} \fxi(D_\RR) \coloneqq \V s \fxi\big( (-1,\infty)_c\big) \times \V s \fxi\big( (\infty,1)_c\big)\,.
\end{equation}

\begin{prop}\label{prop:cocycle}
\begin{enumerate}[{\rm (i)}]
\item\label{prop:coci} 
Each cocycle~$c\in Z^1_\Xi\bigl(\Gamma;\V s \fxi(\proj\RR)\bigr)$ is uniquely determined by its two values 
\[
 c(1,\infty)\vert_{(\infty,1)_c} \quad\text{and}\quad c(-1,\infty)\vert_{(-1,\infty)_c}\,,
\]
and these values can attain every function in the spaces~$\V s \fxi\big( (\infty,1)_c\big)$ and $\V s \fxi\big( (-1,\infty)_c\big)$, respectively. 
The induced map 
\[
 \pc\colon \V {s} \fxi(D_\RR) \to Z^1_\Xi\bigl(\Gamma;\V s \fxi(\proj\RR)\bigr)\,,\quad f=(f_1,f_2) \mapsto c\coloneqq \pc(f)\,,
\]
satisfying 
\begin{align}
 c(1,\infty)  = -f_2  &\qquad\text{on $(\infty,1)_c$}  \label{c1shortprop}
 \\
 c(-1,\infty)  = f_1 &\qquad\text{on $(-1,\infty)_c$\,,} \label{c2shortprop}
\end{align}
is linear and  bijective.
\item Let $f=(f_1,f_2) \in \V {s} \fxi(D_\RR)$ and let $c= \pc(f)$ be the associated cocycle in~$Z_\Xi^1\bigl(\Gamma;\V s \fxi(\proj\RR)\bigr)$. Then there is a (unique) potential~$p\colon\Gamma\to \V s \fxi(\proj\RR)$ of~$c$ determined by $p(\infty) = 0$ and
\begin{equation}\label{pot_p1}
p(1) = 
\begin{cases}
 -f_2 & \text{on $(\infty,1)_c$}
 \\
 f_1 + \tau_s(S)f_1 + \tau_s(S)f_2 & \text{on $(1,\infty)_c$\,.}
\end{cases}
\end{equation}
This potential satisfies
\begin{equation}\label{pot_p2}
p(-1) =
\begin{cases}
-f_2 - \tau_s(S)f_1 - \tau_s(S)f_2 & \text{on $(\infty,-1)_c$}
\\
f_1 & \text{on $(-1,\infty)_c$\,.}
\end{cases}
\end{equation}
The group cocycle~$\psi\in Z^1\bigl(\Gamma;\V s \fxi(\proj\RR)\bigr)$ associated to~$p$ is determined by $\psi_T=0$ and 
\begin{equation}\label{def_grpcoc}
\psi_S = 
\begin{cases}
-\tau_s(S)f_1 - f_2 & \text{on $(\infty,0)_c$}
\\
f_1 + \tau_s(S)f_2 & \text{on $(0,\infty)_c$\,.}
\end{cases}
\end{equation}
\end{enumerate}
\end{prop}

The expressions for the potential~$p$ and the group cocycle~$\psi$ in the statement of Proposition~\ref{prop:cocycle} are consistent with the insights as explained in Section~\ref{sec:heuristic}. (See Figure~\ref{fig:cocycle}.) In the proof of Proposition~\ref{prop:cocycle} we take advantage of these intuitions and use them for the proof that the prescriptions in~\eqref{c1shortprop} and~\eqref{c2shortprop} indeed extend to a cocycle in~$Z^1_\Xi\bigl(\Gamma;\V s \fxi(\proj\RR)\bigr)$. To that end we start by building the cocycle bottom-up from a group cocycle and a potential. Only afterwards we prove uniqueness.

\begin{proof}[Proof of Proposition~\ref{prop:cocycle}]
Let $f \coloneqq (f_1,f_2) \in \V {s} \fxi(D_\RR)$. We first show that there is indeed a cocycle~$c \in Z^1_\Xi\bigl(\Gamma;\V s \fxi(\proj\RR)\bigr)$ satisfying \eqref{c1shortprop} and \eqref{c2shortprop}. To that end we define an inhomogeneous group cocycle~$\psi\in Z^1\bigl(\Gamma;\V s \fxi(\proj\RR)\bigr)$ satisfying \eqref{def_grpcoc}, then we use the group cocycle to define a map~$p\colon \Gamma\to \V s \fxi(\proj\RR)$ which satisfies \eqref{pot_p1}-\eqref{pot_p2} and serves as a potential for a cocycle~$c\in Z^1_\Xi\bigl(\Gamma;\V s \fxi(\proj\RR)\bigr)$ that satisfies \eqref{c1shortprop} and \eqref{c2shortprop}.

Since $\Gamma$ has the presentation 
\[
 \Gamma = \langle T, S \mid S^2 = 1 \rangle\,,
\]
any group cocycle~$\psi\in Z^1\bigl(\Gamma;\V s \fxi(\proj\RR)\bigr)$ is uniquely determined by its values on~$T$ and~$S$, subject to the condition $\psi_S=-\tau(S)\psi_S$. We set $\psi_T\coloneqq  0$ and
\[
\psi_S\coloneqq  
\begin{cases}
-\tau_s(S)f_1-f_2 & \text{on $(\infty,0)_c$}
\\
f_1 + \tau_s(S)f_2 & \text{on $(0,\infty)_c$\,.}
\end{cases}
\]
Obviously, $\psi_S = -\tau_s(S)\psi_S$, and hence these initial values uniquely extend to a group cocycle $\psi\in  Z^1\bigl(\Gamma;\V s \fxi(\proj\RR)\bigr)$. 

We define a map~$p\colon\Gamma\to \V s \fxi$ by taking advantage of Lemma~\ref{lem:potentials}. We set $p(\infty) \coloneqq  0$ and
\[
p(1) \coloneqq  
\begin{cases}
-f_2 & \text{on $(\infty,1)_c$}
\\
f_1 + \tau_s(S)f_1 + \tau_s(S)f_2 & \text{on $(1,\infty)_c$\,.}
\end{cases}
\]
The stabilizer group of~$1$ is trivial. In turn, Lemma~\ref{lem:potentials} shows that these initial values determine (uniquely) a potential~$p$. Let $c$ be the cocycle determined by~$p$.

We show that the cocycle~$c\in Z^1_\Xi\bigl(\Gamma;\V s \fxi(\proj\RR)\bigr)$ with potential~$p$ satisfies \eqref{c1shortprop} and \eqref{c2shortprop}. Since $c(1,\infty) = p(1)-p(\infty) = p(1)$, the relation in~\eqref{c1shortprop} is obviously satisfied. Further we have 
\begin{align*}
 c(-1,\infty) & = p(-1) = p(S.1) = \tau_s(S)p(1) + \psi_S 
 \\
 & = 
 \begin{cases}
 -\tau_s(S)f_2 & \text{on $(0,-1)_c$}
 \\
 \tau_s(S)f_1 + f_1 + f_2 & \text{on $(-1,0)_c$}
 \end{cases}
+
\begin{cases}
 -\tau_s(S)f_1 - f_2 & \text{on $(\infty,0)_c$}
 \\
 f_1 + \tau_s(S)f_2 & \text{on $(0,\infty)_c$}
\end{cases}
\\
& = 
\begin{cases}
-\tau_s(S)f_2 - \tau_s(S)f_1 - f_2 & \text{on $(\infty,-1)_c$}
\\
f_1& \text{on $(-1,0)_c$}
\\
f_1 & \text{on $(0,\infty)_c$\,.}
\end{cases}
\end{align*}
Thus, \eqref{c2shortprop} is satisfied.

It remains to prove uniqueness. Suppose that $\tilde c$ is a cocycle in~$Z^1_\Xi\bigl(\Gamma;\V s \fxi(\proj\RR)\bigr)$ that satisfies \eqref{c1shortprop} and \eqref{c2shortprop}. Let $\tilde p\colon \Gamma\to \V s \fxi(\proj\RR)$ be the potential of~$\tilde c$ which satisfies $\tilde p(\infty) = 0$. In the following we show that $\tilde p$ equals the potential used in the proof of existence. From this it immediately follows that $\tilde c$ is identical to the cocycle $c$ constructed above, and hence the cocycle in~$Z^1_\Xi\bigl(\Gamma;\V s \fxi(\proj\RR)\bigr)$ is uniquely determined by~\eqref{c1shortprop} and~\eqref{c2shortprop}.

Let $\tilde \psi\in Z^1\bigl(\Gamma;\V s \fxi(\proj\RR)\bigr)$ be the group cocycle determined by the potential~$\tilde p$. From~\eqref{c1shortprop} and \eqref{c2shortprop} it follows that 
\[
 \tilde p(1) = -f_2 \quad\text{on $(\infty,1)_c$\,,}
\]
and
\[ 
 \tilde p(-1) = f_1 \quad\text{on $(-1,\infty)_c$\,.}
\]
From $\tilde \psi_S = \tilde p(1) - S.\tilde p(-1)$ it follows that
\[
 \tilde \psi_S  =   -f_2 - \tau_s(S)f_1  \quad\text{on $(\infty,0)_c$\,.}
\]
Note that $S\,(-1,\infty)_c = (1,0)_c$. From $S^2 = 1$ it follows that 
\[
 \tilde \psi_S = \tau_s(S)f_2 + f_1 \quad\text{on $(0,\infty)_c$\,.}
\]
Hence, the group cocycle~$\tilde \psi$ is identical to the group cocycle~$\psi$ used above in the proof of the existence of a cocycle in~$Z^1_\Xi\bigl(\Gamma;\V s \fxi(\proj\RR)\bigr)$ with~\eqref{c1shortprop} and \eqref{c2shortprop}. From 
\[
 \tilde p(1) = \tau_s(S)\tilde p(-1) + \tilde \psi_S 
\]
it now follows that 
\[
 \tilde p(1) = f_1 + \tau_s(S) f_1 + \tau_s(S)f_2 \quad\text{on $(1,\infty)_c$\,.}
\]
Thus, $\tilde p$ is identical to the potential~$p$ used in the proof of existence above. This completes the proof of Proposition~\ref{prop:cocycle}.
\end{proof}

\section{Real period functions and semi-analytic cocycles}\label{sec:pc_real}\markright{19. REAL PERIOD FUNCTIONS AND COCYCLES}

In this section we will show that the map~$\pc$ defined in Proposition~\ref{prop:cocycle} descends to an isomorphism 
\begin{equation}\label{eq:pcfinal}
 \FE_s^\omega(\RR) \to Z^1_\Xi\bigl(\Gamma;\V s {\fxi;-;\aj}(\proj\RR)\bigr)_\sic^\van
\end{equation}
for all $s\in\CC$. In the course of establishing this result, we will see that we can characterize the two properties~\tvan (vanishing) and \tsic (singularities) \emph{separately} in terms of properties of elements in~$\V {s} \fxi(D_\RR)$. To that end we set
\begin{equation}
 \FE_s^\fxi(\RR) \coloneqq \bigl\{ f\in\V {s} \fxi(D_\RR) \setmid f=\TO_s^\slow f\bigr\}
\end{equation}
and prove that the map~$\pc$ induces isomorphisms
\begin{align}
 \FE_s^{\omega(\Xi)}(\RR) &\to Z^1_\Xi\bigl(\Gamma;\V s \fxi(\proj\RR)\bigr)^\van \label{isovan}
 \intertext{and}
 \spaceext &\to Z^1_\Xi\bigl(\Gamma;\V s \fxi(\proj\RR)\bigr)_\sic\,, \label{isosic}
\end{align}
where $\spaceext$ is a subspace of~$\V {s} \fxi(D_\RR)$ of pairs of real-analytic functions that satisfy a certain regularity condition at~$\infty$. The superscript~\teext shall indicate that these regularity conditions yield that a certain function formed from the elements in~$\spaceext$ admits a real-analytic extension to~$\infty$. A precise definition is provided in~\eqref{ext_infty2} below. We will then show that 
\[
 \FE_s^\om(\RR) = \FE_s^\fxi(\RR) \cap \spaceext\,,
\]
and hence deduce from~\eqref{isovan} and~\eqref{isosic} that $\pc$ induces an isomorphism
\[
 \FE_s^\om(\RR) \to Z^1_\Xi\bigl(\Gm;\V s \fxi(\proj\RR)\bigr)^\van_\sic\,.
\]
(See Corollary~\ref{cor:per_sicvan} and its proof.) The combination of~\tvan and \tsic causes a certain rigidity which yields that 
\[
 Z^1_\Xi\bigl(\Gm;\V s \fxi(\proj\RR)\bigr)^\van_\sic = Z^1_\Xi\bigl(\Gm;\V s {\fxi;-;\aj}(\proj\RR)\bigr)^\van_\sic\,,
\]
and which will then conclude the proof that the map in~\eqref{eq:pcfinal} is well-defined and bijective.

We start with the observation that the $1$-eigenfunctions of~$\TO_s^\slow$ correspond under the map~$\pc$ to the cocycles with the vanishing property.

\begin{lem}\label{lem:EFcoc}
Let $s\in\CC$. Suppose that $f\in  \V {s} \fxi(D_\RR)$ and $c\coloneqq\pc(f)$. Then $c$ satisfies the vanishing property~$\van$ if and only if $f$ is a $1$-eigenfunction of~$\TO_s^\slow$.
\end{lem}

\begin{proof}
Let $(f_1,f_2)\coloneqq f$ denote the component functions of~$f$. Let $p$ be the potential of~$c$ which satisfies $p(\infty) =0$. Recall $p(1)$ and $p(-1)$ and the associated group cocycle~$\psi$ from Proposition~\ref{prop:cocycle}. We have
\begin{align*}
c&(1,\lambda-1)  = p(1) - p(\lambda-1) 
\\
& = p(1) - p(T(-1)) 
\\
& = p(1) - \tau_s(T)p(-1) - \psi_{T^{-1}}
\\
& = 
\begin{cases}
-f_2 + \tau_s(T)f_2 + \tau_s(TS)f_1 + \tau_s(TS)f_2 & \text{on $(\infty,1)_c$}
\\
f_1 + \tau_s(S)f_1 + \tau_s(S)f_2 + \tau_s(T)f_2 + \tau_s(TS)f_1 + \tau_s(TS)f_2 & \text{on $(1,\lambda-1)_c$}
\\
f_1 + \tau_s(S)f_1 + \tau_s(S)f_2 - \tau_s(T)f_1 & \text{on $(\lambda-1,\infty)_c$\,.}
\end{cases}
\end{align*}
It follows that $(f_1,f_2)$ is a $1$-eigenfunction of~$\TO_s^\slow$ if and only if $c(1,\lambda-1) = 0$ on~$(\lambda-1,1)_c$.
\end{proof}

Lemma~\ref{lem:EFcoc} immediately implies the following result.

\begin{prop}\label{prop:iso_van}
For any $s\in\CC$, the map~$\pc$ descends to an isomorphism between~$\FE_s^{\omega(\Xi)}(\RR)$ and $Z^1_\Xi\bigl(\Gamma;\V s \fxi(\proj\RR)\bigr)^\van$.
\end{prop}

We now develop an interpretation of the singularity condition~\tsic for the coycles in~$Z^1_\Xi\bigl(\Gm;\V s \fxi(\proj\RR)\bigr)$ in terms of properties of elements in~$\V {s} \fxi(D_\RR)$. From~\eqref{c1shortprop} and \eqref{c2shortprop} it follows immediately that the subset of~$\V {s} \fxi(D_\RR)$ corresponding to
the space $Z^1_\Xi\bigl(\Gm;\V s \fxi(\proj\RR)\bigr)_\sic$ contains only elements~$f=(f_1,f_2)$ whose both component functions~$f_1$ and~$f_2$ are real-analytic on their domains. For that reason we set, for any~$s\in\CC$,
\[
 \V {s} \om(D_\RR) \coloneqq \V s \omega( (-1,\infty)_c ) \times \V s \omega( (\infty, 1)_c )\,.
\]
Let $c\in Z^1_\Xi\bigl(\Gamma;\V s \fxi(\proj\RR)\bigr)_\sic$.  For investigating the location of its boundary singularities we will take advantage of a well-chosen potential of~$c$. For the potential~$p$ from Proposition~\ref{prop:cocycle} we have 
\[
 \bsing p(\infty) = \emptyset
\]
and
\[
 \bsing p(1) \subseteq \{ 1,\infty\}\,,
\]
which most likely is indeed
\[
 \bsing p(1) = \{1,\infty\}\,.
\]
Thus, the boundary singularities of~$p$ are distributed in a rather unbalanced way and does not make it easy to study the singularities of~$c$. Therefore, we establish for a \emph{certain subclass} of cocycles~$c\in Z^1_\Xi\bigl(\Gamma;\V s \fxi(\proj\RR)\bigr)$ the existence of a potential~$\tilde p$ of~$c$ such that 
\[
 \bsing \tilde p(\xi) \subseteq \{\xi\}
\]
for all~$\xi\in\Xi$. (See Lemma~\ref{lem:analytic}.) The idea behind finding $\tilde p$ is to move the singularity of~$p(1)$ at~$\infty$ over to~$p(\infty)$. 
In Proposition~\ref{prop:iso_sic} we complete then the task of identifying the subset of~$\V s \om(D_\RR)$ corresponding to~$Z^1_\Xi\bigl(\Gamma;\V s \fxi(\proj\RR)\bigr)_\sic$ under~$\pc$. We recall from Section~\ref{sect-sars} that $f_1,f_2\in\V s \fxi$ are called equivalent, $f_1\equiv f_2$, if $f_1-f_2$ extends to an element of~$\V s \om(\proj\RR)$.

\begin{lem}\label{lem:analytic}
Let $f=(f_1,f_2)\in \V s \omega(D_\RR)$, let $c\coloneqq  \pc(f) \in Z^1_\Xi\bigl(\Gamma;\V s \fxi(\proj\RR)\bigr)$ be the associated cocycle, and let $p\colon \Gamma\to \V s \fxi$ and $\psi\in Z^1\bigl(\Gamma;\V s \fxi(\proj\RR)\bigr)$ be the potential and group cocycle associated to~$c$ by Proposition~\ref{prop:cocycle}. Then there exist elements $A_0 \in \V s \omega[0]$, $B_1 \in \V s \omega[1]$, $A_\infty, B_\infty\in \V s \omega[\infty]$ such that 
\begin{equation}\label{sdcp} 
\begin{aligned}
\psi_S&\ceqq  c(0,\infty) \ceqq  A_0-A_\infty\,,
\\
p(1) &\ceqq  c(1,\infty) \ceqq  B_1-B_\infty\,.
\end{aligned}
\end{equation}
For any choice of~$A_0, A_\infty, B_1, B_\infty$ we have
\begin{equation}\label{sing_equiv1}
A_\infty \;\equiv\; B_\infty\,,\qquad \tau_s(S) A_0 \;\equiv\; A_\infty\,.
\end{equation}
Moreover, 
\begin{equation}\label{sing_equiv2}
 \tau_s(T^{-1}) A_\infty \;\equiv\; A_\infty
\end{equation}
if and only if the map
\begin{equation}\label{ext_infty}
 \begin{cases}
  \left( 1  - \tau_s(T^{-1}) \right) \left(\tau_s(S)f_1 + f_2 \right) & \text{on $(\infty,-\lambda)_c$}
  \\
  -\left( 1 - \tau_s(T^{-1}) \right) \left( f_1 + \tau_s(S)f_2 \right) & \text{on $(0,\infty)_c$}
 \end{cases}
\end{equation}
extends real-analytically to~$\infty$. In this case, the map 
\begin{equation}\label{newcoc}
 \tilde p\coloneqq  p+A_\infty \colon \Gamma\to \V s \fxi
\end{equation}
is a potential of~$c$ satisfying 
\[
 \tilde p(\xi) \in \V s \omega[\xi]
\]
for all~$\xi\in\Xi$. The associated group cocycle 
\[
 \tilde\psi\coloneqq  \psi- dA_\infty \colon \gamma \mapsto \psi_\gamma + (1-\tau_s(\gamma^{-1}))A_\infty
\]
is analytic, i.\,e., $\tilde\psi\in Z^1(\Gamma;\V s \omega)$.
\end{lem}

\begin{proof}
The formulas in~\eqref{pot_p1} and~\eqref{def_grpcoc} defining $p(1)$ and $\psi_S$, respectively, and the regularity of~$f$ yield that $p(1)\in\V s \omega[1,\infty]$ and $\psi_S\in\V s \omega[0,\infty]$. Thus, Proposition~\ref{prop-sepsing} implies the existence of~$A_0,A_\infty, B_1,B_\infty$ with the properties as claimed in~\eqref{sdcp}. 

We prove the equivalences in~\eqref{sing_equiv1}. Let $I$ be a (small) neighborhood of~$\infty$. Comparing~\eqref{pot_p1} and \eqref{def_grpcoc} shows that on~$I\smallsetminus\{\infty\}$ we have
\[
 p(1) - \psi_S = \tau_s(S)f_1\,.  
\]
Clearly, $\tau_s(S)f_1$ is real-analytic in a neighborhood of~$\infty$ and hence the difference of~$p(1)$ and~$\psi_S$ is so. In turn, the difference between~$A_\infty$ and~$B_\infty$ is real-analytic in such a neighborhood, and therefore $A_\infty\equiv B_\infty$.

The equality $\tau_s(S)\psi_S = - \psi_S$ implies that $A_\infty-\tau_s(S) A_0 = A_0 - \tau_s(S) A_\infty$. The left hand side of this equation is in~$\V s \om[\infty]$ and the right hand side in~$\V s \om[0]$. Hence both sides are in~$\V s \om[0]\cap\V s \om[\infty]=\V s \om(\proj\RR)$, and thus $\tau_s(S)A_0 \equiv A_\infty$.

To prove the equivalence of~\eqref{sing_equiv2} and~\eqref{ext_infty} we note that 
\[
 \tau_s(T^{-1}) A_\infty \equiv A_\infty
\]
if and only if 
\[
\tau_s(T^{-1})\psi_S - \psi_S = \tau_s(T^{-1})A_0 - A_0 + \tau_s(T^{-1})A_\infty - A_\infty
\]
is real-analytic in~$\infty$. By~\eqref{def_grpcoc}, the latter is the case if and only if 
\begin{align*}
 \Big(\tau_s(T^{-1})\psi_S &- \psi_S\Big)\vert_{(\infty,-\lambda)_c\cup (0,\infty)_c}
 \\
 & = 
 \begin{cases}
  -\tau_s(T^{-1}S) f_1 - \tau_s(T^{-1})f_2 + \tau_s(S)f_1 + f_2 & \text{on $(\infty,-\lambda)_c$}
  \\
  \tau_s(T^{-1})f_1 + \tau_s(T^{-1}S)f_2 - f_2 - \tau_s(S)f_2 & \text{on $(0,\infty)_c$}
 \end{cases}
 \\
 & = 
 \begin{cases}
  \left( 1 - \tau_s(T^{-1})\right) \left(\tau_s(S)f_1 + f_2 \right) & \text{on $(\infty,-\lambda)_c$}
  \\
  -\left( 1 - \tau_s(T^{-1})\right) \left( f_1 + \tau_s(S)f_2 \right) & \text{on $(0,\infty)_c$}
 \end{cases}
\end{align*}
extends real-analytically to~$\infty$.

For the rest of this proof we suppose that \eqref{sing_equiv2} or, equivalently, \eqref{ext_infty} is satisfied. Obviously, $\tilde p$ defines the same cocycle in~$Z^1_\Xi\bigl(\Gamma;\V s \fxi(\proj\RR)\bigr)$ as~$p$, and $\tilde\psi$ is the group cocycle associated to~$\tilde p$. We have
\[
 \tilde\psi_S = \psi_S + (1 - \tau_s(S))A_\infty = A_0 - \tau_s(S)A_\infty\,,
\]
which is in~$\V s \omega(\proj\RR)$ by~\eqref{sing_equiv1}. Further, 
\[
 \tilde\psi_T = \psi_T + (1-\tau_s(T^{-1}))A_\infty = (1-\tau_s(T^{-1}))A_\infty\,,
\]
which is in~$\V s \omega(\proj\RR)$ by~\eqref{sing_equiv2}. Thus, $\tilde\psi\in Z^1_\Xi(\Gamma; \V s \omega(\proj\RR))$.

It remains to show the regularity statement for the potential~$\tilde p$. Using the regularity properties of~$\tilde\psi$ we find for any~$\gamma\in\Gamma$,
\begin{align*}
\tilde p(\gamma\infty) & =  \tau_s(\gamma)\tilde p(\infty) + \tilde\psi_{\gamma^{-1}} 
\\
& \equiv \tau_s(\gamma)\big( p(\infty) + A_\infty\big) 
\\
& = \tau_s(\gamma)A_\infty  \in \V s \omega [\gamma\infty]
\end{align*}
and
\begin{align*}
\tilde p(\gamma 1) & = \tau_s(\gamma) \tilde p(1) + \tilde\psi_{\gamma^{-1}}
\\
& \equiv \tau_s(\gamma)\big( p(1) + A_\infty \big)
\\
& = \tau_s(\gamma) \big( B_1 - B_\infty + A_\infty \big)
\\
& \equiv \tau_s(\gamma) B_1 \in \V s \omega [\gamma 1]\,.
\end{align*}
This completes the proof.
\end{proof}

For $s\in\CC$ let $\spaceext$ \index[symbols]{Vaaac@$\spaceext$}\index[symbols]{E@$\mext$}\index[defs]{condition!$\extm$}\index[defs]{ext@$\extm$} denote the space of elements~$f=(f_1,f_2) \in \V s \omega(D_\RR)$ such that the map
\begin{equation}\label{ext_infty2}
 \begin{cases}
  \left( 1  - \tau_s(T^{-1}) \right) \left(\tau_s(S)f_1 + f_2 \right) & \text{on $(\infty,-\lambda)_c$}
  \\
  -\left( 1 - \tau_s(T^{-1}) \right) \left( f_1 + \tau_s(S)f_2 \right) & \text{on $(0,\infty)_c$}
 \end{cases}
\end{equation}
extends real-analytically to~$\infty$.

\begin{prop}\label{prop:iso_sic}
For any $s\in\CC$, the map~$\pc$ descends to an isomorphism between~$\spaceext$ and~$Z^1_\Xi\bigl(\Gamma;\V s \fxi(\proj\RR)\bigr)_\sic$.
\end{prop}

\begin{proof}
By Proposition~\ref{prop:cocycle}\eqref{prop:coci} it suffices to establish the set equality 
\[
\pc(\spaceext) =  Z^1_\Xi\bigl(\Gamma;\V s \fxi(\proj\RR)\bigr)_\sic\,.
\]
Let $f=(f_1,f_2)\in \spaceext$ and let $c\coloneqq  \pc(f)$. By Proposition~\ref{prop:cocycle} the cocycle~$c$ is in~$Z^1_\Xi\bigl(\Gamma;\V s \fxi(\proj\RR)\bigr)$. Thus, it remains to show the singularity condition for~$c$. To that end let $\tilde p$ be the potential of~$c$ given by Lemma~\ref{lem:analytic} (see~\eqref{newcoc}). Recall that for all~$\xi\in \Xi$ we have $\tilde p(\xi) \in \V s \omega[\xi]$. Hence, for all~$\xi,\eta\in \Xi$ we have
\[
 c(\xi,\eta) = \tilde p(\xi) - \tilde p(\eta) \in \V s \omega[\xi,\eta]\,.
\]
This shows that $c\in Z^1_\Xi\bigl(\Gamma;\V s \fxi(\proj\RR)\bigr)_\sic$.

Conversely, let $c\in Z^1_\Xi\bigl(\Gamma;\V s \fxi(\proj\RR)\bigr)_\sic$, and let $f=(f_1,f_2) \coloneqq  \pc^{-1}(c)$, thus (see~\eqref{c1shortprop}-\eqref{c2shortprop})
\[
 f_1 = c(-1,\infty)\vert_{(-1,\infty)_c}\quad\text{and}\quad f_2 = - c(1,\infty)\vert_{(1,\infty)_c}\,.
\]
The singularity condition for~$c$ yields that $f_1$ and $f_2$ are real-analytic. It remains to establish~\eqref{ext_infty2}. To that end choose~$A_0\in \V s \omega[0]$, $A_\infty\in \V s \omega [\infty]$ such that 
\[
 c(0,\infty) = A_0 - A_\infty\,,
\]
see~\eqref{sdcp}. Then 
\begin{align*}
c(0,-\lambda) & = c(0,\infty) - c(-\lambda,\infty) = c(0,\infty) - \tau_s(T^{-1})c(0,\infty)
\\
& = A_0 - \tau_s(T^{-1})A_0 + A_\infty - \tau_s(T^{-1})A_\infty\,.
\end{align*}
From $\bsing c(0,-\lambda)\subseteq\{0,-\lambda\}$ it follows that 
\[
 A_\infty \equiv \tau_s(T^{-1})A_\infty\,,
\]
which, by Lemma~\ref{lem:analytic}, implies~\eqref{ext_infty2}.
\end{proof}

The combination of Propositions~\ref{prop:iso_van} and \ref{prop:iso_sic} shows that $\pc$ descends to an isomorphism
\begin{equation}
 \FE_s^\fxi(\RR)\cap \spaceext \to Z^1_\Xi\bigl(\Gamma;\V s \fxi(\proj\RR)\bigr)_\sic^\van\,.
\end{equation}
With Corollary~\ref{cor:per_sicvan} we will prove that 
\[
 \FE_s^\fxi(\RR) \cap \spaceext = \FE_s^\om(\RR)\,,
\]
yielding that the map~$\pc$ descends to an isomorphism between the spaces~$\FE_s^\omega(\RR)$ and $Z^1_\Xi\bigl(\Gamma;\V s \fxi(\proj\RR)\bigr)_\sic^\van$. The combination of the singularity and the vanishing conditions imply further that all cocycles in~$Z^1_\Xi\bigl(\Gamma;\V s \fxi(\proj\RR)\bigr)_\sic^\van$ have analytic jumps at the points in~$\Gamma\,1$. See Proposition~\ref{prop:an_jump}. The geometric reason behind this phenomenon is the contraction property of the acting elements in~$\TO_s^\slow$ or, equivalently, the real-analytic extendability of the elements in~$\FE_s^\omega(\RR)$ to larger domains. It has the effect that the map~$\pc$ descends to an isomorphism between~$\FE_s^\omega(\RR)$ and~$Z^1_\Xi\bigl(\Gamma;\V s {\fxi;-;\aj}(\proj\RR)\bigr)_\sic^\van$. See Corollary~\ref{cor:iso_aj}.

\begin{cor}\label{cor:per_sicvan} 
For any $s\in\CC$, the map~$\pc$ descends to an isomorphism between $\FE_s^\omega(\RR)$ and  $Z^1_\Xi\bigl(\Gamma;\V s \fxi(\proj\RR)\bigr)_\sic^\van$.
\end{cor}

\begin{proof}
Taking advantage of Propositions~\ref{prop:iso_van} and \ref{prop:iso_sic} it suffices to establish that each element of~$\FE_s^\omega(\RR)$ satisfies~\eqref{ext_infty2}. To that end let $f=(f_1,f_2)\in \FE_s^\omega(\RR)$. We reformulate~\eqref{ext_infty2}. On~$(0,\infty)_c$, \eqref{FEeqa} implies the identity
\begin{align*}
 -\left(1-\tau_s(T^{-1})\right)\left(f_1 + \tau_s(S)f_2\right) & = -f_1 + \tau_s(T^{-1})f_1 + \tau_s(T^{-1}S)f_2 - \tau_s(S)f_2
 \\
 & = -\tau_s(T^{-1}S)f_1 - \tau_s(S)f_2\,.
\end{align*}
On~$(\infty,-\lambda)_c$, \eqref{FEeqb} yields the identity
\begin{align*}
 \left(1-\tau_s(T^{-1})\right)&\left(\tau_s(S)f_1 + f_2\right) \\
 & = -\tau_s(T^{-1}S)f_1 + \tau_s(T^{-1})\left( -f_2 + \tau_s(T)f_2 + \tau_s(TS)f_1\right)
 \\
 & = -\tau_s(T^{-1}S)f_1 - \tau_s(S)f_2\,.
\end{align*}
The map~$\tau_s(T^{-1}S)f_1 + \tau_s(S)f_2$ is real-analytic in~$\infty$. Hence, \eqref{ext_infty2} is satisfied.
\end{proof}

\begin{prop}\label{prop:an_jump}
For all $s\in\CC$ we have
\[
 Z^1_\Xi\bigl(\Gamma;\V s {\fxi;-;\aj}(\proj\RR)\bigr)_\sic^\van = Z^1_\Xi\bigl(\Gamma;\V s \fxi(\proj\RR)\bigr)_\sic^\van\,.
\]
\end{prop}

\begin{proof}
It suffices to show that 
\[ Z^1_\Xi\bigl(\Gamma;\V s \fxi(\proj\RR)\bigr)_\sic^\van\subseteq Z^1_\Xi\bigl(\Gamma;\V s {\fxi;-;\aj}(\proj\RR)\bigr)_\sic^\van\,.\]
 Let $c\in Z^1_\Xi\bigl(\Gamma;\V s \fxi(\proj\RR)\bigr)_\sic^\van$. By Corollary~\ref{cor:per_sicvan} there exists a (unique) period function $f = (f_1,f_2)\in \FE_s^\omega(\RR)$ such that $c=\pc(f)$. Let $\tilde p$ be the potential of~$c$ provided by Lemma~\ref{lem:analytic}. Due to the regularity properties of~$\tilde p$ it suffices to show that $\tilde p(1)$ has an analytic jump at~$1$. 

Let $p$ be the potential of~$c$ provided by Proposition~\ref{prop:cocycle}, and let $A_\infty$ be as in Lemma~\ref{lem:analytic}. Since (see~\eqref{newcoc})
\[
 \tilde p(1) = p(1) + A_\infty
\]
and since $A_\infty$ is real-analytic in~$1$, it suffices to show that $p(1)$ has an analytic jump at~$1$. Recall from Proposition~\ref{prop:cocycle} that 
\[
 p(1) = 
 \begin{cases}
  -f_2 & \text{on $(\infty,1)_c$}
  \\
  f_1 + \tau_s(S)f_1 + \tau_s(S)f_2 & \text{on $(1,\infty)_c$\,.}
 \end{cases}
\]
Since $f_1\in \V s \omega( (-1,\infty)_c )$ and $f_2\in \V s \omega( (\infty, 1)_c)$, the map~$p(1)$ is real-analytic on~$(\infty,1)_c$ and on~$(1,\infty)_c$. By Proposition~\ref{prop-re}, $f_1$ extends real-analytically beyond~$-1$, and $f_2$ extends real-analytically beyond~$1$, which implies that $p(1)\vert_{(\infty,1)_c}$ and $p(1)\vert_{(1,\infty)_c}$ both have real-analytic extensions beyond~$1$. Thus, $p(1)$ has an analytic jump at~$1$. This completes the proof.
\end{proof}

Combining Corollary~\ref{cor:per_sicvan} and Proposition~\ref{prop:an_jump} immediately implies the following isomorphism.

\begin{cor}\label{cor:iso_aj}
 The map~$\pc$ descends to an isomorphism between~$\FE_s^\omega(\RR)$ and $Z^1_\Xi\bigl(\Gamma;\V s {\fxi;-;\aj}(\proj\RR)\bigr)_\sic^\van$.
\end{cor}

\section{Complex period functions and semi-analytic cohomology}\label{sec:pc_complex}
\markright{20. COMPLEX PERIOD FUNCTIONS AND COHOMOLOGY}

In Section~\ref{sec:pc_real} we showed that $\pc$ descends to an isomorphism 
\[
 \FE_s^\om(\RR) \to Z^1_\Xi\bigl(\Gamma;\V s {\fxi;-;\aj}(\proj\RR)\bigr)^\van_\sic\,.
\]
We recall from Section~\ref{sec:periodfunctions} that for any~$s\in\CC$, the space~$\FE_s^\om(\CC)$ of complex period functions can be identified with the subspace of~$\FE_s^\om(\RR)$ consisting of those real period functions~$f=(f_1,f_2)\in\FE_s^\om(\RR)$ for which $f_1$ extends holomorphically to the domain~$\CC\smallsetminus (-\infty,-1]$ and $f_2$ extends holomorphically to~$\CC\smallsetminus [1,\infty)$. Taking advantage of this identification, the map~$\pc$ naturally induces a map 
\[
 \FE_s^\om(\CC) \to Z^1_\Xi\bigl(\Gamma;\V s {\fxi;-;\aj}(\proj\RR)\bigr)^\van_\sic\,,
\]
which we also call~$\pc$. In this section we will show that $\pc$ establishes isomorphisms
\begin{align*}
 \FE_s^\om(\CC) &\to Z^1_\Xi\bigl(\Gamma;\V s {\fxi;\exc;\aj}(\proj\RR)\bigr)^\van_\sic\,,
 \\
 \FE_s^{\om,1}(\CC) & \to Z^1_\Xi\bigl(\Gamma;\V s {\fxi;\exc,\smp;\aj}(\proj\RR)\bigr)^\van_\sic 
 \intertext{and}
 \FE_s^{\om,0}(\CC) & \to Z^1_\Xi\bigl(\Gamma;\V s {\fxi;\exc,\infty;\aj}(\proj\RR)\bigr)^\van_\sic
\end{align*}
for appropriate values of~$s\in\CC$. We refer to Propositions~\ref{prop:iso_exc} and \ref{prop:EF_TO_fast} for precise statements. Moreover, we characterize the coboundaries in terms of properties of complex period functions.

\begin{prop}\label{prop:iso_exc}
For any $s\in\CC$, the map~$\pc$ descends to an isomorphism between $\FE_s^\omega(\CC)$ and $Z^1_\Xi\bigl(\Gamma;\V s {\fxi;\exc;\aj}(\proj\RR)\bigr)_\sic^\van$. 
\end{prop}

\begin{proof}
In view of Proposition~\ref{prop:cocycle} it suffices to establish the equality (of sets) 
\[
\pc(\FE_s^\omega(\CC))=Z^1_\Xi\bigl(\Gamma;\V s {\fxi;\exc;\aj}(\proj\RR)\bigr)_\sic^\van\,.
\]
We first show that $\pc(\FE_s^\omega(\CC))$ is a subset of~$Z^1_\Xi\bigl(\Gamma;\V s {\fxi;\exc;\aj}(\proj\RR)\bigr)_\sic^\van$. To that end let $f=(f_1,f_2)\in \FE_s^\omega(\CC)$ and let $c=\pc(f) \in Z^1_\Xi\bigl(\Gamma;\V s \fxi(\proj\RR)\bigr)$ be the associated cocycle. Since $f\in \FE_s^\omega(\RR)$, Corollary~\ref{cor:iso_aj} shows that $c\in Z^1_\Xi\bigl(\Gamma;\V s {\fxi;-;\aj}(\proj\RR)\bigr)_\sic^\van$. Thus, it remains to prove that~$c$ satisfies the condition~\texc at all cusps, that is, it remains to show that for all~$\xi,\eta\in\Xi$, all $g\in\Gamma$, the map~$c(\xi,\eta)$ satisfies~\texc at~$g\infty$.

Let $\tilde p$ and $\tilde\psi$ be the potential of~$c$ and its associated group cocycle provided by Lemma~\ref{lem:analytic}, respectively. Recall that for all~$\xi\in\Xi$, 
\[
 \tilde p(\xi) \in \V s \omega[\xi]\,.
\]
Thus, it remains to show that for all~$g\in \Gamma$, the map~$\tilde p(g\infty)$ satisfies~\texc at~$g\infty$. Recall further that $\tilde\psi\in Z^1\bigl(\Gamma;\V s \omega(\proj\RR)\bigr)$, implying that $\tilde\psi_{g^{-1}}$ is analytic. Thus, the relation
\[
 \tilde p(g\infty) = \tau_s(g)\tilde p(\infty) + \tilde\psi_{g^{-1}}
\]
yields that it suffices to show that $\tilde p(\infty)$ satisfies~\texc at~$\infty$. To that end recall the potential~$p$ and its associated group cocycle~$\psi$ from Proposition~\ref{prop:cocycle} and let the elements~$A_0\in \V s \omega[0]$, $A_\infty\in \V s \omega[\infty]$ be as in Lemma~\ref{lem:analytic}. Then 
\[
 \tilde p(\infty) = p(\infty) + A_\infty = A_\infty = A_0 - \psi_S\,.
\]
Since $A_0$ obviously satisfies~\texc at~$\infty$, it remains to show that $\psi_S$ does so. Recall from Proposition~\ref{prop:cocycle} (see~\eqref{def_grpcoc}) that 
\[
 \psi_S = 
 \begin{cases}
-\tau_s(S)f_1 - f_2 & \text{on $(\infty,0)_c$}
\\
f_1 + \tau_s(S)f_2 & \text{on $(0,\infty)_c$\,.}
 \end{cases}
\]
Recall that $f_1$ extends holomorphically to~$\CC\smallsetminus (-\infty,-1]$, and $f_2$ extends holomorphically to~$\CC\smallsetminus [1,\infty)$.  On~$(\infty,0)_c$ we have
\[
 \psi_S(x) = -\tau_s(S)f_1(x) - f_2(x) = - \big(x^{-2}\big)^s f_1 \left(-\frac1x\right) - f_2(x)\,.
\]
Since the map~$z\mapsto z^{-2}$ is holomorphic on 
\[
 \{ \Rea z > 0 \} \cup \{ \Rea z<0\}\,,
\]
the extendability properties of~$f_1$ and~$f_2$ yield that $\psi_S$ extends holomorphically to 
\[
 \{ \Rea z < 0\}\,.
\]
On~$(0,\infty)_c$ we have
\[
 \psi_S(x) = f_1(x) + \tau_s(S)f_2(x) = f_1(x) + \big(x^{-2}\big)^s f_2\left(-\frac1x\right)\,, 
\]
which extends holomorphically to 
\[
 \{ \Rea z > 0\}\,.
\]
Thus, $\psi_S$ satisfies~\texc at~$\infty$. In turn, $c\in Z^1_\Xi\bigl(\Gamma;\V s {\fxi;\exc;\aj}(\proj\RR)\bigr)_\sic^\van$.

We now show that $Z^1_\Xi\bigl(\Gamma;\V s {\fxi;\exc;\aj}(\proj\RR)\bigr)_\sic^\van$ is a subset of~$\pc(\FE_s^\omega(\CC))$. To that end let $c\in Z^1_\Xi\bigl(\Gamma;\V s {\fxi;\exc;\aj}(\proj\RR)\bigr)_\sic^\van$ and let $f=(f_1,f_2)\coloneqq  \pc^{-1}(c)$, thus
\[
 f_1 = c(-1,\infty)\vert_{(-1,\infty)_c}\quad\text{and}\quad f_2 = -c(1,\infty)\vert_{(\infty,1)_c}\,.
\]
By Corollary~\ref{cor:iso_aj}, $f\in \FE_s^\omega(\RR)$. The property~\texc yields that $c(-1,\infty)$ is holomorphic on~$\{ \Rea z>x_0\}$ for some~$x_0\in\RR$. Thus, $f_1$ extends holomorphically to a complex neighborhood of~$(-1,\infty)_c$ that is rounded at~$\infty$. Likewise, since $c(1,\infty)$ is holomorphic on~$\{ \Rea z < x_1\}$ for some~$x_1\in\RR$, the map~$f_2$ extends holomorphically to a complex neighborhood of~$(\infty, 1)_c$ that is rounded at~$\infty$. Proposition~\ref{prop-FE-RC} now shows that $f\in\FE_s^\omega(\CC)$. This completes the proof.
\end{proof}

We now show that the map~$\pc$ identifies the subspaces~$\FE_s^{\om,1}(\CC)$ and $\FE_s^{\om,0}(\CC)$ of all complex period functions with the vector  spaces~$Z^1_\Xi\bigl(\Gamma;\V s {\fxi;\exc,\smp;\aj}(\proj\RR)\bigr)^\van_\sic$ and $Z^1_\Xi\bigl(\Gamma;\V s {\fxi;\exc,\infty;\aj}(\proj\RR)\bigr)^\van_\sic$, respectively, of cocycles with values in functions that have a certain singularity behavior at the cuspidal points~$\Gm\infty$. In terms of period functions, the behavior at these points is controlled by smoothing properties of the transfer operator~$\TO_s^\fast$. Since $\TO_s^\fast$ is closely related to the one-sided averaging operators~$\av{s,T}^\pm$ (see Sections~\ref{sect-osav} and \ref{sec:fastconv}), the smoothing properties of these operators are essentially identical. We start with a brief study of the smoothing properties of the averaging operators.

We recall from Section~\ref{sect-psa} that for any interval~$I\subseteq\proj\RR$ with~$\infty\in I$, a map $\varphi\colon I\to\nobreak\CC$ is an element of~$\V s \infty(I)$ if $\varphi\in C^\infty(I\smallsetminus\{\infty\})$ and if it has an asymptotic expansion
\begin{equation}\label{aei}
\varphi(  t) \sim |t|^{-2s} \sum_{m\geq 0} a_m t^{-m} \quad
\text{ as }t\rightarrow\pm \infty\,.
\end{equation}

\begin{lem}\label{lem:asymptotics}
Let $s\in\CC$, $\Rea s > 0$. Let $a,b\in\RR$, $a\leq b$, and $\varphi\in \V s \omega\big( (b,a)_c \big)$. Set
\[
 \psi\coloneqq 
 \begin{cases}
\av{s,T}^+\varphi & \text{on $(b,\infty)_c$}
\\
\av{s,T}^-\varphi & \text{on $(\infty,a)_c$\,.}
 \end{cases}
\]
\begin{enumerate}[{\rm (i)}]
\item\label{asympi} Suppose that $s\not=\frac12$. Then $\psi$ has a simple singularity at $\infty$.
\item\label{asympii} Suppose that in the asymptotic expansion in~\eqref{aei} of $\varphi$ we have $a_0=0$. Then $\psi$ extends to an element in $\V s \infty\big( (b,a)_c\big)$.
\end{enumerate}
\end{lem}

\begin{proof}
Suppose first that $s\not=\tfrac12$. As discussed in Section~\ref{sect-osav},  
\begin{equation}\label{phi_omega}
\psi\in \V s \omega\big( (\infty,a)_c\cup (b,\infty)_c \big)
\end{equation}
and $\psi$ has asymptotic expansions 
\begin{align}\label{asympexp1}
 \psi(x) &\sim |x|^{-2s} \sum_{m\geq -1} c_mx^m \qquad\text{as $x\uparrow \infty$}
 \intertext{and}
 \psi(x) &\sim |x|^{-2s} \sum_{m\geq -1} c_mx^m \qquad\text{as $x\downarrow \infty$}\label{asympexp2}
\end{align}
with the same coefficients in both cases. Thus, $\psi$ has a simple singularity at $\infty$. This proves \eqref{asympi}.

Suppose now that $s\in\CC$ with $\Rea s\in (0,1)$ arbitrary (including $s=\tfrac12$), and suppose that $a_0=0$ in the asymptotic expansion in~\eqref{aei} of $\varphi$. Then \eqref{phi_omega} and the asymptotic expansions in~\eqref{asympexp1}-\eqref{asympexp2} remain valid also for the case $s=\tfrac12$. By \cite[(4.12)]{BLZm}, $a_0=0$ implies $c_{-1}=0$, and hence $\psi$ extends smoothly to $\infty$. This shows \eqref{asympii}.
\end{proof}

\begin{prop}\label{prop:EF_TO_fast}
\begin{enumerate}[{\rm (i)}]
\item\label{eti} For $s\in\CC$, $\Rea s\in (0,1)$, $s\not=\tfrac12$,
the map $\pc$ descends to an isomorphism between $\FE_s^{\omega,1}(\CC)$ and $Z^1_\Xi\bigl(\Gamma;\V s {\fxi;\exc,\smp;\aj}(\proj\RR)\bigr)^\van_\sic$.
\item\label{etii} For any $s\in\CC$, $\Rea s\in (0,1)$, the map $\pc$ descends to an isomorphism between $\FE_s^{\omega,0}(\CC)$ and $ Z^1_\Xi\bigl(\Gamma;\V s {\fxi;\exc,\infty;\aj}(\proj\RR)\bigr)^\van_\sic$.
\end{enumerate}
\end{prop}

\begin{proof}
Major parts of the proofs of~\eqref{eti} and~\eqref{etii} are identical, for which reason we start by  considering both cases simultaneously and then will split the discussion only at the very end. We refer to~\eqref{eti} as `Case~1', and to~\eqref{etii} as `Case~2'.

By Proposition~\ref{prop:cocycle} it suffices to show the two set equalities 
\begin{align*} \pc(\FE_s^{\omega,1}(\CC)) &= Z^1_\Xi\bigl(\Gamma;\V s {\fxi;\exc,\smp;\aj}(\proj\RR)\bigr)_\sic^\van\\
\text{ and }\quad\pc(\FE_s^{\omega,0}(\CC)) &= Z^1_\Xi\bigl(\Gamma;\V s {\fxi;\exc,\infty;\aj}(\proj\RR)\bigr)_\sic^\van\,.\end{align*}
We start by showing that the images of $\FE_s^{\omega,1}(\CC)$ and $\FE_s^{\omega,0}(\CC)$ under $\pc$ are contained in the cocycle spaces as claimed.

Let $f=(f_1,f_2)\in \FE_s^\omega(\CC)$, $f=\ftro s f$, and let $c\coloneqq  \pc(f)$. By Proposition~\ref{prop:iso_exc} the cocycle~$c$ is in~$Z^1_\Xi\bigl(\Gamma;\V s {\fxi;\exc;\aj}(\proj\RR)\bigr)_\sic^\van$. Thus, in Case~1 it remains to show that $c$ has at most simple singularities at all cusps, and in Case~2 that $c$ extends smoothly to all cusps. Let $p$ be the potential of~$c$ provided by Proposition~\ref{prop:cocycle}, let $B_1\in \V s \omega[1]$, $A_\infty,B_\infty\in \V s \omega[\infty]$ be as in Lemma~\ref{lem:analytic}, and let $\tilde p$ be the potential of~$c$ provided by Lemma~\ref{lem:analytic}. Arguing as in the proof of Proposition~\ref{prop:iso_exc} we see that it suffices to show that 
\[
 \tilde p(\infty) = p(\infty) + A_\infty = A_\infty
\]
has a simple singularity at~$\infty$ (for Case~1) or extends smoothly to~$\infty$ (for Case~2). Since $A_\infty$ and $B_\infty$ differ only by an additive element in~$\mc V_s^\omega(\proj\RR)$, it suffices to establish these properties for~$B_\infty$.

Let $h\coloneqq  \tau_s(S)(f_1+f_2)$. From $f= \ftro s f$ it follows that 
\begin{align}\label{equal_f1}
 f_1 &=  \av{s,T}^+ \tau_s(T^{-1}) h = \av{s,T}^+h - h 
 \\
 f_2 & = -\av{s,T}^- h\,. \label{equal_f2}
\end{align}
Let 
\begin{equation}\label{def_g}
 \tilde g \coloneqq  
 \begin{cases}
 \av{s,T}^-h & \text{on $(\infty,-1)_c$}
 \\
  \av{s,T}^+h & \text{on $(1,\infty)_c$\,.}
 \end{cases}
\end{equation}
Comparing~\eqref{pot_p1} with \eqref{equal_f1}-\eqref{def_g} shows that 
\[
 p(1) = \tilde g \quad\text{on $(1,-1)_c$\,.}
\]
Recall that 
\[
 p(1) = B_1 - B_\infty\,.
\]
Thus, $B_\infty$ has the same regularity at~$\infty$ as~$\tilde g$, for which we apply Lemma~\ref{lem:asymptotics} in both cases. We note that  
\[
 h\in \V s \omega\big( (1,-1)_c\big)\,.
\]
Suppose that in a neighborhood of~$0$, the Taylor expansions of~$f_1$ and $f_2$ are given by
\begin{align*}
 f_1(x) &= \sum_{n=0}^\infty a_n x^n\,, 
 \\
 f_2(x) & = \sum_{n=0}^\infty b_nx^n\,,
\end{align*}
respectively. Then 
\[
 h(x) = \tau_s(S)(f_1+f_2)(x) \sim \big(x^{-2}\big)^s \sum_{n=0}^\infty (-1)^n(a_n+b_n)x^{-n}\qquad\text{as $x\to\infty$\,.}
\]
We note that this limit is indeed two-sided. In Case~1, Lemma~\ref{lem:asymptotics} shows that $\tilde g$ has a simple singularity at~$\infty$, and hence $B_\infty \in \V s {\omega,\smp}[\infty]$. In Case~2, Lemma~\ref{lem:asymptotics} shows that $\tilde g$ extends smoothly to~$\infty$, and hence $B_\infty\in \V s {\omega,\infty}[\infty]$. Thus, in both cases, the cocycle~$c$ has the properties claimed.

For the converse direction we again start by considering both cases simultaneously. Let $s\in\CC$, $\Rea s\in (0,1)$ (also allowing~$s=\tfrac12$) and take a cocycle $c\in Z^1_\Xi\bigl(\Gamma;\V s {\fxi;\exc,\smp;\aj}(\proj\RR)\bigr)_\sic^\van$. Set 
\[
f=(f_1,f_2)\coloneqq  \pc^{-1}(c)
\]
and let $h\coloneqq  \tau_s(S)(f_1+f_2)$. Then 
\[
 h\in \V s \omega\big( (1,-1)_c \big)
\]
and hence 
\[
  \tilde f = \begin{pmatrix}\tilde f_1 \\ \tilde f_2 \end{pmatrix} \coloneqq \ftro s f = \begin{pmatrix} \av{s,T}^+\tau_s(T^{-1})h \\ - \av{s,T}^-h\end{pmatrix}
\]
converges. By Proposition~\ref{prop:iso_exc}, $f=\TO_s^\slow f$. Hence \eqref{prto} shows that
\begin{align}\label{smooth1}
 \big(1 - \tau_s(T^{-1})\big) f_1 & = \tau_s(T^{-1})h\,,
 \\
 \big(1 - \tau_s(T)\big) f_2 & = \tau_s(T)h\,. \label{smooth2}
\end{align}

Suppose now that we are in Case~2. The smoothness condition implies that $f_1$ and $f_2$ extend smoothly to a neighborhood of~$\infty$. In what follows we identify the functions~$f_1$ and $f_2$ with their smooth extensions. Then~\eqref{avrel} implies
\begin{align}\label{findef1}
 f_1 & = \av{s,T}^+ \big(1 - \tau_s(T^{-1})\big)f_1 = \av{s,T}^+\tau_s(T^{-1})h = \tilde f_1
 \intertext{and}
 -\tau_s(T) f_2 & = \tau_s(T) \av{s,T}^-\big( \tau_s(T^{-1}) - 1 \big) f_2 = \av{s,T}^-\big(1-\tau_s(T)\big)f_2 \label{findef2}
 \\
 & = \av{s,T}^-\tau_s(T)h = \tau_s(T)\av{s,T}^-h = -\tau_s(T)\tilde f_2\,. \nonumber
\end{align}
Thus, $f=\tilde f$. To complete Case~2 it remains to show that $f_1(0)+f_2(0) = 0$. 

Since $f_1$ and $f_2$ are smooth at~$0$ (even real-analytic), $\tau_s(S)h$ is so. Thus, for suitable coefficients~$c_n\in\CC$, $n\in\NN_0$,
\[
 \tau_s(S)h(x) = (f_1+f_2)(x) \sim \sum_{n=0}^\infty c_nx^n\qquad\text{as $x\to0$\,,}
\]
or equivalently, 
\begin{equation}\label{eq:exp_h_infty}
 h(x) \sim \bigl(x^{-2}\bigr)^s \sum_{n=0}^\infty (-1)^n c_nx^{-n}\qquad\text{as $x\to\infty$\,.}
\end{equation}
Further, from~\eqref{smooth1} it follows that 
\begin{equation}\label{eq:at_infty}
 \bigl( \tau_s(T) - 1\bigr)f_1 = h \qquad\text{on~$(-1+\lambda,\infty)_c$\,.}
\end{equation}
Since $f_1$ and $h$ are smooth at~$\infty$, the equality in~\eqref{eq:at_infty} remains valid in a neighborhood of~$\infty$. We use~\eqref{eq:at_infty} to deduce an asymptotic expansion of~$h$ at~$\infty$. To that end we recall that $f_1$ has an asymptotic expansion at~$\infty$ of the form
\begin{equation}\label{eq:f1_asymp1}
 f_1(x) \sim \bigl(x^{-2}\bigr)^s \sum_{n=0}^\infty a_nx^{-n}\qquad\text{as $x\to\infty$\,.}
\end{equation}
This implies that, as $x\to\infty$, 
\begin{equation}\label{eq:f1_asymp2}
\begin{aligned}
 \tau_s(T)f_1(x) & \sim \bigl( (x-\lambda)^{-2} \bigr)^s \sum_{n=0}^\infty a_n (x-\lambda)^{-n} 
 \\
 & \sim \bigl(x^{-2}\bigr)^s \Bigl( a_0 + \sum_{n=1}^\infty \tilde a_n x^{-n} \Bigr) 
\end{aligned}
\end{equation}
for suitable choices of $\tilde a_n\in\CC$, $n\in\NN$. We stress that the zero-th coefficient of the asymptotic expansion does not change. Combining~\eqref{eq:f1_asymp1} and~\eqref{eq:f1_asymp2} with~\eqref{eq:at_infty} shows that 
\begin{equation}\label{eq:exp_h_new}
 h(x) \sim \bigl(x^{-2}\bigr)^s \sum_{n=1}^\infty \tilde a_nx^{-n}\qquad\text{as $x\to\infty$}
\end{equation}
for suitable choices of~$\tilde a_n\in\CC$, $n\in\NN$. We note that the zero-th coefficient of the asymptotic expansion in~\eqref{eq:exp_h_new} vanishes. Comparing \eqref{eq:exp_h_infty} and~\eqref{eq:exp_h_new} shows that $c_0=0$ in~\eqref{eq:exp_h_infty}, and hence
\[
 f_1(0) + f_2(0) = 0\,.
\]
We now complete the proof of the converse direction in Case~1. Thus, the singularity condition implies that $f_1$ and $f_2$ have an (at most simple) singularity at~$\infty$. Therefore, the conclusions in~\eqref{findef1}-\eqref{findef2} are not necessarily true anymore. However, \eqref{smooth1}-\eqref{smooth2} imply that 
\[
 \big(1-\tau_s(T^{-1})\big) f_1 \quad\text{and}\quad \big(1-\tau_s(T)\big)f_2
\]
are smooth at~$\infty$. Thus, \eqref{aei} yields
\begin{align*}
\big(1 - \tau_s(T^{-1})\big) f_1 & = 
\big(1 - \tau_s(T^{-1})\big)\av{s,T}^+ \big(1-\tau_s(T^{-1})\big) f_1
\\
& = \big(1 - \tau_s(T^{-1})\big)\av{s,T}^+\tau_s(T^{-1})h
\\
& = \big(1 - \tau_s(T^{-1})\big)\tilde f_1\,.
\end{align*}
Thus,
\[
 \big( 1 - \tau_s(T^{-1}) \big) (\tilde f_1 - f_1) = 0 \qquad\text{on $(-1,\infty)_c$}
\]
and, analogously,
\[
 \big( 1 - \tau_s(T^{-1}) \big) (\tilde f_2 - f_2) = 0 \qquad\text{on $(\infty, 1-\lambda)_c$\,.}
\]
Thus, there are $\lambda$-periodic functions $r_1$ on~$(-1,\infty)_c$ and $r_2$ on~$(\infty,1-\lambda)_c$ such that 
\[
 \tilde f_j = f_j + r_j \qquad (j=1,2)\,.
\]
We claim that $r_j=0$ for $j=1,2$. For~$j\in\{1,2\}$, the functions~$\tilde f_j$ and~$f_j$ are both smooth in a punctured neighborhood of~$\infty$ and may have a simple singularity at~$\infty$. Thus we have the two-sided asymptotic expansion
\[
 r_j(x) = \tilde f_j(x) - f_j(x) \sim (x^{-2})^s \sum_{m=-1}^\infty c_m x^{-m}\qquad\text{as $x\to\infty$\,.}
\]
Since $r_j$ is $\lambda$-periodic and $\Rea s \in (0,1)$, it follows that $c_m=0$ for all~$m\geq -1$, and hence $r_j=0$. This shows that $f=\tilde f$. 
\end{proof}

We end this section by characterizing the preimages of the coboundary spaces under the map~$\pc$. This is the last ingredient needed for a proof of Theorem~\ref{thm-FE-coh}. We recall from Section~\ref{sec:periodfunctions} that 
\[
 \BFE_s^\omega(\CC) = \{ (-b,b)\in \FE_s^\omega(\CC) \setmid \text{$b$ holomorphic on $\CC$, $\tau_s(T)b = b$} \}\,.
\]

\begin{prop}\label{prop:coboundaries}
For any $s\in\CC$, the map $\pc$ descends to an isomorphism 
\[
 \pc\colon\BFE_s^\omega(\CC) \to B^1_\Xi\bigl(\Gamma;\V s {\fxi;\exc;\aj}(\proj\RR)\bigr)_\sic^\van\,.
\]
Moreover, for $\Rea s\in (0,1)$,
\begin{equation}\label{cobound_vanish}
 B^1_\Xi\bigl(\Gamma; \V s {\fxi;\exc,\smp;\aj}(\proj\RR)\bigr)_\sic^\van = B^1_\Xi\bigl(\Gamma;\V s {\fxi;\exc,\infty;\aj}(\proj\RR)\bigr)_\sic^\van = \{0\}\,.
\end{equation}
\end{prop}

\begin{proof}
Let $c\in B^1_\Xi\bigl(\Gamma;\V s {\fxi;\exc;\aj}(\proj\RR)\bigr)_\sic^\van$ and suppose that 
\[ p\colon\Xi\to \V s {\fxi;\exc;\aj}(\proj\RR)\]
is a $\Gamma$-equivariant potential of~$c$. Then $p$ is determined by~$p(\infty)$ and~$p(1)$. Since $T$ fixes~$\infty$, $p(\infty)$ is $\tau_s(T)$-invariant, hence $\lambda$-periodic. Further, $p(\infty)$ satisfies the condition~\texc at~$\infty$. Therefore $p(\infty)$ is holomorphic on a domain of the form 
\[
 \{ \Rea z > x_0\} \cup \{ \Rea z < x_1\}
\]
for some~$x_0,x_1\in \RR$. The $\lambda$-periodicity then implies that $p(\infty)$ is indeed holomorphic on all of~$\CC$. 

We claim that for all~$g\in\Gamma$, 
\begin{equation}\label{p1_vanish}
 p(g1) = 0\,.
\end{equation}
To that end we note that the conditions~\tvan and~\tsic imply that all singularities of
\begin{equation}\label{sing_diff}
 c(1,\lambda-1) = p(1) - p(\lambda-1) 
\end{equation}
are contained in~$\{1,\lambda-1\}$, and that the map~$c(1,\lambda-1)$ vanishes on~$(\lambda-1,1)_c$. The $\Gamma$-equivariance implies
\begin{equation}\label{equivariance}
 p(\lambda-1) = \tau_s(TS)p(1)\,.
\end{equation}
We study the set of singularities of~$p(1)$. 

Let $r\in \Xi$ such that $(TS)^nr\not=1$ for all~$n\in\NN_0$. If $r$ were a singularity of~$p(1)$, then~\eqref{equivariance} would imply that $TSr$ is a singularity of~$p(\lambda-1)$. By~\eqref{sing_diff} (note that $TS r\notin \{1,\lambda-1\}$), $TSr$ would be a singularity of~$p(1)$. Thus, by induction, for all~$n\in\NN$, $(TS)^nr$ would be a singularity of~$p(1)$. Since $p(1)$ has only finitely many singularities, $r$ needs to be a fixed point of~$TS$. However, none of these fixed points is contained in~$\Xi$. In turn, $r$ is not a singularity of~$p(1)$.

Suppose now that $r\in\Xi$ and $n_0\in\NN$ such that $(TS)^{n_0}r=1$. If $r$ were a singularity of~$p(1)$, then, since $r\notin\{1,\lambda-1\}$, \eqref{sing_diff} would imply that $r$ is a singularity of~$p(\lambda-1)$. Then~\eqref{equivariance} would show that $(TS)^{-1}r$ is a singularity of~$p(1)$. By induction, for any~$n\in\NN$, $(TS)^{-n}r$ would be a singularity of~$p(1)$. As above we find that this is impossible. In turn
\[
 \bsing p(1) \subseteq \{1\}\,.
\]
Combining the vanishing of~\eqref{sing_diff} on~$(\lambda-1,1)_c$ with~\eqref{equivariance} we find
\[
 p(1) = \tau_s(TS)p(1) \quad\text{on $(\lambda-1,1)_c$\,.}
\]
Since $\fixTS^+$ and $\fixTS^-$ are the attracting and repelling fixed points of~$TS$, respectively, for any interval~$I$ that is compactly contained in~$(\fixTS^-,\fixTS^+)_c$ there exists $n_I\in\NN_0$ such that $(TS)^{n_I}I\subseteq (\lambda-1,\fixTS^+)_c$. We define a map 
\[
a\colon \proj\RR\smallsetminus\{\fixTS^-\} \to \V s {\fxi;\exc;\aj}(\proj\RR)
\]
by setting for any such interval $I$,
\[
 a\coloneqq   \tau_s((TS)^{n_I})p(1) \quad\text{on~$I$\,,}
\]
and
\[
 a\coloneqq  p(1) \quad\text{on~$(\lambda-1,\fixTS^-)_c$\,.}
\]
Then $a$ is well-defined and a $\tau_s(TS)$-invariant element of~$\V s \omega[\fixTS^-]$. By~\cite[Proposition~4.1]{BLZm} (after conjugating $\fixTS^+$ to~$0$, and $\fixTS^-$ to $\infty$), $a=0$. Thus, $p(1) = 0$ on~$(\lambda-1,\fixTS^-)_c$. 

Working analogously with the map~$(TS)^{-1}$ instead of~$TS$ (in which case $\fixTS^-$ becomes the attracting fixed point) we find that $p(1) = 0$ on~$(\fixTS^+,1)_c$. Thus, $p(1) = 0$ on~$\proj\RR$. By $\Gamma$-equivariance, \eqref{p1_vanish} follows.

Let $f=(f_1,f_2)\coloneqq  \pc^{-1}(c)$. Thus
\begin{align*}
 f_1 &= c(-1,\infty)\vert_{(-1,\infty)_c} = -p(\infty)\vert_{(-1,\infty)_c}
 \\
 f_2 & = -c(1,\infty)\vert_{(\infty,1)_c} = p(\infty)\vert_{(\infty,1)_c\,.}
\end{align*}
Hence, $f\in \BFE_s^\omega(\CC)$.

Vice versa, let $f = (-b,b)\in \BFE_s^\omega(\CC)$ and let $c\coloneqq  \pc(f)$. By Proposition~\ref{prop:iso_exc}, the cocycle~$c$ is in $Z^1_\Xi\bigl(\Gamma;\V s {\fxi;\exc;\aj}(\proj\RR)\bigr)_\sic^\van$.  We define a $\Gamma$-equivariant map 
\[
p\colon\Xi\to \V s {\fxi;\exc;\aj}(\proj\RR)
\]
by 
\begin{align*}
 p(\infty) &\coloneqq  -b \quad\text{on $\proj\RR\smallsetminus\{\infty\}$\,,}
 \\
 p(1) &\coloneqq  0 \quad\text{on $\proj\RR$\,.}
\end{align*}
By Proposition~\ref{prop:cocycle}, 
\begin{align*}
 c(1,\infty) &= b \quad\text{on $\proj\RR\smallsetminus\{1,\infty\}$\,,}
 \\
 c(-1,\infty) & = b \quad\text{on $\proj\RR\smallsetminus\{-1,\infty\}$\,.}
\end{align*}
Obviously, $p$ is a $\Gamma$-equivariant potential of $c$, and hence $c$ is a coboundary. This completes the proof of the first statement.

It remains to show the statement in~\eqref{cobound_vanish} on vanishing coboundary spaces. To that end suppose that $\Rea s\in (0,1)$ and suppose that $p(\infty) \in \V s {\fxi;\exc,\smp;\aj}(\proj\RR)$. We claim that $p(\infty)=0$. With~\eqref{p1_vanish} it then follows that $p=0$ and hence the cocycle~$c$ is $0$.

To simplify notation we set $\varphi\coloneqq p(\infty)$. The condition~\tsmp (`simple singularity') yields that $\varphi$ has the two-sided asymptotic expansion
\[
 \varphi(t) \sim |t|^{-2s} \sum_{k\geq -1} a_k t^{-k} \qquad\text{as~$t\to\infty$}
\]
for suitable coefficients~$a_k\in\CC$, $k\geq -1$. From the $\lambda$-periodicity of~$\varphi$ it now follows that $a_k=0$ for all~$k\geq -1$, and hence $\varphi=0$. 
This completes the proof.
\end{proof}

\section{Proof of Theorem E}\label{sec:proof_thm-FE-coh}
\markright{21. PROOF OF THEOREM E}

For the proof of Theorem~\ref{thm-FE-coh} (from p.~\pageref{thm-FE-coh}) we note that Propositions~\ref{prop:iso_exc} and \ref{prop:EF_TO_fast} imply that the map~$\pc$ induces the linear maps 
\begin{align*}
 \FE_s^\om(\CC) & \to H^1_\Xi\bigl(\Gamma;\V s {\fxi;\exc;\aj}(\proj\RR)\bigr)^\van_\sic 
 \\
 \FE_s^{\om,1}(\CC) & \to H^1_\Xi\bigl(\Gamma;\V s {\fxi;\exc,\smp;\aj}(\proj\RR)\bigr)^\van_\sic 
 \\
 \FE_s^{\om,0}(\CC) & \to H^1_\Xi\bigl(\Gamma;\V s {\fxi;\exc,\infty;\aj}(\proj\RR)\bigr)^\van_\sic\,,
\end{align*}
for $s$ being in the range as stated for each case. By Proposition~\ref{prop:coboundaries}, the kernel of the first map is~$\BFE_s^\om(\CC)$, the kernels of the other two maps are~$\{0\}$. This completes the proof. \qed

%% file: Mem-BP-partV.tex

\setcounter{sectold}{\arabic{section}}
\markboth{V. PROOFS AND RECAPITULATION}{V. PROOFS AND RECAPITULATION}

\chapter{Proofs of Theorems A and B, and a recapitulation}\label{part:proof_mainthms}

\setcounter{section}{\arabic{sectold}}
\renewcommand\theequation{\Roman{chapter}.\arabic{equation}}
\setcounter{equation}{0}

Combining the results from Chapters~\ref{part:autom} and \ref{part:TO} now leads immediately to Theorems~\ref{thmA_new} and \ref{thmB_new} (as stated in Section~\ref{sec:intro}, the Introduction). Before we will provide more details of the proofs, we briefly recapitulate the major isomorphisms from Chapters~\ref{part:autom} and \ref{part:TO}.

In Chapter~\ref{part:TO} we showed that each element~$f=(f_1,f_2)\in \FE_s^\om(\CC)$ determines a unique cocycle~$c=c_f\in Z^1_\Xi\bigl(\Gm;\V s {\fxi;\exc;\aj}(\proj\RR)\bigr)_\sic^\van$ as soon as we require
\[
 c(1,\infty)\vert_{(\infty,1)_c} = -f_2 \quad\text{and}\quad c(-1,\infty)\vert_{(-1,\infty)_c} = f_1\,.
\]
The resulting map~$\pc\colon \FE_s^\om(\CC)\to Z^1_\Xi\bigl(\Gm;\V s {\fxi;\exc;\aj}(\proj\RR)\bigr)_\sic^\van$ descends to isomorphisms 
\[
 \FE_s^\om(\CC)/\BFE_s^\om(\CC) \to H^1_\Xi\bigl(\Gm;\V s {\fxi;\exc;\aj}(\proj\RR)\bigr)^\van_\sic 
\]
for all $s\in\CC$,
\[
 \FE_s^{\om,1}(\CC) \to H^1_\Xi\bigl(\Gm;\V s {\fxi;\exc,\smp;\aj}(\proj\RR)\bigr)^\van_\sic 
\]
if $\Rea s\in (0,1)$, $s\not=\frac12$, and
\[
 \FE_s^{\om,0}(\CC) \to H^1_\Xi\bigl(\Gm;\V s {\fxi;\exc,\infty;\aj}(\proj\RR)\bigr)^\van_\sic 
\]
if $\Rea s\in (0,1)$. See Theorem~\ref{thm-FE-coh}, p.~\pageref{thm-FE-coh}.

In Chapter~\ref{part:autom} we showed how to construct 
a funnel form~$u_c\in \A_s$
from a cocycle~$c\in Z^1_\Xi\bigl(\Gm;\V s {\fxi;\exc;\aj}(\proj\RR)\bigr)_\sic^\van$. This construction lead to isomorphisms 
\[
 H_\Xi^1\bigl(\Gm;\V s {\fxi;\exc;\aj}(\proj\RR)\bigr)_\sic^\van \to \A_s
\]
and 
\[
 H_\Xi^1\bigl(\Gm;\V s {\fxi;\exc,\smp;\aj}(\proj\RR)\bigr)_\sic^\van \to \A_s^1
\]
for all~$s\in\CC$ with $\Rea s\in (0,1)$, $s\not=\tfrac12$, and 
\[
 H_\Xi^1\bigl(\Gm;\V s {\fxi;\exc,\infty;\aj}(\proj\RR)\bigr)_\sic^\van \to A_s^0
\]
for any~$s\in\CC$ with $\Rea s\in (0,1)$. See Theorem~\ref{thm:cohominter} (p.~\pageref{thm:cohominter}) and Section~\ref{sec:proof_cohominter}, where we also provided a brief survey of this construction. 

Let us suppose that $c=c_f\in Z^1_\Xi\bigl(\Gm;\V s {\fxi;\exc;\aj}(\proj\RR)\bigr)_\sic^\van$ is the cocycle associated to the period function~$f=(f_1,f_2)\in \FE_s^\om(\CC)$, and let $u_f=u_c \in \A_s$ be the funnel form associated to~$c_f$. In very rough terms, to find the value of~$u_f$ at a point~$z\in\uhp$ we proceed as follows: 
We extend the functions~$f_1$ and~$f_2$ to $C^2$-maps~$\tilde f_1$ and~$\tilde f_2$ on~$\uhp$ that are Laplace eigenfunctions in a (small) neighborhood of~$(-1,\infty)_c$ and~$(\infty,1)_c$, respectively. Then we define $u_f(z)$ as an average of the values of~$\tilde f_1$ and~$\tilde f_2$ at a certain (well-specified) set of points in~$\uhp$. 

As the isomorphisms from above show, the map~$\FE_s^\om(\CC)\to \A_s$, $f\mapsto u_f$, is not invertible; its kernel is~$\BFE_s^\om(\CC)$. (Restricted to~$\FE_s^{\om,1}(\CC)$ or even to~$\FE_s^{\om,0}(\CC)$, this map becomes invertible.) Nevertheless, the inverse map 
\[
 \A_s \to \FE_s^\om(\CC)/\BFE_s^\om(\CC)
\] 
can be stated rather easily. For $u\in \A_s$ an element~$f=(f_1,f_2)\in \FE_s^\om(\CC)$ with $u=u_f$ is given by   
\begin{equation}\label{eq:def_f1}
 f_1(t) = \int_{-1}^\infty \{ u_f, R(t;\cdot)^s\} \quad\text{for $t\in (-1,\infty)_c$}
\end{equation}
and 
\begin{equation}\label{eq:def_f2}
 f_2(t) = -\int_1^\infty \{ u_f, R(t;\cdot)^s \}\quad\text{for $t\in (\infty,1)_c$\,,}
\end{equation}
more precisely, by regularizations of these integrals at~$\infty$. In both cases, the integral is performed along any path in~$\uhp$ with endpoints in~$\proj\RR$. For the case that $u_f$ is a cuspidal funnel form, these integrals converge (and hence no regularization is needed). In Section~\ref{sec:heuristic} we gave a geometric interpretation of the assignments~\eqref{eq:def_f1} and \eqref{eq:def_f2}.

\begin{proof}[Proof of Theorems~\ref{thmA_new} and \ref{thmB_new}]
For $s\in\CC$, $\Rea s\in (0,1)$, $s\not=\frac12$ we set
\[
A_s \coloneqq \ucoc_s\circ\pc\,,
\]
where $\pc$ is given by~\eqref{map_pc_first} and $\ucoc_s$ is defined in~\eqref{def_ucoc}. Proposition~\ref{prop:iso_exc} and Theorem~\ref{thm:cohominter} (p.~\pageref{thm:cohominter}; more precisely, the discussion in Section~\ref{sec:proof_cohominter}) show that $A_s$ constitutes a surjective linear map
\[
 A_s \colon \FE_s^\om(\CC) \to \A_s\,.
\]
The remaining statements from Theorems~\ref{thmA_new} and \ref{thmB_new} now follow immediately from Theorem~\ref{thm-FE-coh} (p.~\pageref{thm-FE-coh}) and Theorem~\ref{thm:cohominter} (p.~\pageref{thm:cohominter}).
\end{proof}

\renewcommand\theequation{\arabic{section}.\arabic{equation}}

%% file: Mem-BP-partVI.tex

\setcounter{sectold}{\arabic{section}}
\markboth{PART VI. PARITY}{PART VI. PARITY}

\chapter{Parity}\label{part:parity}
\setcounter{section}{\arabic{sectold}}

In the previous chapters we considered the Hecke triangle group~$\Gm$ as a subgroup of~$\PSL_2(\RR)$. Since $\PSL_2(\RR)$ embeds canonically into \index[symbols]{P@$\PGL_2(\RR)$}
\[
 \PGL_2(\RR) \coloneqq \GL_2(\RR)/(\RR^\times\one)\,,
\]
the projective group of~$\GL_2(\RR)$, we may consider $\Gm$ as a subgroup of~$\PGL_2(\RR)$ as well. Within~$\PGL_2(\RR)$, the Hecke triangle group~$\Gm$ is normalized by \index[symbols]{J@$J=\textbmat{-1}001$}
\[
 J=\bmat{-1}001\,,
\]
which is the element in~$\PGL_2(\RR)$ that is represented by~$\textmat{-1}{0}{0}{1}\in\GL_2(\RR)$. 

The action of~$J$ on~$\Gm$ by conjugation may be seen as an outer automorphism or exterior symmetry of~$\Gm$, and we may use it to define---in a natural way---an involution on each of the spaces of funnel forms, cocycles and period functions. Each of these involutions can be understood as a realization of the action of~$J$, and hence as an exterior symmetry, on the particular space, which then decomposes into its subspaces of elements that are invariant or anti-invariant under this action of~$J$. In other words, these involutions yield a natural notion of parity and allow us to split the space of funnel forms into even and odd funnel forms, the cohomology spaces into spaces of even and odd cocycle classes, and the space of period functions into even and odd period functions. These subspaces shall be denoted by adding a~`$+$' (for `even') or a~`$-$' (for `odd') to the symbol of the full space. For precise definitions we refer to Section~\ref{sec:def_oddeven}.

In this section we will prove Theorem~\ref{thmC} (from p.~\pageref{thmC}). Thus we will show that the isomorphisms in Theorem~\ref{thm:cohominter} (as given in Section~\ref{sec:proof_cohominter}) and in Theorem~\ref{thm-FE-coh} (p.~\pageref{thm-FE-coh}) are anti-equivariant with respect to these involutions, and hence the isomorphisms in Theorem~\ref{thmA_new} (p.~\pageref{thmA_new}) and Theorem~\ref{thmB_new} (p.~\pageref{thmB_new}) descend to isomorphisms between the subspaces of funnel forms (resonant funnel forms, or cuspidal funnel forms) and period functions of same parity. 

\setcounter{mainthmprim}{2}

\begin{mainthmprim}\label{thm:main_parity}
The isomorphisms from Theorem~\ref{thm-alcoh} (the map~$\ucoc_s$ defined in Section~\ref{sec:proof_cohominter}) and from Theorem~\ref{thm-FE-coh} (induced by the map~$\pc$) descend to the following isomorphisms:
\begin{enumerate}[{\rm (i)}]
\item For any~$s\in\CC$, $\Rea s\in (0,1)$, $s\not=\frac12:$
\[
 \FE_s^{\om,\pm}(\CC)/\BFE_s^{\om,\pm}(\CC) \stackrel{\sim}{\longrightarrow} H_\Xi^{1,\mp}\bigl(\Gm; \V s {\fxi;\exc;\aj}(\proj\RR)\bigr)_\sic^\van \stackrel{\sim}{\longrightarrow} \A_s^{\pm}\,.
\]
\item For any~$s\in\CC$, $\Rea s\in (0,1)$, $s\not=\frac12:$
\[
 \FE_s^{\om,1,\pm}(\CC) \stackrel{\sim}{\longrightarrow} H_\Xi^{1,\mp}\bigl(\Gm; \V s {\fxi;\exc,\smp;\aj}(\proj\RR)\bigr)_\sic^\van \stackrel{\sim}{\longrightarrow} \A_s^{1,\pm}\,.
\]
\item For any~$s\in\CC$, $\Rea s\in (0,1):$
\[
 \FE_s^{\om,0,\pm}(\CC) \stackrel{\sim}{\longrightarrow} H_\Xi^{1,\mp}\bigl(\Gm; \V s {\fxi;\exc,\infty;\aj}(\proj\RR)\bigr)_\sic^\van \stackrel{\sim}{\longrightarrow} \A_s^{0,\pm}\,.
\]
\end{enumerate}
\end{mainthmprim}

Moreover we will see (Remark~\ref{rmk:equalspaces}) that the additional rigidity introduced by the parity yields 
\[
  \FE_s^{\om,1,-}(\CC) = \FE_s^{\om,0,-}(\CC) 
\]
and
\[
 H_\Xi^{1,+}\bigl(\Gm; \V s {\fxi;\exc,\smp;\aj}(\proj\RR) \bigr)^\van_\sic = H_\Xi^{1,+}\bigl(\Gm; \V s {\fxi;\exc,\infty;\aj}(\proj\RR)\bigr)^\van_\sic
\]
for~$s\in\CC$, $\Rea s\in (0,1)$, $s\not=\frac12$.

Alternatively to considering odd period functions as period functions of~$\Gm$ with odd parity under~$J$ we may understand these as genuine period functions for the triangle group~$\tGm$ generated by~$\Gm$ and~$J$ in~$\PGL_2(\RR)$, and analogously for the other spaces. Parity properties under~$J$ of objects defined for~$\Gm$ correspond then to conditions for associated objects defined for~$\tGm$ at the boundary on the space~$\tGm\backslash\uhp$. This point of view was taken in~\cite{Pohl_spectral_hecke} (for cofinite Hecke triangle groups), in~\cite{Pohl_hecke_infinite} (for infinite Hecke triangle groups) and in~\cite{Pohl_representation, Adam_Pohl} (for all Hecke triangle groups). From this point of view, Theorem~\ref{thm:main_parity} is the analogue of Theorems~\ref{thmA_new} and~\ref{thmB_new} for the group~$\tGm$ instead of~$\Gm$. In our treatment we will not rely on this point of view but will briefly indicate, in Section~\ref{sec:def_oddeven}, how the objects necessary for it can be defined as intrinsic objects for the triangle group~$\tGm$ (instead as objects for~$\Gamma$ satisfying an additional invariance).

\section{The triangle group in the projective general linear group}\label{sec:PGL}
\markright{22. THE TRIANGLE GROUP IN PGL(2,R)}

\subsection{Two actions of the projective general linear group}

As mentioned above, we consider $\PSL_2(\RR)$ as a subgroup of~$\PGL_2(\RR)$. As for $\PSL_2(\RR)$, we denote the equivalence class of~$\textmat{a}{b}{c}{d}\in\nobreak\GL_2(\RR)$ in~$\PGL_2(\RR)$ by
\[
 \bmat{a}{b}{c}{d}\,.
\]
For any such element of~$\PGL_2(\RR)$ we will use throughout only representatives with~$ad-bc=\pm 1$.

We extend the action in~\eqref{fraclinaction} of~$\PSL_2(\RR)$ on~$\proj\CC$ in two different ways to an action of~$\PGL_2(\RR)$. The first, denoted by~$g \mapsto g z$, is given by fractional linear transformations, applying the formula~\eqref{fraclinaction} to elements of~$\PGL_2(\RR)$ as well. Thus,
\[
 gz = \frac{az+b}{cz+d}
\]
for any~$g=\textbmat{a}{b}{c}{d} \in \PGL_2(\RR)$ (and $\infty = \frac10$). Under this action, elements of~$\PGL_2(\RR)$ with negative determinant interchange~$\uhp$ and~$\lhp$, and preserve the space~$\proj\RR$. We will use this action whenever we consider $\PGL_2(\CC)$ (or a subgroup) to act on a space of functions defined on (a subset of)~$\proj\CC$.

One special instance is the principal series representations of~$\PGL_2(\RR)$ with spectral parameter~$s\in \CC$. This is the action of the group~$\PGL_2(\RR)$ on the space of functions~$f\colon\proj\RR\to\CC$, given by 
\begin{equation}\label{eq:princ_ser_tilde}
\left(\tau_{s}(g^{-1}) f\right)(t) \ceqq  |c
t+d|^{-2s} \, f\left( gt\right)
\qquad\text{with }g=\bmat abcd \,.
\end{equation}

The second action of~$\PGL_2(\RR)$ on~$\proj\CC$, denoted~$g\mapsto g\cdot z$, is the extension of~\eqref{fraclinaction} to~$\PGL_2(\RR)$ which identifies $\PGL_2(\RR)$ with the group of Riemannian isometries of~$\uhp$. This action is given by 
\begin{equation}\label{gdz2}
\bmat abcd \cdot z
\ceqq 
\begin{cases}
\frac{az+b}{cz+d}&\text{ if $ad-bc>0$}\,,
\\[1mm] 
\frac{a \bar z+b}{c\bar z+d}&\text{ if $ad-bc<0$}\,
\end{cases}
\end{equation}
for~$g=\bmat abcd\in \PGL_2(\RR)$ and $z\in \proj\CC$ (and $\infty = \frac10$). It preserves~$\uhp$, $\lhp$, and~$\proj\RR$. We will use this action whenever we consider $\PGL_2(\CC)$ (or a subgroup) as a group of Riemannian isometries.

\subsection{The triangle group}

The element
\[
 J = \bmat {-1}001
\]
normalizes the Hecke triangle group~$\Gm$. The subgroup of~$\PGL_2(\RR)$ generated by~$\Gm$ and~$J$ is the \emph{triangle group}\index[defs]{triangle group}\index[symbols]{Gaam@$\tGm$}
\[ 
\tGm=\Gm\sqcup J \Gm\,.
\]
Using~\eqref{gdz2} to define an action of~$\tGm$ on~$\uhp$, the group~$\tGm$ becomes a discrete group of Riemannian isometries. A fundamental domain for~$\tGm\backslash\uhp$ is indicated in Figure~\ref{fig-fd2}. \index[symbols]{F@$\tilde\fd$}
\begin{figure}
\centering
\includegraphics[height=4cm]{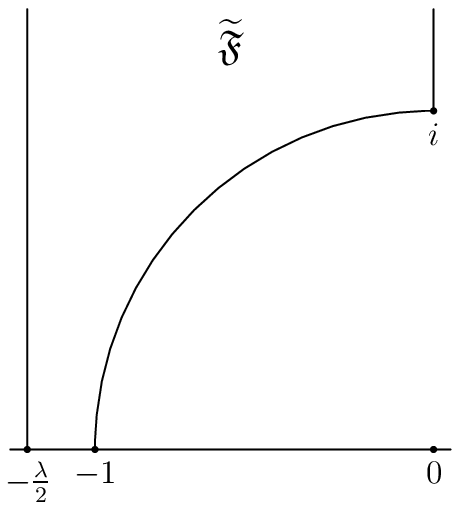}
\caption{Fundamental domain~$\tilde\fd$ for~$\tGm\backslash \uhp$.}\label{fig-fd2}
\end{figure}
The group $\tGm$ is generated by the reflections on the sides of this fundamental domain (which justifies to call it a triangle group). The sets of ordinary points and of cuspidal points of~$\tGm$ are identical to those of~$\Gm$. In other words, the set~$\Om(\Gm)$ of ordinary points of~$\Gm$, the limit set~$\Lambda(\Gamma)$, the set of cuspidal points, and the set of all funnel intervals of~$\Gm$ are invariant under the action of~$J$.

\section{Odd and even funnel forms, cocycles, and period functions}\label{sec:def_oddeven}
\markright{23. ODD AND EVEN}

In this section we will define the involutions on funnel forms, cocycles and period functions that are induced by the action of~$J$. We will denote the involution on funnel forms by~$\mc J^F$ with `F' indicating `\textbf{f}unnel forms'. Further we will denote the involution on cocycles and cocycle classes by~$\mc J_s^H$, where `H' stands for `co\textbf{h}omology' and~$s$ is the spectral parameter from the modules in which the cocycles have their values, and we will denote the involution on period functions by~$\mc J_s^E$, where `E' stands for `\textbf{e}igenfunctions of transfer operators' (period functions) and $s$ is the parameter of the transfer operator.

These involutions induce a notion of parity on each of the considered spaces and allow us to distinguish between odd and even elements, as defined below. For funnel forms and cocycles it is straightforward how even and odd funnel forms and cocycles can be defined as genuine objects of the triangle group~$\tGm$, and any such discussion will be omitted. For odd and even period functions such an alternative definition is not obvious, for which reason we will briefly indicate the natural relation of their definition to the billiard flow on~$\tGm\backslash\uhp$.

\subsection{Odd and even funnel forms}

For any element in the function space \index[symbols]{F@$\Fct(\uhp;\CC)$}
\[
\Fct(\uhp;\CC)\coloneqq \{\,u\colon\uhp\to\CC\,\}
\]
we set \index[symbols]{$u\circ J$}
\begin{equation}
 (u\circ J)(z) \coloneqq u(J\cdot z) = u(-\overline z)\,.
\end{equation}
We define the map~$\mc J^F\colon \Fct(\uhp;\CC)\to \Fct(\uhp;\CC)$ by \index[symbols]{J@$\mc J^F$}
\begin{equation}
 \mc J^F\colon u\mapsto u\circ J\,.
\end{equation}

A straightforward calculation shows that the map~$\mc J^F$ stabilizes any eigenspace of the Laplace--Beltrami operator~$\Delta$. In particular, for any~$s\in\CC$, the space~$\E_s$ of \mbox{$\Gamma$-}invariant Laplace eigenfunctions with spectral parameter~$s$ is invariant under~$\mc J^F$. Since the map~$\mc J^F$ is an involution, the space~$\E_s$ decomposes into the direct sum of the space \index[symbols]{Eab@$\E_s^\pm$}
\[
 \E_s^+ \coloneqq \{\, u\in\E_s \setmid \forall\,z\in\uhp\colon u(z) = u(-\overline{z})\,\}
\]
of $\mc J^F$-eigenfunctions for the eigenvalue~$1$, and the space 
\[
 \E_s^- \coloneqq \{\, u\in\E_s \setmid \forall\,z\in\uhp\colon u(z) = -u(-\overline{z})\,\}
\]
of $\mc J^F$-eigenfunctions for the eigenvalue~$-1$. Thus,
\[
 \E_s = \E_s^+ \oplus \E_s^-\,.
\]
As usual, we call the elements of~$\E_s^+$ the \emph{even} Laplace eigenfunctions with spectral parameter~$s$, and the elements of~$\E_s^-$ the \emph{odd} ones.

The condition of $s$-analytic boundary behavior as well as each of the conditions in~\eqref{resgrwth} and \eqref{expdecay} on asymptotic growth behavior are preserved under the map~$\mc J^F$. Therefore, each of the spaces of funnel forms, resonant funnel forms and cuspidal funnel forms decomposes into a direct sum of even and odd funnel forms of the appropriate type. More precisely,
\[
 \A_s = \A_s^+ \oplus \A_s^-\,,\qquad \A_s^1 = \A_s^{1,+} \oplus \A_s^{1,-}\,,\qquad \A_s^0 = \A_s^{0,+} \oplus \A_s^{0,-}\,,
\]
where \index[symbols]{AAB@$\A_s^+,\A_s^-$}\index[symbols]{AAD@$\A_s^{1,+},\A_s^{1,-}$}\index[symbols]{AAE@$\A_s^{0,+},\A_s^{0,-}$} \index[defs]{even funnel forms}\index[defs]{odd funnel forms}\index[defs]{funnel form!even}\index[defs]{funnel form!odd}\index[defs]{resonant funnel form!even}\index[defs]{resonant funnel form!odd}\index[defs]{cuspidal funnel form!even}\index[defs]{cuspidal funnel form!odd}
\begin{align*}
 \A_s^+ &\coloneqq \A_s \cap \E_s^+ && \text{is the space of \emph{even} funnel forms,}
 \\
 \A_s^- &\coloneqq \A_s \cap \E_s^- && \text{is the space of \emph{odd} funnel forms,}
 \\
 \A_s^{1,+} & \coloneqq \A_s^1 \cap \E_s^+ && \text{is the space of \emph{even} resonant funnel forms,}
 \\
 \A_s^{1,-} &\coloneqq \A_s^1 \cap \E_s^- && \text{is the space of \emph{odd} resonant funnel forms,}
 \\
 \A_s^{0,+} & \coloneqq \A_s^0 \cap \E_s^+ && \text{is the space of \emph{even} cuspidal funnel forms, and }
 \\
 \A_s^{0,-} &\coloneqq \A_s^0 \cap E_s^- && \text{is the space of \emph{odd} cuspidal funnel forms,}
\end{align*}
in each case for the spectral parameter~$s$.

\pagebreak[4]

\subsection{Odd and even cocycles}

\medskip

\subsubsection{Cohomology on an invariant set}\label{sec:extended_cohom} 
We extend the definitions of Section~\ref{sec:setcohom} to~$\tGm$ by defining the action of~$J$ on any element of~$\CC[\Xi^2]$ by \index[symbols]{$(\cdot,\cdot)\vert J$}
\begin{equation}\label{eq:J_on_coc}
 (\xi,\eta)\vert J \coloneqq (J\cdot\xi,J\cdot\eta) = (-\xi,-\eta)\,.
\end{equation}
For any subspace~$M$ of~$\V s \fxi$ that is invariant under the action of~$\tGm$ by~$\tau_s$ (recall the extension of~$\tau_s$ to all of~$\tGm$ from~\eqref{eq:princ_ser_tilde}), the map in~\eqref{eq:J_on_coc} induces an involution on the space~$C^1_\Xi(\Gm;M)$ of $1$-cochains by \index[symbols]{J@$\mc J_s^H$}
\begin{equation}
 \mc J_s^H c \coloneqq \tau_s(J)(c\vert J)\,,
\end{equation}
where
\[
 (c\vert J)(\xi,\eta) \coloneqq c(-\xi,-\eta)\,.
\]
Thus
\begin{equation}
 \bigl(\mc J_s^H c\bigr)(\xi,\eta)(t) = c(-\xi,-\eta)(-t)
\end{equation}
for all~$\xi,\eta\in\Xi$, $t\in\proj\RR$. The map~$\mc J_s^H$ leaves invariant the spaces of cocycles and coboundaries, induces involutions on these spaces and leads to a splitting in $\pm 1$-eigenspaces of the cohomology space~$H^1_\Xi(\Gm;M)$. We let \index[symbols]{H@$H_\Xi^{1,\pm}(\Gamma;M)$}
\begin{align*}
 H_\Xi^{1,\pm}(\Gamma;M) & \coloneqq \bigl\{\, [c] \in H_\Xi^1(\Gamma;M) \setmid \mc J_s^H c = \pm c\,\bigr\}
 \\
 & \ = \bigl\{\, [c]\in H_\Xi^1(\Gamma;M) \setmid \tau_s(J)c = \pm c\vert J\,\bigr\}
\end{align*}
denote the space of \emph{even} (for~`$+$') and \emph{odd} (for~`$-$') cocycle classes. 
\index[defs]{even cocycle}\index[defs]{odd cocycle}\index[defs]{cocycle!even}\index[defs]{cocycle!odd} 
Hence
\begin{equation}\label{eq:splitting_cohom}
H^1_\Xi(\Gm;M) = H_\Xi^{1,+}(\Gm;M) \oplus H_\Xi^{1,-}(\Gm;M)\,.
\end{equation}

For~$M$ we will use 
the vector spaces
\[ \V s \om(\proj\RR)\,,\; \V s {\fxi;\exc;\aj}(\proj\RR)\,,\;
\V s {\fxi;\exc,\smp;\aj}(\proj\RR)\,,\text{ and }
\V s {\fxi;\exc,\infty;\aj}(\proj\RR)\,,
\] 
all of which are indeed $\tGm$-modules under the action of~$\tau_s$.

\begin{lem}\label{lem:sic_van_J}
Let $s\in\CC$ and $M$ be a submodule of~$\V s \fxi$ that is $\tGm$-invariant.
\begin{enumerate}[{\rm (i)}]
\item If $c\in Z^1_\Xi(\Gm;M)$ satisfies the singularity condition~$\sic$, then so does~$\mc J_s^Hc$.
\item\label{svii} If $c\in Z^1_\Xi(\Gm;M)$ satisfies the vanishing condition~$\van$, then so does~$\mc J_s^Hc$.
\end{enumerate}
\end{lem}

\begin{proof}
\begin{enumerate}[{\rm (i)}]
\item Suppose that $c\in Z^1_\Xi(\Gm;M)$ satisfies the singularity condition~\tsic. Then 
\[
 \forall\, \xi,\eta\in \Xi\colon \bsing (c\vert J)(\xi,\eta) = \bsing c(-\xi,-\eta) \subseteq \{-\xi,-\eta\}
\]
and hence 
\[
 \forall\, \xi,\eta\in\Xi\colon \bsing \tau_s(J)(c\vert J)(\xi,\eta) \subseteq \{\xi,\eta\}\,.
\]
Thus, $\mc J_s^Hc$ satisfies the singularity condition.
\item Suppose that $c\in Z^1_\Xi(\Gm;M)$ satisfies the vanishing condition~\tvan. Thus
\[
 c(1,\lambda-1)\vert_{(\lambda-1,1)_c} = 0
\]
by definition (see~\eqref{def_tvan}), which is equivalent to
\[
 c(1-\lambda,-1) = \tau_s(T^{-1})c(1,\lambda-1) = 0 \qquad\text{on $(-1,1-\lambda)_c$\,}
\]
and hence to 
\[
 c(-1,1-\lambda) = 0 \qquad\text{on $(-1,1-\lambda)_c$\,.}
\]
This is equivalent to 
\[
 \tau_s(J)c(-1,1-\lambda) = 0 \qquad\text{on $(\lambda-1,1)_c$\,.}
\]
Finally, this is equivalent to
\[
 (\mc J_s^Hc)(1,\lambda-1) = \tau_s(J)(c\vert J)(1,\lambda-1) = 0 \qquad\text{on $(\lambda-1,1)_c$\,,}
\]
which proves that $\mc J_s^Hc$ satisfies the vanishing condition. \qedhere
\end{enumerate}
\end{proof}

Lemma~\ref{lem:sic_van_J} shows that the splitting~\eqref{eq:splitting_cohom} descends to a $\mc J_s^H$-invariant splitting
\[
 H^1_\Xi(\Gm;M)_\sic^\van = H_\Xi^{1,+}(\Gm;M)_\sic^\van \oplus H_\Xi^{1,-}(\Gm;M)_\sic^\van\,,
\]
where 
\[
 H_\Xi^{1,\pm}(\Gm;M)_\sic^\van = H_\Xi^{1,\pm}(\Gm;M) \cap H_\Xi^1(\Gm;M)_\sic^\van\,.
\]

\medskip

\subsubsection{Other cohomology spaces} 

In a way analogous to that in Section~\ref{sec:extended_cohom} we define an involution for the spaces of the tesselation cohomology and for the spaces of mixed cohomology. For this definition we may rely on the fact that every element in~$X_1^\tess$ (that is, every edge of the tesselation~$\tess$) is uniquely determined by its tail and head, and then can use the characterization of cycles as maps with domain~$X_0^\tess\times X_0^\tess$ (see~\eqref{char2}). Throughout we denote all instances of these involutions by~$\mc J_s^H$, independent of the cohomology spaces or cocycle spaces on which they act.

\subsection{Odd and even period functions}

For any~$s\in\CC$ we define the map \index[symbols]{J@$\mc J_s^E$}
\[
 \mc J_s^E \colon \V s \fxi\bigl( (-1,\infty)_c \bigr) \times \V s \fxi \bigl( (\infty,1)_c \bigr) \to \V s \fxi\bigl( (-1,\infty)_c \bigr) \times \V s \fxi \bigl( (\infty,1)_c \bigr)
\]
by
\begin{equation}
 \mc J_s^E \coloneqq \mat{0}{\tau_s(J)}{\tau_s(J)}{0} \colon (f_1,f_2) \mapsto (\tau_s(J)f_2, \tau_s(J)f_1)\,.
\end{equation}
Since $JSJ=S$ and $JTJ=T^{-1}$, the map~$\mc J_s^E$ commutes with the transfer operator~$\TO_s^\slow$. Further, it preserves the spaces~$C^\om(D_\RR)$ and $C^\om(D_\CC)$ and hence the spaces~$\FE_s^\om(\RR)$ and $\FE_s^\om(\CC)$ of real and complex period functions. Therefore, each of the spaces~$\FE_s^\om(\RR)$ and $\FE_s^\om(\CC)$ decomposes as the direct sum of the \mbox{$1$-}eigenspace and the $(-1)$-eigenspace of~$\mc J_s^E$. For $K\in\{\RR,\CC\}$ we let \index[symbols]{FE@$\FE_s^{\om,\pm}(K)$}
\[
 \FE_s^{\om,+}(K) \coloneqq \bigl\{\, f\in \FE_s^\om(K) \setmid \mc J_s^E f=f \,\bigr\}
\]
denote the space of \emph{even} period functions, \index[defs]{even period function}\index[defs]{period function!even} and 
\[
 \FE_s^{\om,-}(K) \coloneqq \bigl\{\, f\in \FE_s^\om(K) \setmid \mc J_s^E f=-f \,\bigr\}
\]
denote the space of \emph{odd} period functions. \index[defs]{odd period function}\index[defs]{period function!odd} Then 
\begin{equation}\label{decomp_pf}
 \FE_s^\om(K) = \FE_s^{\om,+}(K) \oplus \FE_s^{\om,-}(K)\,.
\end{equation}
Likewise, we let \index[symbols]{BFE@$\BFE_s^{\om,\pm}(K)$}
\[
 \BFE_s^{\om,\pm}(K) \coloneqq \bigl\{\, (-b,b) \setmid b\in C^\om(K),\ \tau_s(T)b=b,\ \tau_s(J)b=\mp b\,\bigr\}
\]
denote the space of \emph{even} (`$+$') and \emph{odd} (`$-$') boundary period functions. (Note that $\pm$ in the notation for~$\BFE_s^{\om,\pm}(K)$ changes to~$\mp$ in the action of~$\tau_s(J)$ on~$b$.) One easily checks that even boundary period functions are indeed even period functions, and that odd boundary period functions are odd period functions. Therefore, the decomposition~\eqref{decomp_pf} restricts to the decomposition
\[
 \BFE_s^\om(K) = \BFE_s^{\om,+}(K) \oplus \BFE_s^{\om,-}(K)\,.
\]
Also the spaces~$\FE_s^{\om,1}(\CC)$ and $\FE_s^{\om,0}(\CC)$ (for all values of~$s\in\CC$ for which they are defined; see Section~\ref{sec:def_cplxperiod})  are preserved by the involution~$\mc J_s^E$ and hence decompose into $(\pm1)$-eigenspaces of~$\mc J_s^E$ because $\mc J_s^E$ commutes also with the transfer operator~$\TO_s^\fast$ and preserves the identity $f_1(0)+f_2(0)=0$ for all~$f=(f_1,f_2)\in\FE_s^\om(\CC)$. We set \index[symbols]{FE@$\FE_s^{\om,1,\pm}(\CC)$}\index[symbols]{FE@$\FE_s^{\om,0,\pm}(\CC)$}
\begin{align*}
 \FE_s^{\om,1,\pm}(\CC) &\coloneqq \FE_s^{\om,1}(\CC) \cap \FE_s^{\om,\pm}(\CC) 
 \intertext{and} 
 \FE_s^{\om,0,\pm}(\CC) &\coloneqq \FE_s^{\om,0}(\CC)\cap \FE_s^{\om,\pm}(\CC)\,.
\end{align*}
We note that an element $f=(f_1,f_2)$ of $\FE_s^{\om,1,+}(\CC)$ is in~$\FE_s^{\om,0,+}(\CC)$ if and only if 
\[
 0 = f_1(0) + f_2(0) = f_1(0) + \tau_s(J)f_1(0) = 2 f_1(0)\,.
\]
Therefore
\[
 \FE_s^{\om,0,+}(\CC) = \bigl\{\, (f_1,f_2)\in\FE_s^{\om,1,+}(\CC) \setmid f_1(0)=0\,\bigr\}\,.
\]
For any~$f=(f_1,f_2)\in\FE_s^{\om,1,-}(\CC)$ we have 
\[
 f_1(0) + f_2(0) = f_1(0) - \tau_s(J)f_1(0)=0\,. 
\]
Thus, 
\begin{equation}\label{eq:samespace}
 \FE_s^{\om,0,-}(\CC) = \FE_s^{\om,1,-}(\CC)\,.
\end{equation}

We end this section with a brief indication how odd and even period functions of~$\Gm$ arise as genuine objects of the triangle group~$\tGm$. To that end we note that the operator
\[
  \mc P \coloneqq  \frac{1}{\sqrt{2}}\begin{pmatrix} 1 & \tau_s(J) \\
 \tau_s(J) & -1\end{pmatrix}
\] 
diagonalizes the transfer operator $\tro s$:
\[
 \mc P \TO_s^{\slow} \mc P^{-1} = \begin{pmatrix} \TO_s^{\slow,+} & \\ &
 \tau_s(J)\TO_s^{\slow,-}\tau_s(J)\end{pmatrix}\,,
\] 
where
\begin{equation}
 \label{def_TO_slow}
 \TO_s^{\slow,\pm} \coloneqq  \tau_s(T^{-1}S) + \tau_s(T^{-1}) \pm
 \tau_s(T^{-1}SJ)\,.
\end{equation}
Then  
\[
 \FE_s^{\om,\pm}(\RR) \cong \bigl\{\, f\in C^\om\big( (-1,\infty) \big) \setmid \TO_s^{\slow,\pm}f=f \,\bigr\}
\]
and
\[
 \FE_s^{\om,\pm}(\CC) \cong \bigl\{\, f\in C^\om\big( \CC\smallsetminus (-\infty,-1] \big) \setmid \TO_s^{\slow,\pm}f=f \,\bigr\}\,.
\]
The operators~$\TO_s^{\slow,\pm}$ arise as transfer operators of a certain discretization for the billiard flow on~$\tGm\backslash\uhp$. The two choices of the sign in~\eqref{def_TO_slow} are expected to encode different boundary conditions. For details we refer to~\cite{Pohl_hecke_infinite}.

\section{Isomorphisms with parity}\label{sec:proof_parity}
\markright{24. ISOMORPHISMS WITH PARITY}

In this section we will provide a proof of Theorem~\ref{thm:main_parity} (from p.~\pageref{thm:main_parity}). As mentioned above we will split in into two steps by first establishing the isomorphisms between period functions and cocycle classes of opposite parity in Proposition~\ref{prop:FE_cohom_pm}, and then showing the isomorphisms between funnel forms and cocycle classes of opposite parity in Proposition~\ref{prop:funnel_cohom_pm}.

\begin{prop}\label{prop:FE_cohom_pm}
The map~$\pc$ from Chapter~\ref{part:TO} induces the following isomorphisms:
\begin{enumerate}[{\rm (i)}]
\item For any~$s\in\CC$, 
\[
 \FE_s^{\om,\pm}(\CC)/\BFE_s^{\om,\pm}(\CC) \stackrel{\sim}{\longrightarrow} H_\Xi^{1,\mp}\bigl( \Gamma; \V s {\fxi;\exc;\aj}(\proj\RR) \bigr)_\sic^\van\,.
\]
\item For~$s\in\CC$, $\Rea s\in (0,1)$, $s\not= 1/2$,
\[
 \FE_s^{\om,1,\pm}(\CC) \stackrel{\sim}{\longrightarrow} H_\Xi^{1,\mp}\bigl( \Gm; \V s {\fxi;\exc,\smp;\aj}(\proj\RR)\bigr)_\sic^\van\,.
\]
\item For~$s\in\CC$, $\Rea s\in (0,1)$,
\[
 \FE_s^{\om,0,\pm}(\CC) \stackrel{\sim}{\longrightarrow} H_\Xi^{1,\mp}\bigl(\Gm; \V s {\fxi;\exc,\infty;\aj}(\proj\RR) \bigr)_\sic^\van\,.
\]
Each of these isomorphisms is a restriction of a corresponding isomorphism from Theorem~\ref{thm-FE-coh} (p.~\pageref{thm-FE-coh}).
\end{enumerate}
\end{prop}

\begin{proof}
The statements follow immediately from Theorem~\ref{thm-FE-coh} and the anti-equi\-var\-iance of~$\pc$ with respect to the action of~$J$, that is, 
\[
 \pc\circ \mc J_s^E = -\tau_s(J) \circ \pc\,. \qedhere
\]
\end{proof}

\begin{rmk}\label{rmk:equalspaces}
For each~$s\in\CC$, $\Rea s\in (0,1)$, $s\not=1/2$, we have
\[
 \FE_s^{\om,1,-}(\CC) = \FE_s^{\om,0,-}(\CC)
\]
by~\eqref{eq:samespace}. Thus, Proposition~\ref{prop:FE_cohom_pm} yields
\[
 H_\Xi^{1,+}\bigl(\Gm; \V s {\fxi;\exc,\smp;\aj}(\proj\RR)\bigr)_\sic^\van = H_\Xi^{1,+}\bigl(\Gm; \V s {\fxi;\exc,\infty;\aj}(\proj\RR) \bigr)_\sic^\van\,.
\]
\end{rmk}

For proving that the isomorphisms between funnel forms and cocycle classes descend to isomorphisms between spaces of opposite parity we recall from~\eqref{def_ucoc} the map 
\[
 \ucoc_s\colon H^1_\Xi\bigl(\Gm; \V s {\fxi;\exc;\aj}(\proj\RR) \bigr)_\sic^\van \to \A_s
\]
of which we showed (see for Theorem~\ref{thm:cohominter} and its proof on p.~\pageref{proof:cohominter}) that it is a linear bijection (for appropriate values of~$s$) and that it descends to isomorphisms between the spaces $H^1_\Xi\bigl(\Gm;\V s {\fxi;\exc,\smp;\aj}(\proj\RR) \bigr)_\sic^\van$ and $\A_s^1$ as well as between the spaces $H^1_\Xi\bigl(\Gm;\V s {\fxi;\exc,\infty;\aj}(\proj\RR)\bigr)_\sic^\van$ and $\A_s^0$. To show the anti-equivariance of~$\ucoc_s$ under the action of~$J$ we will work with the inverse map
\begin{equation}
\cocu_s \colon \A_s \to H_\Xi^1\bigl(\Gm;\V s {\fxi;\exc;\aj}(\proj\RR)\bigr)_\sic^\van
\end{equation}
from~\eqref{def_cocu} instead of directly with~$\ucoc_s$, which makes the proof slightly shorter. We recall that we denote the $J$-action on all arising cohomology spaces by~$\mc J_s^H$.

\begin{prop}\label{prop:funnel_cohom_pm}
The map~$\cocu_s$ from~\eqref{def_cocu} induces the following isomorphisms:
\begin{enumerate}[{\rm (i)}]
\item\label{casearb} For~$s\in\CC$, $\Rea s\in (0,1)$, $s\not=1/2$,
\[
 \A_s^\pm \stackrel{\sim}{\longrightarrow} H_\Xi^{1,\mp}\bigl(\Gm; \V s {\fxi;\exc;\aj}(\proj\RR)\bigr)_\sic^\van\,.
\]
\item For~$s\in\CC$, $\Rea s\in (0,1)$, $s\not=1/2$,
\[
 \A_s^{1,\pm} \stackrel{\sim}{\longrightarrow} H_\Xi^{1,\mp}\bigl(\Gm; \V s {\fxi;\exc,\smp;\aj}(\proj\RR)\bigr)_\sic^\van\,.
\]
\item For~$s\in\CC$, $\Rea s\in (0,1)$, 
\[
 \A_s^{0,\pm} \stackrel{\sim}{\longrightarrow} H_\Xi^{1,\mp}\bigl(\Gm; \V s {\fxi;\exc,\infty;\aj}(\proj\RR)\bigr)_\sic^\van\,.
\]
\end{enumerate}
The inverse maps are given by the suitable restrictions of the map~$\ucoc_s$ from~\eqref{def_ucoc}.
\end{prop}

\begin{proof}
We will show only~\eqref{casearb}. The proofs of the remaining statements are almost identical. By Theorem~\ref{thm:cohominter} it suffices to show that for all~$s\in\CC$, $\Rea s\in (0,1)$, $s\not=1/2$, the map
\[
 \cocu_s\colon\A_s \to H^1_\Xi\bigl(\Gm; \V s {\fxi;\exc;\aj}(\proj\RR) \bigr)_\sic^\van
\]
is anti-equivariant under the action of~$J$, that is 
\[
 \cocu_s\circ\mc J^F = -\mc J_s^H \circ \cocu_s\,.
\]
We recall from Section~\ref{sec:proof_cohominter} that $\cocu_s=\pr_\Xi\circ\coh_s$ with 
\[
\coh_s\colon \A_s \to H^1\bigl(F_\bullet^\tess; \V s \om(\proj\RR), \V s {\fxi;\exc;\aj}(\proj\RR)\bigr)_\sic^\van
\]
from Proposition~\ref{prop-cohAH} and 
\[
\pr_\Xi \colon H^1\bigl(F_\bullet^\tess; \V s \om(\proj\RR), \V s {\fxi;\exc;\aj}(\proj\RR)\bigr)_\sic^\van \to H^1_\Xi\bigl(\Gm;  \V s {\fxi;\exc;\aj}(\proj\RR) \bigr)_\sic^\van
\]
from Proposition~\ref{prop-coh-Tess-Xi}. Thus, to prove that $\cocu_s$ is anti-equivariant under the action of~$J$, it suffices to show that $\coh_s$ is anti-equivariant and $\pr_\Xi$ is equivariant. The latter follows straightforward from the fact that $\pr_\Xi$ is essentially only a restriction map. 

To show the anti-equi\-var\-iance of~$\coh_s$ as a map 
\[
 \coh_s\colon \A_s \to H^1\bigl(F_\bullet^\tess; \V s \om(\proj\RR), \V s {\fxi;\exc;\aj}(\proj\RR)\bigr)_\sic^\van\,,
\]
it suffices by Proposition~\ref{prop-cohAH} to show the anti-equi\-var\-iance of~$\coh_s$ as a map
\begin{equation}\label{eq:large_coh}
 \coh_s\colon \E_s^\Gm \to H^1\bigl(F_\bullet^{\tess,Y};\V s \om(\proj\RR) \bigr)
\end{equation}
and the equivariance of the two embeddings 
\[
 \A_s \hookrightarrow \E_s^\Gm
\]
and 
\[
 H^1\bigl(F_\bullet^\tess; \V s \om(\proj\RR), \V s {\fxi;\exc;\aj}(\proj\RR)\bigr)_\sic^\van \hookrightarrow H^1\bigl(F_\bullet^{\tess,Y}, \V s \om(\proj\RR)\bigr)\,.
\]
The equivariances of these embeddings are obvious.

We recall from Proposition~\ref{prop:map_step1} and Equation~\eqref{cu} that the map~$\coh_s$ (understood as in~\eqref{eq:large_coh}) is determined by assigning to~$u\in\E_s^\Gm$ the cocycle~$c_u\in Z^1\bigl(F_\bullet^{\tess,Y}; \V s \om(\proj\RR)\bigr)$ given by
\[
 c_u(x)(t) = \int_x \bigl\{ u, R(t;\,\cdot\,)^s\bigr\} \qquad (x\in\CC_1[X_1^{\tess,Y}],\ t\in\RR)\,.
\]
To establish the anti-equivariance of~$\coh_s$ we first note that for any two differentiable maps~$u,v\colon\uhp\to\CC$ we have
\[
 \{ u\circ J, v\circ J\} = -\{u,v\}\circ J\,,
\]
as follows directly from~\eqref{Gfacc2}. Further, for any~$t\in\RR$ we have
\[
 R(t;\,\cdot\,)\circ J = R(-t;\,\cdot\,)\,.
\]
For any~$u\in\E_s^\Gm$, $x\in \CC_1[X_1^{\tess,Y}]$ and $t\in\RR$ it follows that 
\begin{align*}
 c_{u\circ J}(x)(t) & =  \int_x \big\{ u\circ J, R(t;\,\cdot\,)^s \big\}  = \int_x \big\{ u\circ J, R(-t;\,\cdot\,)^s\circ J \big\}
 \\
 & = -\int_x \big\{ u, R(-t;\,\cdot\,)^s\big\}\circ J = -\int_{J\cdot x} \big\{ u, R(-t;\,\cdot\,)^s\big\} 
 \\
 & = -c_u(J\cdot x)(-t)
 \\
 & = -\left(\mc J_s^H c_u\right)(x)(t)\,.
\end{align*}
Thus
\[
 \coh_s\circ\mc J^F = -\mc J_s^H\circ\coh_s
\]
on~$\E_s^\Gm$. This completes the proof.
\end{proof}

\begin{proof}[Proof of Theorem~\ref{thm:main_parity}]
The statements of Theorem~\ref{thm:main_parity} follow immediately from combining Propositions~\ref{prop:FE_cohom_pm} and \ref{prop:funnel_cohom_pm}. 
\end{proof}

%% file: Mem-BP-partVII.tex

\setcounter{sectold}{\arabic{section}}
\markboth{VII. COMPLEMENTS AND OUTLOOK}{VII. COMPLEMENTS AND OUTLOOK}
\chapter{Complements and outlook}\label{part:complements}
\setcounter{section}{\arabic{sectold}}

The Selberg zeta function plays an important role in the study of the spectral theory of hyperbolic surfaces and the investigation of the relation between geometric and spectral properties of these surfaces. (See also Section~\ref{sec:SZF}.) In Section~\ref{sec:essential} we stated that the Fredholm determinant of the transfer operator family~$(\TO_s^\fast)_{s\in\CC}$ represents the Selberg zeta function~$Z_X$ for~$X=\Gm\backslash\uhp$, i.\,e., 
\[
 Z_X(s) = \det\left(1-\TO_s^\fast\right)\,.
\]
We used this identity to indicate the importance of the $1$-eigenfunctions of~$\TO_s^\fast$ even before their importance became clear by Theorem~\ref{thmB_new}. In Section~\ref{sec:fredholm} we will provide a short proof of this identity. We will then conclude this article with a brief outlook in Section~\ref{sec:outlook}.

\section{Fredholm determinant of the fast transfer operator}\label{sec:fredholm}
\markright{25. FREDHOLM DETERMINANT}

In this section we will show that the Fredholm determinant of~$\TO_s^\fast$ equals the Selberg zeta function~$Z_X$ by deducing it from the analogous property for the transfer operator~$\ftroP s$ from Section~\ref{sec:fastTOP} (originally from~\cite{Pohl_hecke_infinite}). To simplify all further discussions we use here a different space of domain for~$\TO_s^\fast$ than in Section~\ref{sec:fastconv} (see the discussion below). However, this slight twist still suffices the purpose of supporting the heuristic argument of the importance of the $1$-eigenfunctions of~$\TO_s^\fast$.

In~\cite{Pohl_hecke_infinite}, the transfer operator~$\ftroP s$ is considered as an operator on a certain Banach space 
\[
\mc B^P = \mc B_1 \oplus \mc B_2 \oplus \mc B_3 \oplus \mc B_4\,,
\]
where each of the spaces~$\mc B_j$, $j\in\{1,\ldots, 4\}$, is the Banach space (with respect to the supremums norm) of functions that are holomorphic on a certain open subset~$D_j$ in~$\proj\CC$ and that extend continuously to the boundary of~$D_j$. As operator on~$\mc B^P$, the transfer operator~$\ftroP s$ as given by~\eqref{eq:ftroP} is well-defined for~$\Rea s  \gg 1$, nuclear of order~$0$, and its Fredholm determinant equals the Selberg zeta function, thus
\begin{equation}\label{eq:szf_eq_fredholm}
 Z_X(s) = \det\left(1 - \ftroP s\right)\,.
\end{equation}
The map~$s\mapsto \ftroP s$ admits a meromorphic continuation to all of~$\CC$, which extends the validity of~\eqref{eq:szf_eq_fredholm} to all of~$\CC$ (meaning that also poles on both sides are equal) and, in addition, yields an alternative proof of the meromorphic continuability of the Selberg zeta function~$Z_X$. 

Due to the relation between~$\ftroP s$ and~$\TO_s^\fast$ discussed in Section~\ref{sec:fastTOP}, the properties of~$\ftroP s$ immediately yield that, for~$\Rea s \gg 1$, the transfer operator~$\TO_s^\fast$ is a self-map of~$\mc B^\fast\coloneqq \mc B_1\oplus\mc B_4$, as such nuclear of order~$0$, and the map~$s\mapsto \TO_s^\fast$ admits a meromorphic continuation to all of~$\CC$. 

The validity of~\eqref{eq:szf_eq_fredholm} with~$\TO_s^\fast$ in place of~$\ftroP s$ does not immediately carry over from~$\ftroP s$: For~$\Rea s \gg 1$, we have
\begin{align}
 \det\left(1-\ftroP s\right) &= \exp\left(-\sum_{n\in\NN} \frac1n \Tr \bigl(\ftroP s\bigr)^n \right) \label{eq:detexpP}
\intertext{and}
 \det\left(1-\TO_s^\fast\right) &= \exp\left(-\sum_{n\in\NN} \frac1n \Tr \bigl(\TO_s^\fast\bigr)^n \right) \label{eq:detexpfast}
\end{align}
Since the traces of the iteratives~$\bigl(\ftroP s\bigr)^n$ and~$\bigl(\TO_s^\fast\bigr)^n$ are not identical for~$n\in\NN$, we cannot immediately conclude the equality of~\eqref{eq:detexpP} and~\eqref{eq:detexpfast}. However, the proof of Proposition~\ref{prop:detsequal} will (implicitly) show that a resummation establishes equality.

\begin{prop}\label{prop:detsequal}
For all~$s\in\CC$ we have  $\det\left(1-\TO_s^\fast\right) = Z_X(s)$. 
\end{prop}

\begin{proof}
We take advantage of the results in~\cite{FP_szf} which show that certain properties of the discrete dynamical system underlying~$\TO_s^\fast$ immediately yield the claimed identity~$\det\left(1-\TO_s^\fast\right) = Z_X(s)$. The necessary properties are stated below in~\eqref{allhyp}--\eqref{goodcounting}. All of them can be checked using only properties of~$\TO_s^\fast$. 

For~$n\in\NN$ we set 
\[
 \Perd_n\coloneqq  \left\{ T^{a_1}ST^{a_2}S\cdots T^{a_n}S \setmid a_1,\ldots, a_n\in\ZZ\smallsetminus\{0\}\right\}
\]
and let 
\[
 \Per\coloneqq  \bigcup_{n\in\NN} \Perd_n\,.
\]
(In~\cite{FP_szf} these sets are denoted by~$\Pert_n$ and~$\Pert$.) We then have 
\[
 \Tr\bigl(\TO_s^\fast\bigr)^n = \sum_{g\in \Perd_n} \Tr \tau_s(g)\,.
\]
We recall from Section~\ref{sec:hypclass} that an element~$g\in\Gamma$ is called hyperbolic if its action on~$\overline\uhp$ has exactly two fixed points. We let $[\Gamma]_\hyp$ \index[symbols]{Gam@$[\Gamma]_\hyp$} denote the set of all $\Gamma$-conjugacy classes of hyperbolic elements in~$\Gamma$, and we denote the conjugacy class of~$g$ by~$[g]$. \index[symbols]{$[g]$}  We say that a hyperbolic element $g\in\Gamma$ is \emph{primitive}\index[defs]{primitive hyperbolic element}\index[defs]{element!primitive hyperbolic} if it is not a non-trivial power of a hyperbolic element in~$\Gamma$. In other words, $g$ is primitive if for all $g_0\in\Gamma$ and $m\in\NN$ with $g_0^m=g$ we have $m=1$.

By~\cite{FP_szf} it suffices to show the following four properties in order to complete the proof:
\begin{enumerate}[(a)]
\item\label{allhyp} All elements of~$\Per$ are hyperbolic.
\item\label{min_in} If $h_0\in\Gamma$ is primitive hyperbolic and $m\in\NN$ such that $h_0^m\in \Per$, then $h_0\in \Per$.
\item\label{repre} For each~$[g]\in [\Gamma]_\hyp$ there exists a unique~$n\in\NN$ such that $\Perd_n$ contains an element which represents~$[g]$. 
\item\label{goodcounting} For~$[g]\in [\Gamma]_\hyp$ we set $w(g) \coloneqq  n$ with~$n$ as in~\eqref{repre}, and we let $m(g)\in\NN$ be the unique number such that $g=g_0^m$ for a primitive hyperbolic element~$g_0\in\Gamma$. Then there are exactly~$w(g)/m(g)$ distinct elements in~$\Perd_{w(g)}$ that represent~$[g]$.
\end{enumerate}

In order to establish~\eqref{allhyp}--\eqref{goodcounting}, we denote, for all~$n\in\NN$, the set of elements acting on the diagonal of~$\bigl(\ftroP s\bigr)^n$ by~$\wt \Perd_n$, and let 
\[
 \wt \Per\coloneqq  \bigcup_{n\in\NN} \wt \Perd_n\,.
\]
By~\cite{Pohl_hecke_infinite}, the set~$\wt \Per$ satisfies \eqref{allhyp}-\eqref{goodcounting} (with $\wt \Per$ and $\wt \Perd_n$ instead of~$\Per$ and $\Perd_n$). (As soon as~\eqref{allhyp}--\eqref{goodcounting} are shown for~$\Per$ and hence the Fredholm determinants in~\eqref{eq:detexpP} and \eqref{eq:detexpfast} are both equal to the Selberg zeta function~$Z_X$ on~$\Rea s \gg 1$, it follows that each element of~$\Per$ appears in the calculation of the traces in the right hand sides of~\eqref{eq:detexpP} and \eqref{eq:detexpfast} with the same weight in both expressions.)

One easily shows that $\Per = \wt \Per$, and hence $\Per$ satisfies~\eqref{allhyp} and~\eqref{min_in}. It further follows that for each~$[g]\in[\Gamma]_\hyp$ there is at least one~$n\in\NN$ such that $\Perd_n$ contains a representative of~$[g]$. Since the presentation (see~\eqref{presentation_Gamma})
\[
 \Gm = \left\langle T, S \left\vert\ \ S^2 = I \right.\right\rangle
\]
of~$\Gm$ implies that $\Perd_a\cap \Perd_b=\emptyset$ for $a,b\in\NN, a\not=b$, there is a unique such~$n$. This shows~\eqref{repre}.

It remains to show that $\Per$ satisfies~\eqref{goodcounting}. To that end let $[g]\in [\Gamma]_h$, let $n=w(g)$ and pick a representative of~$[g]$ in~$\Perd_n$, say
\[
 T^{a_1}ST^{a_2}S\cdots T^{a_n}S\,.
\]
Let 
\begin{align*}
 h_1 &\coloneqq  T^{a_1}ST^{a_2}S\cdots T^{a_n}S\,,
 \\
 h_2 & \coloneqq  T^{a_2}S\cdots T^{a_n}ST^{a_1}S\,,
 \\
 & \vdots 
 \\ h_n & \coloneqq  T^{a_n}ST^{a_1}S\cdots T^{a_{n-1}}S\,.
\end{align*}
Then the set
\[
 \mc R \coloneqq \{ h_1,\ldots, h_n\}
\]
consists of representatives of~$[g]$ in~$\Per$, and contains all such. In what follows we calculate the cardinality of~$\mc R$. 

Let $h_0\in\Gamma$ be the (unique) primitive element such that $h_1=h_0^m$ for some (unique) number~$m\in\NN$. Then $h_0\in \Per$ and 
and $m=m(g)$. Necessarily,
\[
 h_0 = T^{a_1}ST^{a_2}S\cdots T^{a_p}S
\]
for some~$p\leq n$. In turn  
\[
 h_1 = h_0^m = \left(T^{a_1}ST^{a_2}S\cdots T^{a_p}S\right)^m\,.
\]
Uniqueness of presentation (see~\eqref{presentation_Gamma}) implies $n=pm$ and hence
\[
 \# \mc R = p = \frac{n}{m} = \frac{w(g)}{m(g)}.
\]
This shows~\eqref{goodcounting} and completes the proof.
\end{proof}

\section{Outlook}\label{sec:outlook}
\markright{26. OUTLOOK}

In this article we restricted all discussions to non-cofinite Hecke triangle groups: The notions of funnel forms, resonant funnel forms and cuspidal funnel forms, the cohomology spaces and the additional requirements on cocycle classes, and the isomorphisms between spaces of funnel forms, cohomology spaces and spaces of period functions are developed only for this class of Fuchsian groups. The decision to restrict the present article to non-cofinite Hecke triangle group was based on the wish to keep the exposition at a reasonable length and to provide a first (explicit) example instead of presenting a general, technically more involved treatment.

However, families of slow/fast transfer operators enjoying properties similar to those used here exist for a much wider class of Fuchsian groups. Moreover, the presentation here already indicates how to generalize all definitions, constructions and proofs to this wider class. We leave the details of this generalization for a future publication.